\documentclass[11pt]{unbthesis}

\usepackage{amsfonts, amsmath, amssymb, amsgen, amsthm, amscd,latexsym}
\usepackage{amsmath,mathrsfs}
\usepackage[all]{xy}

\def\Diff{\mathop{\rm Diff}\nolimits}
\def\End{\mathop{\rm End}\nolimits}

\def\Id{\mathop{\rm Id}\nolimits}
\def\Im{\mathop{\rm Im}\nolimits}

\def\Ad{\mathop{\rm Ad}\nolimits}
\def\ad{\mathop{\rm ad}\nolimits}
\def\det{\mathop{\rm det}\nolimits}

\def\Tr{\mathop{\rm Tr}\nolimits}

\def\Hom{\mathop{\rm Hom}\nolimits}
\def\Tot{\mathop{\rm Tot}\nolimits}

\def\Cb{{\mathbb C}}

\def\Nb{{\mathbb N}}
\def\Rb{{\mathbb R}}

\def\Zb{{\mathbb Z}}

\def\Ac{{\cal A}}

\def\Fc{{\cal F}}

\def\Hc{{\cal H}}

\def\Lc{{\cal L}}
\def\Mc{{\cal M}}

\def\Pc{{\cal P}}

\def\Uc{{\cal U}}
\def\Vc{{\cal V}}

\def\Lc{{\cal L}}

\def\Dc{{\cal D}}

\def\a{\alpha}
\def\b{\beta}
\def\d{\delta}
\def\D{\Delta}
\def\g{\gamma}

\def\G{\Gamma}

\def\om{\omega}
\def\Om{\Omega}
\def\s{\sigma}

\def\t{\theta}
\def\z{\zeta}
\def\ve{\varepsilon}
\def\vp{\varphi}
\def\nr{\natural}
\def\x{\xi}

\def\0b{\bf 0}

\def\fl{\forall}

\def\nb{\nabla}
\def\ot{\otimes}

\def\ra{\rightarrow}

\def\rt{\triangleright}

\def\lt{\triangleleft}
\def\cl{\blacktriangleright\hspace{-4pt} < }
\def\al{>\hspace{-4pt}\vartriangleleft}

\def\acl{\blacktriangleright\hspace{-4pt}\vartriangleleft }
\def\dcp{\vartriangleright\hspace{-4pt}\vartriangleleft }

\def\bD{\blacktriangledown}

\def\bi{\bowtie}
\def\hd{\overset{\ra}{\partial}}
\def \vd{\uparrow\hspace{-4pt}\partial}
\def\hs{\overset{\ra}{\sigma}}
\def \vs{\uparrow\hspace{-4pt}\sigma}
\def\hta{\overset{\ra}{\tau}}
\def \vta{\uparrow\hspace{-4pt}\tau}
\def\hb{\overset{\ra}{b}}
\def \vb{\uparrow\hspace{-4pt}b}

\def\hB{\overset{\ra}{B}}
\def \vB{\uparrow\hspace{-4pt}B}
\def\p{\partial}

\def\bi{\bowtie}
\def\hd{\overset{\ra}{\partial}}

\def \vd{\uparrow\hspace{-4pt}\partial}
\def\hs{\overset{\ra}{\sigma}}
\def \vs{\uparrow\hspace{-4pt}\sigma}
\def\hta{\overset{\ra}{\tau}}
\def \vta{\uparrow\hspace{-4pt}\tau}
\def\hb{\overset{\ra}{b}}
\def\hP{\overset{\ra}{\p}}
\def \vb{\uparrow\hspace{-4pt}b}
\def \vP{{\uparrow\hspace{-1pt}\p}}

\def\hB{\overset{\ra}{B}}
\def \vB{\uparrow\hspace{-4pt}B}
\def\p{\partial}

\def\0D{\Delta^{(0)}}
\def\1D{\Delta^{(1)}}
\def\Db{\blacktriangledown}

\def\wg{\wedge}

\def\td{\tilde}
\def\cop{{^{\rm cop}}}
\newcommand{\wbar}[1]{\overline{#1}}

\newcommand{\FD}{\mathfrak{D}}

\newcommand{\Fa}{\mathfrak{a}}
\newcommand{\Fg}{\mathfrak{g}}
\newcommand{\Fh}{\mathfrak{h}}

\newcommand{\Fm}{\mathfrak{m}}
\newcommand{\Fl}{\mathfrak{l}}

\newcommand{\Fn}{\mathfrak{n}}

\newcommand{\Fd}{\mathfrak{d}}

\newcommand{\Fs}{\mathfrak{s}}
\newcommand{\FZ}{\mathfrak{Z}}

\newtheorem{theorem}{Theorem}[section]
\newtheorem{remark}[theorem]{Remark}
\newtheorem{proposition}[theorem]{Proposition}
\newtheorem{lemma}[theorem]{Lemma}
\newtheorem{corollary}[theorem]{Corollary}

\newtheorem{example}[theorem]{Example}
\newtheorem{definition}[theorem]{Definition}

\def\build#1_#2^#3{\mathrel{
\mathop{\kern 0pt#1}\limits_{#2}^{#3}}}
\newcommand{\ps}[1]{~\hspace{-4pt}_{^{(#1)}}}
\newcommand{\pr}[1]{~\hspace{-4pt}_{_{\{#1\}}}}
\newcommand{\ns}[1]{~\hspace{-4pt}_{_{{<#1>}}}}
\newcommand{\sns}[1]{~\hspace{-4pt}_{_{{<\overline{#1}>}}}}
\newcommand{\nsb}[1]{~\hspace{-4pt}_{^{[#1]}}}

\newcommand{\snsb}[1]{~\hspace{-4pt}_{_{{[\overline{#1}]}}}}

\def\odots{\ot\cdots\ot}
\def\wdots{\wedge\dots\wedge}
\newcommand{\nm}[1]{{\mid}#1{\mid}}
\def\one{{\bf 1}}
\def\bfR{{\bf R}}

\numberwithin{equation}{section}

 \newcommand{\ie}{{\it i.e.\/}\ }

\def\a{\alpha}
\def\b{\beta}
\def\d{\delta}

\def\g{\gamma}

\def\om{\omega}
\def\s{\sigma}
\def\t{\theta}
\def\ve{\varepsilon}

\def\vp{\varphi}

\def\z{\zeta}

\def\D{\Delta}
\def\G{\Gamma}

\def\Om{\Omega}

\def\dtone{\left.\frac{d}{dt_1}\right|_{_{t_1=0}}}
\def\dtn{\left.\frac{d}{dt_n}\right|_{_{t_n=0}}}
\def\dt{\left.\frac{d}{dt}\right|_{_{t=0}}}

\def\ds{\left.\frac{d}{ds}\right|_{_{s=0}}}

\newcommand{\mdt}[1]{\left.\frac{d}{dt_{#1}}\right|_{_{t_{#1}=0}}}

\def\fl{\forall}

\def\nb{\nabla}

\def\ot{\otimes}
\def\part{\partial}

\def\wdg{\wedge}

\def\ra{\rightarrow}

\def\text{\hbox}

\def\fl{\forall}

\def\nb{\nabla}
\def\ot{\otimes}

\def\ra{\rightarrow}

\def\wdg{\wedge}

\def\Ad{\mathop{\rm Ad}\nolimits}

\def\Diff{\mathop{\rm Diff}\nolimits}

\def\End{\mathop{\rm End}\nolimits}

\def\Hom{\mathop{\rm Hom}\nolimits}

\def\Id{\mathop{\rm Id}\nolimits}
\def\exp{\mathop{\rm exp}\nolimits}

\def\build#1_#2^#3{\mathrel{
\mathop{\kern 0pt#1}\limits_{#2}^{#3}}}

\numberwithin{equation}{section}
\begin{document}

\title{Hopf-cyclic cohomology of bicrossed product Hopf algebras}
\author{Serkan S\"utl\"u}
\predegree{Master of Science, Bo\~gazi\c{c}i University, 2007}
\degree{Doctor of Philosophy}
\gau{Mathematics and Statistics}
\supervisor{Bahram Rangipour, Ph.D.,\, Mathematics \& Statistics}
\examboard{David Bremner,\, Ph.D.,\, Computer Science, Chair \\ & Colin Ingalls, Ph.D.,\, Mathematics \& Statistics,\\  & Barry Monson,\, Ph.D.,\, Mathematics \& Statistics}
\externalexam{Tomasz Brzezinski, Ph.D.,\, Mathematics Department, \\ & University of Wales, Swansea, UK}

\date{December, 2012}
\copyrightyear{2013}
\setlength\parindent{0pt}

\unbtitlepage

\tableofcontents

\chapter*{Dedication}
\addcontentsline{toc}{chapter}{Dedication}

This dissertation is dedicated to my mathematics teachers
\begin{itemize}
\item [] Mehmet BUL who taught me mathematics for the first time,
\item [] Prof. Hamdi ARIKAN who inspired me to pursue a career in mathematics, and
\item [] Prof. Cemal KO\c{C}... I wish I could have spent more time with him.
\end{itemize}

\chapter*{Abstract}
\addcontentsline{toc}{chapter}{Abstract}

In this dissertation we study the coefficients spaces (SAYD modules) of Hopf-cyclic cohomology theory over a certain family of bicrossed product Hopf algebras, and we compute the Hopf-cyclic cohomology of such Hopf algebras with coefficients.

\medskip

We associate a Hopf algebra, what we call a Lie-Hopf algebra, to any matched pair of Lie groups, Lie algebras and affine algebraic groups via the semi-dualization procedure of Majid. We then identify the SAYD modules over Lie-Hopf algebras with the representations and corepresentations of the total Lie group, Lie algebra or the affine algebraic group of the matched pair.

\medskip

First we classify the SAYD modules that correspond only to the representations of a total Lie group (algebra). We call them induced SAYD modules. We then generalize this identification, focusing on the matched pair of Lie algebras. We establish a one-to-one correspondence between the SAYD modules over the Lie-Hopf algebra associated to a matched pair of Lie algebras and certain SAYD modules over the total Lie algebra.

\medskip

Once the SAYD modules are associated to the representations and the corepresentations of Lie algebras, nontrivial examples can be constructed. This way, we illustrate a highly nontrivial 4-dimensional SAYD module over the Schwarzian Hopf algebra $\Hc_{\rm 1S}$.

\medskip

In addition, we discuss the periodic cyclic cohomology of Lie-Hopf algebras with nontrivial SAYD coefficients. We obtain a general van Est isomorphism identifying the periodic cyclic cohomology of a Lie-Hopf algebra with the (relative) Lie algebra cohomology of the corresponding total Lie algebra.

\chapter*{Acknowledgements}
\addcontentsline{toc}{chapter}{Acknowledgments}

I would like to acknowledge my indebtedness and render my warmest thanks to my supervisor, Bahram Rangipour. I have been extremely lucky to have a supervisor who cared so much about me. None of this work would have been possible without his guidance and encouragement. I can only say a proper thanks to him through my future work.

\medskip

It is not an easy task to review a thesis. I am grateful for the thoughtful and detailed comments of the members of the examining board: Professor Barry Monson, Professor Colin Ingalls, Professor David Bremner and Professor Tomasz Brzezinski.

\medskip

Completing this work would have been all the more difficult were it not for the support and friendship
provided by the secretaries of the Department of Mathematics \& Statistics. I am grateful to them for assisting me in many different ways.

\medskip

I thank Dr. Atabey Kaygun, Dr. \.{I}smail G\"ulo\~glu, Dr. M\"uge Kanuni Er and Dr. Song\"ul Esin for their endless support throughout my Ph.D. study. I also thank Ali Gharouni, Aydin Sarraf, Mohammad Abu Zaytoon, Mohammad Hassanzadeh and Reza Ghamarshoushtari  for their generous friendship and support all along.

\medskip

I would like to thank T\"UB\.ITAK (The Scientific and Technological Research Council of Turkey) for supporting me in the graduate school.

\medskip

Finally I would like to thank Nimet - Emre S\"utl\"u, Emel - M\"usl\"um - Gizem \c Cetinta\c s, and Eda-Noyan-Ya\~g\i z Bing\"ol, whose encouragement and support has never faltered. Thank you !

\medskip

And Seda, my wonderful wife... Her patience, support, and love have been the reasons why I did not turn around and run away. She was always there for me whenever I needed help. I cannot imagine someone more special in my life.

\chapter*{Table of Notations}
\addcontentsline{toc}{chapter}{Table of Notations}

\begin{table}[!h]
\begin{tabular}{l c l}
$A$ && \text{an algebra}\\
$a,a',b,b'$ && \text{elements of $A$}\\
$\mu$ && \text{multiplication}\\
$\eta$ && \text{unit}\\
$C$ && \text{a coalgebra}\\
$c$ && \text{a typical element in $C$}\\
$\D$ && \text{comultiplication}\\
$\ve$ && \text{counit}\\
$H$ && \text{a Hopf algebra}\\
$h,g$ && \text{elements of $H$}\\
$\Fc$&&\text{a commutative Hopf algebra}\\
$f,f',g,g'$&&\text{elements of $\Fc$}\\
$\Uc,\Vc$&&\text{a cocommutative Hopf algebra}\\
$X,u$&&\text{elements in $\Uc$}\\
$\xi,u'$&&\text{elements in $\Vc$}\\
$V$&&\text{a vector space, a module, a comodule,}\\
&&\text{or a SAYD module over a Hopf algebra}\\
$\Db,\nb$&&\text{coactions over a comodule}\\
$\Db_V$&&\text{coaction over $V$}\\
$\Db_{\Fg}$&&\text{a $\Fg$-coaction}\\
$\rt$&&\text{a left action}\\
$\lt$&&\text{a right action}\\
$\ast, \bullet, \cdot$&&\text{actions}\\
\end{tabular}
\end{table}
\begin{table}[!h]
\begin{tabular}{lcl}
$K,H,G,G_1,G_2$&&\text{(Lie) groups}\\
$\phi$&&\text{a typical element of the group $G$}\\
$\vp$&&\text{a typical element of the group $G_1$}\\
$\psi$&&\text{a typical element of the group $G_2$}\\
${\rm GL}(n)$&&\text{the general linear group}\\
${\rm GL}(n)^{\rm aff}$&&\text{the group of affine motions}\\
${\rm PGL}(n)$&&\text{the projective linear group}\\
$\Cb[G]$&&\text{group algebra}\\
$\Ad$&&\text{the adjoint representation of a Lie group}\\
$\Fg,\Fg_1,\Fg_2,\Fh,\Fs$&&\text{Lie algebras}\\
$X$&&\text{an element in $\Fg_1$}\\
$\xi$&&\text{an element in $\Fg_2$}\\
$s\ell(n)$&&\text{the Lie algebra of the special linear group}\\
$g\ell(n)$&&\text{the Lie algebra of the general linear group}\\
$g\ell(n)^{\rm aff}$&&\text{the Lie algebra of the group ${\rm GL}(n)^{\rm aff}$}\\
$pg\ell(n)$&&\text{the Lie algebra of the projective linear group}\\
$\ad$&&\text{the adjoint representation of a Lie algebra}\\
$\Lc_X$&&\text{action of a Lie algebra element $X$, Lie derivative with respect to $X$}\\
$\Rb P^n$ &&\text{real projective space}\\
$M$ &&\text{a (foliated) manifold}\\
$C^\infty_c(M)$ &&\text{the set of smooth functions on the manifold $M$ with compact support}\\
$\G$ && \text{holonomy pseudogroup of a foliation}\\
$\al$&&\text{left cross product for algebras}\\
$\Ac_\G$ &&\text{a crossed product algebra of the form $C^\infty_c(M)\rtimes \G$}\\
$\cl$&&\text{right cross coproduct for coalgebras}\\
$\acl$&&\text{bicrossed coproduct for Hopf algebras}\\
$\dcp$&&\text{double crossed coproduct for Hopf algebras or groups,}\\
&&\text{double crossed sum for Lie algebras}\\
$\Hc$&&\text{a bicrossed product Hopf algebra}\\
$\Fc\acl \Uc$&&\text{bicrossed product of the Hopf algebras $\Fc$ and $\Uc$}\\
$\Uc\dcp \Vc$&&\text{double crossed product of the Hopf algebras $\Uc$ and $\Vc$}\\
$G_1\dcp G_2$&&\text{double crossed product of the groups $G_1$ and $G_2$}\\
$\Fg_1\dcp \Fg_2$&&\text{double crossed sum of the Lie algebras $\Fg_1$ and $\Fg_2$}\\
$^C\Mc$&&\text{the category of left $C$-comodules}\\
\end{tabular}
\end{table}

\chapter{Introduction}

The basic idea of non-commutative differential geometry is to ``quantize" the calculus of differential forms via the operator theoretic generalization
\begin{equation}\label{aux-noncommutative-differential-geometry}
df = [F,f]
\end{equation}
of differential. In this setting $f$ is an element of an involutive algebra $\Ac$ represented in a Hilbert space $\Hc$ and $F$ is a self-adjoint operator on $\Hc$ that satisfies $F^2 = \Id$.

\medskip

The data required to develop a calculus based on \eqref{aux-noncommutative-differential-geometry} is given by the following property of the triple $(\Ac,\Hc,F)$.

\begin{definition}[\cite{Conn85}]
The triple $(\Ac,\Hc,F)$ is called an (even) $n$-summable Fredholm module if
\begin{enumerate}
\item $\Hc = \Hc^+ \oplus \Hc^-$ is a $\Zb/2$-graded Hilbert space with grading operator $\g$ such that $\g(h) =(-1)^{{\rm deg}\; h}h$ for any homogeneous $h \in\Hc$.
    \item There is a graded homomorphism $\pi:\Ac \to \mathscr{L}(\Hc)$, the algebra of bounded operators on $\Hc$, in other words $\Hc$ is a $\Zb/2$-graded $\Ac$-module.
        \item $F \in \mathscr{L}(\Hc)$, $F^2 =\Id$, $F\g = -\g F$ and for any homogeneous $a \in \Ac$, $Fa - (-1)^{{\rm deg}\; a}aF \in \mathscr{L}^n(\Hc)$, where $\mathscr{L}^n(\Hc)$ is the ideal of $n$th Schatten class operators in $\mathscr{L}(\Hc)$.
\end{enumerate}
\end{definition}

When the algebra $\Ac$ is the algebra $C^\infty(M)$ of smooth functions on a manifold $M$, the basic examples of triples $(\Ac,\Hc,F)$ arises from the elliptic operators on $M$. However, there are examples of $n$-summable Fredholm modules so that $\Ac$ is not commutative.

\medskip

To any $n$-summable Fredholm module there associates an $n$-dimensional cycle over the algebra $\Ac$, \ie a triple $(\Om,d,\int)$ consists of a graded algebra $\Om = \oplus_{j=0}^n\Om^j$, a graded derivation $d:\Om \to \Om$ of degree 1 that satisfies $d^2 = 0$, a closed graded trace $\int:\Om \to \Cb$ and finally an algebra map $\rho:\Ac \to  \Om^0$.

\medskip

Indeed, by definition for any $a \in \Ac$, $da = [F,a] \in \mathscr{L}^n(\Hc)$. For $j \in \Nb$, let $\Om^j$ be the linear span of the elements
\begin{equation}
(a^0 + \lambda\Id)da^1\ldots da^j
\end{equation}
in $\mathscr{L}^{n/j}(\Hc)$. Then the inclusion $\mathscr{L}^{n/j_1}(\Hc) \times \mathscr{L}^{n/j_2}(\Hc) \subset \mathscr{L}^{n/{j_1+j_2}}(\Hc)$ endows $\Om = \oplus_{j=0}^n\Om^j$ with a differential graded algebra structure, the differential being
\begin{equation}
d\om = [F,\om].
\end{equation}
Finally, the linear functional
\begin{equation}
\int:\Om^n \to \Cb, \qquad \int \om = {\rm Trace}(\g\om)
\end{equation}
is a closed linear trace.

\medskip

An $n$-dimensional cycle over $\Ac$ is essentially determined by the $(n+1)$-linear functional
\begin{equation}
\tau(a^0,\ldots,a^n) = \int\rho(a^0)d\rho(a^1)\ldots d\rho(a^n),
\end{equation}
called the character of the cycle. Then it is straightforward to derive a characterization of such functionals. A linear functional $\tau:\Ac^{\ot \;n+1} \to \Cb$ is the character of an $n$-dimensional cycle over an algebra $\Ac$ if and only if
\begin{enumerate}
\item $\tau(a^0,\ldots, a^n) = (-1)^n\tau(a^n,a^0,\ldots, a^{n-1})$,
\item $b\tau = 0$,
\end{enumerate}
where
\begin{align}
\begin{split}
&b\tau(a^0,\ldots, a^{n+1}) = \\
& \sum_{j=0}^n(-1)^j\tau(a^0, \ldots, a^ja^{j+1},\ldots,a^{n+1}) + (-1)^{n+1}\tau(a^{n+1}a^0,a^1,\ldots,a^n).
\end{split}
\end{align}
As a result, the pair $(C_\lambda(\Ac) = \oplus_{n \geq 0}C^n_\lambda(\Ac),b)$ form a chain complex, where $C^n_\lambda(\Ac)$ is the space of characters of $n$-summable Fredholm modules over $\Ac$. The cohomology groups of this complex, called the cyclic cohomology of $\Ac$, are denoted by $H_\lambda^\bullet(\Ac)$.

\medskip

From the inclusion $\mathscr{L}^n(\Hc) \subset \mathscr{L}^m(\Hc)$ when $n \leq m$, it follows that if $(\Ac,\Hc,F)$ is an $n$-summable Fredholm module, then it is also an $(n+2k)$-summable Fredholm module. Moreover, the $(n+2k)$-dimensional character $\tau_{n+2k} \in H^{n+2k}_\lambda(\Ac)$ of this $n$-summable Fredholm module can be obtained from the $n$-dimensional character $\tau_n \in H^n_\lambda(\Ac)$ as $\tau_{n+2k} = S^k\tau_n$. This defines the periodicity map
\begin{equation}
S:H^n_\lambda(\Ac) \to H^{n+2}_\lambda(\Ac)
\end{equation}
and hence the periodic cyclic cohomology
\begin{equation}
HP(\Ac) = \underset{\longrightarrow}{\rm lim}(H^\bullet_\lambda(\Ac),S) = \left(\underset{n\geq 0}{\oplus}H^n_\lambda(\Ac)\right) / \sim, \qquad \tau \sim S\tau.
\end{equation}

The cyclic cohomology theory is an analogue of the de Rham homology of currents, and it plays a similar role in noncommutative geometry to that of de Rham cohomology in differential topology. Indeed, if $\Ac = C^\infty(M)$ for a smooth compact manifold $M$, then focusing on the continuous cyclic cochains,
\begin{equation}
H_{\lambda, {\rm cont}}(C^\infty(M)) = H^{\rm dR}(M).
\end{equation}
Moreover, by the periodicity of the periodic cyclic cohomology groups, when $k \in \Nb$ is large enough the $2k$-dimensional character $\tau_{2k}$ of any finitely summable Fredholm module determines the same element in $HP(\Ac)$. Such an assignment defines a map from the set of finitely summable Fredholm modules over $\Ac$ to the periodic cyclic cohomology of $\Ac$.

\medskip

This assignment is independent of the homotopy class of finitely summable Fredholm modules, and hence yields the map
\begin{equation}
{\rm ch}_\ast: \left\{  \begin{array}{c}
             \text{homotopy classes of finitely summable} \\
             \text{Fredholm modules over}\;\Ac
           \end{array}
 \right\} \ra HP(\Ac),
\end{equation}
which is an analogue of the Chern character
\begin{equation}
{\rm ch}^\ast: K_\ast(M) \ra H^{\rm dR}(M)
\end{equation}
in K-homology \cite{BaumDoug82}, in the sense that for $\Ac =C^\infty(M)$ the diagram
$$
\xymatrix {
 \ar[d] {\left\{  \begin{array}{c}
             \text{homotopy classes of finitely summable} \\
             \text{Fredholm modules over}\;C^\infty(M)
           \end{array}
 \right\} } \ar[r] &  HP_{\rm cont}(C^\infty(M))   \ar[d] \\
 K_\ast(M)   \ar[r]_{{\rm ch}^\ast} &  H^{\rm dR}(M)
}
$$
is commutative.

\medskip

The most important application of the theory is that the index map can be captured via a natural pairing between the periodic cyclic cohomology and the K-theory. More precisely, if $(M,\mathscr{F})$ is a foliated manifold and $D$ is a transversally elliptic operator on $M$, that is an elliptic operator on the leaf space $M/\mathscr{F}$, then
\begin{equation}
{\rm Index}_D(E) = \langle {\rm ch}_\ast(D), {\rm ch}^\ast(E) \rangle, \quad \forall E \in K_\ast(\Ac_\Gamma),
\end{equation}
where $\Ac_\Gamma$ is the convolution algebra of the smooth \'etale groupoid associated to the holonomy pseudogroup $\Gamma$ of the foliation.

\medskip

As a result, it became an important problem to provide a general formula for a representative of the cocycle ${\rm ch}_\ast(D)$, which is in the periodic cyclic cohomology of the convolution algebra $\Ac_\Gamma$.

\medskip

Such a formula in local terms is accomplished in \cite{ConnMosc95} as a finite sum of the expressions of the form
\begin{equation} \label{aux-index}
{\int \!\!\!\!\!\! -}  a^0 [D, a^1]^{(k_1)} \ldots [D, a^n]^{(k_n)}  \, \vert D \vert^{-(n+2k_1+\ldots+2k_n)},
\end{equation}
in terms of residues (generalizing the Wodzicki-Guillemin-Manin residue and the Dixmier trace), where $[D, a]^{(k)}$ stands for the $k$th iterated commutator of $D^2$  with $[D,a]$.

\medskip

While on the one hand this formula is observed to reduce to the local form of the Atiyah-Singer index theorem in case $D$ is a Dirac operator on the manifold $M$, on the other hand it was realized that the explicit computation of the terms of the ``index cocycle'' ${\rm ch}_\ast(D)$ is exceedingly difficult even for foliations of codimension 1. To this end, an ``organizing principle'' was needed.

\medskip

For a foliation of codimension $n$, the above mentioned organizing principle is introduced in \cite{ConnMosc98} as the action of a Hopf algebra $\Hc_n$. The basic idea was to reduce the expressions \eqref{aux-index} to genuine integrals and hence to replace them with the expressions of the form
\begin{equation}
\tau_\Gamma(a^0h^1(a^1)\ldots h^n(a^n)),
\end{equation}
where $\tau_\Gamma:\Ac_\Gamma \to \Cb$ is the canonical trace on the algebra $\Ac_\Gamma$, and $h^1,\ldots, h^n:\Ac_\Gamma \to \Ac_\Gamma$ are the ``transverse" differential operators that form a Hopf algebra $\Hc_n$ depending only on the codimension of the foliation. This action of $\Hc_n$ on $\Ac_\Gamma$ preserves (up to a character) the canonical trace $\tau_\Gamma$ and hence gives rise to a characteristic map in cyclic cohomology whose domain is the cyclic cohomology for Hopf algebras and whose range includes the index cocycle.

\medskip

In \cite{ConnMosc98}, the cyclic cohomology theory for algebras is adapted to the Hopf algebras $\Hc_n$ and the cyclic cohomology of $\Hc_n$ is identified, via an explicit van Est type isomorphism, with the Gelfand-Fuchs cohomology of the Lie algebra of formal vector fields on $\Rb^n$. Moreover, the  cyclic cohomology of the universal enveloping algebra $U(\Fg)$ of a Lie algebra $\Fg$ is identified  with the Lie algebra homology of $\Fg$ with coefficients in the $\Fg$-module $\Cb$ with the module structure given by ${\rm Tr \circ ad}:\Fg \to \Cb$.

\medskip

The remarkable similarity between the passage from cyclic cohomology of algebras to cyclic cohomology of Hopf algebras and the passage from de Rham cohomology to Lie algebra cohomology via the invariant de Rham cohomology is used in \cite{KhalRang03} to introduce the invariant cyclic homology as a generalization of the cyclic cohomology of Hopf algebras and its dual theory \cite{KhalRang02}.

\medskip

The idea of an invariant differential complex is also used in \cite{HajaKhalRangSomm04-II} to provide finally a common denominator to the cyclic theories whose cyclicity is based on Hopf algebraic structures. This, in turn, led to the (co)cyclic complexes of (co)algebras associated to (co)module (co)algebra (co)actions of Hopf algebras on (co)algebras, and hence the name Hopf-cyclic cohomology and the notion of symmetry in noncommutative geometry.

\medskip

Another remarkable discovery made in \cite{HajaKhalRangSomm04-II} concerned the space of coefficients for Hopf-cyclic cohomology theories. These coefficient spaces are called stable-anti-Yetter-Drinfeld modules (SAYD modules in short). A SAYD module is a module and comodule over a Hopf algebra with the compatibility conditions between the action and the coaction obtained from a modification of the Yetter-Drinfeld compatibility conditions.

\medskip

First examples of SAYD modules are introduced and studied in \cite{HajaKhalRangSomm04-I}. It is also observed that the theory introduced in \cite{ConnMosc98} corresponds to the theory with 1-dimensional (trivial) coefficients. More precisely, in \cite{ConnMosc98} the coefficient space $\Cb$ was regarded as a module over $\Hc_n$ via a character $\d:\Hc_n \to \Cb$, and a comodule over $\Hc_n$ by a group-like element $\s \in \Hc_n$. Such a pair $(\d,\s)$ is called a modular pair in involution (MPI for short), and the corresponding SAYD module is denoted by $\,^\s\Cb_\d$.

\medskip

In \cite{MoscRang07}, a direct method is presented to compute the periodic Hopf-cyclic cohomology of Hopf algebras of transverse symmetry of codimension 1, in particular the Connes-Moscovici Hopf algebra $\Hc_1$ and its Schwarzian factor $\Hc_{\rm 1S}$. The method used in \cite{MoscRang07} is based on two computational devices. The first is a Cartan homotopy formula for the Hopf-cyclic cohomology of coalgebras with coefficients in SAYD modules. The second one is the construction of a bicocyclic module for the bicrossed product Hopf algebras. This allows one to compute Hopf-cyclic cohomology by means of an Eilenberg-Zilber theorem \cite{GetzJone93}.

\medskip

The key observation to associate a bicocyclic module to the Hopf algebras of transverse symmetry of codimension 1 is the bicrossed product structure of these Hopf algebras \cite{ConnMosc98,ConnMosc00,HadfMaji07}.

\medskip

The method presented in \cite{MoscRang07} was later upgraded in \cite{MoscRang09} to cover a wider class of Hopf algebras. More precisely, a Hopf algebra $\Hc_\Pi$ was associated to any infinite primitive pseudogroup $\Pi$ of Lie-Cartan \cite{Cart09}. This family of Hopf algebras includes the Connes-Moscovici Hopf algebra $\Hc_n$ by $\Pi = {\rm Diff}(\Rb^n)$. Along with the lines of \cite{SingSter65,Guil70}, the Hopf algebra $\Hc_\Pi$ associated to the pseudogroup $\Pi$ can be regarded as the ``quantum group" counterpart of the infinite dimensional primitive Lie algebra of  $\Pi$.

\medskip

For any Lie-Cartan type pseudogroup $\Pi$, due to a splitting of the group ${\rm Diff}(\Rb^n) \cap \Pi$, it was observed in \cite{MoscRang09} that each member $\Hc_\Pi$ of this class of Hopf algebras admits a bicrossed product decomposition, and hence a bicocyclic module. On the other hand, each Hopf algebra $\Hc_\Pi$ comes equipped with a modular character $\d_\Pi:\Hc_\Pi \to \Cb$ so that $(\d_\Pi,1)$ is a modular pair in involution, \ie $\Cb$ is a SAYD module over the Hopf algebra $\Hc_\Pi$ via the module structure determined by $\d_\Pi$ and the trivial comodule structure. As a result, it becomes possible to present a direct approach to the computation of the periodic Hopf-cyclic cohomology $HP(\Hc_\Pi,\Cb_{\d_\Pi})$.

\medskip

The main application of \cite{MoscRang09} is to focus on the Connes-Moscovici Hopf algebra $\Hc_n$ and to compute the explicit cocycle representatives for a basis of the relative periodic Hopf-cyclic cohomology of $\Hc_n$ relative to $U(g\ell(n))$, namely $HP(\Hc_n,U(g\ell(n)),\Cb_{\d_\Pi})$. These representatives correspond to universal Chern classes.

\medskip

By \cite{BottHaef72} it is known that the Gelfand-Fuchs cohomology is the most effective framework to describe characteristic classes of foliations. On the other hand, the Gelfand-Fuchs cohomology of the Lie algebra of formal vector fields on $\Rb^n$ is identified in \cite{ConnMosc98} with the Hopf-cyclic cohomology of the Hopf algebra $\Hc_n$ with coefficients in the 1 dimensional SAYD module $\Cb_\d$. It is also illustrated in \cite{ConnMosc} by an explicit computation that, in the codimension 1 case, the characteristic classes of transversally oriented foliations are captured in the range of the characteristic map constructed in \cite{ConnMosc98}.

\medskip

This correspondence revealed a link between this new cohomology theory and the classical theory of characteristic classes of foliations as well as index theory.

\medskip

That the entire transverse structure is captured by the action of the Hopf algebra $\Hc_n$ and the lack of explicit examples of higher dimensional SAYD modules led us to search for the representations and the corepresentations of bicrossed product Hopf algebras as well as to investigate the effect of the higher dimensional coefficient spaces on transverse geometry.

\medskip

Our perspective is much in the same spirit of \cite{MoscRang09} philosophically, yet completely different - more algebraic and more direct - in method. We associate a Hopf algebra to any matched pair of Lie groups, Lie algebras and affine algebraic groups via the semidualization \cite{Maji90,Majid-book} and any such Hopf algebra a priori admits a bicrossed product decomposition.

\medskip

The building blocks for this bicrossed product decomposition consist of two relatively naive Hopf algebras, one of which is a universal enveloping algebra  while the other is an algebra of representative functions. For the sake of brevity we call these Hopf algebras the ``Lie-Hopf algebras".

\medskip

The merit of this point of view manifests itself while studying the SAYD modules over Lie-Hopf algebras. It turns out that the SAYD modules can be identified with the representations and corepresentations of the total Lie group, Lie algebra or the affine algebraic group of the matched pair. The key tool we developed is the Lie algebra coaction.

\medskip

We first observe that any Lie-Hopf algebra comes equipped with a modular pair $(\d,\s)$ in involution. Nontriviality of the group-like element for general Lie-Hopf algebras points out the refinement of our approach even in this early stage.

\medskip

We then complete our classification of SAYD modules in two steps. First we classify the SAYD modules that correspond only to the representations of the total Lie group, Lie algebra or affine algebraic group. We call them ``induced" SAYD modules. In the second step, focusing on the matched pair of Lie algebras, we establish a one-to-one correspondence between the SAYD modules over the Lie-Hopf algebra associated to a matched pair of Lie algebras and SAYD modules over the total Lie algebra. This is the place where we use the notion of Lie algebra coaction.

\medskip

Once the SAYD modules are associated to the representations and the corepresentations of Lie algebras, it becomes easy to produce nontrivial examples. We illustrate a highly nontrivial 4-dimensional SAYD module over the Schwarzian Hopf algebra $\Hc_{\rm 1S}$. As for the Connes-Moscovici Hopf algebras $\Hc_n$ on the other hand, we conclude that the only SAYD module on $\Hc_n$ is the 1-dimensional $\Cb_\d$ introduced in \cite{ConnMosc98}.

\medskip

In addition, we discuss the periodic cyclic cohomology of Lie-Hopf algebras with nontrivial SAYD coefficients. In the case of induced SAYD coefficients, we upgrade the technique used in \cite{MoscRang09} to obtain a van Est type isomorphism identifying the periodic cyclic cohomology of a Lie-Hopf algebra with the (relative) Lie algebra cohomology.

\medskip

In the general case, however, the isomorphism that identifies the periodic cyclic cohomology of the Lie-Hopf algebra associated to a matched pair of Lie algebras with the (relative) Lie algebra cohomology of the total Lie algebra is on the level of $E_1$-terms of spectral sequences associated to natural filtrations on the corresponding coefficient spaces.

\medskip

Let us now give a brief outline.

\medskip

In order to be able to give as self-contained as possible an exposition, in Chapter \ref{chapter-preliminaries} we provide a rather detailed background material that will be needed in the sequel. Starting from the (co)actions of two Hopf algebras on each other, we recall the crossed product Hopf algebras. These are the types of Hopf algebras whose (co)representations we want to study under the name of SAYD modules. Then we recall the definition of such (co)representation spaces as the coefficient spaces for the Hopf-cyclic cohomology. We conclude this chapter with a summary of Lie algebra (co)homology.

\medskip

Chapter \ref{chapter-geometric-hopf-algebras} is about the bicrossed product Hopf algebras in which we are interested. We callthem Lie-Hopf algebras. They are built on a commutative Hopf algebra of the representative functions (on a Lie group, Lie algebra or an affine algebraic group) and the universal enveloping algebra of a Lie algebra. So we first recall these building blocks and then the construction of the Lie-Hopf algebras by the semidualization of a Lie group, Lie algebra and an affine algebraic group.

\medskip

Next, in Chapter \ref{chapter-hopf-cyclic-coefficients}, we deal with the SAYD modules over the Lie-Hopf algebras in terms of the (co)representations of the relevant Lie group, Lie algebra or the affine algebraic group that gives rise to the Lie-Hopf algebra under consideration. In the first part, we consider those which correspond to the representations (and trivial corepresentation) under the name of induced SAYD modules. Then in the second part we study the general case (in full generality). In order to illustrate our theory we construct a highly nontrivial 4-dimensional SAYD module over the Schwarzian quotient $\Hc_{\rm 1S}$ of the Connes-Moscovici Hopf algebra $\Hc_1$.

\medskip

In Chapter \ref{chapter-hopf-cyclic-cohomology} we study the Hopf-cyclic cohomology of Lie-Hopf algebras with coefficients in SAYD modules. At first we discuss the Hopf-cyclic cohomology of the commutative representative Hopf algebras. Then we continue with the Hopf-cyclic cohomology of Lie-Hopf algebras with coefficients in induced SAYD modules. Finally we deal with the Hopf-cyclic cohomology with SAYD coefficients (of nontrivial corepresentation) of a Lie-Hopf algebra that corresponds to the Lie algebra decomposition. As an example, we compute the periodic Hopf-cyclic cohomology of $\Hc_{\rm 1S}$ with coefficients.

\medskip

Finally Chapter \ref{chapter-future-research} is devoted to a discussion of future research towards a noncommutative analogue of the theory of characteristic classes of foliations.

\chapter{Preliminaries}\label{chapter-preliminaries}

In this chapter we provide background material that will be needed in the sequel. We first recall the bicrossed product and double crossed product constructions. Our terminology is along the lines of \cite{Majid-book}. Such Hopf algebras are our main objects. We want to study their (co)representations under the title of SAYD modules. Since SAYD modules serve as the coefficient spaces for the Hopf-cyclic cohomology, it is natural that we recall also this theory here. To this end we follow \cite{ConnMosc98,ConnMosc,HajaKhalRangSomm04-II,HajaKhalRangSomm04-I}. Finally we include a summary of Lie algebra (co)homology following \cite{ChevEile,HochSerr53,Knap} as it is naturally related to the Hopf-cyclic cohomology \cite{ConnMosc98}.

\medskip

All vector spaces and their tensor product are over $\Cb$, unless otherwise specified. An algebra is a triple $(A,\mu,\eta)$, where the structure maps $\mu:A \ot A \to A$ and $\eta: \Cb \to A$ denote the multiplication and the unit. A coalgebra is also a triple $(C,\D,\ve)$, where the structure maps $\D:C \to C \ot C$ and $\ve: C \to \Cb$ denote the comultiplication and the counit. A Hopf algebra is a six-tuple $(H,\mu,\eta,\D,\ve,S)$, where $(H,\mu,\eta)$ is an algebra, $(H,\D,\ve)$ is a coalgebra and $S:H \to H$ is the antipode. We assume the antipode to be invertible.

\medskip

We will use the Sweedler's notation \cite{Sweed-book}. We denote a comultiplication by $\D(c) = c\ps{1} \ot c\ps{2}$, a left coaction by $\nabla(v) = v\ns{-1} \ot v\ns{0}$ and a right coaction by $\nabla(v) = v\ns{0} \ot v\ns{1}$, summation suppressed. By coassociativity, we simply write
\begin{equation}
\D(c\ps{1}) \ot c\ps{2} = c\ps{1} \ot \D(c\ps{2}) = c\ps{1} \ot c\ps{2} \ot c\ps{3},
\end{equation}
as well as
\begin{equation}
\nabla^2(v) = v\ns{-1} \ot \nb(v\ns{0}) = v\ns{-2} \ot v\ns{-1} \ot v\ns{0}
\end{equation}
and
\begin{equation}
\nabla^2(v) = \nb(v\ns{0}) \ot v\ns{1} = v\ns{0} \ot v\ns{1} \ot v\ns{2}.
\end{equation}

More detailed information on coalgebras, Hopf algebras and their comodules can be found in \cite{BrzeWisb}.

\medskip

Unless stated otherwise, a Lie algebra $\Fg$ is always assumed to be finite dimensional with a basis
\begin{equation}
\Big\{X_1, \ldots, X_N\Big\}
\end{equation}
and a dual basis
\begin{equation}
\Big\{\t^1, \ldots, \t^N\Big\}.
\end{equation}
We denote the structure constants of the Lie algebra $\Fg$ by $C^k_{ij}$, \ie
\begin{equation}
[X_i,X_j]=\sum_kC_{ij}^k X_k\;.
\end{equation}

All Lie groups and affine algebraic groups are assumed to be complex and connected.

\medskip

If $A$ is an algebra and $I \subseteq A$ is an (two-sided) ideal, we simply write $I \leq A$.

\medskip

For an algebra $A$, by $A^{\rm op}$ we mean the opposite algebra, \ie the same vector space with the reverse order of multiplication. Similarly, for a coalgebra $C$, by $C\cop$ we mean to the cooposite coalgebra, \ie the same vector space with the reversed comultiplication.

\section{Crossed product Hopf algebras}

Starting from the Hopf algebra (co)actions on (co)algebras that provide the notion of symmetry in noncommutative geometry, in this section we will recall two crossed product constructions for Hopf algebras: the bicrossed product and the double crossed product constructions.

\subsection{Hopf algebra (co)actions on (co)algebras}

Let $V$ be a vector space and $A$ be an algebra. A left action of $A$ on $V$ is a linear map
\begin{equation}
\rt:A \ot V \to V, \quad a \ot v \mapsto a \rt v,
\end{equation}
such that
\begin{equation}
(ab) \rt v = a \rt (b \rt v), \qquad 1 \rt v = v
\end{equation}
for any $a,b \in A$ and any $v \in V$. In this case we say that $V$ is a left $A$-module, or equivalently, $V$ is a representation space of $A$.

\medskip

Let $V$ and $W$ be two left $H$-modules. A linear map $f:V \to W$ is called left $H$-linear, or a left $H$-module map, if \begin{equation}
f(h \rt v) = h \rt f(v).
\end{equation}

\begin{definition}\label{definition-module-algebra}
Let $A$ be an algebra and $H$ be a Hopf algebra. Then $A$ is called a left $H$-module algebra if it is a left $H$-module and
\begin{equation}
h \rt (ab) = (h\ps{1} \rt a)(h\ps{2} \rt b), \quad h \rt 1 = \ve(h)1.
\end{equation}
\end{definition}

\begin{remark}\rm{
If $A$ is a left $H$-module, then $A \ot A$ is also a left $H$-module via
\begin{equation}\label{aux-tensor-product-H-module}
h \rt (a \ot b) := h\ps{1} \rt a \ot h\ps{2} \rt b,
\end{equation}
for any $h \in H$, and any $a,b \in A$. On the other hand, the ground field $\Cb$ is a left $H$-module via counit, \ie for any $h \in H$, and $\a \in \Cb$,
\begin{equation}\label{aux-trivial-H-module}
h \rt \a := \ve(h)\a.
\end{equation}
As a result, $A$ is a left $H$-module algebra if and only if the algebra structure maps $\mu:A \ot A \to A$ and $\eta: \Cb \to A$ are $H$-linear.
}\end{remark}

\begin{example}\rm{
A Hopf algebra $H$ is a left $H$-module algebra by the left adjoint action, \ie
\begin{equation}
h \rt g = h\ps{1}gS(h\ps{2}),
\end{equation}
for any $h,g \in H$.
}\end{example}

\begin{proposition}[\cite{Majid-book}]
Let $A$ be an algebra and $H$ be a Hopf algebra. Then $A$ is a left $H$-module algebra if and only if $A \ot H$ is an algebra, which is called the left cross product algebra and is denoted by $A \al H$, with multiplication
\begin{equation}
(a \al h)(a' \al h') = a(h\ps{1} \rt a') \al h\ps{2}h'
\end{equation}
and unit $1 \ot 1 \in A \ot H$.
\end{proposition}

\begin{definition}\label{definition-module-coalgebra}
Let $C$ be a coalgebra and $H$ be a Hopf algebra. Then $C$ is called a left $H$-module coalgebra if it is a left $H$-module and
\begin{equation}
\D(h \rt c) = (h\ps{1} \rt c\ps{1}) \ot (h\ps{2} \rt c\ps{2}), \quad \ve(h \rt c) = \ve(h)\ve(c).
\end{equation}
\end{definition}

\begin{remark}\rm{
For a coalgebra $C$, we now have the left $H$-modules $C \ot C$ by \eqref{aux-tensor-product-H-module} and $\Cb$ by \eqref{aux-trivial-H-module}. As a result, $C$ is a left $H$-module coalgebra if and only if the coalgebra structure maps $\D:C \ot C \to C$ and $\ve: C \to \Cb$ are $H$-linear.
}\end{remark}

\begin{example}\rm{
A Hopf algebra $H$ is a left $H$-module coalgebra by its own multiplication, \ie
\begin{equation}
h \rt g = hg,
\end{equation}
for any $h,g \in H$.
}\end{example}

Let $V$ be a vector space and $C$ be a coalgebra. A right coaction of $C$ on $V$ is a linear map
\begin{equation}
\nabla:V \to V \ot C, \quad v \mapsto v\ns{0} \ot v\ns{1},
\end{equation}
such that
\begin{equation}
v\ns{0} \ot \D(v\ns{1}) = \nb(v\ns{0}) \ot v\ns{1}, \quad v=v\ns{0}\ve(v\ns{1})
\end{equation}
for any $v \in V$. In this case we say that $V$ is a right $C$-comodule, or equivalently, $V$ is a corepresentation space of $C$. A left $C$-comodule is defined similarly.

\medskip

A linear map $f:V \to W$ of right $C$-comodules is called right $C$-colinear, or equivalently a right $C$-comodule map, if
\begin{equation}
\nb(f(v)) = f(v\ns{0}) \ot v\ns{1},
\end{equation}
for any $v \in V$.

\begin{definition}\label{definition-comodule-algebra}
Let $H$ be a Hopf algebra and $A$ be an algebra. Then $A$ is called a right $H$-comodule algebra if it is a right $H$-comodule and
\begin{equation}
\nb(ab) = a\ns{0}b\ns{0} \ot a\ns{1}b\ns{1}, \quad \nb(1) = 1 \ot 1,
\end{equation}
for any $a,b \in A$.
\end{definition}

\begin{remark}\rm{
If $A$ is a right $H$-comodule, then $A \ot A$ is also a right $H$-comodule via
\begin{equation}\label{aux-tensor-product-H-comodule}
a \ot b \mapsto a\ns{0} \ot b\ns{0} \ot a\ns{1}b\ns{1},
\end{equation}
for any $a,b \in A$. On the other hand, the ground field $\Cb$ is a right $H$-comodule via
\begin{equation}\label{aux-trivial-H-comodule}
\a \mapsto \a \ot 1,
\end{equation}
for any $\a \in \Cb$. As a result, $A$ is a right $H$-comodule algebra if and only if the algebra structure maps $\mu:A \ot A \to A$ and $\eta: \Cb \to A$ are right $H$-colinear.
}\end{remark}

\begin{example}\rm{
A Hopf algebra $H$ is a right $H$-comodule algebra by its own comultiplication $\D:H \to H \ot H$.
}\end{example}

\begin{definition}\label{definition-comodule-coalgebra}
Let $H$ be a Hopf algebra and $C$ be a coalgebra. Then $C$ is called a right $H$-comodule coalgebra if it is a right $H$-comodule and
\begin{equation}\label{aux-eq-comodule-coalgebra}
c\ns{0}\ps{1} \ot c\ns{0}\ps{2} \ot c\ns{1} = c\ps{1}\ns{0} \ot c\ps{2}\ns{0} \ot c\ps{1}\ns{1}c\ps{2}\ns{1}, \quad \ve(c\ns{0})c\ns{1} = \ve(c),
\end{equation}
for any $c \in C$.
\end{definition}

\begin{remark}\rm{
For any right $H$-comodule $C$, the tensor product $C \ot C$ and the ground field $\Cb$ are also a right $H$-comodules via \eqref{aux-tensor-product-H-comodule} and \eqref{aux-trivial-H-comodule} respectively. As a result, $C$ is a right $H$-comodule coalgebra if and only if the coalgebra structure maps $\D:C \to C \ot C$ and $\ve: C \to \Cb$ are right $H$-colinear.
}\end{remark}

\begin{example}\label{example-comodule-coalgebra-adjoint}\rm{
A Hopf algebra $H$ is a right $H$-comodule coalgebra by the right adjoint coaction, \ie
\begin{equation}
\nb:H \to H \ot H, \quad h \mapsto h\ps{2} \ot S(h\ps{1})h\ps{3}
\end{equation}
for any $h \in H$.
}\end{example}

\begin{proposition}[\cite{Majid-book}]
Let $C$ be a coalgebra and $H$ be a Hopf algebra. Then $C$ is a right $H$-comodule coalgebra if and only if $H \ot C$ is a coalgebra, which is called a right cross coproduct coalgebra and is denoted by $H \cl C$, via
\begin{equation}
\Delta(h\cl c)= h\ps{1}\cl c\ps{1}\ns{0}\ot  h\ps{2}c\ps{1}\ns{1}\cl c\ps{2}, \quad \ve(h\cl c)=\ve(h)\ve(c).
\end{equation}
\end{proposition}

\subsection{Double crossed product Hopf algebras}

In this subsection we recall the double crossed product construction. A double crossed product Hopf algebra is built on the tensor product of two Hopf algebras and is uniquely determined by this factorization property. The data summarizing the double crossed product construction is called a ``mutual pair of Hopf algebras''.

\begin{definition}
A pair of Hopf algebras $(\Uc,\Vc)$ is called a mutual pair of Hopf algebras if $\Vc$ is a
right $\Uc$-module coalgebra, $\Uc$ is  a left  $\Vc$-module coalgebra and
\begin{align}\label{aux-mutual-pair-1}
&v\rt(uu')= (v\ps{1}\rt u\ps{1})((v\ps{2}\lt u\ps{2})\rt  u'),\quad  1\lt u=\ve(u),\\ \label{aux-mutual-pair-2}
&(vv')\lt u= (v\lt(v'\ps{1}\rt u\ps{1}))(v'\ps{2}\lt u\ps{2}),\quad
v\rt 1=\ve(v),\\\label{aux-mutual-pair-3}
& v\ps{1}\lt u\ps{1}\ot v\ps{2}\rt u\ps{2}= v\ps{2}\lt u\ps{2}\ot v\ps{1}\rt u\ps{1},
\end{align}
for any $u,u' \in \Uc$ and $v,v' \in \Vc$.
\end{definition}
Having a  mutual   pair of Hopf algebras, we construct a Hopf algebra on $\Uc \ot \Vc$, called the double crossed product Hopf algebra of the mutual pair $(\Uc,\Vc)$. This Hopf algebra is denoted by $\Uc\dcp \Vc$. As a coalgebra $\Uc\dcp\Vc$ is isomorphic to $\Uc\ot\Vc$, \ie it is equipped with the tensor product coalgebra structure. However, its algebra structure is defined by the rule
\begin{equation}
(u\dcp v)(u'\dcp v'):= u(v\ps{1}\rt u'\ps{1})\dcp (v\ps{2}\lt u'\ps{2})v'
\end{equation}
for any $u,u' \in \Uc$ and $v,v' \in \Vc$, where $1\dcp 1$ is the multiplicative unit. Finally, the antipode of $\Uc\dcp \Vc$ is given by
\begin{equation}
S(u\dcp v)= S(v\ps{1})\rt S(u\ps{1})\dcp S(v\ps{2})\lt S(u\ps{2}).
\end{equation}

It follows from the construction that $\Uc$ and $\Vc$ are Hopf subalgebras of the double crossed product $\Uc \dcp \Vc$. Moreover, the converse is also true.

\begin{theorem}[\cite{Majid-book}]\label{theorem-mutual-pair-universal-property}
Let $H$ be a Hopf algebra that factorizes into two Hopf subalgebras
\begin{equation}
\Uc \hookrightarrow H \hookleftarrow \Vc,
\end{equation}
and that
\begin{equation}
\Uc \ot \Vc \to H, \quad u \ot v \mapsto uv
\end{equation}
is an isomorphism of vector spaces. Then $(\Uc,\Vc)$ form a mutual pair of Hopf algebras and $H \cong \Uc \dcp \Vc$ as Hopf algebras.
\end{theorem}

An important example of the double crossed product construction is the quantum double $D(H)$ that we recall here for a finite dimensional Hopf algebra $H$. Later on we will consider the representations of $D(H)$.

\begin{example}\rm{
Let $H$ be a finite dimensional Hopf algebra. Then $(H^{\ast \rm op},H)$ is a mutual pair of Hopf algebras by the coadjoint actions, \ie
\begin{equation}
h \rt \phi = \phi\ps{2} (S(\phi\ps{1})\phi\ps{3})(h)
\end{equation}
and
\begin{equation}
\phi \rt h = h\ps{2} \phi(S(h\ps{1})h\ps{3})
\end{equation}
viewed as a right action of $H^{\ast \rm op}$. As a result, we have the double crossed product Hopf algebra $D(H) = H^{\ast \rm op} \dcp H$.
}\end{example}

The examples of the double crossed product construction of interest here follow from the group or Lie algebra decompositions described below. We first recall such decompositions in an abstract setting.

\begin{definition}\label{definition-matched-pair-of-groups}
A pair of groups $(G,H)$ is called a matched pair of groups if there are maps
\begin{equation}\label{aux-matched-pair-group-action-I}
H \times G \to G, \quad (h,g) \mapsto h \rt g
\end{equation}
and
\begin{equation}\label{aux-matched-pair-group-action-II}
H \times G \to H, \quad (h,g) \mapsto h \lt g
\end{equation}
such that
\begin{align}\label{aux-matched-pair-group-actions-I}
& (hh') \rt g = h \rt (h' \rt g), \quad 1 \rt g = g, \\\label{aux-matched-pair-group-actions-II}
& h \lt (gg') = (h \lt g) \lt g', \quad h \lt 1 = h, \\\label{aux-matched-pair-group-actions-III}
& h \rt (gg') = (h \rt g)((h \lt g) \rt g'), \quad h \rt 1 = 1, \\\label{aux-matched-pair-group-actions-IV}
& (hh') \lt g = (h \lt (h' \rt g))(h' \lt g), \quad 1 \lt g = 1,
\end{align}
for any $h,h' \in H$ and $g,g' \in G$.
\end{definition}

Given a matched pair of groups, we construct a double crossed product group $G \dcp H$, which is isomorphic to $G \times H$ as the underlying set, with multiplication
\begin{equation}
(g \dcp h)(g' \dcp h') = g(h \rt g') \dcp (h \lt g')h', \quad \forall g,g' \in G,\;h,h' \in H
\end{equation}
and unit $1 \dcp 1 \in G \dcp H$. In this case, $G$ and $H$ are both subgroups of $G \dcp H$. Conversely, if  a group $K$ has two subgroups
\begin{equation}
G \hookrightarrow K \hookleftarrow H,
\end{equation}
such that $K = G\times H$ as sets, then the pair $(G,H)$ of subgroups is a matched pair of groups and $K \cong G \dcp H$ as groups. In this case, the mutual actions of the subgroups are obtained by
\begin{equation}\label{aux-matched-pair-groups-action}
hg = (h \rt g)(h \lt g), \quad \forall g \in G,\; h \in H.
\end{equation}

\begin{remark}\label{remark-matched-pair-group-algebras}\rm{
The pair of groups $(G,H)$ is a matched pair of groups if and only if the pair $(\Cb [G], \Cb [H])$ of group algebras is a mutual pair of Hopf algebras.
}\end{remark}

Let us conclude this subsection by reviewing the analogue discussion for Lie algebras.

\begin{definition}\label{definition-matched-pair-of-Lie-algebras}
A pair of Lie algebras  $(\Fg_1, \Fg_2)$  is called a matched pair of Lie algebras if there are linear maps
\begin{equation}\label{aux-matched-pair-Lie-algebra-action-I}
\Fg_2 \ot \Fg_1 \ra \Fg_1, \quad \x \ot X  \mapsto \x\rt X
\end{equation}
and
\begin{equation}\label{aux-matched-pair-Lie-algebra-action-II}
\Fg_2 \ot \Fg_1\ra \Fg_2, \quad \x \ot X  \mapsto \x \lt X
\end{equation}
such that
\begin{align}
&[\z,\x]\rt X=\z\rt(\x\rt X)-\x\rt(\z\rt X),\\
 & \z\lt[X, Y]=(\z\lt X)\lt Y-(\z\lt Y)\lt X, \\
 &\z\rt[X, Y]=[\z\rt X, Y]+[X,\z\rt Y] +
(\z\lt X)\rt Y-(\z\lt Y)\rt X\\
&[\z,\x]\lt X=[\z\lt X,\x]+[\z,\x\lt X]+ \z\lt(\x\rt X)-\x\lt(\z\rt X),
\end{align}
for any $X,Y \in \Fg_1$ and any $\z,\x \in \Fg_2$.
\end{definition}
Given a matched pair of  Lie algebras $(\Fg_1,\Fg_2)$, we define the double crossed sum Lie algebra $\Fg_1\bowtie \Fg_2$ whose underlying vector space is $\Fg_1\oplus\Fg_2$ and whose Lie bracket is
\begin{equation}
[X\oplus\z,  Z\oplus\x]=([X, Z]+\z\rt Z-\x\rt X)\oplus ([\z,\x]+\z\lt Z-\x\lt X).
\end{equation}
It is immediate that both $\Fg_1$ and $\Fg_2$ are Lie subalgebras of $\Fg_1\bowtie\Fg_2$ via obvious inclusions. Conversely,  if for a Lie algebra $\Fg$ there are two Lie subalgebras $\Fg_1$ and $\Fg_2$ so that $\Fg=\Fg_1\oplus\Fg_2$ as vector  spaces, then $(\Fg_1, \Fg_2)$  forms a matched pair of Lie algebras and $\Fg\cong \Fg_1\bowtie \Fg_2$ as Lie algebras. In this case the actions of $\Fg_1$ on $\Fg_2$ and $\Fg_2$ on $\Fg_1$ are uniquely determined  by
\begin{equation}\label{aux-actions-matched-pair-Lie-algebras}
[\x,X]=\x\rt X+\x\lt X, \quad \forall \x\in \Fg_2\; X\in\Fg_1.
\end{equation}

\begin{remark}\label{remark-matched-pair-enveloping-algebras}\rm{
For a matched pair of Lie algebras $(\Fg_1,\Fg_2)$, the pair of universal enveloping algebras $(U(\Fg_1),U(\Fg_2))$  becomes a mutual pair of Hopf algebras. Moreover, $U(\Fg)$ and $ U(\Fg_1)\dcp U(\Fg_2)$ are isomorphic as Hopf algebras.
}\end{remark}

In terms of the inclusions
\begin{equation}
i_1:U(\Fg_1) \to U(\Fg_1 \dcp \Fg_2) \quad \mbox{ and } \quad i_2:U(\Fg_2) \to U(\Fg_1 \dcp \Fg_2),
\end{equation}
the isomorphism $U(\Fg) \cong U(\Fg_1)\dcp U(\Fg_2)$ is
\begin{equation}
\mu\circ (i_1 \ot i_2):U(\Fg_1) \dcp U(\Fg_2) \to U(\Fg).
\end{equation}
Here $\mu$ is the multiplication on $U(\Fg)$. In this case, there is a linear map
\begin{equation}
\Psi:U(\Fg_2) \dcp U(\Fg_1) \to U(\Fg_1) \dcp U(\Fg_2),
\end{equation}
satisfying
\begin{equation}
\mu \circ (i_2 \ot i_1) = \mu \circ (i_1 \ot i_2) \circ \Psi\,.
\end{equation}
The mutual actions of $U(\Fg_1)$ and $U(\Fg_2)$ are then defined as
\begin{equation}\label{aux-mutual-actions-U(g1)-U(g2)}
\rt := (\Id_{U(\Fg_2)} \ot \ve) \circ \Psi, \qquad \lt := (\ve \ot \Id_{U(\Fg_1)}) \circ \Psi\,.
\end{equation}

\subsection{Bicrossed product Hopf algebras}

In this subsection we recall the bicrossed product construction. Similar to the double crossed product construction, a bicrossed product Hopf algebra also admits the tensor product of two Hopf algebras as the underlying vector space. The data summarizing this construction is called the matched pair of Hopf algebras.

\begin{definition}
A pair of Hopf algebras $(\Fc,\Uc)$ is called a matched pair of Hopf algebras if $\Fc$ is a left $\Uc$-module algebra, $\Uc$ is a right $\Fc$-comodule coalgebra and
\begin{align}\label{aux-matched-pair-1}
&\ve(u\rt f)=\ve(u)\ve(f), \\  \label{aux-matched-pair-2}
&\Delta(u\rt f)= u\ps{1}\ns{0} \rt f\ps{1}\ot u\ps{1}\ns{1}(u\ps{2}\rt f\ps{2}), \\  \label{aux-matched-pair-3}
 &\nb(1)=1\ot
1, \\ \label{aux-matched-pair-4}
&\nb(uv)= u\ps{1}\ns{0} v\ns{0}\ot u\ps{1}\ns{1}(u\ps{2}\rt v\ns{1}),\\  \label{aux-matched-pair-5} &
 u\ps{2}\ns{0}\ot (u\ps{1}\rt
f) u\ps{2}\ns{1}= u\ps{1}\ns{0}\ot u\ps{1}\ns{1}(u\ps{2}\rt f),
\end{align}
for any $u \in \Uc$ and $f \in \Fc$.
\end{definition}

We then form a new Hopf algebra $\Fc\acl \Uc$, called the  bicrossed product of the matched pair  $(\Fc , \Uc)$. It has $\Fc\cl \Uc$ as the underlying coalgebra, and $\Fc\al \Uc$ as the underlying algebra. The antipode is defined by
\begin{equation}\label{aux-antipode-bicrossed-product}
S(f\acl u)=(1\acl S(u\ns{0}))(S(fu\ns{1})\acl 1) , \qquad f \in \Fc , \, u \in \Uc.
\end{equation}

\begin{example}\rm{
Let $H$ be a Hopf algebra and $H^{\rm op}$ be the Hopf algebra with the opposite algebra structure. Then $H$ is a left $H^{\rm op}$-module algebra by
\begin{equation}
h^{\rm op} \rt g := S(h\ps{1})gh\ps{2},
\end{equation}
and $H$ is a $H^{\rm op}$-comodule coalgebra by the right adjoint coaction of Example \ref{example-comodule-coalgebra-adjoint}. Moreover, $(H^{\rm op},H)$ forms a matched pair of Hopf algebras and hence the bicrossed product Hopf algebra $H^{\rm op} \acl H$.
}\end{example}

The bicrossed product Hopf algebras of concern here follow from the Lie group, affine algebraic group or Lie algebra decompositions via semidualization. We conclude this subsection by mentioning this procedure briefly.

\begin{definition}
A pair of Hopf algebras $(H,K)$ are called a dual pair of Hopf algebras if there is a linear map
\begin{equation}\label{aux-dual-pairing-Hopf-algebras}
\langle,\rangle:H \ot K \to \Cb, \quad h \ot k \mapsto \langle h, k\rangle
\end{equation}
such that
\begin{align}
&\langle h, kk'\rangle =  \langle h\ps{1}, k\rangle\langle h\ps{2}, k'\rangle,\quad \langle h, 1\rangle=\ve(h),\\
&\langle hh', k\rangle=\langle h, k\ps{1}\rangle\langle h', k\ps{2}\rangle,\quad \langle 1, k\rangle=\ve(k), \\
& \langle h, S(k)\rangle = \langle S(h), k\rangle,
\end{align}
for any $h,h' \in H$ and any $k,k' \in K$.
\end{definition}

If the pairing \eqref{aux-dual-pairing-Hopf-algebras} is non-degenerate, \ie
\begin{itemize}
\item [(i).] $\langle h, k\rangle = 0, \forall k \in K \Rightarrow h = 0$,
\item [(ii).] $\langle h, k\rangle = 0, \forall h \in H \Rightarrow k = 0$,
\end{itemize}
then the pair $(H,K)$ is called a non-degenerate dual pair of Hopf algebras.

\medskip

For a finite dimensional Hopf algebra $H$ and its algebraic dual $H^\ast$, the pair $(H,H^\ast)$ is a non-degenerate dual pair of Hopf algebras. In the infinite dimensional case, the dual Hopf algebra of $H$ is defined to be
\begin{equation}
H^\circ = \Big\{f\in H^\ast \mid \exists \; I\leq \ker f \text{ such that } \; \dim((\ker f)/I) < \infty  \Big\}.
\end{equation}
However, the pairing between $H$ and $H^\circ$ is not necessarily non-degenerate, see for instance \cite[Section 2.2]{Abe-book}.

\medskip

We use the duality for the following result.

\begin{proposition}
Let $(H,K)$ be a non-degenerate dual pair of Hopf algebras. Then for a Hopf algebra $L$, the pair $(L,H)$ is a mutual pair Hopf algebras if and only if $(K,L)$ is a matched pair of Hopf algebras.
\end{proposition}

\section{Hopf-cyclic cohomology}

As our main interest is the Hopf cyclic cohomology of certain bicrossed product Hopf algebras with coefficients, in this section we recall this theory briefly. We first outline the theory of coefficients by recalling the modular pair in involutions (MPI) as well as the stable anti-Yetter-Drinfeld (SAYD) modules from \cite{ConnMosc,ConnMosc00,HajaKhalRangSomm04-I}. We then summarize the cyclic cohomology theory for Hopf algebras. More explicitly, to any of the four Hopf symmetries described in definitions \ref{definition-module-algebra}, \ref{definition-module-coalgebra}, \ref{definition-comodule-algebra} and \ref{definition-comodule-coalgebra}, there corresponds a Hopf-cyclic complex that computes the Hopf-cyclic (co)homology of a (co)algebra with coefficients in a SAYD module. In this subsection we recall from \cite{ConnMosc98,ConnMosc00,HajaKhalRangSomm04-II,MoscRang09} the Hopf-cyclic cohomology of an $H$-module coalgebra with coefficients in a SAYD module over a Hopf algebra $H$.

\subsection{Hopf-cyclic coefficients}

Let $H$ be a Hopf algebra. By definition, a character $\d: H\ra \Cb$ is an algebra map. On the other hand, a group-like element $\s\in H$ is the dual object of the character, \ie $\D(\s)=\s\ot \s$ and $\ve(\s)=1$.

\begin{definition}
Let $H$ be a Hopf algebra, $\d: H\ra \Cb$ be an algebra map and $\s\in H$ be a group-like element. The pair $(\d,\s)$ is called a modular pair in involution if
\begin{equation}
\d(\s)=1, \quad \text{and}\quad  S_\d^2=\Ad_\s,
\end{equation}
where ${\rm Ad}_\s(h)= \s h\s^{-1}$, for any $h \in H$ and  $S_\d$ is defined by
\begin{equation}\label{aux-twisted-antipode}
S_\d(h)=\d(h\ps{1})S(h\ps{2}), \quad h \in H.
\end{equation}
\end{definition}

Let us next recall the definition of a right-left stable-anti-Yetter-Drinfeld module over a Hopf algebra.

\begin{definition}
Let $H$ be a Hopf algebra and $V$ be a vector space. Then $V$ is called a right-left anti-Yetter-Drinfeld module over $H$ if it is a right $H$-module, a left $H$-comodule via $\Db:V \to H \ot V$, and
\begin{equation}\label{aux-SAYD-condition}
\Db(v\cdot h)= S(h\ps{3})v\ns{-1}h\ps{1}\ot v\ns{0}\cdot h\ps{2},
\end{equation}
for any $v\in V$ and $h\in H$. Moreover, $V$ is called stable if
\begin{equation}
v\ns{0}\cdot v\ns{-1}=v,
\end{equation}
\end{definition}

Replacing $S$ by $S^{-1}$ we arrive at the definition of Yetter-Drinfeld module (YD module in short). For the sake of completeness we include the definition below.

\begin{definition}
Let $H$ be a Hopf algebra and $V$ be a left $H$-module and right $H$-comodule via $\Db(v)= v\ns{-1}\ot v\ns{0}$. Then $V$ is called a left-right Yetter-Drinfeld module (YD module for short) over $H$  if
\begin{equation}\label{aux-left-right-YD-module}
(h\ps{2}\cdot v)\ns{0}\ot (h\ps{2}\cdot v)\ns{1}h\ps{1}= h\ps{1}\cdot  v\ns{0}\ot h\ps{2}v\ns{1},
\end{equation}
or equivalently
\begin{equation}\label{aux-left-right-YD-module-equivalent-condition}
\Db(h\cdot v)= h\ps{2}\cdot v\ns{0}\ot  h\ps{3}v\ns{1} S^{-1}(h\ps{1}).
\end{equation}
Similarly, $V$ is called a right-left Yetter-Drinfeld module over $V$ if
\begin{equation}\label{aux-right-left-YD-module}
h\ps{2}(v\cdot h\ps{1})\ns{-1}\ot (v\cdot h\ps{1})\ns{0}= v\ns{-1}h\ps{1}\ot v\ns{0}\cdot h\ps{2},
\end{equation}
equivalently
\begin{equation}\label{aux-right-left-YD-module-equivalent-condition}
\Db(v\cdot h)= S^{-1}(h\ps{3})v\ns{-1} h\ps{1}\ot v\ns{0}\cdot h\ps{2}.
\end{equation}
\end{definition}

Finally we recall \cite[Lemma 2.2]{HajaKhalRangSomm04-I} which reveals the link between the notions of MPI and SAYD.

\begin{lemma}[\cite{HajaKhalRangSomm04-I}]
Let $H$ be a Hopf algebra, $\d: H\ra \Cb$ be a character and $\s\in H$ be a group-like element. Then $\Cb$ is a right-left SAYD module over the Hopf algebra $H$ via $\d$ and $\s$ if and only if $(\d,\s)$ is an MPI.
\end{lemma}

\subsection{Hopf-cyclic complex}

Let $V$ be a right-left SAYD module over a Hopf algebra $H$. Then we have the graded space
\begin{equation}
C(H,V) := \bigoplus_{q\geq 0} C^q(H,V), \quad C^q(H,V):= V\ot H^{\ot q}
\end{equation}
with the coface operators
\begin{align}
&\p_i: C^q(H,V)\ra C^{q+1}(H,V), \quad 0\le i\le q+1\\
&\p_0(v\ot h^1\odots h^q)=v\ot 1\ot h^1\odots h^q,\\
&\p_i(v\ot h^1\odots h^q)=v\ot h^1\odots h^i\ps{1}\ot h^i\ps{2}\odots h^q, \quad 1\le i\le q \\
&\p_{q+1}(v\ot h^1\odots h^q)=v\ns{0}\ot h^1\odots h^q\ot v\ns{-1},
\end{align}
the codegeneracy operators
\begin{align}
&\s_j: C^q(H,V)\ra C^{q-1}(H,V), \quad 0\le j\le q-1 \\
&\s_j (v\ot h^1\odots h^q)= (v\ot h^1\odots \ve(h^{j+1})\odots h^q),
\end{align}
and the cyclic operator
\begin{align}
&\tau: C^q(H,V)\ra C^q(H,V),\\
&\tau(v\ot h^1\odots h^q)=v\ns{0}\cdot h^1\ps{1}\ot S(h^1\ps{2})\cdot(h^2\odots h^q\ot v\ns{-1}),
\end{align}
where $H$ acts on $H^{\ot q}$ diagonally, see \eqref{aux-tensor-product-H-module}.

\medskip

The graded space $C(H,V)$ endowed with the above operators is then a cocyclic module. This means that $\p_i,$ $\s_j$ and $\tau$ satisfy
\begin{equation}
\p_j  \p_i = \p_i  \p_{j-1}, \, \, i < j  , \qquad \s_j \s_i = \s_i
\s_{j+1},  \, \,  i \leq j
\end{equation}
\begin{equation}
\s_j  \p_i = \left\{ \begin{matrix} \p_i  \s_{j-1} \hfill &i < j
\hfill \cr 1_n \hfill &\hbox{if} \ i=j \ \hbox{or} \ i = j+1 \cr
\p_{i-1}  \s_j \hfill &i > j+1  ;  \hfill \cr
\end{matrix} \right.
\end{equation}

\begin{eqnarray}
\tau_n  \p_i  = \p_{i-1}  \tau_{n-1} ,
 \quad && 1 \leq i \leq n ,  \quad \tau_n  \p_0 =
\p_n \\ \label{cj} \tau_n  \s_i = \s_{i-1} \tau_{n+1} , \quad &&
1 \leq i \leq n , \quad \tau_n  \s_0 = \s_n  \tau_{n+1}^2 \\
\label{ce} \tau_n^{n+1} &=& 1_n  \, .
\end{eqnarray}

Using the coface operators we define the Hochschild coboundary
\begin{equation}
b: C^{q}(H,V)\ra C^{q+1}(H,V), \qquad b:=\sum_{i=0}^{q+1}(-1)^i\p_i.
\end{equation}
Then $b^2=0$ and as a result we have the Hochschild complex $(C^\bullet(H,V),b)$ of the coalgebra $H$ with coefficients in the bicomodule $V$. Here, we view $V$ as a bicomodule with the trivial right $H$-comodule structure. The cohomology of the complex $(C(H,V),b)$ is denoted by $H_{\rm coalg}(H,V)$.

\medskip

Using the rest of the operators we define the Connes boundary operator,
\begin{equation}
B: C^{q}(H,V)\ra C^{q-1}(H,V), \qquad B:=\left(\sum_{i=0}^{q-1}(-1)^{qi}\tau^{i}\right) \s_{q-1}\tau.
\end{equation}

It is shown in \cite{Conn83} that for any cocyclic module we have $b^2=B^2=(b+B)^2=0$. As a result, we have the cyclic cohomology of $H$ with coefficients in the SAYD module $V$, which is  denoted by $HC(H,V)$, as the total cohomology of the bicomplex
\begin{align}
C^{p,q}(H,V)= \left\{ \begin{matrix} V\ot H^{\ot q-p},&\quad\text{if}\quad 0\le p\le q, \\
&\\
0, & \text{otherwise.}
\end{matrix}\right.
\end{align}

One also defines the periodic cyclic cohomology of $H$ with coefficients in $V$, denoted by $HP(H,V)$, as the total cohomology of the bicomplex
\begin{align}
C^{p,q}(H,V)= \left\{ \begin{matrix} V\ot H^{\ot q-p},&\quad\text{if}\quad  p\le q, \\
&\\
0, & \text{otherwise.}
\end{matrix}\right.
\end{align}

\section{Lie algebra (co)homology}

Much in the same way the cyclic cohomology of algebras generalizes de Rham homology of currents, the Hopf-cyclic cohomology of Hopf algebras generalizes Lie algebra homology. Extending this result to the theory with coefficients, our main objective here is to identify the Hopf-cyclic cohomology of certain bicrossed product Hopf algebras in terms of (relative) Lie algebra homology and cohomology. We thus recall the (relative) Lie algebra (co)homology from \cite{ChevEile,Knap} and \cite{HochSerr53}.

\medskip

Let  $\Fg$ be a Lie algebra and $V$ be a right  $\Fg$-module.  We first recall the Lie algebra homology complex
\begin{equation}
C(\Fg, V) = \bigoplus_{q\geq0}C_q(\Fg, V), \quad C_q(\Fg, V) := \wg^q \Fg \ot V
\end{equation}
with the Chevalley-Eilenberg boundary map
\begin{equation}
\xymatrix{\cdots \ar[r]^{\p_{\rm CE}\;\;\;\;\;\;\;\;} & C_2(\Fg,V) \ar[r]^{\p_{\rm CE}\;\;\;} & C_1(\Fg,V) \ar[r]^{\;\;\;\;\;\;\;\p_{\rm CE}} & V},
\end{equation}
which is explicitly defined by
\begin{align}
\begin{split}
& \p_{\rm CE}(X_0 \wdots X_{q-1} \ot v ) = \sum_{i= 0}^{q-1} (-1)^i  X_0 \wdots \widehat{X}_i \wdots
X_{q-1} \ot v \cdot X_i + \\
& \sum_{0 \leq i<j \leq q-1} (-1)^{i+j} [X_i, X_j] \wg X_0 \wdots \widehat{X}_i \wdots
\widehat{X}_j \wdots X_{q-1} \ot v,
\end{split}
\end{align}
where $\widehat{X}_i$ indicates the omission of the argument $X_i$. The homology of the complex $(C(\Fg, V),\p_{\rm CE})$ is called the Lie algebra homology of $\Fg$ with coefficients in $V$ and is denoted by $H_{\bullet}(\Fg,V)$.

\medskip

Next we recall the Lie algebra cohomology complex
\begin{equation}
W(\Fg, V) = \bigoplus_{q\geq0}W^q(\Fg, V), \quad W^q(\Fg,V)=\Hom(\wedge^q \Fg ,V),
\end{equation}
where $\Hom(\wedge^q \Fg ,V)$ is the vector space of all alternating linear maps on $\Fg^{\ot q}$ with values in $V$. The Chevalley-Eilenberg coboundary
\begin{equation}
\xymatrix{V\ar[r]^{d_{\rm CE}\;\;\;\;\;\;\;\;}&W^1(\Fg,V)\ar[r]^{d_{\rm CE}}& W^2(\Fg,V)\ar[r]^{\;\;\;\;\;\;d_{\rm CE}}&\cdots\;,}
\end{equation}
is defined by
 \begin{align}\label{aux-Chevalley-Eilenberg-coboundary}
 \begin{split}
& d_{\rm CE}(\a)(X_0, \ldots,X_q)=\sum_{0 \leq i<j \leq q} (-1)^{i+j}\a([X_i,X_j], X_0\ldots \widehat{X}_i, \ldots, \widehat{X}_j, \ldots, X_q)+\\
& ~~~~~~~~~~~~~~~~~~~~~~~\sum_{i=0}^q (-1)^{i+1}\a(X_0,\ldots,\widehat{X}_i,\ldots X_q)\cdot X_i,
\end{split}
 \end{align}
for any $\a \in W^q(\Fg,V)$.

\medskip

Alternatively, we may identify  $W^q(\Fg,V)$ with $\wg^q\Fg^\ast \ot V$ and the coboundary $d_{\rm CE}$ with
\begin{align}
\begin{split}
& d_{\rm CE}(v)= - \t^i \ot v \cdot X_i,\\
& d_{\rm CE}(\b \ot v)= d_{\rm dR}(\b)\ot v  - \t^i \wg \b \ot v \cdot X_i,
\end{split}
\end{align}
where
\begin{equation}
d_{\rm dR}:\wedge^p\Fg^\ast\ra \wedge^{p+1}\Fg^\ast, \quad d_{\rm dR}(\t^i)=-\frac{1}{2}C^i_{jk}\t^j\wg\t^k
\end{equation}
is the de Rham differential. The cohomology of the complex $(W(\Fg,V),d_{\rm CE})$ is denoted by $H^\bullet(\Fg,V)$ and is called the Lie algebra cohomology of $\Fg$ with coefficients in $V$.

\medskip

Finally we recall the relative Lie algebra cohomology. For a Lie subalgebra $\Fh\subseteq \Fg$ we define the relative complex by
\begin{equation}
W(\Fg,\Fh,V) = \bigoplus_{q\geq0}W^q(\Fg,\Fh,V),
\end{equation}
where
\begin{equation}
W^q(\Fg,\Fh,M)=\left\{\theta \in W^q(\Fg,M) \,|\, \iota(X) \theta = \Lc_X(\theta) = 0, \, X\in \Fh\right\}.
\end{equation}
Here,
\begin{equation}
\iota(X)(\theta)(X_1,\ldots,X_{q})=\theta(X,X_1,\ldots,X_q)
\end{equation}
is the contraction operator and
\begin{align}
& \Lc_X(\theta)(X_1,\ldots,X_{q})=\\\notag
& \sum_{i=1}^q(-1)^i\theta([X,X_i], X_1,\ldots, \widehat{X}_i,\ldots,X_q)+\theta(X_1,\ldots,X_q)X
\end{align}
is the Lie derivative.

\medskip

We can identify $W^q(\Fg,\Fh,V)$ with $\Hom_{\Fh}(\wg^q(\Fg/\Fh),V)=(\wedge^q(\Fg/\Fh)^\ast \ot V)^\Fh$, where the action of $\Fh$ on $\Fg/\Fh$ is induced  by the  adjoint action of $\Fh$ on $\Fg$.

\medskip

It is shown in \cite{ChevEile} that the Chevalley-Eilenberg coboundary  \eqref{aux-Chevalley-Eilenberg-coboundary} is well defined on $W(\Fg,\Fh,V)$. The cohomology of the relative subcomplex $(W(\Fg,\Fh,V),d_{\rm CE})$ is denoted by $H^\bullet(\Fg,\Fh,V)$ and is called the relative Lie algebra cohomology of $\Fh\subseteq \Fg$ with coefficients in $V$.

\chapter{Geometric Hopf algebras}\label{chapter-geometric-hopf-algebras}

In \cite{HochMost57} and \cite{Hoch59}, Hochschild and Mostow showed that the algebra of representative functions on Lie groups, Lie algebras and algebraic groups form a commutative Hopf algebra whose coalgebra structure  encodes the  group (Lie algebra) structure of the group (Lie algebra). In \cite{Hoch61} and \cite{HochMost62} they invented a cohomology theory  for these objects, called  the representative cohomology. This cohomology theory is defined on the representatively injective resolutions and they proved a van Est  theorem  by computing the representative cohomology  in terms of the relative Lie algebra cohomology  for suitable pairs of Lie algebras.

\medskip

Here we extend the work of Hochschild-Mostow by associating a bicrossed product Hopf algebra, which we call a Lie-Hopf algebra, to any matched pair of Lie algebras, Lie groups, and  algebraic groups. A Lie-Hopf algebra is built on a commutative Hopf algebra of the representative functions (on a Lie group, Lie algebra or an affine algebraic group) and a universal enveloping algebra of a Lie algebra. Hence, we first recall these building blocks and then we describe the construction of the Lie-Hopf algebras by the semidualization of Lie groups, Lie algebras and affine algebraic groups.

\section{Commutative representative Hopf algebras}

In this section we recall the commutative Hopf algebras $R(G)$ of representative functions on a Lie group $G$, $R(\Fg)$ of representative functions on the universal enveloping algebra $U(\Fg)$ and $\mathscr{P}(G)$ of polynomial functions on an affine algebraic group $G$.

\medskip

Let $G$ be a Lie group and $\rho: G \to {\rm GL}(V)$ be a finite dimensional smooth representation of $G$. Then for any linear functional $\pi \in \End(V)^\ast$, the composition $\pi \circ \rho$ is called a representative function on $G$ associated with the representation space $V$. We denote by $R(G)$ the set of all representative functions on $G$.

\medskip

The group $G$ acts on $R(G)$ by the left and the right translations, defined by
\begin{equation}
(\psi \cdot f)(\vp) := f(\vp\psi), \quad (f \cdot \psi)(\vp) := f(\psi\vp),
\end{equation}
for any $\psi,\vp \in G$ and any $f \in R(G)$. We have the following characterization of the representative functions.

\begin{proposition}[\cite{HochMost57}]\label{proposition-elements-of-R(G)}
Let $f:G \to \mathbb{C}$ be a smooth function. Then the following are equivalent:
\begin{itemize}
\item [(i).] $f$ is a representative function.
\item [(ii).] The right translations  $\{f \cdot \psi \,|\, \psi \in G\}$ span a finite dimensional vector space over $\mathbb{C}$.
\item [(iii).] The left translations  $\{\psi \cdot f \,|\, \psi \in G\}$ span a finite dimensional vector space over $\mathbb{C}$.
\end{itemize}
\end{proposition}

By this proposition, the set $R(G)$ of representative functions on the group $G$ forms a subalgebra of the algebra $C^\infty(G)$ of smooth functions on $G$.

\medskip

The coalgebra structure of $R(G)$ follows from the observation that for
\begin{equation}
\d:C^\infty(G) \to C^\infty(G\times G), \quad \d(f)(\psi,\vp) = f(\psi\vp)
\end{equation}
and
\begin{equation}
\pi:C^\infty(G) \ot C^\infty(G) \to C^\infty(G\times G), \quad \pi(f \ot g)(\psi,\vp) = f(\psi)g(\vp),
\end{equation}
we have $\d(f) \subseteq \Im\pi$ if and only if $f \in R(G)$, \cite[Theorem 2.2.7]{Abe-book}.

\medskip

As a result, the set $R(G)$ of representative functions on the Lie group $G$ forms a commutative Hopf algebra via
\begin{align}
&\mu: R(G)\ot R(G)\ra R(G), && \mu(f\ot g)(\psi)=f(\psi)g(\psi),\\
&\eta:\Cb\ra R(G),&& \eta(1)=\ve:\Cb [G]  \to \Cb,\\
&\D:R(G)\ra R(G)\ot R(G),&& \D(f)=\pi^{-1}(\d(f)), \\
&\ve:R(G) \to \Cb,&& \ve(f) = f(1), \\
& S: R(G)\ra R(G), &&S(f)(\psi)=f(\psi^{-1}).
\end{align}

Next we recall the commutative Hopf algebra of representative functions on the universal enveloping algebra of a Lie algebra $\Fg$. Let us define
\begin{equation}
R(\Fg)=\Big\{f\in \Hom(U(\Fg),\Cb)\,|\, \exists I\leq U(\Fg), I \subseteq \ker f, \, \dim((\ker f)/I) < \infty   \Big\}.
\end{equation}
On one hand, $R(\Fg)$ has the dual algebra structure of the cocommutative coalgebra $U(\Fg)$. On the other hand, the finite codimensionality condition for the ideal $I \subseteq \ker f$ guarantees, for instance by
\cite[Theorem 2.2.12]{Abe-book} or \cite{Hoch59}, that  for any $f\in R(\Fg)$  there exist  a  finite number of
functions $f_i', f_i'' \in R(\Fg)$ such that
\begin{equation}
f(u^1u^2)=\sum_{i}f_i'(u^1)f_i''(u^2),\quad u^1,u^2\in U(\Fg).
\end{equation}

In short, the Hopf algebra structure of $R(\Fg)$ is given by
\begin{align}
&\mu: R(\Fg)\ot R(\Fg)\ra R(\Fg), && \mu(f\ot g)(u)=f(u\ps{1})g(u\ps{2}),\\
&\eta:\Cb\ra R(\Fg),&& \eta(1)=\ve:U(\Fg) \to \Cb,\\
&\D:R(\Fg)\ra R(\Fg)\ot R(\Fg),&& \D(f)=\sum_if_i'\ot f_i'',\\
& \ve:R(\Fg)\ra\Cb,&&\ve(f) = f(1), \\
& S: R(\Fg)\ra R(\Fg), &&S(f)(u)=f(S(u)).
\end{align}

Finally we recall the commutative Hopf algebra of polynomial functions on an affine algebraic group $G$.

\medskip

Let $G$ be an affine algebraic group and $\mathscr{P}(G)$ be the set of polynomial maps on $G$, which is a finitely generated subalgebra of the set ${\rm Map}(G,\Cb)$ of all maps $G \to \Cb$, that separates the points of $G$. In other words, if for $\psi, \vp \in G$, $\psi \neq \vp$ then there exists $f \in \mathscr{P}(G)$ such that $f(\psi) \neq f(\vp)$, and $G$ is the group $\widehat{\mathscr{P}(G)}$ of all algebra homomorphisms $\mathscr{P}(G) \to \Cb$, \ie
\begin{equation}
G \cong \widehat{\mathscr{P}(G)}, \quad \psi \mapsto \widehat{\psi}, \quad \widehat{\psi}(f) = f(\psi)
\end{equation}
for any $f \in \mathscr{P}(G)$.

\medskip

By definition $\mathscr{P}(G)$ is an algebra. Moreover, the multiplication, inverse and the unit of the group $G$ induce a Hopf algebra structure on $\mathscr{P}(G)$ that we recall here by
\begin{align}
&\mu: \mathscr{P}(G)\ot \mathscr{P}(G)\ra \mathscr{P}(G), && \mu(f\ot g)(\psi)=f(\psi)g(\psi),\\
&\eta:\Cb\ra \mathscr{P}(G),&& \eta(1)=\ve:\Cb [G]  \to \Cb,\\
&\D:\mathscr{P}(G)\ra \mathscr{P}(G)\ot \mathscr{P}(G),&& \D(f)(\psi,\vp)=f(\psi\vp), \\
&\ve:\mathscr{P}(G) \to \Cb,&& \ve(f) = f(1), \\
& S: \mathscr{P}(G)\ra \mathscr{P}(G), &&S(f)(\psi)=f(\psi^{-1}).
\end{align}

Finally we remark that $\mathscr{P}(G)$ is a Hopf subalgebra of $R(G)$, \cite[Section 1.3]{Hoch-book}. In the definition of $R(G)$ for an affine algebraic group $G$, we consider the rational representations of $G$, \cite[Section 4]{Hoch59} and \cite[Section 2]{HochMost57}.

\section{Noncommutative geometric Hopf algebras}

In this section we construct a matched pair of Hopf algebras (hence a bicrossed product Hopf algebra) out of a matched pair of Lie algebras, Lie groups or affine algebraic groups. More precisely, to the matched pair of Lie algebras $(\Fg_1,\Fg_2)$ we associate $R(\Fg_2)\acl U(\Fg_1)$. Similarly we associate the Hopf algebra $R(G_2)\acl U(\Fg_1)$ to a matched pair of Lie groups $(G_1,G_2)$, where $\Fg_1$ is the Lie algebra of the Lie group $G_1$. Finally we construct the Hopf algebra $\mathscr{P}(G_2)\acl U(\Fg_1)$ for a matched pair of affine algebraic groups $(G_1,G_2)$, where $\Fg_1$ is the Lie algebra of $G_1$.

\subsection{Lie-Hopf algebras}

In order not to check the matched pair conditions \eqref{aux-matched-pair-1} to \eqref{aux-matched-pair-5} repeatedly, we use the fact that we are interested in only representative Hopf algebras and universal enveloping algebras. We therefore introduce the notion of a Lie-Hopf algebra.

\medskip

Let $\Fc$ be a commutative Hopf algebra on which a Lie algebra  $\Fg$ acts  by derivations. We endow the vector space $\Fg\ot \Fc$ with the bracket
\begin{equation}\label{aux-bracket-on-g-ot-F}
[X\ot f, Y\ot g]= [X,Y]\ot fg+ Y\ot \ve(f)X\rt g- X\ot \ve(g) Y\rt f
\end{equation}
for any $X,Y \in \Fg$ and $f,g \in \Fc$.

\begin{lemma}
Let $\Fg$ act on a commutative Hopf algebra $\Fc$ and suppose $\ve(X\rt f)=0$ for any $X\in \Fg$ and $f\in \Fc$. Then  the bracket  \eqref{aux-bracket-on-g-ot-F} endows  $\Fg\ot \Fc$ with a  Lie algebra structure.
\end{lemma}

\begin{proof}
That the bracket is antisymmetric is straightforward.  We then need to check the Jacobi identity. To this end, we observe
\begin{align}
\begin{split}\label{aux-jacobi-1}
&[[X\ot f, Y\ot g], Z\ot h]=[[X,Y],Z]\ot fgh+Z\ot \ve(fg)[X,Y]\rt h-\\
&[X,Y]\ot \ve(h)Z\rt (fg)+[Y,Z]\ot \ve(f)hX\rt g - Y\ot \ve(h)\ve(f)Z\rt (X\rt g)-\\
&[X,Z]\ot \ve(g)hY\rt f+X\ot \ve(h)\ve(g)Z\rt(Y\rt f),
\end{split}
\end{align}
\begin{align}
\begin{split}\label{aux-jacobi-2}
&[[ Y\ot g, Z\ot h],X\ot f]= [[Y,Z],X]\ot fgh+ X\ot \ve(gh)[Y,Z]\rt f-\\
 &[Y,Z]\ot \ve(f) X\rt (gh)+[Z,X]\ot \ve(g)f Y\rt h- Z\ot \ve(f)\ve(g) X\rt(Y\rt h)-\\
&[Y,X]\ot \ve(h)fZ\rt g+ Y\ot \ve(f)\ve(h) X\rt(Z\rt g),
\end{split}
\end{align}
and
\begin{align}
\begin{split}\label{aux-jacobi-3}
&[[Z\ot h,X\ot f],Y\ot g]=[[Z,X],Y]\ot fgh+ Y\ot \ve(hf)[Z,X]\rt g- \\
&[Z,X]\ot \ve(g)Y\rt(hf)+[X,Y]\ot \ve(h)g Z\rt f- X\ot \ve(h)\ve(g)Y\rt(Z\rt f)-\\
&[Z,Y]\ot \ve(f)g X\rt h+ Z\ot \ve(g)\ve(f) Y\rt(X\rt h).
\end{split}
\end{align}
Summing up \eqref{aux-jacobi-1}, \eqref{aux-jacobi-2} and \eqref{aux-jacobi-3}, and using the fact that $\Fg$ is a Lie algebra  acting on $\Fc$ by derivations, we get
\begin{equation}
[[X\ot f, Y\ot g], Z\ot h]+[[ Y\ot g, Z\ot h],X\ot f]+ [[Z\ot h,X\ot f],Y\ot g]=0.
\end{equation}
\end{proof}

Now we assume that $\Fc$ coacts on $\Fg$ from the right via $\Db_\Fg:\Fg\ra \Fg\ot \Fc$. We define the first-order  matrix coefficients $ f^i_j\in \Fc$ of this coaction by
\begin{equation}\label{aux-first-order-matrix-coefficients}
\Db_\Fg(X_j)= \sum_i X_i\ot f_j^i.
\end{equation}
Using the coassociativity of the coaction $\Db_\Fg:\Fg\ra \Fg\ot \Fc$ we observe that
\begin{equation}\label{aux-comultiplication-matrix-coefficients}
\D(f_i^j)=\sum_{k=1}^n f_k^j\ot f_i^k.
\end{equation}

We then define the second-order  matrix coefficients by
\begin{equation}
 f_{j,k}^i:=  X_{k}\rt f^{i}_j.
\end{equation}

\begin{definition}
We say that the coaction $\Db_\Fg:\Fg\ra \Fg\ot \Fc$ satisfies the structure identity of $\Fg$  if
\begin{equation}\label{aux-Bianchi}
f_{j,i}^k-f_{i,j}^k=\sum_{s,r} C_{sr}^k f_{i}^rf_{j}^s+\sum_lC_{ij}^lf_l^k\,.
\end{equation}
\end{definition}

\begin{lemma}\label{lemma-structure-identity}
The coaction  $\Db_\Fg:\Fg\ra \Fg\ot \Fc$ satisfies the structure identity of $\Fg$  if and only  if  $\Db_\Fg:\Fg\ra \Fg\ot \Fc$ is a Lie algebra map.
\end{lemma}

\begin{proof}
Assuming that the structure identity \eqref{aux-Bianchi} holds, we need to  show that
\begin{equation}
[\Db_\Fg(X_i), \Db_\Fg(X_j)] =\Db_\Fg([X_i,X_j])=\Db_\Fg(\sum_k C^k_{ij}X_k)= \sum_{k,l} C^k_{ij}X_l\ot f^l_k.
\end{equation}
To this end, using \eqref{aux-Bianchi} we observe that
\begin{align}
\begin{split}
&[\Db_\Fg(X_i), \Db_\Fg(X_j)]=\sum_{p,q}[X_p\ot f^p_i, X_q\ot f^q_j]\\
&=\sum_{p,q}[X_p,X_q]\ot f^p_if^q_j+\sum_{p,q} X_q\ot \ve(f^p_i)X_p\rt f^q_j- \sum_{p,q}X_p\ot \ve(f^q_j)X_q\rt f^p_i\\
&=\sum_{k,r,s} C^k_{rs}X_k\ot f^r_if^s_j+ \sum_k X_k\ot X_i\rt f^k_j- \sum_kX_k\ot X_j\rt f^k_i\\
&=\sum_{k,r,s}(C^k_{rs}X_k\ot f^r_if^s_j+ \sum_kX_k\ot (f^k_{j,i}- f^k_{i,j}))\\
&=\sum_{k,r,s}C^k_{rs}X_k\ot f^r_if^s_j+ \sum_{k,r,s}C^k_{sr}X_k\ot f^r_if^s_j+\sum_{k,l} C^k_{ij}X_l\ot f^l_k=\sum_{k,l}C^k_{ij}X_l\ot f^l_k.
\end{split}
\end{align}
The reverse  statement is similar.
\end{proof}

Using the  action of $\Fg$ on $\Fc$ and the coaction of  $\Fc$  on $\Fg$ we define
\begin{equation}
X\bullet (f^1 \ot f^2)= X\ns{0}\rt f^1\ot X\ns{1} f^2 + f^1\ot X\rt f^2,
\end{equation}
which turns out to be an action of $\Fg$ on $\Fc\ot \Fc$.

\begin{definition}\label{definition-Lie-Hopf}
Let a Lie algebra $\Fg$ act on a commutative Hopf algebra $\Fc$ by derivations. We say that $\Fc$ is a $\Fg$-Hopf algebra if
 \begin{enumerate}
   \item $\Fc$ coacts on $\Fg$ and its coaction satisfies the structure identity of $\Fg$.
   \item $\D$ and $\ve$ are $\Fg$-linear in the sense that $\D(X\rt f)=X\bullet\D(f)$ and $\ve(X\rt f)=0$, for any $f\in \Fc$ and $X\in \Fg$.
    \end{enumerate}
\end{definition}

Let $\Fc$ be a  $\Fg$-Hopf algebra. Since the Lie algebra $\Fg$ acts on $\Fc$ by derivations, $\Fc$ becomes a $U(\Fg)$-module algebra. On the other hand, we extend the coaction $\Db_\Fg:\Fg \to \Fg \ot \Fc$ to a coaction $\Db:U(\Fg) \to U(\Fg) \ot \Fc$ inductively via the rule \eqref{aux-matched-pair-4} and $\Db(1)=1\ot 1$.

\begin{lemma}\label{lemma-U-coaction}
The extension of $\Db_\Fg:\Fg\ra \Fg\ot \Fc$  to $\Db:U(\Fg)\ra U(\Fg)\ot \Fc$  via \eqref{aux-matched-pair-4} is well-defined.
\end{lemma}

\begin{proof}
We need to prove that $\Db_\Fg([X,Y])=\Db(XY-YX)$. Using \eqref{aux-matched-pair-4}, Lemma \ref{lemma-structure-identity} and the fact that $\Fc$ is commutative, we see that
\begin{align}\label{aux-eq-in-lemma-U-coaction}
\begin{split}
&\Db(XY-YX) \\
& = [X\ns{0},Y\ns{0}]\ot X\ns{1}Y\ns{1}+ Y\ns{0}\ot X\rt Y\ns{1}-  X\ns{0}\ot Y\rt X\ns{1} \\
&=\Db_\Fg([X,Y]).
\end{split}
\end{align}
\end{proof}

We are now ready to express the main result of this subsection.

\begin{theorem}\label{theorem-Lie-Hopf-matched-pair}
Let $\Fc$ be a $\Fg$-Hopf algebra. Then via the coaction of $\Fc$ on $U(\Fg)$ defined above and the natural action of $U(\Fg)$ on $\Fc$, the pair $(\Fc,U(\Fg))$ becomes a matched pair of Hopf algebras. Conversely, for a commutative Hopf algebra $\Fc$, if $(\Fc,U(\Fg))$ is a matched pair of Hopf algebras then $\Fc$ is a $\Fg$-Hopf algebra.
\end{theorem}

\begin{proof}
Let $\Fc$ be a $\Fg$-Hopf algebra. We need to verify that the matched pair conditions are satisfied. The axioms \eqref{aux-matched-pair-1} and \eqref{aux-matched-pair-3} hold by definition. Hence we need only to prove \eqref{aux-matched-pair-2} and \eqref{aux-matched-pair-4}.

\medskip

By definition of the coaction $\Db:U(\Fg)\ra U(\Fg)\ot \Fc$, the axiom \eqref{aux-matched-pair-4} holds for any $u,v\in U(\Fg)$.

\medskip

Next we check  \eqref{aux-matched-pair-2}, which holds for $X\in \Fg$ and $f\in \Fc$ by Definition \ref{definition-Lie-Hopf}. Let us assume that  it  is satisfied for $u,v\in U(\Fg)$, and any $f\in \Fc$. By using \eqref{aux-matched-pair-4}  we see that
\begin{align}\label{aux-eq-in-proof-1}
\begin{split}
& (uv)\ps{1}\ns{0}\ot (uv)\ps{1}\ns{1}\ot (uv)\ps{2}\\
&= u\ps{1}\ns{0}v\ps{1}\ns{0}\ot u\ps{1}\ns{1}(u\ps{2}\rt v\ps{1}\ns{1})\ot u\ps{3}v\ps{2}.
\end{split}
\end{align}
Using \eqref{aux-eq-in-proof-1} and the fact that $\Fc$ is $U(\Fg)$-module algebra we prove our claim by observing that
\begin{align}
\begin{split}
&\D(uv\rt f)= u\ps{1}\ns{0}\rt (v\rt f)\ps{1}\ot u\ps{1}\ns{1}(u\ps{2}\rt (v\rt f)\ps{2})\\
&= u\ps{1}\ns{0}\rt (v\ps{1}\ns{0}\rt f\ps{1})\ot u\ps{1}\ns{1}(u\ps{2}\rt (v\ps{1}\ns{1}(v\ps{2}\rt f\ps{2})\\
&= u\ps{1}\ns{0}\rt (v\ps{1}\ns{0}\rt f\ps{1})\ot u\ps{1}\ns{1}(u\ps{2}\rt (v\ps{1}\ns{1})(u\ps{3}v\ps{2}\rt f\ps{2})\\
&= (uv)\ps{1}\ns{0}\rt f\ps{1}\ot (uv)\ps{1}\ns{1}((uv)\ps{2}\rt f\ps{2}).
\end{split}
\end{align}

Finally,  we check that $U(\Fg)$ is  a $\Fc$-comodule coalgebra . To this end, we need to verify \eqref{aux-eq-comodule-coalgebra}, which is obviously satisfied  for any $X\in \Fg$. Let us assume that \eqref{aux-eq-comodule-coalgebra} is satisfied for $u,v\in U(\Fg)$, and prove that it is satisfied for $uv \in U(\Fg)$ too. Indeed, by using \eqref{aux-matched-pair-4}, the fact that $U(\Fg)$ is cocommutative, $\Fc$ is commutative, and that $\Fc$ is a $U(\Fg)$-module algebra, we observe that
\begin{align}
\begin{split}
& (uv)\ps{1}\ns{0}\ot (uv)\ps{2}\ns{0}\ot (uv)\ps{1}\ns{1}(uv)\ps{2}\ns{1}\\
&= (u\ps{1}v\ps{1})\ns{0}\ot (u\ps{2}v\ps{2})\ns{0}\ot (u\ps{1}v\ps{1})\ns{1}(u\ps{2}v\ps{2})\ns{1}\\
&= u\ps{1}\ns{0}v\ps{1}\ns{0}\ot u\ps{3}\ns{0}v\ps{2}\ns{0}\ot u\ps{1}\ns{1}(u\ps{2}\rt v\ps{1}\ns{1})u\ps{3}\ns{1}(u\ps{4}\rt v\ps{2}\ns{1})\\
&= u\ps{1}\ns{0}v\ps{1}\ns{0}\ot u\ps{2}\ns{0}v\ps{2}\ns{0}\ot u\ps{1}\ns{1} u\ps{2}\ns{1}(u\ps{3}\rt (v\ps{1}\ns{1}v\ps{2}\ns{1})\\
&= u\ps{1}\ns{0}\ps{1}v\ns{0}\ps{1}\ot u\ps{1}\ns{0}\ps{2}v\ns{0}\ps{2}\ot u\ps{1}\ns{1} (u\ps{2}\rt (v\ns{1})\\
&= (u\ps{1}\ns{0}v\ns{0})\ps{1}\ot (u\ps{1}\ns{0}v\ns{0})\ps{2}\ot u\ps{1}\ns{1}(u\ps{2}\rt v\ns{1})\\
&= (uv)\ns{0}\ps{1}\ot (uv)\ns{0}\ps{2}\ot (uv)\ns{1}.
\end{split}
\end{align}

Conversely, let $\Fc$ be a commutative Hopf algebra and $(\Fc,U(\Fg))$ be a matched pair of Hopf algebras. Let us denote the coaction of  $\Fc$ on $U(\Fg)$ by $\Db_{\Uc}:U(\Fg) \to U(\Fg) \ot \Fc$. First we prove that the restriction of $\Db_\Uc:U(\Fg)\ra U(\Fg)\ot \Fc$ on $\Fg$ lands in $\Fg\ot \Fc$. Indeed, since $U(\Fg)$ is $\Fc$-comodule coalgebra, we see that
\begin{align}
\begin{split}
& X\ns{0}\ot 1\ot X\ns{1}+ 1\ot X\ns{0}\ot X\ns{1}\\
 &= X\ps{1}\ns{0}\ot X\ps{2}\ns{0} \ot X\ps{1}\ns{1}X\ps{2}\ns{1}\\
&= X\ns{0}\ps{1}\ot X\ns{0}\ps{2}\ot X\ns{1}.
\end{split}
\end{align}
This shows that for any $X\in\Fg$, $\Db_\Uc(X)$ belongs to $\Pc\ot \Fc$, where $\Pc$ is the Lie algebra of primitive elements of $U(\Fg)$. Since $\Pc=\Fg$, we get a coaction $\Db_\Fg:\Fg\ra \Fg\ot \Fc$ which is the restriction of $\Db_\Uc$.

\medskip

Since $\Fc$ is  a $U(\Fg)$-module algebra, $\Fg$ acts on $\Fc$ by derivations. The equality \eqref{aux-eq-in-lemma-U-coaction} together with \eqref{aux-matched-pair-4} shows that the coaction $\Db_g:\Fg\ra \Fg\ot \Fc$ is a map of Lie algebras. Finally \eqref{aux-matched-pair-2} implies that $\D$ is $\Fg$-linear. So we have proved that $\Fc$ is a $\Fg$-Hopf algebra.
\end{proof}

\subsection{Matched pair of Hopf algebras associated to matched pair of Lie algebras}

In this subsection we associate  a bicrossed product Hopf algebra to any matched pair of Lie algebras $(\Fg_1,\Fg_2)$.

\medskip

Let $(\Fg_1,\Fg_2)$ be a matched pair of Lie algebras. Using the left action of $U(\Fg_2)$ on $\Fg_1$ we define the functionals $f_i^j:U(\Fg_2) \to \Cb$ for $1 \leq i,j \leq \dim\Fg_1 = N$ by
\begin{equation}
u'\rt X_i=f_i^j(u') X_j, \quad \forall X_i,X_j \in \Fg_1,\; u' \in U(\Fg_2).
\end{equation}
Equivalently, using the dual basis elements,
\begin{equation}
f_i^j(u')=<u'\rt X_i,\t^j>.
\end{equation}

\begin{lemma}
For $1\le i,j\le N$, the functions $f_i^j:U(\Fg_2) \to \Cb$ are representative functions.
\end{lemma}

\begin{proof}
 For $1\le i\le N$, we set $I_i=\{u'\in U(\Fg_2)\mid u'\rt X_i=0\}$. Since $\Fg_1$ is finite dimensional, $I_i$ is a finite codimensional left ideal of $U(\Fg_2)$, being the kernel of the linear map
\begin{equation}
U(\Fg_2) \to \Fg_1, \quad v \mapsto u' \rt X_i.
\end{equation}
Moreover, $I_i \subseteq \ker f_i^j$ for any $1\le j\le N$.
\end{proof}

As a result, we have the coaction
\begin{equation}\label{aux-g-coaction}
\Db_{\rm Alg}:\Fg_1\ra \Fg_1\ot R(\Fg_2),\quad \Db_{\rm Alg}(X_i)=\sum_{k=1}^N X_j\ot f_i^j,
\end{equation}
admitting $f_i^j:U(\Fg_2) \to \Cb$ for $1\le i,j\le N$ as the first order matrix coefficients.

\medskip

By the work of  Harish-Chandra  \cite{Harish-Chandra49} we know that $R(\Fg_2)$ separates elements of $U(\Fg_2)$. This results in a non-degenerate dual pair of Hopf algebras $(R(\Fg_2),U(\Fg_2))$ with the pairing
\begin{equation}\label{aux-F-V-pairing}
<f,u'>:=f(u'), \quad \forall f \in R(\Fg_2),\; u' \in U(\Fg_2).
\end{equation}
We use the pairing \eqref{aux-F-V-pairing} to define the action
\begin{equation}\label{aux-action-U-to-F}
U(\Fg_1)\ot R(\Fg_2) \ra R(\Fg_2), \quad <u\rt f, u'>=<f, u'\lt u>
\end{equation}
of $U(\Fg_1)$ on $R(\Fg_2)$. This way, we define the higher order matrix coefficients in $R(\Fg_2)$ by
\begin{equation}
 f_{i,i_1,\dots,i_k}^j:= X_{i_k}\cdots X_{i_1}\rt f_i^j.
\end{equation}

\begin{lemma}\label{lemma-structure-identity-g-1}
The coaction \eqref{aux-g-coaction} satisfies the structure identity of $\Fg_1$.
\end{lemma}

\begin{proof}
We should prove \eqref{aux-Bianchi}. We first use \eqref{aux-mutual-pair-1} to observe that for any $u' \in U(\Fg_2)$,
\begin{equation}\label{aux-action-V-on-bracket}
u'\rt[X,Y]= [u'\ps{1}\rt X,u'\ps{2}\rt Y]+ (u'\lt X)\rt Y-(u'\lt Y)\rt X.
\end{equation}
We then apply both sides of \eqref{aux-Bianchi} to an arbitrary element of $u'\in U(\Fg_2)$ and use \eqref{aux-action-V-on-bracket} to observe
\begin{align}
\begin{split}
& \sum_{l,k}C^l_{ij}f^k_l(v)X_k= \sum_l C^l_{i,j}u'\rt X_l =u'\rt[X_i, X_j]\\
&= [v\ps{1}\rt X_i, u'\ps{2}\rt X_j]+ (u'\lt X_i)\rt X_j-(u'\lt X_j)\rt X_i\\
& =\sum_{r,s} f^r_i(v\ps{1})f^s_j(u'\ps{2})[X_r,X_s]+ \sum_k  f^k_j(u'\lt X_i)X_k-f^k_i(u'\lt X_j)X_k\\
&=\sum_{k,r,s} C^k_{rs} (f^r_if^s_j)(u')X_k+\sum_k f^k_{j,i}(u')X_k- \sum_k f^k_{i,j}(u')X_k.
\end{split}
\end{align}
\end{proof}

\begin{proposition}\label{Proposition-matched-Lie-Hopf-Lie}
For any matched pair of Lie algebras $(\Fg_1,\Fg_2)$, the Hopf algebra $R(\Fg_2)$ is a $\Fg_1$-Hopf algebra.
\end{proposition}

\begin{proof}
By Lemma \ref{lemma-structure-identity-g-1}, the coaction \eqref{aux-g-coaction} satisfies the structure identity of $\Fg_1$. Hence, in view of Theorem \ref{theorem-Lie-Hopf-matched-pair}, it suffices to prove that $\ve(X\rt f)=0$ and $\D(X\rt f)=X\bullet \D(f)$.

\medskip

We observe that
\begin{equation}
\ve(X\rt f)= (X\rt f)(1)= f(1\lt X)=0.
\end{equation}
Then it remains to show that $\D(X\rt f)=X\bullet\D(f)$. Indeed,
\begin{align}
\begin{split}
&\D(X\rt f)(u'^1\ot u'^2)= X\rt f(u'^1u'^2)= f(u'^1u'^2\lt X) \\
&= f(u'^1\lt(u'^2\ps{1}\rt X)u'^2\ps{2})+  f(u'^1(u'^2\rt X))=\\
&= f\ps{1}(u'^1\lt(X\ns{1})(u'^2\ps{1}) X\ns{0}) f\ps{2}(u'^2\ps{2}) +  f\ps{1}(u'^1) f\ps{2}(u'^2\lt X)\\
&= (X\ns{0}\rt f\ps{1})(u'^1)X\ns{1}(u'^2\ps{1}) f\ps{2}(u'^2\ps{2}) +  f\ps{1}(u'^1) (X\rt f\ps{2})(u'^2)\\
&=(X\bullet \D(f))(u'^1\ot u'^2).
\end{split}
\end{align}
\end{proof}

The following theorem summarizes the main result of this subsection.

\begin{theorem}
Let $(\Fg_1, \Fg_2)$ be a  matched pair Lie algebras. Then,   via the canonical action and coaction defined in \eqref{aux-action-U-to-F} and   \eqref{aux-g-coaction} respectively,  $(R(\Fg_2),U(\Fg_1))$ is a matched pair of Hopf algebras.
\end{theorem}

\begin{proof}
By Proposition \ref{Proposition-matched-Lie-Hopf-Lie},  $R(\Fg_2)$ is a $\Fg_1$-Hopf algebra, and hence the claim follows from Theorem \ref{theorem-Lie-Hopf-matched-pair}.
\end{proof}

As a result, to any  matched pair $(\Fg_1,\Fg_2)$ of Lie algebras, we associate a Hopf algebra
\begin{equation}
\Hc(\Fg_1,\Fg_2):=R(\Fg_2)\acl U(\Fg_1).
\end{equation}

Let us now provide a concrete example of this construction.

\subsubsection{Projective Hopf algebra via Lie algebra decomposition}

We will construct a bicrossed product Hopf algebra, more precisely a $g\ell(n)^{\rm aff}$-Hopf algebra $\Fc$, out of a decomposition of the Lie algebra $pg\ell(n)$ of the Lie group $\rm{PGL}(n,\mathbb{R})$ of projective transformations of $\Rb P^n$. Recall that $\rm{PGL}(n,\mathbb{R})$ is defined as $\rm{GL}(n+1)$ modulo its center.

\medskip

It is noted in \cite{KobaNaga64} that the Lie algebra $pg\ell(n)$ is of the form
\begin{equation}
pg\ell(n) = \Fm \oplus g\ell(n,\Rb) \oplus \Fm^*
\end{equation}
where, $\Fm$ is the $n$-dimensional abelian Lie algebra of the column vectors whereas $\Fm^*$ is the $n$-dimensional abelian Lie algebra of the row vectors.

\medskip

Let us fix a basis of $\Fm$ as $\{X_s\, |\, 1 \leq s \leq n\}$, a basis of $g\ell(n)$ as $\{X_p^q\, |\, 1 \leq p,q \leq n\}$, and a basis of $\Fm^*$ as $\{X^r\, |\, 1 \leq r \leq n\}$. Accordingly, the bracket of the Lie algebra $pg\ell(n)$ is given by \begin{align}\label{commutatorrelationsforglnaffine}
\begin{split}
& [X_s,X_l] = 0, \qquad [X_p^q, X_i^j] = \delta_i^qX_p^j - \delta_p^jX_i^q \\
& [X^r,X^t] = 0, \qquad [X_p^q, X_s] = \delta_s^qX_p \\
& [X^r, X_p^q] = \delta_p^rX^q, \qquad [X_s, X^r] = X_s^r + \delta_s^r\sum_a X_a^a.
\end{split}
\end{align}
As a result, we have the (vector space) decomposition
\begin{equation}
pg\ell(n) = g\ell(n)^{\rm aff} \oplus \Fm^*.
\end{equation}
By \eqref{aux-actions-matched-pair-Lie-algebras} we immediately deduce the complete list of mutual actions of the Lie algebras $g\ell(n)^{\rm aff}$ and $\Fm^\ast$ on each other. As for the left action of $\Fm^\ast$ on $g\ell(n)^{\rm aff}$, we have
\begin{equation}
X^r \rt X_k = -X^r_k - \delta^r_k \sum_a X^a_a, \qquad X^r \rt X_p^q = 0.
\end{equation}
For the functions $f^i_{jk} \in R(\Fm^\ast)$ defined by
\begin{equation}
X^r \rt X_k = f^i_{jk}(X^r)X_i^j,
\end{equation}
it is straightforward to observe
\begin{equation}\label{aux-f-ijk}
f^i_{jk}(X^r) = -\delta^i_j\delta^r_k - \delta^i_k\delta^r_j,
\end{equation}
which immediately implies that
\begin{equation}
f^i_{jk} = \frac{1}{2}(\delta^i_jf^k_{kk} + \delta^i_kf^j_{jj}).
\end{equation}
As a result, we have the coaction
\begin{align}
\begin{split}
& \Db_{\rm Alg}:g\ell(n)^{\rm aff} \to g\ell(n)^{\rm aff} \ot R(\Fm^*) \\
& X_k \mapsto X_k \ot 1 + X_i^j \ot f^i_{jk} \\
& X_p^q \mapsto X_p^q \ot 1.
\end{split}
\end{align}
Let us next determine the left action of $g\ell(n)^{\rm aff}$ on $R(\Fm^*)$. To this end, we consider the right action
\begin{equation}\label{aux-action-glnaff-on-m-ast}
X^r \lt X_l = 0, \qquad X^r \lt X_p^q = \delta^r_pX^q
\end{equation}
of $g\ell(n)^{\rm aff}$ on $\Fm^*$. The first equation then yields
\begin{equation}\label{aux-f-ikjl-relation}
f^i_{jkl} + \frac{1}{2}(f^i_{al}f^a_{jk} - f^i_{aj}f^a_{kl} - f^i_{ak}f^a_{jl}) =0.
\end{equation}
Indeed, it is immediate to deduce from \eqref{aux-action-glnaff-on-m-ast} that $f^i_{jkl} := X_l \rt f^i_{jk}$ and that the multiplications $f^i_{al}f^a_{jk}, f^i_{aj}f^a_{kl}$ and $f^i_{ak}f^a_{jl}$ are zero on the Lie algebra elements. As for an element $X^rX^s \in U(\Fm^*)$, by \eqref{aux-mutual-actions-U(g1)-U(g2)} we first observe
\begin{align}
\begin{split}
& X^rX^s \lt X_l = X^r(X^s \lt X_l) + (X^r \lt X_l)X^s + X^r \lt (X^s \rt X_l) = \\
& X^r \lt (X^s \rt X_l) = X^r \lt (-X^s_l - \delta^s_l\sum_a X^a_a) = -\delta^r_lX^s - \delta^s_lX^r.
\end{split}
\end{align}
Then on the one hand
\begin{align}
\begin{split}
& f^i_{jkl}(X^rX^s) = -\delta^r_lf^i_{jk}(X^s) - \delta^s_lf^i_{jk}(X^r) \\
& = -\delta^r_l(-\delta^i_j\delta^s_k - \delta^i_k\delta^s_j) - \delta^s_l(-\delta^i_j\delta^r_k - \delta^i_k\delta^r_j) \\
& = \delta^i_j\delta^r_l\delta^s_k + \delta^i_k\delta^r_l\delta^s_j + \delta^i_j\delta^r_k\delta^s_l + \delta^i_k\delta^r_j\delta^s_l,
\end{split}
\end{align}
and on the other hand
\begin{align}
\begin{split}
& (f^i_{al}f^a_{jk} - f^i_{aj}f^a_{kl} - f^i_{ak}f^a_{jl})(X^rX^s) = f^i_{al}(X^r)f^a_{jk}(X^s) + f^i_{al}(X^s)f^a_{jk}(X^r) \\
& - f^i_{aj}(X^r)f^a_{kl}(X^s) - f^i_{aj}(X^s)f^a_{kl}(X^r) - f^i_{ak}(X^r)f^a_{jl}(X^s) - f^i_{ak}(X^s)f^a_{jl}(X^r) \\
& = -2\delta^i_j\delta^r_l\delta^s_k - 2\delta^i_k\delta^r_l\delta^s_j - 2\delta^i_j\delta^r_k\delta^s_l - 2\delta^i_k\delta^r_j\delta^s_l.
\end{split}
\end{align}
Hence, an induction implies \eqref{aux-f-ikjl-relation}.

\medskip

Next we see that the latter in \eqref{aux-action-glnaff-on-m-ast} yields
\begin{equation}
X_p^q \rt f^i_{jk} = \delta^q_jf^i_{pk} + \delta^q_kf^i_{jp} - \delta^i_pf^q_{jk}.
\end{equation}
Indeed,
\begin{align}
\begin{split}
& (X_p^q \rt f^i_{jk})(X^r) = f^i_{jk}(X^r \lt X_p^q) = \delta^r_pf^i_{jk}(X^q) = -\delta^r_p\delta^i_j\delta^q_k - \delta^r_p\delta^i_k\delta^q_j \\
& = \delta^q_j(-\delta^i_p\delta^r_k - \delta^i_k\delta^r_p) + \delta^q_k(-\delta^i_j\delta^r_p - \delta^i_p\delta^r_j) - \delta^i_p(-\delta^q_j\delta^r_k - \delta^q_k\delta^r_j) \\
& = \delta^q_jf^i_{pk}(X^r) + \delta^q_kf^i_{jp}(X^r) - \delta^i_pf^q_{jk}(X^r) = (\delta^q_jf^i_{pk} + \delta^q_kf^i_{jp} - \delta^i_pf^q_{jk})(X^r),
\end{split}
\end{align}
and the result follows once again by induction.

\medskip

Finally we define the projective Hopf algebra
\begin{equation}
\Hc_{n\rm Proj} := R(\Fm^\ast) \acl U(g\ell(n)^{\rm aff}).
\end{equation}
Using the notation $f_k := f^k_{kk}$, the Hopf algebra $R(\Fm^\ast)$ of representative functions is the polynomial Hopf algebra generated by $\{f_k\, |\, 1 \leq k \leq n\}$ with the coaction
\begin{align}\label{coactionLieproj}
\begin{split}
&\Db_{\rm Alg}: g\ell(n)^{\rm aff} \to g\ell(n)^{\rm aff} \ot R(\Fm^\ast) \\
& X_k \mapsto X_k \ot 1 + \frac{1}{2}X_k^a \ot f_a + \frac{1}{2}\sum_a X_a^a \ot f_k \\
& X_p^q \mapsto X_p^q \ot 1,
\end{split}
\end{align}
and the action
\begin{equation}
X_l \rt f_k = \frac{1}{2}f_kf_l, \qquad X_i^j \rt f_k = \delta_k^jf_i.
\end{equation}
It is now fairly straightforward to check that the coaction \eqref{coactionLieproj} is a map of Lie algebras (equivalently the coaction \eqref{coactionLieproj} satisfies the structure identity of $g\ell(n)^{\rm aff}$) and the coalgebra structure maps of $R(\Fm^\ast)$ are $g\ell(n)^{\rm aff}$-linear.

\medskip

As a result, the Hopf algebra $\Hc_{\rm n\, Proj}$ is generated by
\begin{equation}
\{1,f_l,X_k,X^j_i\,|\, 1 \leq i,j,k,l \leq n\}
\end{equation}
and the Hopf algebra structure is given by
\begin{align}
\begin{split}
& [X_i^j,X_p^q] = \d^j_pX_i^q - \d^q_iX_p^j, \qquad [X_k,X_i^j] = \d^j_kX_i, \\
& [X_k,f_l] = \frac{1}{2}f_kf_l, \qquad [X_i^j,f_l] = \d^j_lf_i,\\
& \D(f_l) = f_l \ot 1 + 1 \ot f_l, \qquad \D(X_i^j) = X_i^j \ot 1 + 1 \ot X_i^j,\\
& \D(X_k) = X_k \ot 1 + 1 \ot X_k + \frac{1}{2}X_k^a \ot f_a + \frac{1}{2}\sum_a X_a^a \ot f_k,\\
& \ve(f_l) = \ve(X_k) = \ve(X_i^j) = 0,\\
& S(f_l) = -f_l, \quad S(X_i^j) = -X_i^j, \quad S(X_k) = -X_k + \frac{1}{2}X_k^a f_a + \frac{1}{2}\sum_a X_a^a f_k.
\end{split}
\end{align}

\medskip

Finally, for $n=1$ the Hopf algebra $\Hc_{\rm n\,Proj}$ coincides with the Schwarzian quotient of the Connes-Moscovici Hopf algebra $\Hc_1$, \ie
\begin{equation}
\Hc_{1\rm Proj} = \Hc_{1\rm S}.
\end{equation}

\subsection{Matched pair of  Hopf algebras associated to matched pair of Lie groups}

In this subsection, our aim is to associate a bicrossed product Hopf algebra to any matched pair of Lie groups $(G_1,G_2)$.

\medskip

Let $(G_1,G_2)$ be a matched pair of Lie groups, in which case the mutual actions $\rt: G_2 \times G_1 \to G_1$ and $\lt: G_2 \times G_1 \to G_2$ are assumed to be smooth.

\medskip

Let  $\Fg_1$ and $\Fg_2$ be the Lie algebras of $G_1$ and $G_2$ respectively. Then we can define a right action of $\Fg_1$ on $\Fg_2$  by taking the derivative of the right action of $G_1$ on $G_2$. Similarly, we define a left action of $\Fg_2$ on $\Fg_1$. By a routine argument it follows that if $(G_1,G_2)$ is a matched pair of Lie groups, then $(\Fg_1,\Fg_2)$ is a matched pair of Lie algebras.

\medskip

Our first objective is to prove that $R(G_2)$ is a $\Fg_1$-Hopf algebra. We first define a right coaction of $R(G_2)$ on $\Fg_1$. To do so, we differentiate the right action of $G_2$ on $G_1$ to get a right action of $G_2$ on $\Fg_1$. We then  introduce the functions $f_i^j:G_2 \to \mathbb{C}$, for $1 \leq i,j \leq N = \dim\Fg_1$ by
\begin{align}\label{psi-act-X-j}
\psi \rt X_i = \sum_{j} X_jf_i^j(\psi), \quad X\in \Fg_1, \psi\in G_2.
\end{align}

\begin{lemma}\label{Lemma f-i-j-in-R(G)}
The functions $f_i^j:G_2 \to \mathbb{C}$, $1 \leq i,j \leq N$, defined by \eqref{psi-act-X-j} are representative functions.
\end{lemma}

\begin{proof}
For $\psi_1,\psi_2 \in G_2$, we observe that
\begin{align}
\begin{split}
& \sum_j X_jf_i^j(\psi_1\psi_2) = \psi_1\psi_2 \rt X_i = \psi_1 \rt (\psi_2 \rt X_i) \\
 &=\sum_{l}  \psi_1 \rt X_lf_i^l(\psi_2) = \sum_{j,l} X_jf_l^j(\psi_1)f_i^l(\psi_2).
 \end{split}
\end{align}

Therefore,
\begin{equation}
\psi_2 \cdot f_i^j = \sum _lf_i^l(\psi_2)f_l^j\;.
\end{equation}
In other words, $\psi_2 \cdot f_i^j \in {\rm Span}\{f_i^j\}$ for any $\psi_2 \in G_2$. The claim then follows from Proposition \ref{proposition-elements-of-R(G)}.
\end{proof}

As a result of Lemma \ref{Lemma f-i-j-in-R(G)}, the equation \eqref{psi-act-X-j} determines the coaction
\begin{align}\label{aux-coaction-for-lie-group}
\Db_{\rm Gr}: \Fg_1 \to \Fg_1 \otimes R(G_2), \quad \Db_{\rm Gr}(X_i):= \sum_jX_j \otimes f_i^j
\end{align}
that admits $f_i^j:G_2 \to \Cb$, $1 \leq i,j \leq N$, as the matrix coefficients.

\medskip

Let us recall the natural left action of $G_1$ on $C^{\infty}(G_2)$ defined by
 \begin{equation}
 \vp\rt f(\psi):= f(\psi\lt \vp),
 \end{equation}
 and define the derivative of this action by
\begin{equation}\label{aux-action-for-lie-group}
X \rt f := \dt {\rm exp}(tX) \rt f, \quad X\in \Fg_1, f\in R(G_2).
\end{equation}
In fact, considering $R(G_2) \subseteq C^{\infty}(G_2)$, this is nothing but the derivative $d_e \rho(X)|_{R(G_2)}$ of the representation $\rho: G_1 \to {\rm GL}(C^{\infty}(G_2))$ at the identity.

\begin{lemma}
For any $X \in \Fg_1$ and $f \in R(G_2)$, we have $X \rt f \in R(G_2)$. Moreover we have
\begin{equation}\label{aux-Db-equivariancy}
\D(X\rt f)= X \bullet \D(f).
\end{equation}
\end{lemma}

\begin{proof}
For any $\psi_1,\psi_2 \in G_2$, using ${(\rt )\circ \rm exp }={\rm exp}\circ {d_e\rt}$ we observe that
\begin{align}
\begin{split}
& (X \rt f)(\psi_1\psi_2) = \dt  f(\psi_1\psi_2 \lt {\rm exp}(tX))  \\
&= \dt  f((\psi_1 \lt (\psi_2 \rt {\rm exp}(tX)))(\psi_2 \lt {\rm exp}(tX))) \\
 &= \dt  f((\psi_1 \lt (\psi_2 \rt {\rm exp}(tX))\psi_2) + \dt f(\psi_1(\psi_2 \lt {\rm exp}(tX)))  \\
 & =\dt  (\psi_2 \cdot f)(\psi_1 \lt (\psi_2 \rt {\rm exp}(tX))) + \dt  (f \cdot \psi_1)(\psi_2 \lt {\rm exp}(tX))\\
 &=f\ps{2}(\psi_2)((\psi_2 \rt X) \rt f\ps{1})(\psi_1) + f\ps{1}(\psi_1)(X \rt f\ps{2})(\psi_2).
 \end{split}
\end{align}
As a result,
\begin{equation}
\psi_2 \cdot (X \rt f) =  f\ps{2}(\psi_2)(\psi_2 \rt X) \rt f\ps{1} + (X \rt f\ps{2})(\psi_2)f\ps{1}.
\end{equation}
Hence, we conclude that $\psi_2 \cdot (X \rt f) \in {\rm Span}\{X_i \rt f'_k, f'_k\},$ by writing
\begin{equation}
\Db_{\rm Gr}(X) = \sum_i X_i \ot g^i, \quad \D(f) = \sum_k f'_k \ot f''_k.
\end{equation}
We have proved that the left translations of $X \rt f$ span a finite dimensional vector space, and therefore by Proposition \ref{proposition-elements-of-R(G)} the element $X \rt f$ is a representative function.

\medskip

The $\Fg_1$-linearity  of $\D$ follows from
\begin{align}
\begin{split}
& f\ps{2}(\psi_2)((\psi_2 \rt X) \rt f\ps{1})(\psi_1) +  f\ps{1}(\psi_1)(X \rt f\ps{2})(\psi_2)\\
 &= (X\ns{0}\rt f\ps{1})(\psi_1) X\ns{1}(\psi_2)f\ps{2}(\psi_2)+  f\ps{1}(\psi_1)(X \rt f\ps{2})(\psi_2).
\end{split}
\end{align}
\end{proof}

\begin{proposition}
The map $\Fg_1 \otimes R(G_2) \to R(G_2)$ defined by \eqref{aux-action-for-lie-group} is a left action.
\end{proposition}

\begin{proof}
Using $\Ad\; \circ\; {\rm exp}= {\rm exp}\; \circ\; \ad$, we prove
\begin{align}
\begin{split}
& [X,Y] \rt f = \dt  {\rm exp}([tX,Y]) \rt f  =\dt  \Ad_{{\rm exp}(tX)}(Y) \rt f  \\
 & =\dt \ds  {\rm exp}(tX){\rm exp}(sY){\rm exp}(-tX) \rt f  \\
 &= \dt \ds {\rm exp}(tX){\rm exp}(sY) \rt f - \dt \ds {\rm exp}(sY){\rm exp}(tX) \rt f  \\
 &= X \rt (Y \rt f) - Y \rt (X \rt f).
 \end{split}
\end{align}
\end{proof}

\begin{lemma}\label{lemma-Lie-group-Bianchi}
The coaction \eqref{aux-coaction-for-lie-group} satisfies the structure identity of $\Fg_1$.
\end{lemma}

\begin{proof}
We shall prove that \eqref{aux-Bianchi} holds. Realizing the elements of $\Fg_1$ as local derivations  on $C^{\infty}(G_1)$, we have
\begin{align}
\begin{split}
& (\psi \rt [X_i,X_j])(\hat{f}) = \dt  \hat{f}( {\rm exp}(t(\psi \rt [X_i,X_j])))  \\
 & =\dt  \hat{f}(\psi \rt {\rm exp}(\ad(tX_i)(X_j)) \\
 &  =[\psi \rt X_i, \psi \rt X_j](\hat{f}) + \dt  ((\psi \lt {\rm exp}(tX_i)) \rt X_j)(\hat{f})  \\
 &  -\ds  ((\psi \lt {\rm exp}(sX_j)) \rt X_i)(\hat{f}),
\end{split}
\end{align}
for any $\psi \in G_2$, $\hat{f} \in C^{\infty}(G_1)$ and $X_i,X_j \in \Fg_1$. Hence,  we conclude that
\begin{align}\label{aux-str-id-0}
\begin{split}
& \psi \rt [X_i,X_j] = [\psi \rt X_i, \psi \rt X_j] + \dt  ((\psi \lt {\rm exp}(tX_i)) \rt X_j) \\
& -\ds  ((\psi \lt {\rm exp}(sX_j)) \rt X_i).
\end{split}
\end{align}
By the definition of the coaction we have
\begin{equation}\label{aux-str-id-1}
\psi \rt [X_i,X_j] = X_kC_{ij}^lf_l^k(\psi).
\end{equation}
Similarly,
\begin{equation}\label{aux-str-id-2}
 [\psi \rt X_i, \psi \rt X_j] = [X_rf_i^r(\psi), X_sf_j^s(\psi)] =  C_{rs}^kX_kf_i^r(\psi)f_j^s(\psi) = X_kC_{rs}^k(f_j^sf_i^r)(\psi).
\end{equation}
Finally
\begin{align}\label{aux-str-id-3}
\begin{split}
& \dt  ((\psi \lt {\rm exp}(tX_i)) \rt X_j) = \dt  X_kf_j^k(\psi \lt {\rm exp}(tX_i))  \\
 &= X_k (X_i \rt f_j^k)(\psi) = X_kf_{j,i}^k(\psi),
 \end{split}
\end{align}
and in the same fashion
\begin{equation}\label{aux-str-id-4}
\ds ((\psi \lt {\rm exp}(sX_j)) \rt X_i) = X_kf_{i,j}^k(\psi).
\end{equation}
The result follows from \eqref{aux-str-id-1} to \eqref{aux-str-id-4}, in view of \eqref{aux-str-id-0}.
\end{proof}

We are now ready for the main result of this subsection.

\begin{theorem}\label{theorem-R(G_2)-Lie-Hopf}
Let $(G_1, G_2)$ be a matched pair of Lie groups. Then by the action \eqref{aux-action-for-lie-group} and the coaction \eqref{aux-coaction-for-lie-group}, the pair  $(R(G_2),U (\Fg_1))$ is a matched pair of Hopf algebras.
\end{theorem}

\begin{proof}
In view of Theorem \ref{theorem-Lie-Hopf-matched-pair},  we need to prove that  the Hopf algebra $R(G_2)$ is a $\Fg_1$-Hopf algebra.

\medskip

Considering the Hopf algebra structure of $R(G_2)$, we see that
\begin{equation}
 \ve(X \rt f) = (X \rt f)(e) = \dt  f(e \lt {\rm exp}(tX)) =  \dt  f(e) = 0.
\end{equation}
By  Lemma \ref{lemma-Lie-group-Bianchi}, we know that $\Db_{\rm Gr}$  satisfies the structure identity of $\Fg_1$. On the other hand, the equation \eqref{aux-Db-equivariancy} proves that  $\Delta(X \rt f) = X \bullet \Delta(f)$.
\end{proof}

As a result, to any matched pair of Lie groups $(G_1, G_2)$, we associate a Hopf algebra
\begin{equation}
\Hc (G_1, G_2) := R(G_2) \acl U (\Fg_1).
\end{equation}

We proceed by providing a relation between the Hopf algebras $\mathcal{H} (G_1, G_2)$ and $\mathcal{H} (\Fg_1, \Fg_2)$. To this end, we first introduce a map
\begin{equation}\label{aux-map-R(G)to-R(g)}
\theta: R(G_2) \to R(\Fg_2), \qquad \pi \circ \rho \to \pi \circ d_e\rho,
\end{equation}
for any finite dimensional representation $\rho: G_2 \to {\rm GL}(V)$, and a linear functional $\pi: \End(V) \to \mathbb{C}$. Here we identify the representation $d_e\rho: \Fg_2 \to g\ell(V)$ of $\Fg_2$ and the unique algebra map $ d_e\rho:U(\Fg_2) \to g\ell(V)$ making the diagram
$$
\xymatrix {
\ar[d]_i \Fg_2 \ar[r]^{d_e\rho} & gl(V) \\
U(\Fg_2)\ar[ur]_{ d_e\rho}
}
$$
commutative.

\medskip

Let $G$ and $H$ be two Lie groups, where $G$ is  simply connected. Let also $\Fg$ and $\Fh$ be the corresponding Lie algebras, respectively. Then, a linear map $\sigma: \Fg \to \Fh$ is the differential of a map $\rho:G \to H$ of Lie groups if and only if it is a map of Lie algebras \cite{FultHarr-book}. Therefore, when $G_2$ is  simply connected, the map $\theta: R(G_2) \to R(\Fg_2)$ is bijective.

\medskip

We can express $\theta: R(G_2) \to R(\Fg_2)$ explicitly. The map $d_e\rho:U(\Fg_2) \to g\ell(V)$ sends $1 \in U(\Fg_2)$ to $\Id_V \in g\ell(V)$, hence for $f \in R(G_2)$
\begin{equation}
\theta(f)(1) = f(e).
\end{equation}
Since it is multiplicative, for any $\xi_1,\ldots,\xi_n \in \Fg_2$ we have
\begin{equation}\label{aux-map-theta}
\theta(f)(\xi^1\ldots\xi^n) = \dtone \ldots \dtn f({\rm exp}(t_1\xi_1)\ldots{\rm exp}(t_n\xi_n)).
\end{equation}

\begin{proposition}
The map
\begin{equation}\label{aux-map-Theta}
\Theta: \mathcal{H}(G_1, G_2) \to \mathcal{H}(\Fg_1, \Fg_2), \qquad \Theta(f \acl u)=\theta(f) \acl u
\end{equation}
is a map of Hopf algebras. Moreover, if  $G_2$ is simply connected, then $\mathcal{H}(G_1, G_2) \cong \mathcal{H}(\Fg_1, \Fg_2)$ as Hopf algebras.
\end{proposition}

\begin{proof}
First we show that \eqref{aux-map-Theta} is an algebra map. To this end, we need to prove that \eqref{aux-map-theta} is a map of $U(\Fg_1)$-module algebras. It is straightforward to observe that $\theta: R(G_2) \to R(\Fg_2)$ is a map of Hopf algebras.

\medskip

Let us now prove that $\theta$ is a $U(\Fg_1)$-module map. Indeed for any $X \in \Fg_1$ and any $\xi \in \Fg_2$,
\begin{align}\label{aux-theta-is-g-linear}
\begin{split}
& \theta(X \rt f)(\xi) = \dt\ds f({\rm exp}(s\xi) \lt {\rm exp}(tX))  \\
 & =\dt f({\rm exp}(t(\xi \lt X))) = X \rt \theta(f)(\xi).
 \end{split}
\end{align}
Next we prove  that the diagram
\begin{equation}\label{aux-diagram-R(G)-R(g)}
\xymatrix{  \ar[rrd]_{\Db_{\rm Alg}}\Fg_1\ar[rr]^{\Db_{\rm Gr}}&& \Fg_1\ot R(G_2)\ar[d]^{\theta}\\
&&    \Fg_1\ot R(\Fg_2).  }
\end{equation}
is commutative. Indeed, by evaluating on $\xi \in \Fg_2$ we have
\begin{align}
\begin{split}
 &\sum_jX^j\theta(g^i_j)(\xi) =\sum_j \dt X^jg^i_j({\rm exp}(t\xi)) \\
 &= \dt({\rm exp}(t\xi) \rt X^i) =  \xi \rt X^i =  \sum_jX^jf^i_j(\xi).
\end{split}
\end{align}
As a result, \eqref{aux-map-Theta} is a map of coalgebras. That \eqref{aux-map-Theta} commutes with the antipodes follows from \eqref{aux-antipode-bicrossed-product} and the fact that \eqref{aux-map-theta} is a map of Hopf algebras.
\end{proof}

In order to illustrate the theory, let us now discuss the projective Hopf algebra defined in the previous subsection from the Lie group decomposition point of view.

\subsubsection{Projective Hopf algebra via Lie group decomposition}

Let $\xi^0, \cdots ,\xi^n$ be a homogeneous coordinate system on $\mathbb{R}P^n$ with $\xi^0 \neq 0$. Then, a coordinate system $x^1, \ldots , x^n$ of $\Rb P^n$ can be defined by $x^i := \xi^i / \xi^0$. If $[s^{\alpha}_{\beta}]_{\alpha,\beta = 0, \ldots, n} \in \rm{GL}(n+1)$, then the induced projective transformation is given by the linear fractional transformation
\begin{equation}
x^i \mapsto y^i := \frac{s^i_0 + s^i_jx^j}{s^0_0 + s^0_jx^j} \qquad i=1, \ldots, n.
\end{equation}
If $s^0_0 \neq 0$, then setting
\begin{equation}\label{aux-local-coordinates}
u^i := s^i_0/s^0_0, \quad u^i_j := s^i_j/s^0_0, \quad u_j := s^0_j/s^0_0 \quad i,j = 1, \ldots n
\end{equation}
we obtain
\begin{equation}
x^i \mapsto y^i := \frac{u^i + u^i_jx^j}{1 + u_jx^j} \qquad i=1, \ldots, n.
\end{equation}
We shall take $(u^i, u^i_j, u_j)$,$1 \leq i,j \leq n$, as a local coordinate system around the identity of $\rm{PGL}(n,\mathbb{R})$.

\medskip

Let $o \in \Rb P^n$ be the point with homogeneous coordinates $[1,0,\ldots,0]$, and $H$ be the isotropy subgroup of ${\rm PGL}(n,\Rb)$ so that $\Rb P^n = {\rm PGL}(n,\Rb)/H$. In terms of the local coordinates $(u^i, u^i_j, u_j)$, $H$ is defined by $u^i = 0$, $1 \leq i \leq n$.

\medskip

The group $H$ contains ${\rm GL}(n)$ by $u_j = 0$, $1 \leq j \leq n$, and
\begin{equation}
N' = \{\psi = (0, u^i_j, u_j) \in H\, |\, u^i_j = \delta^i_j\} \cong \Rb^n
\end{equation}
is a subgroup of $H$. As a result, the subgroup of ${\rm PGL}(n,\Rb)$ generated by $H$ and the translations
\begin{equation}
x^i\mapsto u^i+x^i
\end{equation}
admits the decomposition
\begin{equation}
{\rm GL}(n)^{\rm aff} \cdot N'.
\end{equation}
Explicitly, if $\phi$ is the transformation
\begin{equation}
\phi^i:x \mapsto \frac{u^i + u^i_ax^a}{1 + u_ax^a}
\end{equation}
then we write $\phi = \vp \circ \psi$ where $\vp \in \rm{GL}(n)^{aff}$ is given by
\begin{equation}
\vp^i:x \mapsto u^i + \hat{u}^i_ax^a,
\end{equation}
and $\psi := \vp^{-1} \circ \phi$, \ie
\begin{equation}
\psi^i:x \mapsto \frac{x^i}{1+u_ax^a}.
\end{equation}
We first form the coaction $\Db_{\rm Gr}:g\ell(n)^{\rm aff} \to g\ell(n)^{\rm aff} \ot R(N')$. To this end, we need to determine the left action of the group $N'$ on the Lie algebra $g\ell(n)^{\rm aff}$.

\medskip

For $\psi \in N'$ and $X \in g\ell(n)^{\rm aff}$,
\begin{equation}
\psi \rt X := \dt \psi \rt exp(tX)
\end{equation}
Hence for $X_k \in g\ell(n)^{\rm aff}$, let
\begin{equation}\label{aux2}
\vp_t := \exp(tX_k) = (u(t); \Id_n) \in {\rm GL}(n)^{\rm aff} \quad \ie \quad \vp_t^i:x \mapsto u^i(t) + x^i.
\end{equation}
Then using the definition \eqref{aux-matched-pair-groups-action} of the mutual actions, which is repeated here as
\begin{equation}
\psi \circ \vp_t = (\psi \rt \vp_t) \circ (\psi \lt \vp_t),
\end{equation}
for $\psi = (0,\Id_n,u_j) \in N'$ we have
\begin{equation}
\psi \rt \vp_t = (\psi(u(t)); \psi'(u(t))) \in \rm{GL}(n)^{aff}
\end{equation}
where
\begin{equation}
\psi^i(u(t)) = \frac{u^i(t)}{1 + u_au^a(t)}
\end{equation}
and
\begin{equation}
\psi'(u(t))^i_j = \part_j\psi^i(u(t)) = \frac{\delta^i_j(1 + u_au^a(t)) - u^i(t)u_j}{(1 + u_au^a(t))^2}.
\end{equation}
Therefore,
\begin{align}
\begin{split}
& \psi \rt X_k = \dt (\psi \rt \vp_t) = \dt \psi^i(u(t)) X_i + \dt \psi'(u(t))^i_j X_i^j \\
& = X_k + \beta^i_{jk}(\psi)X_i^j,
\end{split}
\end{align}
where
\begin{equation}\label{aux-beta}
\beta^i_{jk}(\psi) = \dt \psi'(u(t))^i_j = -\delta^i_ju_k-\delta^i_ku_j.
\end{equation}
We note that as a result of \eqref{aux-beta} we have the equality
\begin{equation}
\beta^i_{jk} = \frac{1}{2}(\delta^i_j\beta^k_{kk} + \delta^i_k\beta^j_{jj}).
\end{equation}
Next, we exponentiate $X_p^q \in gl(n)^{\rm aff}$ as
\begin{equation}
\vp_t := \exp(tX_p^q) = (0; v(t)) \in {\rm GL}(n)^{\rm aff} \quad \ie \quad \vp_t^i:x \mapsto v^i_a(t)x^a.
\end{equation}
As before,
\begin{equation}
\psi \rt \vp_t = (\psi(0); \psi'(0)v(t)) \in \rm{GL}(n)^{aff},
\end{equation}
where
\begin{equation}
\psi^i(0) = 0
\end{equation}
and
\begin{equation}
\psi'(0)v(t) = v(t).
\end{equation}
Therefore,
\begin{equation}
\psi \rt X_p^q = \dt (\psi \rt \vp_t) = \dt v^i_j(t) X_i^j = X_p^q.
\end{equation}
Accordingly we have the coaction
\begin{align}
\begin{split}
& \Db_{\rm Gr}:g\ell(n)^{\rm aff} \to g\ell(n)^{\rm aff} \ot R(N') \\
& X_k \mapsto X_k \ot 1 + X_i^j \ot \beta^i_{jk} \\
& X_p^q \mapsto X_p^q \ot 1.
\end{split}
\end{align}
Let us next deal with the left action of $g\ell(n)^{\rm aff}$ on $R(N')$.  Using the expression in \eqref{aux2} for $X_l$, we find for $\psi = (0,\Id_n,u_j) \in N'$ that
\begin{equation}
(\psi \lt \vp_t)^i:x \mapsto \frac{x^i}{1 + \frac{u_a}{1+u_bu^b(t)}x^a}.
\end{equation}
As a result, for $\beta^i_{jk} \in R(N')$ we have
\begin{align}
\begin{split}
& (X_l \rt \beta^i_{jk})(\psi) = \dt \beta^i_{jk}(\psi \lt \vp_t) = \dt (-\delta^i_j\frac{u_k}{1+u_bu^b(t)} - \delta^i_k\frac{u_j}{1+u_bu^b(t)}) \\
& = \delta^i_ju_ku_l + \delta^i_ku_ju_l.
\end{split}
\end{align}
Denoting $\beta^i_{jkl} := X_l \rt \beta^i_{jk}$, we can observe as before that
\begin{equation}
\beta^i_{jkl} + \frac{1}{2}(\beta^i_{la}\beta^a_{jk} - \beta^i_{ka}\beta^a_{lj} - \beta^i_{ja}\beta^a_{kl}) = 0.
\end{equation}
On the other hand, starting from the exponentiation of $X_p^q \in g\ell(n)^{\rm aff}$, we have
\begin{equation}
(\psi \lt \vp_t)^i:x \mapsto \frac{x^i}{1 + u_av^a_b(t)x^b}.
\end{equation}
Therefore
\begin{align}
\begin{split}
& (X_p^q \rt \beta^i_{jk})(\psi) = \dt \beta^i_{jk}(\psi \lt \vp_t) = \dt (-\delta^i_ju_av^a_k(t) - \delta^i_ku_av^a_j(t)) \\
& = -\delta^i_j\delta^q_ku_p - \delta^i_k\delta^q_ju_p \\
& = \delta^q_j(-\delta^i_pu_k - \delta^i_ku_p) + \delta^q_k(-\delta^i_ju_p - \delta^i_pu_j) - \delta^i_p(-\delta^q_ju_k - \delta^q_ku_j) \\
& = \delta^q_j\beta^i_{pk}(\psi) + \delta^q_k\beta^i_{jp}(\psi) - \delta^i_p\beta^q_{jk}(\psi) = (\delta^q_j\beta^i_{pk} + \delta^q_k\beta^i_{jp} - \delta^i_p\beta^q_{jk})(\psi).
\end{split}
\end{align}
As a result,
\begin{equation}
X_p^q \rt \beta^i_{jk} = \delta^q_j\beta^i_{pk} + \delta^q_k\beta^i_{jp} - \delta^i_p\beta^q_{jk}.
\end{equation}
Finally we define the projective Hopf algebra (via the group decomposition) by
\begin{equation}
\Hc_{n\rm Proj} := R(N') \acl U(g\ell(n)^{\rm aff}).
\end{equation}
Using the notation $\beta_k := \beta^k_{kk}$, the Hopf algebra $R(N')$ is the polynomial Hopf algebra generated by $\{\beta_k\, |\, 1 \leq k \leq n\}$. We then write the coaction as
\begin{align}\label{coactionGrproj}
\begin{split}
&\Db_{\rm Gr}: g\ell(n)^{\rm aff} \to g\ell(n)^{\rm aff} \ot R(N') \\
& X_k \mapsto X_k \ot 1 + \frac{1}{2}\sum_a X_a^a \ot \beta_k + \frac{1}{2}X_k^a \ot \beta_a \\
& X_p^q \mapsto X_p^q \ot 1,
\end{split}
\end{align}
and the action as
\begin{equation}
X_l \rt \beta_k = \frac{1}{2}\beta_k\beta_l, \qquad X_i^j \rt \beta_k = \delta_k^j\beta_i.
\end{equation}
Let us now check that the coaction \eqref{coactionGrproj} is a Lie algebra map. To this end, we first observe that
\begin{align}
\begin{split}
& [\Db_{\rm Gr}(X_k), \Db_{\rm Gr}(X_l)] = \\
& [X_k \ot 1 + \frac{1}{2}\sum_a X_a^a \ot \beta_k + \frac{1}{2}X_k^a \ot \beta_a, X_l \ot 1 + \frac{1}{2}\sum_b X_b^b \ot \beta_l + \frac{1}{2}X_l^b \ot \beta_b] \\
& = [X_k \ot 1, X_l \ot 1] + \frac{1}{2}\sum_b[X_k \ot 1, X_b^b \ot \beta_k] + \frac{1}{2}\sum_b[X_k \ot 1, X_l^b \ot \beta_b] + \\
& \frac{1}{2}\sum_a[X_a^a \ot \beta_k, X_l \ot 1] + \frac{1}{4}\sum_{a,b}[X_a^a \ot \beta_k, X_b^b \ot \beta_l] + \frac{1}{4}\sum_{a,b}[X_a^a \ot \beta_k, X_l^b \ot \beta_b] + \\
& \frac{1}{2}\sum_a[X_k^a \ot \beta_a, X_l \ot 1] + \frac{1}{4}\sum_{a,b}[X_k^a \ot \beta_a, X_b^b \ot \beta_l] + \frac{1}{4}\sum_{a,b}[X_k^a \ot \beta_a, X_l^b \ot \beta_b] \\
& = 0 + \frac{1}{2}(-X_k \ot \beta_l + \frac{1}{2}\sum_bX_b^b \ot \beta_l\beta_k) + \frac{1}{2}(-X_l \ot \beta_k + \frac{1}{2}X_l^b \ot \beta_k\beta_b) + \\
& \frac{1}{2}(X_l \ot \beta_k - \frac{1}{2}\sum_aX_a^a \ot \beta_k\beta_l) + 0 + 0 + \frac{1}{2}(X_k \ot \beta_l - \frac{1}{2}X_k^b \ot \beta_a\beta_l) + \\
& 0 + \frac{1}{4}(X_k^b \ot \beta_l\beta_b - X_l^a \ot \beta_a\beta_k) \\
& = 0 = \Db_{\rm Gr}([X_k, X_l]).
\end{split}
\end{align}
Similarly,
\begin{align}
\begin{split}
& [\Db_{\rm Gr}(X_i^j), \Db_{\rm Gr}(X_k)] = [X_i^j \ot 1, X_k \ot 1 + \frac{1}{2}\sum_a X_a^a \ot \beta_k + \frac{1}{2}X_k^a \ot \beta_a] \\
& = [X_i^j \ot 1, X_k \ot 1] + \frac{1}{2}\sum_a[X_i^j \ot 1, X_a^a \ot \beta_k] + \frac{1}{2}[X_i^j \ot 1, X_k^a \ot \beta_a] \\
& = \delta_k^jX_i \ot 1 + \frac{1}{2}\delta_k^j\sum_aX_a^a \ot f_i +\frac{1}{2}(\delta_k^jX_i^a - \delta_i^aX_k^j) \ot f_a + \frac{1}{2}X_k^j \ot f_i \\
& =  \delta_k^j\Db_{\rm Gr}(X_i) = \Db_{\rm Gr}([X_i^j, X_k]),
\end{split}
\end{align}
and finally
\begin{equation}
[\Db_{\rm Gr}(X_i^j), \Db_{\rm Gr}(X_p^q)] = [X_i^j \ot 1, X_p^q \ot 1] = [X_i^j, X_k] \ot 1 = \Db_{\rm Gr}([X_i^j, X_k]).
\end{equation}
As a result, we can say that the coaction \eqref{coactionGrproj} is a Lie algebra map. In the same way as the Lie algebra decomposition case, we can verify that the coalgebra structure maps of $R(N')$ are $g\ell(n)^{\rm aff}$-linear.

\subsection{Matched pair of  Hopf algebras associated to matched pair of affine algebraic groups}

In this subsection we associate a bicrossed product Hopf algebra to any matched pair of affine algebraic groups.

\medskip

Let $G_1$ and $ G_2$ be two affine algebraic groups with mutual actions. In this case the actions
\begin{equation}
\rt: G_2 \times G_1 \to G_1,\qquad\lt: G_2 \times G_1 \to G_2,
\end{equation}
are assumed to be the  maps of affine algebraic sets, which by \cite[Chap 22]{TauvYu-book} means the existence of the maps
\begin{equation}
\mathscr{P}(\rt):\mathscr{P}(G_1) \to \mathscr{P}(G_2 \times G_1) = \mathscr{P}(G_2) \ot \mathscr{P}(G_1), \quad f \mapsto f{\ns{-1}} \ot f{\ns{0}},
\end{equation}
such that $f{\ns{-1}}(\psi) f{\ns{0}}(\varphi) = f(\psi \rt \varphi)$ and
\begin{equation}
 \mathscr{P}(\lt):\mathscr{P}(G_2) \to \mathscr{P}(G_2 \times G_1) = \mathscr{P}(G_2) \ot \mathscr{P}(G_1), \quad f \mapsto  f{\ns{0}} \ot f{\ns{1}},
\end{equation}
such that $ f{\ns{0}}(\psi) f{\ns{1}}(\varphi) = f(\psi \lt \varphi)$.

\medskip

We call the pair $(G_1,G_2)$ a matched pair of affine algebraic groups if the mutual actions satisfy \eqref{aux-matched-pair-group-actions-I} to \eqref{aux-matched-pair-group-actions-IV}.

\medskip

Next we  define the representations
\begin{equation}
f \lt \psi :=  f{\ns{-1}}(\psi)f{\ns{0}}, \qquad\text{and }\qquad \varphi \rt f := f{\ns{0}} f{\ns{1}}(\varphi)
\end{equation}
of $G_2$ on $\mathscr{P}(G_1)$ and  $G_1$ on $\mathscr{P}(G_2)$ respectively. We denote the action of $G_1$ on $\mathscr{P}(G_2)$ by $\rho$.

\medskip

As in the Lie group case, we define the action of $\Fg_1$ as the derivative of the action of $G_1$. Denoting the derivative of $\rho$ by $\rho^{\circ}$, by \cite{Hoch-book71} we have
\begin{equation}\label{aux-polynomial-action}
X \rt f := \rho^{\circ}(X)(f) =  f{\ns{0}}X(f{\ns{1}}).
\end{equation}

The action of $G_2$ on $\Fg_1$ is as before defined by
\begin{equation}
\psi \rt X := (L_{\psi})^{\circ}(X),
\end{equation}
where
\begin{equation}
L_{\psi}: G_1 \to G_1, \qquad \varphi \mapsto \psi \rt \varphi.
\end{equation}

We now prove that $\mathscr{P}(G_2)$ coacts on $\Fg_1$. Using the left action of $G_2$ on $\Fg_1$ we introduce the functions $f_i^j:G_2 \to \mathbb{C}$ exactly as before:
\begin{equation}\label{aux-polynomial-matrix-coefficients}
f_i^j(\psi) := <\psi \rt X_i, \theta^j>,
\end{equation}
that is
\begin{equation}
\psi \rt X_i =\sum_j f_i^j(\psi)X_j.
\end{equation}

\begin{lemma}
The functions $f_i^j:G_2 \to \mathbb{C}$ defined in \eqref{aux-polynomial-matrix-coefficients} are polynomial functions.
\end{lemma}

\begin{proof}
In view of Lemma 1.1 of \cite{Hoch-book}, there exists a basis $\{X_1, \cdots, X_n\}$ of $\Fg_1$ and a corresponding subset $S = \{f^1, \ldots, f^n\} \subseteq \mathscr{P}(G_1)$ such that $X_j(f^i) = \delta_j^i$. Now on the one hand,
\begin{equation}
(\psi \rt X_i)(f) = X_i(f \lt \psi) =  X_i(f{\ns{-1}}(\psi)f{\ns{0}}) =  f{\ns{-1}}(\psi)X_i(f{\ns{0}}),
\end{equation}
while on the other hand
\begin{equation}
(\psi \rt X_i)(f) = \sum_j f_i^j(\psi)X_j(f).
\end{equation}
For $f = f^k \in S$ we have $f_i^k(\psi) = (f^k){\ns{-1}}(\psi)X_i((f^k){\ns{0}}),$
that is,  $f_i^k = X_i((f^k){\ns{0}})(f^k){\ns{-1}} \in \mathscr{P}(G_2).$
\end{proof}

Finally, we define the coaction
\begin{align}\label{aux-polynomial-coaction}
\begin{split}
& \Db_{\rm Pol}:\Fg_1 \to \Fg_1 \ot \mathscr{P}(G_2),\qquad
 \Db_{\rm Pol}(X_i)= \sum_j X_j \ot f_i^j,
\end{split}
\end{align}
as well as the second order matrix coefficients
\begin{equation}
X_k \rt f_i^j = f_{i,k}^j.
\end{equation}

\begin{proposition}\label{proposition-polynomial-structure-identity}
The coaction defined by \eqref{aux-polynomial-coaction} satisfies the structure identity of $\Fg_1$.
\end{proposition}

\begin{proof}
We have to show \eqref{aux-Bianchi}. We first observe
\begin{align}
\begin{split}
& (f \lt \psi)\ps{1}(\varphi)(f \lt \psi)\ps{2}(\varphi') = (f \lt \psi)(\varphi \varphi') = f(\psi \rt \varphi \varphi')  \\
& = f\ps{1}(\psi \rt \varphi)f\ps{2}((\psi \lt \varphi) \rt \varphi')  \\
& = (f\ps{2}){\ns{-1}\ns{0}}(\psi)((f\ps{1} \lt \psi) \cdot (f\ps{2}){\ns{-1}\ns{1}})(\varphi)(f\ps{2}){\ns{0}}(\varphi'),
\end{split}
\end{align}
which implies that
\begin{equation}\label{aux-proof-polynomial}
 (f \lt \psi)\ps{1} \ot (f \lt \psi)\ps{2} =  (f\ps{2}){\ns{-1}\ns{0}}(\psi)(f\ps{1} \lt \psi) \cdot (f\ps{2}){\ns{-1}\ns{1}} \ot (f\ps{2}){\ns{0}}.
\end{equation}
Using \eqref{aux-proof-polynomial} we have
\begin{align}
\begin{split}
& (\psi \rt [X_i, X_j])(f) = [X_i, X_j](f \lt \psi) = (X_i \cdot X_j - X_j \cdot X_i)(f \lt \psi)  \\
&=  X_i((f \lt \psi)\ps{1})X_j((f \lt \psi)\ps{2}) - X_j((f \lt \psi)\ps{1})X_i((f \lt \psi)\ps{2})\\
  &  = (f\ps{2}){\ns{-1}\ns{0}}(\psi)[X_i(f\ps{1} \lt \psi)(f\ps{2}){\ns{-1}\ns{1}}(e_1)  \\
 & +(f\ps{1} \lt \psi)(e_1)X_i((f\ps{2}){\ns{-1}\ns{1}})]X_j((f\ps{2}){\ns{0}})  \\
 & - (f\ps{2}){\ns{-1}\ns{0}}(\psi)[X_j(f\ps{1} \lt \psi)(f\ps{2}){\ns{-1}\ns{0}}(e_1)  \\
  & + (f\ps{1} \lt \psi)(e_1)X_j((f\ps{2}){\ns{-1}\ns{1}})]X_i((f\ps{2}){\ns{0}})  \\
 &= [\psi \rt X_i, \psi \rt X_j](f) +  f{\ns{-1}\ns{0}}(\psi)X_i(f{\ns{-1}\ns{1}})X_j(f{\ns{0}})  \\
 & - f{\ns{-1}\ns{0}}(\psi)X_j(f{\ns{-1}\ns{1}})X_i(f{\ns{0}}).
\end{split}
\end{align}
We finally notice that
\begin{align}
\begin{split}
&  (f{\ns{-1}}X_j(f{\ns{0}}))(\psi) = f{\ns{-1}}(\psi)X_j(f{\ns{0}}) = X_j(f \lt \psi) = (\psi \rt X_j)(f) \\
& =\sum_k(X_kf_j^k(\psi))(f) =\sum_k (X_k(f)f_j^k)(\psi).
\end{split}
\end{align}
Hence,
\begin{align}
\begin{split}
&  f{\ns{-1}\ns{0}}(\psi)X_i(f{\ns{-1}\ns{1}})X_j(f{\ns{0}}) = (X_i \rt f{\ns{-1}})(\psi)X_j(f{\ns{0}})  \\
&= (X_i \rt f{\ns{-1}}X_j(f{\ns{0}}))(\psi) = \sum_k (X_k(X_i \rt f_j^k)(\psi))(f),
\end{split}
\end{align}
and similarly
\begin{equation}
f{\ns{-1}\ns{0}}(\psi)X_j(f{\ns{-1}\ns{1}})X_i(f^{\ns{0}}) = \sum_k(X_k(X_j \rt f_i^k)(\psi))(f).
\end{equation}
So we have observed that
\begin{equation}
\psi \rt [X_i, X_j] = [\psi \rt X_i, \psi \rt X_j] +\sum_k X_k(X_i \rt f_j^k)(\psi) -\sum_k X_k(X_j \rt f_i^k)(\psi),
\end{equation}
which immediately implies the structure equality.
\end{proof}

We now express the main result of this subsection.

\begin{theorem}\label{theorem-affine-algebraic-Lie-Hopf}
Let $(G_1, G_2)$ be a matched pair of affine algebraic groups. Then via the action \eqref{aux-polynomial-action} and the coaction \eqref{aux-polynomial-coaction}, $(\mathscr{P}(G_2),U(\Fg_1))$ is a matched pair of Hopf algebras.
\end{theorem}

\begin{proof}
In view of Theorem \ref{theorem-Lie-Hopf-matched-pair},  it is enough to show that $\mathscr{P}(G_2)$ is a $\Fg_1$-Hopf algebra. The structure identity follows from Proposition \ref{proposition-polynomial-structure-identity}.  Hence, we need to prove
\begin{equation}
\Delta(X \rt f) = X \bullet \Delta(f) \quad \mbox{ and } \quad \ve(X \rt f) = 0.
\end{equation}
First we observe that
\begin{align}
\begin{split}
&  f{\ns{0}}\ps{1}(\psi_1)f{\ns{0}}\ps{2}(\psi_2)f{\ns{1}}(\varphi) =  f{\ns{0}}(\psi_1\psi_2)f{\ns{1}}(\varphi) = \\
& f(\psi_1\psi_2 \lt \varphi) =  f\ps{1}(\psi_1 \lt (\psi_2 \rt \varphi))f\ps{2}(\psi_2 \lt \varphi) = \\
&  (f\ps{1}){\ns{0}}(\psi_1)((f\ps{1}){\ns{1}\ns{-1}} \cdot (f\ps{2}){\ns{0}})(\psi_2)((f\ps{1}){\ns{1}\ns{0}} \cdot (f\ps{2}){\ns{1}})(\varphi),
\end{split}
\end{align}
which implies
\begin{align}
\begin{split}
& f{\ns{0}}\ps{1} \ot f{\ns{0}}\ps{2} \ot f{\ns{1}}  \\
&= (f\ps{1}){\ns{0}} \ot (f\ps{1}){\ns{1}\ns{-1}} \cdot (f\ps{2}){\ns{0}} \ot (f\ps{1}){\ns{1}\ns{0}} \cdot (f\ps{2}){\ns{1}}.
\end{split}
\end{align}
Therefore,
\begin{align}
\begin{split}
& \D(X\rt f)(\psi_1,\psi_2)=(X \rt f)(\psi_1\psi_2) =  f{\ns{0}}\ps{1}(\psi_1)f{\ns{0}}\ps{2}(\psi_2)X(f{\ns{1}})  \\
&=  (f\ps{1}){\ns{0}}(\psi_1) (\psi_2 \rt X)((f\ps{1}){\ns{1}})f\ps{2}(\psi_2) +  f\ps{1}(\psi_1 \lt e)(X \rt f\ps{2})(\psi_2)  \\
&=  ((\psi_2 \rt X) \rt f\ps{1})(\psi_1)f\ps{2}(\psi_2) +  f\ps{1}(\psi_1)(X \rt f\ps{2})(\psi_2)  \\
&=  (X\ns{0} \rt f\ps{1})(\psi_1)(X\ns{1} \cdot f\ps{2})(\psi_2) +  f\ps{1}(\psi_1)(X \rt f\ps{2})(\psi_2)=\\
&=(X\bullet \D(f))(\psi_1,\psi_2).
\end{split}
\end{align}
 Next, we want to prove that $\ve(X \rt f) = 0$. To this end we  notice
\begin{equation}
\ve(X \rt f) = (X \rt f)(e_2) =  f{\ns{0}}(e_2)X(f{\ns{1}}) = X(f{\ns{0}}(e_2)f{\ns{1}})=0,
\end{equation}
since $f{\ns{0}}(e_2)f{\ns{1}} \in \mathscr{P}(G_1)$ is constant.
\end{proof}

\chapter{Hopf-cyclic coefficients}\label{chapter-hopf-cyclic-coefficients}

In this chapter we study the SAYD modules over the Lie-Hopf algebras in terms of the (co)representations of the relevant Lie group, Lie algebra or the affine algebraic group that gives rise to the Lie-Hopf algebra under consideration. More explicitly, we canonically associate a stable anti-Yetter-Drinfeld (SAYD) module to any representation of the Lie algebra, Lie group or the affine algebraic group that gives rise to the relevant Lie-Hopf algebra via semidualization. We call these coefficients the induced SAYD modules. We then take our arguments a step further by introducing the concept of corepresentation of a Lie algebra. This way, we associate to any representation and corepresentation of  a matched pair object (a Lie group, Lie algebra or an affine algebraic group), a SAYD module over the corresponding Lie-Hopf algebra. The case of trivial corepresentation reduces to the induced SAYD modules. Finally, we also discuss the AYD modules over the Connes-Moscovici Hopf algebras $\Hc_n$.

\section{Induced Hopf-cyclic coefficients}

In this section we study the SAYD modules that are associated to the representations of the ambient matched pair object, a double crossed sum Lie algebra, a double crossed product Lie group or affine algebraic group. We first consider the problem in the abstract setting of Lie-Hopf algebras. We associate a  canonical modular pair in involution to any bicrossed product Hopf algebra $\Fc\acl U(\Fg)$ that corresponds to a $\Fg$-Hopf algebra. We then study the SAYD modules over such bicrossed product Hopf algebras.

\subsection{Induced Hopf-cyclic coefficients over Lie-Hopf algebras}\label{subsection-induced hopf-cyclic}

Let us first define
\begin{equation}\label{aux-trace-adjoint}
\d_\Fg:\Fg\ra\Cb, \quad \d_\Fg(X)=\Tr({\rm ad}_X).
\end{equation}
Since \eqref{aux-trace-adjoint} is a derivation of the Lie algebra $\Fg$, we can extend it to an algebra map on $U(\Fg)$ that we also denote by $\d_\Fg$. Finally we extend \eqref{aux-trace-adjoint} to an algebra map
\begin{equation}\label{aux-delta}
\d:=\ve\acl \d_\Fg:\Fc\acl U(\Fg)\ra \Cb, \quad \d(f\acl u)=\ve(f)\d_\Fg(u).
\end{equation}
Let us introduce a canonical  element  $\s_\Fc\in\Fc$ in terms of the first order matrix coefficients \eqref{aux-first-order-matrix-coefficients} by
\begin{equation}\label{aux-sigma}
\s_\Fc:= \det [f^i_j]=\sum_{\pi\in S_m}(-1)^\pi f_1^{\pi(1)}\dots f_m^{\pi(m)}.
\end{equation}

\begin{lemma}\label{lemma-group-like}
The element $\s_\Fc\in \Fc$ is group-like.
\end{lemma}
\begin{proof}
We know that  $\D(f_i^j)=\sum_k f_k^j\ot f_i^k$.  So,
\begin{align}
\begin{split}
&\D(\s_\Fc)= \sum_{\pi\in S_m}\sum_{i_1,\dots, i_m}(-1)^\pi f_{i_1}^{\pi(1)}\cdots f_{i_m}^{\pi(m)} \ot f_1^{i_1}\cdots f_m^{i_m}\\
&=\sum_{\pi\in
S_m}\sum_{\underset{\text{  \footnotesize are  distinct }}{i_1,\dots, i_m}}(-1)^\pi f_{i_1}^{\pi(1)}\cdots f_{i_m}^{\pi(m)} \ot f_1^{i_1}\cdots f_m^{i_m}\\
&+\sum_{\pi\in S_m}\sum_{\underset{\text{  \footnotesize are not  distinct }}{i_1,\dots, i_m}}(-1)^\pi f_{i_1}^{\pi(1)}\cdots f_{i_m}^{\pi(m)} \ot f_1^{i_1}\cdots
f_m^{i_m},
\end{split}
\end{align}
where the second sum is zero since $\Fc$ is commutative. In order to deal with the  first sum we associate a unique permutation $\mu\in S_m$ to each $m$-tuple $(i_1,\dots,i_m)$ with distinct entries by the rule $\mu(j)=i_j$.  So we have
\begin{align}
\begin{split}
&\D(\s_\Fc)= \sum_{\pi\in S_m} \sum_{\mu\in S_m}(-1)^\pi f_{\mu(1)}^{\pi(1)}\cdots f_{\mu(m)}^{\pi(m)} \ot f_1^{\mu(1)}\cdots f_m^{\mu(m)}\\
 &= \sum_{\pi\in S_m}
\sum_{\mu\in S_m}(-1)^\pi(-1)^{\mu}(-1)^{\mu^{-1}} f_{1}^{\pi(\mu^{-1}(1))}\cdots f_{m}^{\pi(\mu^{-1}(m))} \ot f_1^{\mu(1)}\cdots f_m^{\mu(m)}\\ &
=\sum_{\eta,\mu\in S_m} (-1)^\eta(-1)^{\mu} f_{1}^{\eta(1)}\cdots f_{m}^{\eta(m)} \ot f_1^{\mu(1)}\cdots f_m^{\mu(m)}=\s_\Fc\ot\s_\Fc.
\end{split}
\end{align}
\end{proof}

It is immediate to see that $\s:=\s_\Fc\acl 1$ is a group-like element in the Hopf algebra  $\Fc\acl U(\Fg)$.

\begin{theorem}\label{theorem-MPI}
For any $\Fg$-Hopf algebra  $\Fc$, the pair $(\delta, \sigma)$ is a modular  pair in involution  for the Hopf algebra $\Fc\acl U(\Fg)$.
\end{theorem}
\begin{proof}
For an element $1 \acl X_i \in \Hc:=\Fc\acl U(\Fg)$,  the action of iterated comultiplication $\D^{(2)}$ is calculated by
\begin{align}
\begin{split}
&\D^{(2)}(1 \acl X_i) =  (1 \acl X_i)\ps{1} \ot (1 \acl X_i)\ps{2} \ot (1 \acl X_i)\ps{3} \\
 & =  1 \acl X_i\ns{0} \ot X_i\ns{1} \acl 1 \ot X_i\ns{2} \acl 1  \\
 &+1 \acl 1 \ot 1 \acl X_i\ns{0} \ot X_i\ns{1} \acl 1   + 1 \acl 1 \ot 1 \acl1 \ot 1 \acl X_i.
\end{split}
\end{align}
By definition of the antipode \eqref{aux-antipode-bicrossed-product}, we observe that
\begin{align}
\begin{split}
&S(1 \acl X_i)  =  (1 \acl S(X_i\ns{0})) (S(X_i\ns{1}) \acl 1) =  - (1 \acl X_j)  (S(f_i^j) \acl 1) \\
 & =  - X_j\ps{1} \rt S(f_i^j) \acl X_j\ps{2}  =  - X_j \rt S(f_i^j) \acl 1 - S(f_i^j) \acl  X_j,
\end{split}
\end{align}
and hence
\begin{align}
\begin{split}
&S^2(1 \acl  X_i)  = - S(X_j \rt S(f_i^j)) \acl  1 - (1 \acl  S(X_j\ns{0})) \cdot (S(S(f_i^j)X_j\ns{1}) \acl  1) \\
 & =  - S(X_j\ns{1})(X_j\ns{0} \rt f_i^j) \acl  1 + (1 \acl  X_k) \cdot (f_i^jS(f_j^k) \acl  1) \\
  &=  - S(f_j^k)(X_k \rt f_i^j) \acl  1 + 1 \acl  X_i   =  - S(f_j^k)f_{i,k}^j \acl  1 + 1 \acl  X_i.
\end{split}
\end{align}
Finally, for the twisted antipode $S_{\delta}:\cal H \to \cal H$ defined in \eqref{aux-twisted-antipode}, we  simplify   its square action by
\begin{equation}
  S_{\delta}^2(h) =  \delta(h\ps{1})\delta(S(h\ps{3}))S^2(h\ps{2}), \quad h\in \Hc.
\end{equation}
We aim to prove that
\begin{equation}\label{aux-MPI-equation}
S_{ \delta}^2 = \Ad_{\sigma}.
\end{equation}
Since the twisted antipode is an anti-algebra map, it is enough to  prove that \eqref{aux-MPI-equation} holds for the elements of the form $1 \acl X_i$ and $f\acl 1$. For the latter elements, we have  $S_\d(f\acl 1)=S(f)\acl 1$. Hence the claim is established since $S^2(f)=f=\s f\s^{-1}$.

\medskip

According to the above preliminary calculations, we have
\begin{equation}\label{aux-twisted-antipode-square-action}
S_{ \delta}^2(1 \acl X_i)  =   \delta_\Fg(X_j)(f_i^j \acl 1) + (1 \acl X_i) - S(f_j^k)f_{i,k}^j \acl 1 - \delta_\Fg(X_i)1 \acl 1.
\end{equation}
Multiplying both hand sides of  the structure  identity \eqref{aux-Bianchi} by $S(f_j^k)$, we have
\begin{align}\notag
&- S(f_j^k)f_{k,i}^j =  - S(f_j^k)f_{i,k}^j + \sum_{r,s}C_{sr}^jS(f_j^k)f_k^rf_i^s - \sum_lC_{ik}^lS(f_j^k)f_l^j \\\label{aux5}
 & =  - S(f_j^k)f_{i,k}^j + \sum_{s,j}C_{sj}^jf_i^s - \sum_lC_{il}^l  1_\Fc  =  - S(f_j^k)f_{i,k}^j + \delta_\Fg(X_s)f_i^s - \delta_\Fg(X_i) 1_\Fc.
\end{align}
Combining \eqref{aux5} and  \eqref{aux-twisted-antipode-square-action} we get
\begin{equation}\label{aux7}
S_{\delta}^2(1 \acl X_i) = - S(f_j^k)f_{k,i}^j \acl 1 + 1 \acl X_i.
\end{equation}
On the other hand, since $\Fg$ acts on $\Fc$ by derivations, we see that
\begin{align}\notag
&0  =  X_i \rt (S(f_j^k)f_k^j)  =  f_k^j (X_i \rt S(f_j^k)) + S(f_j^k)(X_i \rt f_k^j) \\\label{aux8}
& =  f_k^j(X_i \rt S(f_j^k)) + S(f_j^k)f_{k,i}^j.
\end{align}
From  \eqref{aux8} and \eqref{aux7}  we deduce  that
\begin{equation}\label{aux9}
S_{\delta}^2(1 \acl X_i) = f_k^j(X_i \rt S(f_j^k)) \acl 1 + 1 \acl X_i.
\end{equation}
Now we  consider the element
\begin{equation}
  \sigma^{-1} = \det[S(f_j^k)] = \sum_{\pi \in S_m}(-1)^{\pi}S(f_1^{\pi(1)})S(f_2^{\pi(2)}) \dots S(f_m^{\pi(m)}).
\end{equation}
Using once again the fact that   $\Fg$ acts on $\Fc$   by  derivation we observe that
\begin{align}\notag
&X_i \rt \sigma^{-1} = X\rt\det[S(f_j^k)] =\\\label{aux10}
  &\sum_{1\le j\le m,\;\pi \in S_m}(-1)^{\pi}S(f_1^{\pi(1)})S(f_2^{\pi(2)}) \cdots\; X_i \rt S(f_j^{\pi(j)})\;\cdots S(f_m^{\pi(m)}).
\end{align}
Since   $\s=\det[f_i^j]$ and $\Fc$ is commutative, we further observe that $\s=\det [f_i^j]^T$. Here  $[f_i^j]^T$ denotes the  transpose of the matrix $[f_i^j]$. We can then conclude that
\begin{equation}
\sigma\left(\sum_{\pi \in S_m}(-1)^{\pi}[X_i \rt S(f_1^{\pi(1)})]S(f_2^{\pi(2)}) \dots S(f_m^{\pi(m)})\right) = f_k^1(X_i \rt S(f_1^k)),
\end{equation}
which  implies
\begin{equation}\label{aux11}
\sigma(X_i \rt \sigma^{-1}) = f_k^j(X_i \rt S(f_j^k)).
\end{equation}
Finally,
\begin{equation}
{\rm Ad}_{\sigma}(1 \acl X_i) =  (\sigma \acl 1)(1 \acl X_i)(\sigma^{-1} \acl 1)= \sigma X_i \rt \sigma^{-1} \acl 1 + 1 \acl X_i,
\end{equation}
followed  by the substitution of \eqref{aux11} in \eqref{aux9}  finishes the proof.
\end{proof}

In order to obtain a SAYD  module over a bicrossed product Hopf algebra $\Hc:=\Fc\acl U(\Fg)$, we first find a YD module over $\Hc$, then we tensor it with the canonical modular pair in involution and finally we use \cite[Lemma 2.3]{HajaKhalRangSomm04-I}.

\begin{definition}
Let $V$ be a left   $\Fg$-module and a  right $\Fc$-comodule via $\Db_V:V\ra V\ot \Fc$. We say that $V$ is an induced $(\Fg,\Fc)$-module if
\begin{equation}\label{definition-induced-module}
  \Db_V(X\cdot v)= X\bullet \Db_V(v),
\end{equation}
where, $X\bullet (v\ot f)= X\ns{0}v\ot X\ns{1}f+ v\ot X\rt f$.
\end{definition}

\begin{lemma}
If $V$ is an induced $(\Fg,\Fc)$-module, then the coaction $\Db_V:V\ra \Fc\ot V$ is $U(\Fg)$-linear in the sense that, for any $u \in U(\Fg)$,
\begin{equation}\label{aux-coaction-U-linear}
\Db_V(u\cdot v)= u\bullet \Db_V(v):= u\ps{1}\ns{0}\cdot v\ns{0}\ot u\ps{1}\ns{1}(u\ps{2}\rt v\ns{1}).
\end{equation}
\end{lemma}

\begin{proof}
For any $X\in \Fg$, the condition \eqref{aux-coaction-U-linear} is obviously satisfied. Let us assume that it is satisfied for $u^1, u^2\in U(\Fg)$ and $v\in V$, and show that it also holds for $u^1u^2$ and any $v\in V$. Using \eqref{aux-matched-pair-4} we observe that
\begin{align}
\begin{split}
&\Db_V(u^1u^2\cdot v)= u^1\ps{1}(u^2\cdot v)\ns{0}\ot u^1\ps{1}\ns{1}(u^1\ps{2}\rt (u^2\cdot v)\ns{1})\\
& = u^1\ps{1}\ns{0} u^2\ps{1}\ns{0}\cdot v\ns{0}\ot u^1\ps{1}\ns{1} (u^1\ps{2}\rt (u^2\ps{1}\ns{1}(u^2\ps{1}\rt v\ns{1})))\\
& =u^1\ps{1}\ns{0} u^2\ps{1}\ns{0} \cdot v\ns{0} \ot u^1\ps{1}\ns{1} u^1\ps{2}\rt u^2\ps{1}\ns{1}( u^1\ps{3}u^2\ps{1}\rt v\ns{1})\\
& =(u^1u^2)\ps{1}\ns{0}\cdot v\ns{0} \ot (u^1u^2)\ps{1}\ns{1} (u^1u^2)\ps{2}\rt v\ns{1}.
\end{split}
\end{align}
\end{proof}

Now we let $\Hc$ act on $M$ via
\begin{equation}\label{aux-left-action-H-on-M}
 \Hc\ot V \ra V, \quad (f\acl u)\cdot v=\ve(f)u\cdot v.
\end{equation}
On the other hand, since $\Fc$ is a Hopf subalgebra of $\Hc$ we can extend the right coaction of $\Fc$ on $V$ to a right coaction of $\Hc$ on $V$ by
\begin{equation}\label{aux-right-coaction-H-on-M}
\Db_V: V\ra  V\ot \Hc, \quad \Db_V(v)= v\ns{0}\ot v\ns{1}\acl 1.
\end{equation}

\begin{proposition}
Let $\Fc$ be a $\Fg$-Hopf algebra and  $V$ an induced $(\Fg, \Fc)$-module. Then via the action and coaction  defined in \eqref{aux-left-action-H-on-M} and \eqref{aux-right-coaction-H-on-M},  $V$ is a left-right YD-module over $\Fc\acl U(\Fg)$.
\end{proposition}
\begin{proof}
Since the condition \eqref{aux-left-right-YD-module} is multiplicative, it suffices to check the condition for
$f\acl 1$ and $1\acl u$. As it is obviously satisfied for the elements of the form $f \acl 1$, we check it for $1 \acl u$.

\medskip

We see that $\D(1\acl u)= 1\acl u\ps{1}\ns{0}\ot u\ps{1}\ns{1}\acl u\ps{2}$. Using \eqref{aux-coaction-U-linear} and the fact that $U(\Fg)$ is cocommutative, we observe for $h= 1\acl u$ that
\begin{align}
\begin{split}
& (h\ps{2}\cdot v)\ns{0}\ot (h\ps{2}\cdot v)\ns{1}h\ps{1}\\
&= ((u\ps{1}\ns{1}\acl u\ps{2})\cdot v)\ns{0}\ot ((u\ps{1}\ns{1}\acl u\ps{2})\cdot v)\ns{1}1\acl u\ps{1}\ns{0}\\
&= u\ps{2}\ns{0}\cdot v\ns{0}\ot u\ps{2}\ns{1}(u\ps{3}\rt v\ns{1})\acl u\ps{1}\\
&= h\ps{1}\cdot v\ns{0}\ot h\ps{2}v\ns{1}.
\end{split}
\end{align}
\end{proof}

It is known that by composing $S$ with the action and  $S^{-1}$with the coaction, we turn a left-right YD module over $\Hc$ into a right-left YD module over $\Hc$. Since the coaction always lands in $\Fc$ and $S^{-1}(f\acl 1)= S^{-1}(f)\acl 1= S(f)\acl 1$,  we  conclude that
\begin{equation}\label{right-action-M}
 V \ot \Hc\ra V, \quad v \cdot (f\acl u)=\ve(f)S(u)\cdot v,
\end{equation}
and
\begin{equation}
\Db_V: V\ra \Hc\ot  V, \quad \Db_V(v)= S(v\ns{1})\acl 1\ot v\ns{0},
\end{equation}
define a right-left YD module over $\Hc$.

\medskip

Since any modular pair in involution (MPI) determines a one dimensional AYD module,  $^\s\Cb_\d\ot V$ is an AYD module over $\Fc\acl U(\Fg)$.  We denote $^\s\Cb_\d\ot V$ by $^\s{V}_\d$. Finally, we explicitly write the action
\begin{equation}\label{aux-action-SAYD-over-Lie-Hopf}
^\s{V}_\d\ot \Fc\acl U(\Fg)\ra\; ^\s{V}_\d,\quad v\lt_{(\d,\s)}(f\acl u):= \ve(f) \d(u\ps{1}) S(u\ps{2})\cdot v
 \end{equation}
and the coaction
\begin{align}\label{aux-coaction-SAYD-over-Lie-Hopf}
\begin{split}
&^\s\Db_\d: \; ^\s{V}_\d\ra  \Fc\acl U(\Fg)\ot\; ^\s{V_\d},\\
&^\s\Db_\d(v):=  \s S(v\ns{1})\acl 1\ot v\ns{0}
\end{split}
\end{align}
of $\Fc\acl U(\Fg)$. Moreover, we can immediately verify that $^\s{V}_\d$ is stable.

\medskip

The following theorem summarizes this subsection.

\begin{theorem}\label{theorem-SAYD-over-Lie-Hopf}
For any $\Fg$-Hopf algebra $\Fc$ and any induced $(\Fg, \Fc)$-module $V$, there is a  SAYD module structure on $^\s{V}_\d:= \; ^\s\Cb_\d\ot V$ over $\Fc\acl U(\Fg)$ defined by \eqref{aux-action-SAYD-over-Lie-Hopf} and \eqref{aux-coaction-SAYD-over-Lie-Hopf}. Here,  $(\d,\s)$ is the canonical modular pair in involution associated to $(\Fg,\Fc)$.
\end{theorem}

\subsection{Induced Hopf-cyclic coefficients of geometric noncommutative Hopf algebras}\label{subsection-induced-SAYD}

In this subsection, we associate to any representation of a double crossed sum Lie algebra, double crossed product Lie group or a double crossed product affine algebraic group a SAYD module over a geometric noncommutative Hopf algebra that corresponds to the relevant matched pair object.

 \medskip

We start with the $\Fg_1$-Hopf algebra $R(\Fg_2)$, where $(\Fg_1,\Fg_2)$ is a matched pair of Lie algebras. A left module $V$ over the double crossed sum Lie algebra $\Fg_1\bi\Fg_2$ is naturally a left  $\Fg_1$-module as well as a left $\Fg_2$-module. In addition, it satisfies
\begin{equation}\label{Lie-module-compatibility}
\z\cdot(X\cdot v)- X\cdot (\z\cdot v)= (\x\rt X)\cdot v+ (\z\lt X)\cdot v.
\end{equation}
Conversely, if $V$ is a left module over $\Fg_1$ and $\Fg_2$ satisfying \eqref{Lie-module-compatibility}, then $V$ is a $\Fg_1\bi\Fg_2$-module with
\begin{equation}
(X\oplus \z)\cdot v:= X\cdot v+\z\cdot v.
\end{equation}

This is generalized in the following lemma.

\begin{lemma}\label{lemma-U-V-module}
Let $\Uc$ and $\Vc$ be a mutual pair of Hopf algebras and  $V$ be a left module over both Hopf algebras. Then $V$ is a left $\Uc \bi \Vc$-module via
 \begin{equation}
 (u \bi u')\cdot v:=u \cdot (u' \cdot v),
  \end{equation}
  if and only if
  \begin{equation}\label{Hopf-module-compatibility}
  u'\cdot(u \cdot v)= (u'\ps{1}\rt u\ps{1})\cdot ((u'\ps{2}\lt u\ps{2}) \cdot v),
   \end{equation}
    for any $u \in \Uc$, $u'\in \Vc$ and $v \in V$.
\end{lemma}

\begin{proof}
Let $V$ be a left $\Uc \bi \Vc$-module. Then by the associativity of the action, we have
\begin{align}
\begin{split}
& u' \cdot (u \cdot v) = (1 \bi u') \cdot ((u \bi 1) \cdot v) = ((1 \bi u') \cdot (u \bi 1)) \cdot v\\
& = ((u'\ps{1}\rt u\ps{1}) \bi (u'\ps{2}\lt u\ps{2})) \cdot v\\
& = (u'\ps{1}\rt u\ps{1})\cdot ((u'\ps{2}\lt u\ps{2}) \cdot v).
\end{split}
\end{align}
Conversely, if \eqref{Hopf-module-compatibility} is satisfied, then we conclude that
\begin{equation}
(1 \bi u') \cdot ((u \bi 1) \cdot v) = ((1 \bi u')\cdot (u \bi 1)) \cdot v,
\end{equation}
and hence the action is associative.
\end{proof}

Now let $V$ be a left $\Fg:=\Fg_1\bi\Fg_2$-module such that the action of $\Fg_2$ is  locally finite \cite{Hoch-book}.  Then $V$ is left module over $U(\Fg_1)$ and $U(\Fg_2)$ and the resulting actions satisfy the condition \eqref{Hopf-module-compatibility}.

\medskip

We know that the category of locally finite $\Fg_2$ modules is equivalent with the category of $R(\Fg_2)$-comodules \cite{Hoch-book} via the functor that assigns the right $R(\Fg_2)$-comodule $V$ with the coaction $\Db_V:V\ra V\ot R(\Fg_2)$ to a locally finite left $U(\Fg_2)$-module $V$ via
\begin{equation}\label{aux-U(g)-dual-coaction}
\Db(v)= v\ns{0}\ot v\ns{1}, \quad \text{if } \quad  u'\cdot v= v\ns{1}(u')v\ns{0}.
\end{equation}

Conversely, if $V$ is a right $R(\Fg_2)$-comodule via the coaction $\Db_V:V\ra V\ot R(\Fg_2)$, then we define the left action of $U(\Fg_2)$ on $V$ by
\begin{equation}\label{aux-action-coaction}
u'\cdot v:=  v\ns{1}(u')v\ns{0}.
\end{equation}

\begin{proposition}\label{proposition-module-induced-module-Lie-algebras}
Let $V$ be a left $\Fg:=\Fg_1\bi\Fg_2$-module such that the $\Fg_2$-action is locally finite. Then as a $\Fg_1$-module by restriction, and a $R(\Fg_2)$-comodule by the coaction defined in  \eqref{aux-U(g)-dual-coaction}, $V$ becomes an induced $(\Fg_1,R(\Fg_2))$-module.  Conversely, every induced $(\Fg_1, R(\Fg_2))$-module arises in this way.
\end{proposition}

\begin{proof}
Let $V$ satisfy  the criteria of the proposition. We prove that it is an induced $(\Fg_1, R(\Fg_2))$-module, \ie  $\Db_V(X\cdot v)=X\bullet \Db_V(v)$. Using the compatibility condition \eqref{Hopf-module-compatibility} for $u'\in U(\Fg_2)$ and $X \in \Fg_1$, we observe that
\begin{equation}\label{proof-compatibility}
(u'\ps{1}\rt X)\cdot (u'\ps{2} \cdot v)+ (u'\lt X)\cdot v= u'\cdot (X\cdot v).
\end{equation}
Translating \eqref{proof-compatibility} via \eqref{aux-U(g)-dual-coaction}, we observe that
\begin{align}\label{proof-label-2}
\begin{split}
& (X\bullet \Db_V(v))(u')= \\
&  (X\ns{0}\cdot v\ns{0})\ot (X\ns{1}v\ns{1})(u')+  v\ns{0}\ot (X\rt v\ns{1})(u') =\\
&   X\ns{1}(u'\ps{1})v\ns{1}(u'\ps{2}) X\ns{0}\cdot v\ns{0}+ (X\rt v\ns{1})(u') v\ns{0} =\\
&   v\ns{1}(u'\ps{2}) (u'\ps{1}\rt X)\cdot v\ns{0}+ (X\rt v\ns{1})(u') v\ns{0} =\\
&  (u'\ps{1}\rt X)\cdot(u'\ps{2}\cdot v)  + v\ns{0}v\ns{1}(u'\lt X) =\\
&  (u'\ps{1}\rt X)\cdot (u'\ps{2}\cdot v) + (u'\lt X)\cdot v = u'\rt (X\cdot v) =\\
& (X \cdot v)\ns{0} (X\cdot v)\ns{1}(u')= \Db_V(X\cdot v)(u').
\end{split}
\end{align}
Conversely, from  a comodule over $R(\Fg_2)$ one obtains  a locally finite module over $\Fg_2$ by  \eqref{aux-action-coaction}. One shows that the compatibility \eqref{Lie-module-compatibility} follows from $\Db_V(X\cdot v)=X\bullet \Db_V(v)$ via   \eqref{proof-label-2}.
\end{proof}

Next we investigate the same correspondence  for a matched pair of Lie groups.
Let $(G_1, G_2)$ be a matched pair of Lie groups. Then $(\mathbb{C}[G_1], \mathbb{C}[G_2])$ is a mutual pair of Hopf algebras, where $\Cb [G_1]$ and $\Cb[ G_2]$ are the group algebras of $G_1$ and $G_2$ respectively. In addition, it is straightforward to see that  for $G=G_1 \bi G_2$ we have $\mathbb{C}[G] = \mathbb{C}[G_1] \dcp \mathbb{C}[G_2]$. Therefore, as indicated by Lemma \ref{lemma-U-V-module}, a  module $V$ over  $G_1$ and $G_2$  is a module over $G$ with
\begin{equation}\label{aux-G1G2module}
(\varphi \cdot \psi) \cdot v = \varphi \cdot (\psi \cdot v), \qquad \varphi \in G_1, \psi \in G_2
\end{equation}
if and only if
\begin{equation}\label{aux-compatibility-for-groups}
(\psi \rt \varphi) \cdot ((\psi \lt \varphi) \cdot v) = \psi \cdot (\varphi \cdot v).
\end{equation}
Now let $\Fg_1$ and $\Fg_2$ be the corresponding Lie algebras of the Lie groups $G_1$ and $G_2$ respectively. We define the $\Fg_1$-module structure of $V$ by
\begin{equation}\label{aux-g1action}
X \cdot v = \dt \exp(tX) \cdot v
\end{equation}
for any $X \in \Fg_1$ and $v \in V$.

\medskip

Assuming $V$ to be locally finite as a left $G_2$-module, we define the right coaction $\nb:V \to V \ot R(G_2)$ by
\begin{equation}\label{aux-RG2coaction}
\nb(v) = v\ns{0} \ot v\ns{1} \qquad \mbox{ if and only if } \quad \psi \cdot v = v\ns{1}(\psi)v\ns{0}.
\end{equation}
In this case, we can express an infinitesimal version of the compatibility condition (\ref{aux-compatibility-for-groups}) as
\begin{equation}\label{auxinfinitesimal-compatibility}
(\psi \rt X) \cdot (\psi \cdot v) + \psi \cdot v\ns{0}(X \rt v\ns{1})(\psi) = \psi \cdot (X \cdot v),
\end{equation}
for any $\psi \in G_2$, $X \in \Fg_1$ and $v \in V$.

\medskip

Now we state the following  proposition whose proof is similar to that of Proposition \ref{proposition-module-induced-module-Lie-algebras}.

\begin{proposition}\label{proposition-module-induced-module-Lie groups}
For a matched pair of Lie groups $(G_1,G_2)$, let $V$ be a left $G = G_1 \bi G_2$-module under \eqref{aux-G1G2module} such that $G_2$-action is locally finite. Then with the $\Fg_1$-action \eqref{aux-g1action} and the $R(G_2)$-coaction \eqref{aux-RG2coaction}, $V$ becomes an induced $(\Fg_1, R(G_2))$-module. Conversely, every induced $(\Fg_1, R(G_2))$-module arises in this way.
\end{proposition}

We conclude with the case of affine algebraic groups. Let $(G_1, G_2)$ be a matched pair of affine algebraic groups and $V$ a locally finite polynomial representation of $G_1$ and $G_2$. Let us write the dual comodule structures as
\begin{equation}
\bD_1:V \to V \ot \mathscr{P}(G_1), \qquad v \mapsto v\ps{0} \ot v\ps{1},
\end{equation}
and
\begin{equation}\label{aux-affine-PG2-coaction}
\bD_2:V \to V \ot \mathscr{P}(G_2), \qquad v \mapsto v\ns{0} \ot v\ns{1}.
\end{equation}
On the other hand, if we call the $G_1$-module structure
\begin{equation}
\rho:G_1 \to {\rm GL}(V),
\end{equation}
then we have the $\Fg_1$-module structure
\begin{equation}\label{aux-affineg1action}
\rho^{\circ}:\Fg_1 \to g\ell(V).
\end{equation}
The compatibility condition is the same as the Lie group case and we now present the main result.

\begin{theorem}
For a matched pair $(G_1,G_2)$ of affine algebraic groups, let $V$ be a locally finite polynomial $G = G_1 \bi G_2$-module. Then via the action \eqref{aux-affineg1action} and the coaction \eqref{aux-affine-PG2-coaction}, $V$ becomes an induced $(\Fg_1, \mathscr{P}(G_2))$-module. Conversely, any induced $(\Fg_1, \mathscr{P}(G_2))$-module arises in this way.
\end{theorem}

\begin{proof}
For any $\psi \in G_2$, $X \in \Fg_1$ and $v \in V$
\begin{align}\label{aux-nabla2-X-m-psi}
\begin{split}
\bD_2(X \cdot v)(\psi) = \bD_2(v\ps{0})(\psi)X(v\ps{1}) = v\ps{0}\ns{0}v\ps{0}\ns{1}(\psi)X(v\ps{1}).
\end{split}
\end{align}
On the other hand, for any $\vp \in G_1$,
\begin{align}
\begin{split}
& v\ps{0}\ns{0}v\ps{0}\ns{1}(\psi)v\ps{1}(\vp) = (\vp \cdot v)\ns{0}(\vp \cdot v)\ns{1}(\psi) = \psi \cdot (\vp \cdot v) \\
& =(\psi \rt \vp) \cdot ((\psi \lt \vp) \cdot v) = ((\psi \rt \vp) \cdot v\ns{0}) v\ns{1}(\psi \lt \vp) \\
& = v\ns{0}\ps{0}(v\ns{0}\ps{1} \lt \psi)(\vp)v\ns{1}^{\ns{0}}(\psi)v\ns{1}^{\ns{1}}(\vp)  \\
& = v\ns{0}\ps{0}v\ns{1}^{\ns{0}}(\psi)((v\ns{0}\ps{1} \lt \psi) \cdot v\ns{1}^{\ns{1}})(\vp).
\end{split}
\end{align}
Therefore,
\begin{equation}
v\ps{0}\ns{0} \ot v\ps{0}\ns{1} \ot v\ps{1} = v\ns{0}\ps{0} \ot v\ns{1}^{\ns{0}} \ot (v\ns{0}\ps{1} \lt \psi) \cdot v\ns{1}^{\ns{1}}.
\end{equation}
Plugging this result into \eqref{aux-nabla2-X-m-psi},
\begin{align}
\begin{split}
& \bD_2(X \cdot v)(\psi) = v\ns{0}\ps{0}v\ns{1}^{\ns{0}}(\psi)X((v\ns{0}\ps{1} \lt \psi) \cdot v\ns{1}^{\ns{1}})  \\
& = v\ns{0}\ps{0}v\ns{1}^{\ns{0}}(\psi)X(v\ns{0}\ps{1} \lt \psi)v\ns{1}^{\ns{1}}(e_1)  \\
& + v\ns{0}\ps{0}v\ns{1}^{\ns{0}}(\psi)(v\ns{0}\ps{1} \lt \psi)(e_1)X(v\ns{1}^{\ns{1}}) \\
&=  v\ns{0}\ps{0}v\ns{1}(\psi)(\psi \rt X)(v\ns{0}\ps{1}) + v\ns{0}(X \rt v\ns{1})(\psi)  \\
&=  v\ns{1}(\psi)(\psi \rt X) \cdot v\ns{0} + v\ns{0}(X \rt v\ns{1})(\psi)  \\
& =(\psi \rt X) \cdot (\psi \cdot v) + v\ns{0}(X \rt v\ns{1})(\psi).
\end{split}
\end{align}
Also as before, we evidently have
\begin{equation}
(X \bullet \bD_2(v))(\psi) = (\psi \rt X) \cdot (\psi \cdot v) + v\ns{0}(X \rt v\ns{1})(\psi).
\end{equation}
So for any $v\in V$ and $X\in\Fg_1$ the equality
\begin{equation}
\bD_2(X \cdot v) = X \bullet \bD_2(v)
\end{equation}
 is proved.
\end{proof}

\section{Hopf-cyclic coefficients - the general case}

In this section we study the SAYD modules over $R(\Fg_2) \acl U(\Fg_1)$ in terms of the corepresentations as well as the representations of $\Fg_1 \oplus \Fg_2$.

\subsection{Lie algebra coaction and SAYD modules over Lie algebras}

In this subsection we focus on the stable-anti-Yetter-Drinfeld modules over the universal enveloping algebra of a Lie algebra. To this end, we introduce the notion of comodule over a Lie algebra.

\begin{definition}
A vector space $V$ is a left comodule over the Lie algebra $\Fg$ if there is a linear map
\begin{equation}
\Db_{\Fg}: V \ra \Fg \ot V,\qquad \Db_{\Fg}(v)=v\nsb{-1}\ot v\nsb{0}
\end{equation}
such that
\begin{equation}\label{aux-g-comod}
v\nsb{-2}\wg v\nsb{-1}\ot v\nsb{0}=0,
\end{equation}
where
\begin{equation}
v\nsb{-2}\ot v\nsb{-1}\ot v\nsb{0}= v\nsb{-1}\ot (v\nsb{0})\nsb{-1}\ot (v\nsb{0})\nsb{0}.
\end{equation}
\end{definition}

\begin{proposition}\label{proposition-mod-comod}
Let $\Fg$ be a  Lie algebra and $V$ be  a vector space. Then $V$ is a right $S(\Fg^*)$-module if and only if it is a left $\Fg$-comodule.
\end{proposition}

\begin{proof}
Assume  that $V$ is a right module over the symmetric algebra $S(\Fg^*)$. Then for any $v \in V$ there exists an element $v\nsb{-1} \ot v\nsb{0} \in \Fg^{**} \ot V \cong \Fg \ot V$ such that for any $\theta \in \Fg^*$
\begin{equation}
v \lt \theta = v\nsb{-1}(\theta)v\nsb{0} = \theta(v\nsb{-1})v\nsb{0}.
\end{equation}
Hence we define the linear map $\Db_{\Fg}: V \ra \Fg \ot V$ by
\begin{equation}
 v \mapsto v\nsb{-1}\ot v\nsb{0}.
\end{equation}
Thus, the compatibility needed for $V$ to be a right module over $S(\Fg^*)$, which is  $(v \lt \theta) \lt \eta - (v \lt \eta) \lt \theta = 0$, translates directly into
\begin{equation}
\alpha(v\nsb{-2} \wedge v\nsb{-1}) \ot v\nsb{0} = (v\nsb{-2} \ot v\nsb{-1} - v\nsb{-1} \ot v\nsb{-2}) \ot v\nsb{0} = 0,
\end{equation}
where $\alpha:\wedge^2\Fg \to U(\Fg)^{\ot\, 2}$ is the anti-symmetrization map. Since the anti-symmetrization is injective, we have
\begin{equation}
v\nsb{-2} \wedge v\nsb{-1} \ot v\nsb{0} = 0.
\end{equation}
As a result, $V$ is a left $\Fg$-comodule.

\medskip

Conversely, assume that $V$ is a left $\Fg$-comodule via the map $\Db_{\Fg}:V \to \Fg \ot V$ defined by $  v \mapsto v\nsb{-1} \ot v\nsb{0}$. We define the right action
\begin{equation}
V \ot S(\Fg^*) \to V, \quad v \ot \theta \mapsto v \lt \theta := \theta(v\nsb{-1})v\nsb{0},
\end{equation}
for any $\theta \in \Fg^*$ and any $v \in V$. Thus,
\begin{equation}
(v \lt \theta) \lt \eta - (v \lt \eta) \lt \theta = (v\nsb{-2} \ot v\nsb{-1} - v\nsb{-1} \ot v\nsb{-2})(\theta \ot \eta) \ot v\nsb{0} = 0,
\end{equation}
proving that $V$ is a right module over $S(\Fg^*)$.
\end{proof}

We now investigate the relation between the left $\Fg$-coaction and left $U(\Fg)$-coaction.

\medskip

Let $\Db:V \to U(\Fg) \ot V$ be a left $U(\Fg)$-comodule structure on a linear space $V$. Then composing via the canonical projection $\pi:U(\Fg) \to \Fg$,
\begin{equation}
\xymatrix {
\ar[dr]_{\Db_{\Fg}} V \ar[r]^{\Db\hspace{.5cm}} & U(\Fg) \ot V \ar[d]^{\pi \ot \Id} \\
& \Fg \ot V
}
\end{equation}
we get a linear map $\Db_{\Fg}:V \to \Fg \ot V$.

\begin{lemma}
If $\Db:V \to U(\Fg) \ot V$ is a coaction, then so is $\Db_{\Fg}:V \to \Fg \ot V$.
\end{lemma}

\begin{proof}
If we write $\Db(v) = v\snsb{-1} \ot v\snsb{0}$, then by the cocommutativity of $U(\Fg)$
\begin{align}
\begin{split}
& v\nsb{-2} \wedge v\nsb{-1} \ot v\nsb{0} = \pi(v\snsb{-2}) \wedge \pi(v\snsb{-1}) \ot v\snsb{0} = \\
& \pi(v\snsb{-1}\ps{1}) \wedge \pi(v\snsb{-1}\ps{2}) \ot v\snsb{0} = 0.
\end{split}
\end{align}
\end{proof}

In order to obtain a $U(\Fg)$-comodule out of  a $\Fg$-comodule, we will need the following concept.

\begin{definition}\label{definition-locally-nilpotent-comodule}
Let $V$ be a $\Fg$-comodule via $\Db_{\Fg}:V \to \Fg \ot V$. Then we call the coaction locally conilpotent if for any $v \in V$ there exists $n \in \mathbb{N}$ such that $\Db^n_{\Fg}(v) = 0$.
\end{definition}

\begin{example}\rm{
If $V$ is an SAYD module on $U(\Fg)$, then by \cite[Lemma 6.2]{JaraStef} we have the filtration $V = \cup_{p \in \Zb}F_pV$ defined inductively by $F_0V = V^{\rm co\,U(\Fg)}$ and
\begin{equation}
F_{p+1}V/F_pV = (V/F_pV)^{\rm co\,U(\Fg)},
\end{equation}
where $V^{\rm co\,U(\Fg)}$ is the subspace of coaction invariants. As a result, the induced $\Fg$-comodule $V$ is locally conilpotent.
}\end{example}

\begin{example}\rm{
Let $\Fg$ be a Lie algebra and $S(\Fg^\ast)$ be the symmetric algebra on $\Fg^\ast$. For $V = S(\Fg^\ast)$, consider the coaction
\begin{equation}\label{aux-Koszul-coaction}
S(\Fg^\ast) \to \Fg \ot S(\Fg^\ast),\quad \a \mapsto X_i \ot \a\t^i.
\end{equation}
This is called the Koszul coaction. The corresponding $S(\Fg^*)$-action on $V$ coincides with the multiplication of $S(\Fg^\ast)$. Therefore, the Koszul coaction on $S(\Fg^\ast)$ is not locally conilpotent.

\medskip

On the other hand we note that the Koszul coaction is locally conilpotent on any truncation of the symmetric algebra.
}\end{example}

Let  $\{U_k(\Fg)\}_{k \geq 0}$ be the canonical filtration of $U(\Fg)$, \ie
\begin{equation}
U_0(\Fg) = \Cb \cdot 1, \quad U_1(\Fg) = \Cb \cdot 1 \oplus \Fg, \quad U_p(\Fg) \cdot U_q(\Fg) \subseteq U_{p+q}(\Fg)
\end{equation}

Let us say an element in $U(\Fg)$ is symmetric homogeneous of degree $k$ if it is the canonical image of a symmetric homogeneous tensor of degree $k$ over $\Fg$. Let  $U^k(\Fg)$  be the set of all symmetric elements of degree $k$ in $U(\Fg)$.

\medskip

We recall from  \cite[ Proposition 2.4.4]{Dixm-book} that
\begin{equation}\label{aux-1}
U_k(\Fg) = U_{k-1}(\Fg) \oplus U^k(\Fg).
\end{equation}
In other words, there is a (canonical) projection
\begin{align}\label{aux-the-map-theta-k}
\begin{split}
& \theta_k:U_k(\Fg) \to U^k(\Fg) \cong U_k(\Fg)/U_{k-1}(\Fg) \\
& X_1 \cdots X_k \mapsto \sum_{\sigma \in S_k}X_{\sigma(1)} \cdots X_{\sigma(k)}.
\end{split}
\end{align}
So, fixing an ordered basis of the Lie algebra $\Fg$, we can say that the map \eqref{aux-the-map-theta-k} is bijective on the PBW-basis elements.

\medskip

Let us  consider the unique derivation of $U(\Fg)$ extending the adjoint action of the Lie algebra $\Fg$ on itself,  and call it  $\ad(X):U(\Fg) \to U(\Fg)$ for any $X \in \Fg$.

\medskip

By \cite[Proposition 2.4.9]{Dixm-book},  $\ad(X)(U^k(\Fg)) \subseteq U^k(\Fg)$ and $\ad(X)(U_k(\Fg)) \subseteq U_k(\Fg)$. So by applying  $\ad(X)$ to both sides of \eqref{aux-1},  we observe that  the preimage of $\ad(Y)(\sum_{\sigma \in S_k}X_{\sigma(1)} \cdots X_{\sigma(k)})$  is  $\ad(Y)(X_1 \cdots X_k)$.

\begin{proposition}\label{proposition-construct-U(g)-coaction}
For a locally conilpotent $\Fg$-comodule $V$, the linear map
\begin{align}
\begin{split}
& \Db:V \to U(\Fg) \ot V \\
& v \mapsto 1 \ot v + \sum_{k \geq 1}\theta_k^{-1}(v\nsb{-k} \cdots v\nsb{-1}) \ot v\nsb{0}
\end{split}
\end{align}
defines a $U(\Fg)$-comodule structure.
\end{proposition}

\begin{proof}
For an arbitrary basis element $v^i \in V$, let us write
\begin{equation}\label{aux-54}
v^i\nsb{-1} \ot v^i\nsb{0} = \alpha^{ij}_k X_j \ot v^k
\end{equation}
where $\alpha^{ij}_k \in \mathbb{C}$. Then, by the coaction compatibility $v\nsb{-2} \wedge v\nsb{-1} \ot v\nsb{0} = 0$ we have
\begin{equation}
v^i\nsb{-2} \ot v^i\nsb{-1} \ot v^i\nsb{0} = \sum_{j_1 , j_2} \alpha^{ij_1j_2}_{l_2} X_{j_1} \ot X_{j_2} \ot v^{l_2},
\end{equation}
such that $\alpha^{ij_1j_2}_{l_2} := \alpha^{ij_1}_{l_1}\alpha^{l_1j_2}_{l_2}$ and $\alpha^{ij_1j_2}_{l_2} = \alpha^{ij_2j_1}_{l_2}$.

\medskip

We have
\begin{equation}
\Db(v^i) = 1 \ot v^i + \sum_{k \geq 1}\sum_{j_1 \leq \cdots \leq j_k} \alpha^{ij_1 \cdots j_k}_{l_k} X_{j_1} \cdots X_{j_k} \ot v^{l_k},
\end{equation}
because  for $k \geq 1$
\begin{equation}
v^i\nsb{-k} \ot \cdots \ot v^i\nsb{-1} \ot v^i\nsb{0} = \sum_{j_1 , \cdots , j_k} \alpha^{ij_1 \cdots j_k}_{l_k} X_{j_1} \ot \cdots \ot X_{j_k} \ot v^{l_k},
\end{equation}
where $\alpha^{ij_1 \cdots j_k}_{l_k} := \alpha^{ij_1}_{l_1} \cdots \alpha^{l_{k-1}j_k}_{l_k}$. Moreover, for any $\sigma \in S_k$ we have
\begin{equation}\label{aux-55}
\alpha^{ij_1 \cdots j_k}_{l_k} = \alpha^{ij_{\sigma(1)} \cdots j_{\sigma(k)}}_{l_k}.
\end{equation}
At this point, the counitality
\begin{equation}
(\ve \ot \Id) \circ \Db (v^i) = v^i
\end{equation}
is immediate. On the other hand, to prove the coassociativity we first observe that
\begin{align}
\begin{split}
& (\Id \ot \Db) \circ \Db(v^i) = 1 \ot \Db(v^i) + \sum_{k \geq 1}\sum_{j_1 \leq \cdots \leq j_k} \alpha^{ij_1 \cdots j_k}_{l_k} X_{j_1} \cdots X_{j_k} \ot \Db(v^{l_k}) \\
& = 1 \ot 1 \ot v^i + \sum_{k \geq 1}\sum_{j_1 \leq \cdots \leq j_k} \alpha^{ij_1 \cdots j_k}_{l_k} 1 \ot X_{j_1} \cdots X_{j_k} \ot v^{l_k} + \\
& \sum_{k \geq 1}\sum_{j_1 \leq \cdots \leq j_k} \alpha^{ij_1 \cdots j_k}_{l_k} X_{j_1} \cdots X_{j_k} \ot 1 \ot v^{l_k} + \\
& \sum_{k \geq 1}\sum_{j_1 \leq \cdots \leq j_k} \alpha^{ij_1 \cdots j_k}_{l_k} X_{j_1} \cdots X_{j_k} \ot (\sum_{t \geq 1}\sum_{r_1 \leq \cdots \leq r_t} \alpha^{l_kr_1 \cdots r_t}_{s_t} X_{r_1} \cdots X_{r_t} \ot v^{s_t}),
\end{split}
\end{align}
where $\alpha^{l_kr_1 \cdots r_t}_{s_t} := \alpha^{l_kr_1}_{s_1} \cdots \alpha^{s_{t-1}r_t}_{s_t}$. Then we notice that
\begin{align}
\begin{split}
& \Delta(\sum_{k \geq 1}\sum_{j_1 \leq \cdots \leq j_k} \alpha^{ij_1 \cdots j_k}_{l_k} X_{j_1} \cdots X_{j_k}) \ot v^{l_k} = \sum_{k \geq 1}\sum_{j_1 \leq \cdots \leq j_k} \alpha^{ij_1 \cdots j_k}_{l_k} 1 \ot X_{j_1} \cdots X_{j_k} \ot v^{l_k} \\
& + \sum_{k \geq 1}\sum_{j_1 \leq \cdots \leq j_k} \alpha^{ij_1 \cdots j_k}_{l_k} X_{j_1} \cdots X_{j_k} \ot 1 \ot v^{l_k} \\
& + \sum_{k \geq 2}\sum_{j_1 \leq \cdots \leq r_1 \leq \cdots \leq r_p \leq \cdots \leq j_k} \alpha^{ij_1 \cdots j_k}_{l_k} X_{r_1} \cdots X_{r_p} \ot X_{j_1} \cdots \widehat{X}_{r_1} \cdots \widehat{X}_{r_p} \cdots X_{j_k} \ot v^{l_k} = \\
& \sum_{k \geq 1}\sum_{j_1 \leq \cdots \leq j_k} \alpha^{ij_1 \cdots j_k}_{l_k} 1 \ot X_{j_1} \cdots X_{j_k} \ot v^{l_k} + \\
& \sum_{k \geq 1}\sum_{j_1 \leq \cdots \leq j_k} \alpha^{ij_1 \cdots j_k}_{l_k} X_{j_1} \cdots X_{j_k} \ot 1 \ot v^{l_k} + \\
& \sum_{p \geq 1}\sum_{k-p \geq 1}\sum_{q_1 \leq \cdots \leq q_{k-p}}\sum_{r_1 \leq \cdots \leq r_p} \alpha^{ir_1 \cdots r_p}_{l_p}\alpha^{l_pq_1 \cdots q_{k-p}}_{l_k} X_{r_1} \cdots X_{r_p} \ot X_{q_1} \cdots X_{q_{k-p}} \ot v^{l_k},
\end{split}
\end{align}
where for the last equality we write the complement of $r_1 \leq \cdots \leq r_p$ in $j_1 \leq \cdots \leq j_k$ as $q_1 \leq \cdots \leq q_{k-p}$. Then \eqref{aux-55} implies that
\begin{equation}
\alpha^{ij_1 \cdots j_k}_{l_k} = \alpha^{ir_1 \cdots r_p q_1 \cdots q_{k-p}}_{l_k} = \alpha^{ir_1 \cdots r_p}_{l_p}\alpha^{l_pq_1 \cdots q_{k-p}}_{l_k}.
\end{equation}
As a result,
\begin{equation}
(\Id \ot \Db) \circ \Db(v^i) = (\Delta \ot \Id) \circ \Db(v^i).
\end{equation}
which is coassociativity.
\end{proof}

Let us denote by $\rm  ^{\Fg}conil\Mc$ the subcategory of locally conilpotent left $\Fg$-comodules in the category of left $\Fg$-comodules $\, ^{\Fg}\Mc$ with colinear maps.

\medskip

Assigning a $\Fg$-comodule $\Db_{\Fg}:V \to \Fg \ot V$ to a $U(\Fg)$-comodule $\Db:V \to U(\Fg) \ot V$ determines a functor
\begin{equation}
\xymatrix{
^{U(\Fg)}\Mc  \ar@<.1 ex>[r]^{P} &  \, \rm  ^{\Fg}conil\Mc.
}
\end{equation}
Similarly, constructing a $U(\Fg)$-comodule from a $\Fg$-comodule determines a functor
\begin{equation}
\xymatrix{
\rm  ^{\Fg}conil\Mc \ar@<.1 ex>[r]^{E} &  \, ^{U(\Fg)}\Mc.
}
\end{equation}

As a result, we have the following proposition.

\begin{proposition}\label{proposition-categories-equivalent}
The categories $\, \rm  ^{U(\Fg)}\Mc$ and $\, \rm  ^{\Fg}conil\Mc$ are isomorphic.
\end{proposition}

\begin{proof}
We show that the functors
$$
\begin{xy}
\xymatrix{
^{U(\Fg)}\Mc  \ar@<.6 ex>[r]^{P} & \ar@<.6 ex>[l]^{E}  \, \rm  ^{\Fg}conil\Mc
}
\end{xy}
$$
are inverses to each other.

\medskip

If $\Db_{\Fg}:V \to \Fg \ot V$ is a locally conilpotent $\Fg$-comodule and $\Db:V \to U(\Fg) \ot V$ the corresponding $U(\Fg)$-comodule, then by the very definition  the $\Fg$-comodule corresponding to $\Db:V \to U(\Fg) \ot V$ is exactly $\Db_{\Fg}:V \to \Fg \ot V$. This proves that
\begin{equation}
P \circ E\; =\; \Id_{\rm  ^{\Fg}conil\Mc}.
\end{equation}
Conversely, let us start with a $U(\Fg)$-comodule $\Db:V \to U(\Fg) \ot V$ and write the coaction by using the PBW-basis of $U(\Fg)$ as follows:
\begin{equation}
v^i\ps{-1} \ot v^i\ps{0} = 1 \ot v^i + \sum_{k \geq 1}\sum_{j_1 \leq \cdots \leq j_k}\gamma^{ij_1 \cdots j_k}_{l_k} X_{j_1} \cdots X_{j_k} \ot v^{l_k}.
\end{equation}
So, the corresponding $\Fg$-comodule $\Db_{\Fg}:V \to \Fg \ot V$ is given by
\begin{equation}
v^i\nsb{-1} \ot v^i\nsb{0} = \pi(v^i\ps{-1}) \ot v^i\ps{0} = \sum_{j,k} \gamma^{ij}_k X_j \ot v^k.
\end{equation}
Finally, the $U(\Fg)$-coaction corresponding to this $\Fg$-coaction is defined on $v^i \in V$ by
\begin{equation}
v^i \mapsto 1 \ot v + \sum_{k \geq 1}\sum_{j_1 \leq \cdots \leq j_k}\gamma^{ij_1}_{l_1}\gamma^{l_1j_2}_{l_2} \cdots \gamma^{l_{k-1}j_k}_{l_k} X_{j_1} \cdots X_{j_k} \ot v^{l_k}.
\end{equation}
Therefore, we recover the $U(\Fg)$-coaction we started with if and only if
\begin{equation}\label{aux-40}
\gamma^{ij_1 \cdots j_k}_{l_k} = \gamma^{ij_1}_{l_1}\gamma^{l_1j_2}_{l_2} \cdots \gamma^{l_{k-1}j_k}_{l_k}, \quad \forall k \geq 1.
\end{equation}
The equation \eqref{aux-40} is a consequence of the coassociativity of $\Db$. Indeed,  applying the coassociativity as
\begin{equation}
(\Delta^{k-1} \ot \Id) \circ \Db = \Db^k
\end{equation}
and comparing the coefficients of $X_{j_1} \ot \cdots \ot X_{j_k}$ we  deduce \eqref{aux-40} for any $k \geq 1$. Hence,  we have proved
\begin{equation}
E \circ P = \Id_{^{U(\Fg)}\Mc}.
\end{equation}
The equation \eqref{aux-40} implies that if $\Db:V \to U(\Fg) \ot V$ is a left coaction, then its associated  $\Fg$-coaction  $\Db_{\Fg}:V \to \Fg \ot V$ is locally conilpotent.
\end{proof}

For a $\Fg$-coaction
\begin{equation}
v \mapsto v\nsb{-1} \ot v\nsb{0},
\end{equation}
we denote the associated  $U(\Fg)$-coaction by
\begin{equation}
v \mapsto v\snsb{-1} \ot v\snsb{0}.
\end{equation}

\begin{definition}
Let $V$ be a right module and left comodule over a Lie algebra $\Fg$. We call $V$ a right-left AYD module over $\Fg$ if
\begin{equation}
\Db_{\Fg}(v \cdot X) = v\nsb{-1} \ot v\nsb{0} \cdot X + [v\nsb{-1}, X] \ot v\nsb{0}.
\end{equation}
Moreover, $V$ is called stable if
\begin{equation}
v\nsb{0} \cdot v\nsb{-1} = 0.
\end{equation}
\end{definition}

\begin{proposition}\label{proposition-g-AYD-U(g)-AYD}
Let $\Db_{\Fg}:V \to \Fg \ot V$ be a locally conilpotent $\Fg$-comodule and $\Db:V \to U(\Fg) \ot V$ the corresponding $U(\Fg)$-comodule structure. Then $V$ is a right-left AYD module over $\Fg$ if and only if it is a right-left AYD module over $U(\Fg)$.
\end{proposition}

\begin{proof}
Let us first assume $V$ to be a right-left AYD  module over $\Fg$. For  $X \in \Fg$ and an element $v \in V$, AYD compatibility implies that
\begin{align}
\begin{split}
& (v \cdot X)\nsb{-k} \ot \cdots \ot (v \cdot X)\nsb{-1} \ot (v \cdot X)\nsb{0} = v\nsb{-k} \ot \cdots \ot v\nsb{-1} \ot v\nsb{0} \cdot X \\
& + [v\nsb{-k}, X] \ot \cdots \ot v\nsb{-1} \ot v\nsb{0} + v\nsb{-k} \ot \cdots \ot [v\nsb{-1}, X] \ot v\nsb{0}.
\end{split}
\end{align}
Multiplying in $U(\Fg)$, we get
\begin{align}
\begin{split}
& (v \cdot X)\nsb{-k} \cdots (v \cdot X)\nsb{-1} \ot (v \cdot X)\nsb{0}  = \\
& v\nsb{-k} \cdots v\nsb{-1} \ot v\nsb{0} \cdot X - \ad(X)(v\nsb{-k} \cdots v\nsb{-1}) \ot v\nsb{0}.
\end{split}
\end{align}
So, for the extension $\Db:V \to U(\Fg) \ot V$ we have
\begin{align}
\begin{split}
& (v \cdot X)\snsb{-1} \ot (v \cdot X)\snsb{0} = 1 \ot v \cdot X + \sum_{k \geq 1}\theta_k^{-1}((v \cdot X)\nsb{-k} \cdots (v \cdot X)\nsb{-1}) \ot (v \cdot X)\nsb{0} \\
& = 1 \ot v \cdot X + \sum_{k \geq 1}\theta_k^{-1}(v\nsb{-k} \cdots v\nsb{-1}) \ot v\nsb{0} \cdot X - \sum_{k \geq 1}\theta_k^{-1}(\ad(X)(v\nsb{-k} \cdots v\nsb{-1})) \ot v\nsb{0} \\
& = v\snsb{-1} \ot v\snsb{0} \cdot X - \sum_{k \geq 1}\ad(X)(\theta_k^{-1}(v\nsb{-k} \cdots v\nsb{-1})) \ot v\nsb{0} \\
& = v\snsb{-1} \ot v\snsb{0} \cdot X - \ad(X)(v\snsb{-1}) \ot v\snsb{0} = S(X\ps{3})v\snsb{-1}X\ps{1} \ot v\snsb{0} \cdot X\ps{2}.
\end{split}
\end{align}
In the third equality we used the fact that the operator $\ad:U(\Fg) \to U(\Fg)$  commutes with $\theta_k$, and in the fourth equality we used
\begin{align}
\begin{split}
& \sum_{k \geq 1}\ad(X)(\theta_k^{-1}(v\nsb{-k} \cdots v\nsb{-1})) \ot v\nsb{0} = \\
& \sum_{k \geq 1}\ad(X)(\theta_k^{-1}(v\nsb{-k} \cdots v\nsb{-1})) \ot v\nsb{0} + \ad(X)(1) \ot v = \ad(X)(v\snsb{-1}) \ot v\snsb{0}.
\end{split}
\end{align}
By  using the fact that  AYD condition is multiplicative,  we conclude that $\Db:V \to U(\Fg) \ot V$ satisfies the AYD condition on $U(\Fg)$.

\medskip

Conversely, assume that $V$ is a right-left AYD over $U(\Fg)$. We first observe that
\begin{equation}
(\Delta \ot \Id) \circ \Delta (X) = X \ot 1 \ot 1 + 1 \ot X \ot 1 + 1 \ot 1 \ot X
\end{equation}
Accordingly,
\begin{align}
\begin{split}
& \Db(v \cdot X) = v\snsb{-1}X \ot v\snsb{0} + v\snsb{-1} \ot v\snsb{0} \cdot X - Xv\snsb{-1} \ot v\snsb{0} \\
& = -\ad(X)(v\snsb{-1}) \ot v\snsb{0} + v\snsb{-1} \ot v\snsb{0} \cdot X.
\end{split}
\end{align}
It is known that  the projection map $\pi:U(\Fg) \to \Fg$ commutes with the adjoint representation. So,
\begin{align}
\begin{split}
& \Db_{\Fg}(v \cdot X) = -\pi(\ad(X)(v\snsb{-1})) \ot v\snsb{0} + \pi(v\snsb{-1}) \ot v\snsb{0} \cdot X \\
& = -\ad(X)\pi(v\snsb{-1}) \ot v\snsb{0} + \pi(v\snsb{-1}) \ot v\snsb{0} \cdot X \\
& = [v\nsb{-1},X] \ot v\nsb{0} + v\nsb{-1} \ot v\nsb{0} \cdot X.
\end{split}
\end{align}
As a result, $V$ is a right-left AYD over $\Fg$.
\end{proof}

\begin{lemma}
Let $\Db_{\Fg}:V \to \Fg \ot V$ be a locally conilpotent $\Fg$-comodule and $\Db:V \to U(\Fg) \ot V$ be the corresponding $U(\Fg)$-comodule structure. If $V$ is stable over $\Fg$, then it is stable over $U(\Fg)$.
\end{lemma}

\begin{proof}
Writing the $\Fg$-coaction in terms of the basis elements as in \eqref{aux-54}, the stability condition becomes
\begin{equation}
v^i\nsb{0}v^i\nsb{-1} = \alpha^{ij}_kv^k \cdot X_j = 0, \quad \forall i.
\end{equation}
Therefore, for the corresponding $U(\Fg)$-coaction we have
\begin{align}
\begin{split}
& \sum_{j_1 \leq \cdots \leq j_k} \alpha^{ij_1}_{l_1} \cdots \alpha^{l_{k-1}j_k}_{l_k} v^{l_k} \cdot  (X_{j_1} \cdots X_{j_k}) = \\
& \sum_{j_1 \leq \cdots \leq j_{k-1}} \alpha^{ij_1}_{l_1} \cdots \alpha^{l_{k-2}j_{k-1}}_{l_{k-1}} (\sum_{j_k} \alpha^{l_{k-1}j_k}_{l_k} v^{l_k} \cdot X_{j_1}) \cdot (X_{j_2} \cdots X_{j_k}) = \\
& \sum_{j_2, \cdots , j_k} \alpha^{ij_k}_{l_1} \cdots \alpha^{l_{k-2}j_{k-1}}_{l_{k-1}} (\sum_{j_1} \alpha^{l_{k-1}j_1}_{l_k} v^{l_k} \cdot X_{j_1}) \cdot (X_{j_2} \cdots X_{j_k}),
\end{split}
\end{align}
where in the second equality we used \eqref{aux-55}. This  immediately implies that
\begin{equation}
v^i\snsb{0} \cdot v^i\snsb{-1} = v^i,
\end{equation}
which is stability over $U(\Fg)$.
\end{proof}

However, the converse is not true.

\begin{example}{\rm
It is known that $U(\Fg)$, as a left $U(\Fg)$-comodule via $\Delta:U(\Fg) \to U(\Fg) \ot U(\Fg)$ and a right $\Fg$-module via $\ad:U(\Fg) \ot \Fg \to U(\Fg)$ is stable. However, the associated  $\Fg$-comodule is  no longer  stable. Indeed, for $u = X_1X_2X_3 \in U(\Fg)$, we have
\begin{equation}
u\nsb{-1} \ot u\nsb{0} = X_1 \ot X_2X_3 + X_2 \ot X_1X_3 + X_3 \ot X_1X_2
\end{equation}
Then,
\begin{equation}
u\nsb{0} \cdot u\nsb{-1} = [[X_1,X_2],X_3] + [[X_2,X_1],X_3] + [[X_1,X_3],X_2] = [[X_1,X_3],X_2],
\end{equation}
which is not necessarily zero.
}\end{example}

An alternative interpretation of SAYD modules over Lie algebras is given by modules over the Lie algebra $\widetilde\Fg\; :=\; \Fg^\ast \rtimes \Fg$.

\medskip

Considering the dual $\Fg^\ast$ of the Lie algebra $\Fg$ as a commutative Lie algebra, we define the Lie  bracket  on $\widetilde\Fg\; :=\; \Fg^\ast \rtimes \Fg$ by
\begin{equation}\label{tdg}
\big[\a \oplus X\;,\; \b \oplus Y\big] \;:= \;\big(\Lc_X(\b) - \Lc_Y(\a)\big) \oplus \big[X\;,\;Y\big].
\end{equation}

\begin{definition}
Let $V$ be a module over the Lie algebra $\widetilde\Fg$. Then
\begin{align}\label{aux-unimodular-stable}
\begin{split}
&\text{$V$ is called unimodular stable if }\qquad \sum_k (v \cdot X_k) \cdot \theta^k = 0 ; \\
&\text{$V$ is called  stable if }\qquad \sum_k (v \cdot \theta^k) \cdot X_k = 0.
\end{split}
\end{align}
\end{definition}

The following proposition connects these two interpretations.

\begin{proposition}\label{23}
Let $V$ be a vector space, and $\Fg$ be a Lie algebra. Then $V$ is a stable right $\widetilde\Fg$-module if and only if it is a right-left SAYD module over the Lie algebra $\Fg$.
\end{proposition}

\begin{proof}
Let us first assume that $V$ is a stable right $\widetilde\Fg$-module. Since  $V$ is a right $S(\Fg^*)$-module, it is  a left $\Fg$-comodule by Proposition \ref{proposition-mod-comod}. Accordingly,
\begin{align}
\begin{split}
& [v\nsb{-1}, X_j] \ot v\nsb{0} + v\nsb{-1} \ot v\nsb{0} \cdot X_j = \\
& [X_l,X_j]\theta^l(v\nsb{-1}) \ot v\nsb{0} + X_t\theta^t(v\nsb{-1}) \ot v\nsb{0} \cdot X_j = \\
& X_tC^t_{lj}\theta^l(v\nsb{-1}) \ot v\nsb{0} + X_t\theta^t(v\nsb{-1}) \ot v\nsb{0} \cdot X_j = \\
& X_t \ot [v \lt (X_j \rt \theta^t) + (v \lt \theta^t) \cdot X_j] = \\
& X_t \ot (v \cdot X_j) \lt \theta^t = X_t\theta^t((v \cdot X_j)\nsb{-1}) \ot (v \cdot X_j)\nsb{0} = \\
& (v \cdot X_j)\nsb{-1} \ot (v \cdot X_j)\nsb{0}.
\end{split}
\end{align}
This proves that $V$ is a right-left AYD module over $\Fg$. On the other hand, for any $v \in V$
\begin{equation}
v\nsb{0} \cdot v\nsb{-1} = \sum_i v\nsb{0} \cdot X_i\theta^i(v\nsb{-1}) = \sum_i (v \lt \theta^i) \cdot X_i = 0.
\end{equation}
Hence, $V$ is stable too. As a result, $V$ is a SAYD module over $\Fg$.

\medskip

Conversely, assume that  $V$ is a  right-left SAYD module over $\Fg$. So, $V$  is  a right module over $S(\Fg^*)$ and a right module over $\Fg$. In addition, we see that
\begin{align}
\begin{split}
& v \lt (X_j \rt \theta^i) + (v \lt \theta^i) \cdot X_j = C^i_{kj}v \lt \theta^k + (v \lt \theta^i) \cdot X_j = \\
& C^i_{kj}\theta^k(v\nsb{-1})v\nsb{0} + \theta^i(v\nsb{-1})v\nsb{0} \cdot X_j = \\
& \theta^i([v\nsb{-1},X_j])v\nsb{0} + \theta^i(v\nsb{-1})v\nsb{0} \cdot X_j = \\
& (\theta^i \ot id)([v\nsb{-1}, X_j] \ot v\nsb{0} + v\nsb{-1} \ot v\nsb{0} \cdot X_j) = \\
& \theta^t((v \cdot X_j)\nsb{-1}) (v \cdot X_j)\nsb{0} = (v \cdot X_j) \lt \theta^i.
\end{split}
\end{align}
Thus, $V$ is a right $\widetilde\Fg$-module. Finally, we prove stability using
\begin{equation}
\sum_i (v \lt \theta^i) \cdot X_i = \sum_i v\nsb{0} \cdot X_i\theta^i(v\nsb{-1}) = v\nsb{0} \cdot v\nsb{-1} = 0.
\end{equation}
\end{proof}

\begin{corollary}
Any right module over the Weyl algebra $D(\Fg)$ is a right-left SAYD module over the Lie algebra $\Fg$.
\end{corollary}

Finally, we state an analogue of \cite[Lemma 2.3]{HajaKhalRangSomm04-I} to show that the category of right-left AYD modules over a Lie algebra $\Fg$ is monoidal.

\begin{proposition}
Let $V$ and $W$ be two right-left AYD modules over $\Fg$. Then $V \ot W$ is also a right-left AYD module over $\Fg$ via the coaction
\begin{equation}
\Db_{\Fg}:V \ot W \to \Fg \ot V \ot W, \quad v \ot w \mapsto v\nsb{-1} \ot v\nsb{0} \ot w + w\nsb{-1} \ot w \ot w\nsb{0}
\end{equation}
and the action
\begin{equation}
V \ot W \ot \Fg \to V \ot W, \quad (v \ot w) \cdot X = v \cdot X \ot w + v \ot w \cdot X.
\end{equation}
\end{proposition}

\begin{proof}
We directly verify that
\begin{align}
\begin{split}
& [(v \ot w)\nsb{-1},X] \ot (v \ot w)\nsb{0} + (v \ot w)\nsb{-1} \ot (v \ot w)\nsb{0} \cdot X = \\
& [v\nsb{-1},X] \ot v\nsb{0} \ot w + [w\nsb{-1},X] \ot v \ot w\nsb{0} + \\
& v\nsb{-1} \ot (v\nsb{0} \ot w) \cdot X + w\nsb{-1} \ot (v \ot w\nsb{0}) \cdot X = \\
& (v \cdot X)\nsb{-1} \ot (v \cdot X)\nsb{0} \ot w + w\nsb{-1} \ot v \cdot X \ot w\nsb{0} + \\
& v\nsb{-1} \ot v\nsb{0} \ot w\cdot X + (w\cdot X)\nsb{-1} \ot v \ot (w \cdot X)\nsb{0} = \\
& \Db_{\Fg}(v \cdot X \ot w + v \ot w \cdot X) = \Db_{\Fg}((v \ot w) \cdot X).
\end{split}
\end{align}
\end{proof}

The rest of this subsection is devoted to examples which illustrate the notion of SAYD module over a Lie algebra. We consider the representations and corepresentations of a Lie algebra $\Fg$ on a finite dimensional vector space $V$ in terms of matrices. We then investigate the SAYD conditions as relations between these matrices and the Lie algebra structure of $\Fg$.

\medskip

Let $V$ be a $n$ dimensional $\Fg$-module with a vector space basis $\{v^1, \cdots, v^n\}$. We  express the module structure by
\begin{equation}
v^i \cdot X_j = \beta^i_{jk}v^k, \quad \beta^i_{jk} \in \Cb.
\end{equation}
This way, for any basis element $X_j \in \Fg$ we obtain a matrix $B_j \in M_n(\mathbb{C})$ such that
\begin{equation}
(B_j)^i_k := \beta^i_{jk}.
\end{equation}
Let $\Db_{\Fg}:V \to \Fg \ot V$ be a coaction. We write  the coaction as
\begin{equation}
\Db_{\Fg}(v^i) = \alpha^{ij}_k X_j \ot v^k, \quad \alpha^{ij}_k\in \Cb.
\end{equation}
Hence we get a matrix $A^j \in M_n(\mathbb{C})$ for any basis element $X_j \in \Fg$ such that
\begin{equation}
(A^j)^i_k := \alpha^{ij}_k.
\end{equation}

\begin{lemma}
A linear map $\Db_{\Fg}:V \to \Fg \ot V$ forms a right $\Fg$-comodule if and only if for any $1 \leq j_1,j_2 \leq N$
\begin{equation}
A^{j_1} \cdot A^{j_2} = A^{j_2} \cdot A^{j_1}.
\end{equation}
\end{lemma}

\begin{proof}
It is just the translation of the coaction compatibility \eqref{aux-g-comod} for $v^i \in V$ in terms of the matrices $A^j$.
\end{proof}

\begin{lemma}
A right $\Fg$-module left $\Fg$-comodule $V$ is stable if and only if
\begin{equation}\label{aux-4}
\sum_j A^j \cdot B_j = 0.
\end{equation}
\end{lemma}

\begin{proof}
By the definition of the stability we have
\begin{equation}
v^i\nsb{0} \cdot v^i\nsb{-1} = \alpha^{ij}_{k} v^k \cdot X_j = \alpha^{ij}_{k}\beta^k_{jl} v^l = 0.
\end{equation}
Therefore,
\begin{equation}
\alpha^{ij}_{k}\beta^k_{jl} = (A^j)^i_k (B_j)^k_l = (A^j \cdot B_j)^i_l = 0.
\end{equation}
\end{proof}

We proceed to express the AYD condition.

\begin{lemma}
A right $\Fg$-module left $\Fg$-comodule $V$ is a right-left AYD module  if and only if
\begin{equation}\label{aux-2}
[B_q, A^j] = \sum_s A^sC^j_{sq}.
\end{equation}
\end{lemma}

\begin{proof}
We first observe that
\begin{align}
\begin{split}
& \Db_{\Fg}(v^p \cdot X_q) = \Db_{\Fg}(\beta^p_{qk}v^k) = \beta^p_{qk}\alpha^{kj}_l X_j \ot v^l \\
& = (B_q)^p_k(A^j)^k_l X_j \ot v^l = (B_q \cdot A^j)^p_l X_j \ot v^l.
\end{split}
\end{align}
On the other hand, writing $\Db_{\Fg}(v^p) = \alpha^{pj}_{l} X_j \ot v^l$,
\begin{align}
\begin{split}
& [v^p\nsb{-1},X_q] \ot v^p\nsb{0} + v^p\nsb{-1} \ot v^p\nsb{0} \cdot X_q = \alpha^{ps}_{l}[X_s,X_q] \ot v^l + \alpha^{pj}_{t} X_j \ot v^t \cdot X_q \\
& = \alpha^{ps}_{l}C^j_{sq}X_j \ot v^l + \alpha^{pj}_{t}\beta^t_{ql} X_j \ot v^l = (\alpha^{ps}_{l}C^j_{sq} + (A^j \cdot B_q)^p_l)X_j \ot v^l.
\end{split}
\end{align}
\end{proof}

Since we have expressed the AYD and the stability conditions after choosing bases, we find it convenient to record the following remark.

\begin{remark}\rm{
The stability and the AYD conditions are independent of the choice of basis. Let $\{Y_j\}$ be another basis with
\begin{equation}
Y_j = \gamma^l_jX_l, \qquad X_j = (\gamma^{-1})^l_jY_l.
\end{equation}
Hence, the action and coaction matrices are
\begin{equation}
\widetilde{B}_q = \gamma^l_qB_l, \quad \widetilde{A}^j = A^l(\gamma^{-1})^j_l,
\end{equation}
respectively. Then first of all,
\begin{equation}
\sum_j \widetilde{A}^j \cdot \widetilde{B}_j = \sum_{j,l,s} A^l(\gamma^{-1})^j_l\gamma^s_jB_s = \sum_{l,s} A^lB_s\delta^l_s = \sum_j A^j \cdot B_j = 0
\end{equation}
proves that the stability is independent of the choice of basis. Secondly, we have
\begin{equation}
[\widetilde{B}_q,\widetilde{A}^j] = \gamma^s_q(\gamma^{-1})^j_r[B_s,A^r] = \gamma^s_q(\gamma^{-1})^j_rA^lC^r_{ls}.
\end{equation}
If we write $[Y_p,Y_q] = \widetilde{C}^r_{pq}Y_r$, then it is immediate that
\begin{equation}
\gamma^s_qC^r_{ls}(\gamma^{-1})^j_r = (\gamma^{-1})^s_l\widetilde{C}^j_{sq}.
\end{equation}
Therefore,
\begin{equation}
[\widetilde{B}_q,\widetilde{A}^j] = A^l\gamma^s_qC^r_{ls}(\gamma^{-1})^j_r = A^l(\gamma^{-1})^s_l\widetilde{C}^j_{sq} = \widetilde{A}^s\widetilde{C}^j_{sq}.
\end{equation}
This observation proves that the AYD condition is independent of the choice of basis.
}\end{remark}

Next, considering the Lie algebra $s\ell(2)$, we determine the SAYD modules over simple $s\ell(2)$-modules. First of all, we fix a basis of $s\ell(2)$ as follows.
\begin{equation}
e = \left(
        \begin{array}{cc}
          0 & 1 \\
          0 & 0 \\
        \end{array}
      \right), \qquad f = \left(
        \begin{array}{cc}
          0 & 0 \\
          1 & 0 \\
        \end{array}
      \right), \qquad h = \left(
        \begin{array}{cc}
          1 & 0 \\
          0 & -1 \\
        \end{array}
      \right).
\end{equation}

\begin{example}\rm{
Let $V = <\{v^1,v^2\}>$ be a two dimensional simple $s\ell(2)$-module. Then by  \cite{Kass-book}, the representation
\begin{equation}
\rho:s\ell(2) \to g\ell(V)
\end{equation}
is the inclusion $\rho:s\ell(2) \hookrightarrow g\ell(2)$. Therefore, we have
\begin{equation}
B_1 = \left(
        \begin{array}{cc}
          0 & 0 \\
          1 & 0 \\
        \end{array}
      \right), \quad B_2 = \left(
        \begin{array}{cc}
          0 & 1 \\
          0 & 0 \\
        \end{array}
      \right), \quad B_3 = \left(
        \begin{array}{cc}
          1 & 0 \\
          0 & -1 \\
        \end{array}
      \right).
\end{equation}

We want to find
\begin{align}
A^1 = \left(
        \begin{array}{cc}
          x^1_1 & x^1_2 \\
          x^2_1 & x^2_2 \\
        \end{array}
      \right), \qquad A^2 = \left(
        \begin{array}{cc}
          y^1_1 & y^1_2 \\
          y^2_1 & y^2_2 \\
        \end{array}
      \right), \qquad A^3 = \left(
        \begin{array}{cc}
          z^1_1 & z^1_2 \\
          z^2_1 & z^2_2 \\
        \end{array}
      \right),
\end{align}
such that together with the $\Fg$-coaction $\Db_{s\ell(2)}:V \to s\ell(2) \ot V$, defined by $v^i \mapsto (A^j)^i_k X_j \ot v^k$, $V$ becomes a right-left SAYD over $s\ell(2)$. We first express the stability condition. To this end,
\begin{align}
A^1 \cdot B_1 = \left(
        \begin{array}{cc}
          x^1_2 & 0 \\
          x^2_2 & 0 \\
        \end{array}
      \right), \qquad A^2 \cdot B_2 = \left(
        \begin{array}{cc}
          0 & y^1_1 \\
          0 & y^2_1 \\
        \end{array}
      \right), \qquad A^3 \cdot B_3 = \left(
        \begin{array}{cc}
          z^1_1 & -z^1_2 \\
          z^2_1 & -z^2_2 \\
        \end{array}
      \right),
\end{align}
and hence, the stability condition is
\begin{align}
\sum_j A^j \cdot B_j = \left(
                         \begin{array}{cc}
                           x^1_2 + z^1_1 & y^1_1 - z^1_2 \\
                           x^2_2 + z^2_1 & y^2_1 - z^2_2 \\
                         \end{array}
                       \right) = 0.
\end{align}
Next, we consider the AYD condition
\begin{equation}
[B_q, A^j] = \sum_s A^sC^j_{sq}.
\end{equation}
For $j = 1 = q$,
\begin{equation}
A^1 = \left(
        \begin{array}{cc}
          x^1_1 & 0 \\
          x^2_1 & x^2_2 \\
        \end{array}
      \right), \qquad A^2 = \left(
        \begin{array}{cc}
          0 & y^1_2 \\
          0 & y^2_2 \\
        \end{array}
      \right), \qquad A^3 = \left(
        \begin{array}{cc}
          0 & 0 \\
          z^2_1 & 0 \\
        \end{array}
      \right).
\end{equation}
Similarly, for $q = 2$ and $j = 1$, we arrive at
\begin{equation}
A^1 = \left(
        \begin{array}{cc}
          0 & 0 \\
          0 & 0 \\
        \end{array}
      \right), \qquad A^2 = \left(
        \begin{array}{cc}
          0 & y^1_2 \\
          0 & y^2_2 \\
        \end{array}
      \right), \qquad A^3 = \left(
        \begin{array}{cc}
          0 & 0 \\
          0 & 0 \\
        \end{array}
      \right).
\end{equation}
Finally, for $j = 1$ and $q=2$ we conclude that
\begin{equation}
A^1 = \left(
        \begin{array}{cc}
          0 & 0 \\
          0 & 0 \\
        \end{array}
      \right), \qquad A^2 = \left(
        \begin{array}{cc}
          0 & 0 \\
          0 & 0 \\
        \end{array}
      \right), \qquad A^3 = \left(
        \begin{array}{cc}
          0 & 0 \\
          0 & 0 \\
        \end{array}
      \right).
\end{equation}
Thus, the only $s\ell(2)$-comodule structure that makes a 2-dimensional simple $s\ell(2)$-module $V$ into a right-left SAYD module over $s\ell(2)$ is the trivial comodule structure.
}\end{example}

\begin{example}\label{aux-62}{\rm
We investigate all possible coactions that make the truncated symmetric algebra $S(s\ell(2)^*)_{[1]}$ into an SAYD module over $s\ell(2)$.

\medskip

A vector space basis  of $S(s\ell(2)^*)_{[1]}$ is $\{\one, \theta^e, \theta^f, \theta^h\}$ and the Koszul coaction \eqref{aux-Koszul-coaction} is given by
\begin{align}
\begin{split}
& S(s\ell(2)^*)_{[1]} \to s\ell(2) \ot S(s\ell(2)^*)_{[1]} \\
& \one \mapsto e \ot \theta^e + f \ot \theta^f + h \ot \theta^h \\
& \theta^i \mapsto 0 ,\qquad i \in \{e,f,h\}.
\end{split}
\end{align}
We first consider the right $s\ell(2,\Cb)$-action to find the matrices $B_1, B_2, B_3$.  By the definition of the coadjoint action of $s\ell(2)$ on $s\ell(2)^\ast$, we have
\begin{align}
B_1 = \left(
                    \begin{array}{cccc}
                      0 & 0 & 0 & 0 \\
                      0 & 0 & 0 & -2 \\
                      0 & 0 & 0 & 0 \\
                      0 & 0 & 1 & 0 \\
                    \end{array}
                  \right), \quad B_2 = \left(
                    \begin{array}{cccc}
                      0 & 0 & 0 & 0 \\
                      0 & 0 & 0 & 0 \\
                      0 & 0 & 0 & 2 \\
                      0 & -1 & 0 & 0 \\
                    \end{array}
                  \right), \quad B_3 = \left(
                    \begin{array}{cccc}
                      0 & 0 & 0 & 0 \\
                      0 & 2 & 0 & 0 \\
                      0 & 0 & -2 & 0 \\
                      0 & 0 & 0 & 0 \\
                    \end{array}
                  \right).
\end{align}
Let  $A^1 = (x^i_k), A^2 = (y^i_k), A^3 = (z^i_k)$ represent the $\Fg$-coaction on $V$. Using $B_1, B_2, B_3$ above, the stability condition becomes
\begin{align}
\sum_j A^j \cdot B_j = \left(
                    \begin{array}{cccc}
                      0 & y^0_3 + 2z^0_1 & x^0_3 - 2z^0_2 & -2x^0_1 + 2y^0_2 \\
                      0 & y^1_3 + 2z^1_1 & x^1_3 - 2z^1_2 & -2x^1_1 + 2y^1_2 \\
                      0 & y^2_3 + 2z^2_1 & x^2_3 - 2z^2_2 & -2x^2_1 + 2y^2_2 \\
                      0 & y^3_3 + 2z^3_1 & x^3_3 - 2z^3_2 & -2x^3_1 + 2y^3_2 \\
                    \end{array}
                  \right) = 0.
\end{align}
As before, we make the following observations. First,
\begin{align}
[B_1, A^1] = \left(
                    \begin{array}{cccc}
                      0 & 0 & -x^0_3 & 2x^0_1 \\
                      -2x^3_0 & -2x^3_1 & -2x^3_2 -x^1_3 & -2x^3_3 + 2x^1_1 \\
                      0 & 0 & -x^2_3 & 2x^2_1 \\
                      x^2_0 & x^2_1 & x^2_2 - x^3_3 & x^2_3 - 2x^3_1 \\
                    \end{array}
                  \right) = 2A^3.
\end{align}
Next we have
\begin{align}
[B_2, A^1] = \left(
                    \begin{array}{cccc}
                      0 & x^0_3 & 0 & -2x^0_2 \\
                      0 & x^1_3 & 0 & -2x^1_2 \\
                      2x^3_0 & 2x^3_1 + x^2_3 & 2x^3_2 & 2x^3_3 - 2x^2_2 \\
                      -x^1_0 & -x^1_1 + x^3_3 & -x^1_2 & -x^1_3 - 2x^3_2 \\
                    \end{array}
                  \right) = 0,
\end{align}
 and finally
\begin{align}
[B_3, A^1] = \left(
                    \begin{array}{cccc}
                      0 & -2x^0_1 & 0 & 0 \\
                      0 & 0 & 0 & 0 \\
                      -2x^2_0 & -4x^2_1 & 0 & -2x^2_3 \\
                      0 & -2x^3_1 & 0 & 0 \\
                    \end{array}
                  \right) = -2A^1.
\end{align}
Hence, together with the stability, we get
\begin{align}
A^1 = \left(
                    \begin{array}{cccc}
                      0 & x^0_1 & 0 & 0 \\
                      0 & 0 & 0 & 0 \\
                      x^2_0 & 0 & 0 & 0 \\
                      0 & 0 & 0 & 0 \\
                    \end{array}
                  \right)
\end{align}
and
\begin{align}
[B_1, A^1] = \left(
                    \begin{array}{cccc}
                      0 & 0 & 0 & 2x^0_1 \\
                      0 & 0 & 0 & 0 \\
                      0 & 0 & 0 & 0 \\
                      x^2_0 & 0 & 0 & 0 \\
                    \end{array}
                  \right) = 2A^3.
\end{align}
Similarly we compute
\begin{align}
[B_1, A^2] = \left(
                    \begin{array}{cccc}
                      0 & 0 & 0 & 2y^0_1 \\
                      -2y^3_0 & -2y^3_1 & 0 & 2y^1_1 \\
                      0 & 0 & 0 & 2y^2_1 \\
                      y^2_0 & y^2_1 & 0 & 2y^3_1 \\
                    \end{array}
                  \right) = 0,
\end{align}
as well as
\begin{align}
[B_2, A^2] = \left(
                    \begin{array}{cccc}
                      0 & 0 & 0 & -2y^0_2 \\
                      0 & 0 & 0 & 0 \\
                      0 & 0 & 0 & 0 \\
                      -y^1_0 &  & 0 & 0 \\
                    \end{array}
                  \right) = -2A^3,
\end{align}
and $[B_3, A^2] = 2A^2$. As a result
\begin{align}
A^1 = \left(
                    \begin{array}{cccc}
                      0 & c & 0 & 0 \\
                      0 & 0 & 0 & 0 \\
                      d & 0 & 0 & 0 \\
                      0 & 0 & 0 & 0 \\
                    \end{array}
                  \right), \mbox{  } A^2 = \left(
                    \begin{array}{cccc}
                      0 & 0 & c & 0 \\
                      d & 0 & 0 & 0 \\
                      0 & 0 & 0 & 0 \\
                      0 & 0 & 0 & 0 \\
                    \end{array}
                  \right), \mbox{  } A^3 = \left(
                    \begin{array}{cccc}
                      0 & 0 & 0 & c \\
                      0 & 0 & 0 & 0 \\
                      0 & 0 & 0 & 0 \\
                      \frac{1}{2} d & 0 & 0 & 0 \\
                    \end{array}
                  \right).
\end{align}
We note that for $c = 1, d = 0$ we recover  the Koszul coaction.
}\end{example}

\subsection{SAYD modules over double crossed sum Lie algebras}

In order to investigate SAYD modules over $R(\Fg_2) \acl U(\Fg_1)$ in terms of the representations and the corepresentations of the Lie algebra $\Fg_1 \bowtie \Fg_2$, we develop in this subsection the required compatibility conditions.

\medskip

Let $(\Fg_1,\Fg_2)$ be a matched pair of Lie algebras and  $\Fa := \Fg_1 \bowtie \Fg_2$ be their  double crossed sum Lie algebra. In Subsection \ref{subsection-induced-SAYD} we have seen the compatibility condition \eqref{Lie-module-compatibility} for a module over the double crossed sum $\Fg_1 \bowtie \Fg_2$.

\medskip

Let us consider the compatibility condition for a comodule over $\Fg_1 \oplus \Fg_2$. If $V$ is a comodule over $\Fg_1$ and $\Fg_2$  via
\begin{equation}
v \mapsto v\nsb{-1} \ot v\nsb{0} \in \Fg_1 \ot V \quad \text{ and } \quad v \mapsto v\ns{-1} \ot v\ns{0} \in \Fg_2 \ot V,
\end{equation}
 then we define the linear map
\begin{equation}\label{aux-comodule-on-double-crossed-sum}
v \mapsto v\nsb{-1} \ot v\nsb{0} + v\ns{-1} \ot v\ns{0} \in \Fa \ot V.
\end{equation}
Conversely, if $V$ is a $\Fa$-comodule via  $\Db_\Fa:V\ra \Fa\ot V$,  then we define the  linear maps
\begin{equation}\label{auxy2}
\xymatrix{
\ar[r]^{\Db_{\Fa}} V \ar[dr]_{\Db_{\Fg_1}} & \Fa \ot V \ar[d]^{p_1 \ot \Id} \\
& \Fg_1 \ot V}
\qquad \text{and }\qquad \xymatrix{
\ar[r]^{\Db_{\Fa}} V \ar[dr]_{\Db_{\Fg_2}} & \Fa \ot V \ar[d]^{p_2 \ot \Id} \\
& \Fg_2 \ot V},
\end{equation}
where
\begin{equation}
p_i:\Fa \to \Fg_i, \quad i = 1,2
\end{equation}
are natural projections.

\begin{proposition}\label{proposition-comodule-doublecrossed-sum}
Let $(\Fg_1,\Fg_2)$ be a matched pair of Lie algebras and $\Fa = \Fg_1 \bowtie \Fg_2$ be the double crossed sum. A vector space $V$ is an $\Fa$-comodule if and only if it is a $\Fg_1$-comodule and $\Fg_2$-comodule such that
\begin{equation}\label{aux-18}
v\nsb{-1} \ot v\nsb{0}\ns{-1} \ot v\nsb{0}\ns{0} = v\ns{0}\nsb{-1} \ot v\ns{-1} \ot v\ns{0}\nsb{0}.
\end{equation}
\end{proposition}

\begin{proof}
Assume first that $V$ is an $\Fa$-comodule. By the $\Fa$-coaction compatibility, we have
\begin{align}\label{aux-20}
\begin{split}
& v\nsb{-2} \wg v\nsb{-1} \ot v\nsb{0} + v\nsb{-1} \wg v\nsb{0}\ns{-1} \ot v\nsb{0}\ns{0} \\
& + v\ns{-1} \wg v\ns{0}\nsb{-1} \ot v\ns{0}\nsb{0} + v\ns{-2} \wg v\ns{-1} \ot v\ns{0} = 0.
\end{split}
\end{align}
Applying the antisymmetrization map $\alpha:\Fa \wg \Fa \to U(\Fa) \ot U(\Fa)$, we get
\begin{align}\label{aux-19}
\begin{split}
& (v\nsb{-1} \ot 1) \ot (1 \ot v\nsb{0}\ns{-1}) \ot v\nsb{0}\ns{0} - (1 \ot v\nsb{0}\ns{-1}) \ot (v\nsb{-1} \ot 1) \ot v\nsb{0}\ns{0} \\
& + (1 \ot v\ns{-1}) \ot (v\ns{0}\nsb{-1} \ot 1) \ot v\ns{0}\nsb{0} +\\
& - (v\ns{0}\nsb{-1} \ot 1) \ot (1 \ot v\ns{-1}) \ot v\ns{0}\nsb{0} = 0.
\end{split}
\end{align}
Finally, applying $\Id \ot \ve_{U(\Fg_1)} \ot \ve_{U(\Fg_2)} \ot \Id$ to both sides of the above equation we get \eqref{aux-18}.

\medskip

Let $v \mapsto v\pr{-1} \ot v\pr{0} \in \Fa \ot V$ denote the $\Fa$-coaction on $V$. It can be written as
\begin{equation}
v\pr{-1} \ot v\pr{0} = p_1(v\pr{-1}) \ot v\pr{0} + p_2(v\pr{-1}) \ot v\pr{0}.
\end{equation}
Next, we  shall  prove that
\begin{equation}
v \mapsto p_1(v\pr{-1}) \ot v\pr{0} \in \Fg_1 \ot V \quad \text{ and } \quad v \mapsto p_2(v\pr{-1}) \ot v\pr{0} \in \Fg_2 \ot V
\end{equation}
are coactions. To this end we observe that
\begin{equation}
\alpha(p_1(v\pr{-2}) \wg p_1(v\pr{-1})) \ot v\pr{0} = (p_1 \ot p_1)(\alpha(v\pr{-2} \wg v\pr{-1})) \ot v\pr{0} = 0.
\end{equation}

\medskip

Since the antisymmetrization map $\alpha:\Fg_1 \wg \Fg_1 \to U(\Fg_1) \ot U(\Fg_1)$ is injective, we have
\begin{equation}
p_1(v\pr{-2}) \wg p_1(v\pr{-1}) \ot v\pr{0} = 0,
\end{equation}
proving that $v \mapsto p_1(v\pr{-1}) \ot v\pr{0}$ is a $\Fg_1$-coaction. Similarly $v \mapsto p_2(v\pr{-1}) \ot v\pr{0}$ is a $\Fg_2$-coaction on $V$.

\medskip

Conversely, assume that $V$ is a $\Fg_1$-comodule and $\Fg_2$-comodule such that the compatibility \eqref{aux-18} is satisfied. Then obviously \eqref{aux-20} is true, which is the $\Fa$-comodule compatibility for the coaction \eqref{aux-comodule-on-double-crossed-sum}.
\end{proof}

We proceed to  investigate the relations between AYD modules over the Lie algebras $\Fg_1$ and $\Fg_2$, and AYD modules over the double crossed sum Lie algebra $\Fa = \Fg_1 \bowtie \Fg_2$.

\begin{proposition}\label{auxx-2}
Let $(\Fg_1,\Fg_2)$ be a matched pair of Lie algebras, $\Fa = \Fg_1 \bowtie \Fg_2$,  and $V \in \, ^{\Fa}\rm{conil}\Mc_{\Fa}$. Then $V$ is an AYD module over $\Fa$ if and  only if $V$ is an AYD module over $\Fg_1$ and $\Fg_2$, and the conditions
\begin{align}\label{prop-ax2-1}
&(v \cdot X)\ns{-1} \ot (v \cdot X)\ns{0} = v\ns{-1} \lt X \ot v\ns{0} + v\ns{-1} \ot v\ns{0} \cdot X, \\[.1cm]\label{prop-ax2-2}
& v\ns{-1} \rt X \ot v\ns{0} = 0, \\[.1cm]\label{prop-ax2-3}
& (v \cdot Y)\nsb{-1} \ot (v \cdot Y)\nsb{0} = - Y \rt v\nsb{-1} \ot v\nsb{0} + v\nsb{-1} \ot v\nsb{0} \cdot Y, \\[.1cm]\label{prop-ax2-4}
& Y \lt v\nsb{-1} \ot v\nsb{0} = 0,
\end{align}
are satisfied for any $X \in \Fg_1$, $Y \in \Fg_2$ and any $v \in V$.
\end{proposition}

\begin{proof}
For $V \in \, ^{\Fa}\rm{conil}\Mc_{\Fa}$, we first assume that $V$ is an AYD module over the double crossed sum Lie algebra $\Fa$ via the coaction
\begin{equation}
v \mapsto v\pr{-1} \ot v\pr{0} = v\nsb{-1} \ot v\nsb{0} + v\ns{-1} \ot v\ns{0}.
\end{equation}
As the $\Fa$-coaction is locally conilpotent, by Proposition \ref{proposition-g-AYD-U(g)-AYD}, $V$ is an AYD module over $U(\Fa)$. Then since the projections
\begin{equation}
\pi_1:U(\Fa) = U(\Fg_1) \bowtie U(\Fg_2) \to U(\Fg_1), \quad \pi_2:U(\Fa) = U(\Fg_1) \bowtie U(\Fg_2) \to U(\Fg_2)
\end{equation}
are coalgebra maps, we conclude that $V$ is a comodule over $U(\Fg_1)$ and $U(\Fg_2)$. Finally, since $U(\Fg_1)$ and $U(\Fg_2)$ are Hopf subalgebras of $U(\Fa)$, the AYD conditions on $U(\Fg_1)$ and $U(\Fg_2)$ are immediate, and thus $V$ is an AYD module over $\Fg_1$ and $\Fg_2$.

\medskip

We now prove the compatibility conditions \eqref{prop-ax2-1} to \eqref{prop-ax2-4}. To this end, we will make use of the AYD condition for an arbitrary $X \oplus Y \in \Fa$ and $v \in V$. On the one hand we have
\begin{align}\label{aux-21}
\begin{split}
& [v\pr{-1}, X \oplus Y] \ot v\pr{0} + v\pr{-1} \ot v\pr{0} \cdot (X \oplus Y) = \\
& [v\nsb{-1} \oplus 0, X \oplus Y] \ot v\nsb{0} + [0 \oplus v\ns{-1}, X \oplus Y] \ot v\ns{0} \\
& + v\nsb{-1} \ot v\nsb{0} \cdot (X \oplus Y) + v\ns{-1} \ot v\ns{0} \cdot (X \oplus Y) \\
& = ([v\nsb{-1},X] - Y \rt v\nsb{-1} \oplus -Y \lt v\nsb{-1}) \ot v\nsb{0} \\
& + (v\ns{-1} \rt X \oplus [v\ns{-1},Y] + v\ns{-1} \lt X) \ot v\ns{0} \\
& + (v\nsb{-1} \oplus 0) \ot v\nsb{0} \cdot (X \oplus Y) + (0 \oplus v\ns{-1}) \ot v\ns{0} \cdot (X \oplus Y) \\
& = ((v \cdot X)\nsb{-1} \oplus 0) \ot (v \cdot X)\nsb{0} + (0 \oplus (v \cdot Y)\ns{-1}) \ot (v \cdot Y)\ns{0} \\
& + (- Y \rt v\nsb{-1} \oplus -Y \lt v\nsb{-1}) \ot v\nsb{0} + (v\ns{-1} \rt X \oplus v\ns{-1} \lt X) \ot v\ns{0} \\
& + (v\nsb{-1} \oplus 0) \ot v\nsb{0} \cdot Y + (0 \oplus v\ns{-1}) \ot v\ns{0} \cdot X.
\end{split}
\end{align}
On the other hand, we have
\begin{align}\label{aux-22}
\begin{split}
& (v \cdot (X \oplus Y))\pr{-1} \ot (v \cdot (X \oplus Y))\pr{0} = ((v \cdot X)\nsb{-1} \oplus 0) \ot (v \cdot X)\nsb{0} + \\
& ((v \cdot Y)\nsb{-1} \oplus 0) \ot (v \cdot Y)\nsb{0} + (0 \oplus (v \cdot X)\ns{-1}) \ot (v \cdot X)\ns{0} \\
& + (0 \oplus (v \cdot Y)\ns{-1}) \ot (v \cdot Y)\ns{0}.
\end{split}
\end{align}
Since  $V$ is an AYD module over $\Fg_1$ and $\Fg_2$, the AYD compatibility over $\Fg_1 \oplus \Fg_2$ translates into
\begin{align}\label{aux-23}
\begin{split}
& ((v \cdot Y)\nsb{-1} \oplus 0) \ot (v \cdot Y)\nsb{0} + (0 \oplus (v \cdot X)\ns{-1}) \ot (v \cdot X)\ns{0} = \\
& (- Y \rt v\nsb{-1} \oplus -Y \lt v\nsb{-1}) \ot v\nsb{0} + (v\ns{-1} \rt X \oplus v\ns{-1} \lt X) \ot v\ns{0} \\
&  + (v\nsb{-1} \oplus 0) \ot v\nsb{0} \cdot Y + (0 \oplus v\ns{-1}) \ot v\ns{0} \cdot X.
\end{split}
\end{align}
Finally, we set $Y = 0$ to get \eqref{prop-ax2-1} and \eqref{prop-ax2-2}. The equations \eqref{prop-ax2-3} and \eqref{prop-ax2-4} are similarly obtained by setting $X=0$.

\medskip

The converse argument is clear.
\end{proof}

In general, if $V$ is an AYD module over the double crossed sum Lie algebra $\Fa = \Fg_1 \oplus \Fg_2$, then $V$ is not necessarily an AYD module over the Lie algebras $\Fg_1$ and $\Fg_2$.

\begin{example}\label{example-sl2}\rm{
Consider the Lie algebra $s\ell(2) = \langle X,Y,Z \rangle$,
\begin{equation}
[Y,X] = X, \quad [Z,X] = Y, \quad [Z,Y] = Z.
\end{equation}
Then $s\ell(2) = \Fg_1 \bowtie \Fg_2$ for $\Fg_1 = \langle X,Y \rangle$ and $\Fg_2 = \langle Z \rangle$.

\medskip

The symmetric algebra $V = S({s\ell(2)}^\ast)$ is a right-left AYD module over $s\ell(2)$. The module structure  is defined by the coadjoint action, and the comodule structure is given by  the Koszul coaction (which is not locally conilpotent).

\medskip

We now show that $V$ is not an AYD module over $\Fg_1$.  Let $\{\t^X,\t^Y,\t^Z\}$ be a dual basis for $s\ell(2)$. The linear map
\begin{equation}
\Db_{\Fg_1}: V \to \Fg_1 \ot V, \quad v \mapsto X \ot v\t^X + Y \ot v\t^Y,
\end{equation}
which is the projection onto the Lie algebra $\Fg_1$, endows $V$ with a left $\Fg_1$-comodule structure. However,  the AYD compatibility on $\Fg_1$ is not satisfied. Indeed, on the one hand we have
\begin{equation}
\Db_{\Fg_1}(v \cdot X) = X \ot (v \cdot X)\t^X + Y \ot (v \cdot X)\t^Y,
\end{equation}
and on the other hand we get
\begin{align}
\begin{split}
& [X,X] \ot v\t^X + [Y,X] \ot v\t^Y + X \ot (v\t^X) \cdot X + Y \ot (v\t^Y) \cdot X \\
& =X \ot (v \cdot X)\t^X + Y \ot (v \cdot X)\t^Y - Y \ot v\t^Z.
\end{split}
\end{align}
}\end{example}

\begin{remark}{\rm
Assume that the mutual actions of $\Fg_1$ and $\Fg_2$ are trivial. In this case, if $V$ is an AYD module over $\Fg_1 \bowtie \Fg_2$, then it is an AYD module over $\Fg_1$ and $\Fg_2$.

\medskip

To see this, let us apply $p_1\ot \Id_V$ to the both hand sides of the AYD condition \eqref{aux-SAYD-condition}  for  $X \oplus 0 \in \Fa$, where  $p_1:\Fa \to \Fg_1$ is the obvious projection. We get
\begin{align}\label{auxy1}
\begin{split}
& p_1([v\pr{-1}, X \oplus 0]) \ot v\pr{0} + p_1(v\pr{-1}) \ot v\pr{0} \cdot (X \oplus 0) \\
& = p_1((v \cdot (X \oplus 0))\pr{-1}) \ot (v \cdot (X \oplus 0))\pr{0}.
\end{split}
\end{align}
Since in this case the projection $p_1:\Fa \to \Fg_1$ is a map of Lie algebras, the equation \eqref{auxy1} reads
\begin{equation}
[p_1(v\pr{-1}), X] \ot v\pr{0} + p_1(v\pr{-1}) \ot v\pr{0} \cdot X = p_1((v \cdot X)\pr{-1}) \ot (v \cdot X)\pr{0},
\end{equation}
which is the AYD compatibility for the $\Fg_1$-coaction.

\medskip

Similarly we prove  that $V$ is an  AYD module over the Lie algebra $\Fg_2$.
}\end{remark}

Let $\Fa = \Fg_1 \bowtie \Fg_2$ be a double crossed sum  Lie algebra and $V$ be an SAYD module over $\Fa$.  In the next example we show that $V$ is not necessarily stable over  $\Fg_1$ and $\Fg_2$.

\begin{example}{\rm
Consider the Lie algebra $\Fa = g\ell(2) = \langle Y^1_1, Y^1_2, Y^2_1, Y^2_2 \rangle$ with a dual basis $\{\t^1_1, \t^2_1, \t^1_2, \t^2_2\}$.

\medskip

We have a decomposition $g\ell(2) = \Fg_1 \bowtie \Fg_2$, where $\Fg_1 = \langle Y^1_1, Y^1_2 \rangle$ and $\Fg_2 = \langle Y^2_1, Y^2_2 \rangle$. Let  $V := S({g\ell(2)}^\ast)$ be the symmetric algebra as an SAYD module over $g\ell(2)$ with the coadjoint action and the Koszul coaction.  Then the $\Fg_1$-coaction on $V$ becomes
\begin{equation}
v \mapsto v\ns{-1} \ot v\ns{0} = Y^1_1 \ot v\t^1_1 + Y^1_2 \ot v\t^2_1.
\end{equation}
Accordingly we have
\begin{equation}
{\t^2_1}\ns{0} \cdot {\t^2_1}\ns{-1} = - \Lc_{Y^1_1}\t^1_1 - \Lc_{Y^1_2}\t^2_1 = - \t^2_1\t^1_1 \neq 0.
\end{equation}
}\end{example}

We know that if a comodule   over a  Lie algebra $\Fg$  is locally conilpotent then it can be  lifted to a comodule  over  $U(\Fg)$. In the rest of this subsection, we will be interested  in translating  Proposition \ref{auxx-2} to AYD modules over universal enveloping algebras.

\begin{proposition}\label{auxx-1}
Let $\Fa = \Fg_1 \bowtie \Fg_2$ be a double crossed sum Lie algebra and $V$ be a left comodule over $\Fa$. Then the $\Fa$-coaction  is locally conilpotent if and only if the corresponding $\Fg_1$-coaction and $\Fg_2$-coaction are locally conilpotent.
\end{proposition}

\begin{proof}
By \eqref{auxy2} we know that $\Db_{\Fa} = \Db_{\Fg_1} + \Db_{\Fg_2}$. Therefore,
\begin{equation}
\Db_{\Fa}^2(v) = v\nsb{-2} \ot v\nsb{-1} \ot v\nsb{0} + v\ns{-2} \ot v\ns{-1} \ot v\ns{0} + v\nsb{-1} \ot v\nsb{0}\ns{-1} \ot v\nsb{0}\ns{0}.
\end{equation}
By induction we can assume that
\begin{align}
\begin{split}
& \Db_{\Fa}^k(v) = v\nsb{-k} \ot \ldots \ot v\nsb{-1} \ot v\nsb{0} + v\ns{-k} \ot \ldots \ot v\ns{-1} \ot v\ns{0} \\
& + \sum_{p + q = k} v\nsb{-p} \ot \ldots \ot v\nsb{-1} \ot v\nsb{0}\ns{-q} \ot \ldots \ot v\nsb{0}\ns{-1} \ot v\nsb{0}\ns{0},
\end{split}
\end{align}
and we apply the coaction one more time to get
\begin{align}
\begin{split}
& \Db_{\Fa}^{k+1}(v) = v\nsb{-k-1} \ot \ldots \ot v\nsb{-1} \ot v\nsb{0} + v\ns{-k-1} \ot \ldots \ot v\ns{-1} \ot v\ns{0} \\
& + \sum_{p + q = k} v\nsb{-p} \ot \ldots \ot v\nsb{-1} \ot v\nsb{0}\ns{-q} \ot \ldots \ot v\nsb{0}\ns{-1} \ot v\nsb{0}\ns{0}\nsb{-1} \ot v\nsb{0}\ns{0}\nsb{0} \\
& + \sum_{p + q = k} v\nsb{-p} \ot \ldots \ot v\nsb{-1} \ot v\nsb{0}\ns{-q-1} \ot \ldots \ot v\nsb{0}\ns{-1} \ot v\nsb{0}\ns{0} \\
& = v\nsb{-k-1} \ot \ldots \ot v\nsb{-1} \ot v\nsb{0} + v\ns{-k-1} \ot \ldots \ot v\ns{-1} \ot v\ns{0} \\
& + \sum_{p + q = k} v\nsb{-p-1} \ot \ldots \ot v\nsb{-1} \ot v\nsb{0}\ns{-q} \ot \ldots \ot v\nsb{0}\ns{-1} \ot v\nsb{0}\ns{0} \ot v\nsb{0}\ns{0} \\
& + \sum_{p + q = k} v\nsb{-p} \ot \ldots \ot v\nsb{-1} \ot v\nsb{0}\ns{-q-1} \ot \ldots \ot v\nsb{0}\ns{-1} \ot v\nsb{0}\ns{0}.
\end{split}
\end{align}
In the second equality we used \eqref{aux-18}. This result immediately implies the claim.
\end{proof}

Let $V$ be a locally conilpotent comodule over $\Fg_1$ and $\Fg_2$. We denote by
\begin{equation}
V \to U(\Fg_1) \ot V, \quad  v \mapsto v\snsb{-1} \ot v\snsb{0}
\end{equation}
the lift  of the $\Fg_1$-coaction and similarly by
\begin{equation}
V \to U(\Fg_2) \ot V, \quad  v \mapsto v\sns{-1} \ot v\sns{0}
\end{equation}
the lift of the $\Fg_2$-coaction.

\begin{corollary}\label{corollary-58}
Let $\Fa = \Fg_1 \bowtie \Fg_2$ be a double crossed sum Lie algebra and $V \in \, ^{\Fa}\rm{conil}\Mc_{\Fa}$. Then the $\Fa$-coaction lifts  to the $U(\Fa)$-coaction
\begin{equation}\label{aux-corollary58-coaction}
v \mapsto v\snsb{-1} \ot v\snsb{0}\sns{-1} \ot v\snsb{0}\sns{0} \in U(\Fg_1) \bowtie U(\Fg_2) \ot V.
\end{equation}
\end{corollary}

\begin{proposition}\label{proposition-24}
Let $(\Fg_1,\Fg_2)$ be a matched pair of Lie algebras, $\Fa = \Fg_1 \bowtie \Fg_2$ be their double crossed sum and $V \in \, ^{\Fa}\rm{conil}\Mc_{\Fa}$. Then $V$ is an  AYD module  over $\Fa$ if and only if $V$ is an AYD module over $\Fg_1$ and $\Fg_2$, and
\begin{align}\label{auxy3}
&(v \cdot u)\sns{-1} \ot (v \cdot u)\sns{0} = v\sns{-1} \lt u\ps{1} \ot v\sns{0} \cdot u\ps{2}, \\\label{auxy4}
&v\sns{-1} \rt u \ot v\sns{0} = u \ot v, \\\label{auxy5}
&(v \cdot u')\snsb{-1} \ot (v \cdot u')\snsb{0} = S(u'\ps{2}) \rt v\snsb{-1} \ot v\snsb{0} \cdot v\ps{1}, \\\label{auxy6}
&u' \lt v\snsb{-1} \ot v\snsb{0} = u' \ot v.
\end{align}
are satisfied for any $v \in V$, any $u \in U(\Fg_1)$ and $u' \in U(\Fg_2)$.
\end{proposition}

\begin{proof}
Let $V$ be an AYD module over $\Fa$. Since the coaction is conilpotent, it lifts to an  AYD module over $U(\Fa)$ by Proposition \ref{proposition-g-AYD-U(g)-AYD}. We write the AYD condition for $u \bowtie 1 \in U(\Fa)$ as
\begin{align}\label{auxy7}
\begin{split}
& (v \cdot u)\snsb{-1} \ot (v \cdot u)\snsb{0}\sns{-1} \ot (v \cdot u)\snsb{0}\sns{0}  \\
&= (S(u\ps{3}) \ot 1)(v\snsb{-1} \ot v\snsb{0}\sns{-1})(u\ps{1} \ot 1) \ot v\snsb{0}\sns{0} \cdot (u\ps{2} \ot 1)  \\
& =S(u\ps{4})v\snsb{-1}(v\snsb{0}\sns{-2} \rt u\ps{1}) \ot v\snsb{0}\sns{-1} \lt u\ps{2} \ot v\snsb{0}\sns{0} \cdot u\ps{3}.
\end{split}
\end{align}
Applying $\ve \ot \Id \ot \Id$ to both sides of \eqref{auxy7},  we get \eqref{auxy3}. Similarly we get
\begin{equation}
(v \cdot u)\snsb{-1} \ot (v \cdot u)\snsb{0} = S(u\ps{3})v\snsb{-1}(v\snsb{0}\sns{-1} \rt u\ps{1}) \ot v\snsb{0}\sns{0} \cdot u\ps{2},
\end{equation}
which yields, using AYD condition on the left hand side,
 \begin{equation}
S(u\ps{3})v\snsb{-1}u\ps{1} \ot v\snsb{0}u\ps{2} = S(u\ps{3})v\snsb{-1}(v\snsb{0}\sns{-1} \rt u\ps{1}) \ot v\snsb{0}\sns{0} \cdot u\ps{2}.
\end{equation}
This  immediately implies \eqref{auxy4}. Switching to the Lie algebra $\Fg_2$ and writing the AYD condition for $1 \bowtie u' \in U(\Fa)$, we obtain \eqref{auxy5} and  \eqref{auxy6}.

\medskip

Conversely, for $V \in \, ^{\Fa}\rm{conil}\Mc_{\Fa}$ which is also an AYD module over  $\Fg_1$ and $\Fg_2$, assume that \eqref{auxy3} to \eqref{auxy6} are satisfied. Then  $V$ is an AYD module over $U(\Fg_1)$ and $U(\Fg_2)$. We  show that \eqref{auxy3} and \eqref{auxy4} together imply the AYD condition for the elements of the form $u \bowtie 1 \in U(\Fg_1) \bowtie U(\Fg_2)$. Indeed,
\begin{align}
\begin{split}
& (v \cdot u)\snsb{-1} \ot (v \cdot u)\snsb{0}\sns{-1} \ot (v \cdot u)\snsb{0}\sns{0} \\
& =S(u\ps{3})v\snsb{-1}u\ps{1} \ot (v\snsb{0} \cdot u\ps{2})\sns{-1} \ot (v\snsb{0} \cdot u\ps{2})\sns{0}  \\
& =S(u\ps{4})v\snsb{-1}u\ps{1} \ot v\snsb{0}\sns{-1} \lt u\ps{2} \ot v\snsb{0}\sns{0} \cdot u\ps{3}  \\
& =S(u\ps{4})v\snsb{-1}(v\snsb{0}\sns{-2} \rt u\ps{1}) \ot v\snsb{0}\sns{-1} \lt u\ps{2} \ot v\snsb{0}\sns{0} \cdot u\ps{3},
\end{split}
\end{align}
where the first equality follows from the AYD condition on $U(\Fg_1)$, the second equality follows from \eqref{auxy3}, and the last equality is obtained by using \eqref{auxy4}. Similarly, using \eqref{auxy5} and \eqref{auxy6} we prove the AYD condition for  the elements of the form $1 \bowtie u' \in U(\Fg_1) \bowtie U(\Fg_2)$. The proof is then complete, since the AYD condition is multiplicative.
\end{proof}

The following generalization of Proposition \ref{proposition-24} is now straightforward.

\begin{corollary}\label{auxx-24}
Let $(\Uc,\Vc)$ be a mutual pair of Hopf algebras and $V$ a linear space. Then $V$ is an AYD module over $\Uc\bi \Vc$ if and only if $V$ is an AYD module over $\Uc$ and $\Vc$, and  the conditions
\begin{align}\label{auxy13}
&(v \cdot u)\sns{-1} \ot (v \cdot u)\sns{0} = v\sns{-1} \lt u\ps{1} \ot v\sns{0} \cdot u\ps{2}, \\\label{auxy14}
&v\sns{-1} \rt u \ot v\sns{0} = u \ot v, \\\label{auxy15}
& (v \cdot u')\snsb{-1} \ot (v \cdot u')\snsb{0} = S(u'\ps{2}) \rt v\snsb{-1} \ot v\snsb{0} \cdot u'\ps{1}, \\\label{auxy16}
&v \lt v\snsb{-1} \ot v\snsb{0} = u' \ot v
\end{align}
are satisfied for any $v \in V$, any $u \in \Uc$ and $u' \in \Vc$.
\end{corollary}

\subsection{SAYD modules over Lie-Hopf algebras}

In this subsection we investigate SAYD modules over the Lie-Hopf algebras $R(\Fg_2)\acl U(\Fg_1)$ associated to the matched pair of Lie algebras $(\Fg_1,\Fg_2)$ in terms of representations and corepresentations of the double crossed sum $\Fa = \Fg_1 \bowtie \Fg_2$.

\medskip

If $V$ is a module over a bicrossed product Hopf algebra $\mathcal{F} \acl \mathcal{U}$, then since $\mathcal{F}$ and $\mathcal{U}$ are subalgebras of $\mathcal{F} \acl \mathcal{U}$, $V$ is a module over $\mathcal{F}$ and $\mathcal{U}$. More explicitly, we have the following lemma.

\begin{lemma}\label{module on bicrossed product}
Let $(\mathcal{F}, \mathcal{U})$ be a matched pair of Hopf algebras and $V$ a linear space. Then $V$ is a right module over the bicrossed product Hopf algebra $\mathcal{F} \acl \mathcal{U}$ if and only if $V$ is a right module over $\mathcal{F}$ and  a right module over $\mathcal{U}$, such that
\begin{equation}\label{aux-25}
(v \cdot u) \cdot f = (v \cdot (u\ps{1} \rt f)) \cdot u\ps{2}.
\end{equation}
\end{lemma}

\begin{proof}
Let $V$ be a right module over $\Fc \acl \Uc$. Then for any $v \in V$, $f \in \Fc$ and $u \in \Uc$,
\begin{align}
\begin{split}
& (v \cdot u) \cdot f = (v \cdot (1 \acl u)) \cdot (f \acl 1) = v \cdot ((1 \acl u) \cdot (f \acl 1))\\
&= v \cdot (u\ps{1} \rt f \acl u\ps{2}) = (v \cdot (u\ps{1} \rt f)) \cdot u\ps{2}.
\end{split}
\end{align}
Conversely, if \eqref{aux-25} holds, then
\begin{equation}
(v \cdot (1 \acl u)) \cdot (f \acl 1) = v \cdot ((1 \acl u) \cdot (f \acl 1))
\end{equation}
from which we conclude that $V$ is a right module over $\Fc \acl \Uc$.
\end{proof}

Let $(\Fg_1,\Fg_2)$ be a matched pair of Lie algebras and $V$ be a  module over the double crossed sum $\Fg_1 \bowtie \Fg_2$ such that $\Fg_1 \bowtie \Fg_2$-coaction is locally conilpotent.

\medskip

Being a right $\Fg_1$-module, $M$ has a right $U(\Fg_1)$-module structure. Similarly, since it is a locally conilpotent left $\Fg_2$-comodule, $M$ is a right $R(\Fg_2)$-module. Then we define
\begin{align}\label{aux-17}
\begin{split}
& V \ot R(\Fg_2) \acl U(\Fg_1) \to V \\
& v \ot (f \acl u) \mapsto (v \cdot f) \cdot u = f(v\sns{-1})v\sns{0} \cdot u.
\end{split}
\end{align}

\begin{corollary}
Let $(\Fg_1,\Fg_2)$ be a matched pair of Lie algebras and $V$ be an AYD module over the double crossed sum $\Fg_1 \bowtie \Fg_2$ such that $\Fg_1 \bowtie \Fg_2$-coaction is locally conilpotent. Then $V$ has a right $R(\Fg_2) \acl U(\Fg_1)$-module structure via \eqref{aux-17}.
\end{corollary}

\begin{proof}
For $f \in R(\Fg_2)$, $u \in U(\Fg_1)$ and $v \in V$, we have
\begin{align}
\begin{split}
& (v \cdot u) \cdot f = f((v \cdot u)\sns{-1})(v \cdot u)\sns{0} = f(v\sns{-1} \lt u\ps{1})v\sns{0} \cdot u\ps{2} \\
& = (u\ps{1} \rt f)(v\sns{-1})v\sns{0} \cdot u\ps{2} = (v \cdot (u\ps{1} \rt f)) \cdot u\ps{2},
\end{split}
\end{align}
where in the second equality we used Proposition \ref{proposition-24}. So the proof is complete by Lemma \ref{module on bicrossed product}.
\end{proof}

Let us assume that $V$ is a left comodule over the bicrossed product $\mathcal{F} \acl \mathcal{U}$. Since the projections $\pi_1 := \Id_{\mathcal{F}} \ot \ve_{\mathcal{U}}:\mathcal{F} \acl \mathcal{U} \to \mathcal{F}$ and $\pi_2 := \ve_{\mathcal{F}} \ot \Id_{\mathcal{U}}:\mathcal{F} \acl \mathcal{U} \to \mathcal{U}$ are coalegbra maps, $V$ becomes  a left $\mathcal{F}$-comodule as well as a left $\mathcal{U}$-comodule via $\pi_1$ and $\pi_2$. Denoting these comodule structures by
\begin{equation}
v \mapsto v^{\sns{-1}} \ot v^{\sns{0}} \in \mathcal{F} \ot V \quad \text{and} \quad v \mapsto v\snsb{-1} \ot v\snsb{0} \in \mathcal{U} \ot V,
\end{equation}
we consider the linear map
\begin{equation}\label{auxy12}
v \mapsto v^{\sns{-1}} \ot v^{\sns{0}}\snsb{-1} \ot v^{\sns{0}}\snsb{0} \in \mathcal{F} \acl \mathcal{U} \ot V.
\end{equation}

\begin{lemma}\label{comodule on bicrossed product}
Let $(\mathcal{F}, \mathcal{U})$ be a matched pair of Hopf algebras and $M$ a linear space. Then $V$ is a left comodule over the bicrossed product Hopf algebra $\mathcal{F} \acl \mathcal{U}$ via \eqref{auxy12} if and only if it is a left comodule over $\mathcal{F}$ and  a left comodule over $\mathcal{U}$, such that  for any $v \in V$
\begin{equation}\label{aux-27}
(v^{\sns{0}}\snsb{-1})^{\pr{0}} \ot v^{\sns{-1}} \cdot (v^{\sns{0}}\snsb{-1})^{\pr{1}} \ot v^{\sns{0}}\snsb{0} = v\snsb{-1} \ot (v\snsb{0})^{\sns{-1}} \ot (v\snsb{0})^{\sns{0}},
\end{equation}
where $u \mapsto u^{\pr{0}} \ot u^{\pr{1}} \in \mathcal{U} \ot \mathcal{F}$ is the right $\mathcal{F}$-coaction on $\mathcal{U}$.
\end{lemma}

\begin{proof}
Let us assume that $M$ is a comodule over the bicrossed product Hopf algebra $\mathcal{F} \acl \mathcal{U}$. Then by the coassociativity of the coaction, we have
\begin{align}\label{aux-28}
\begin{split}
& v^{\sns{-2}} \acl (v^{\sns{0}}\snsb{-2})^{\pr{0}} \ot v^{\sns{-1}} \cdot (v^{\sns{0}}\snsb{-2})^{\pr{1}} \acl v^{\sns{0}}\snsb{-1} \ot v^{\sns{0}}\snsb{0} \\
& = v^{\sns{-1}} \acl v^{\sns{0}}\snsb{-1} \ot (v^{\sns{0}}\snsb{0})^{\sns{-1}} \acl (v^{\sns{0}}\snsb{0})^{\sns{0}}\snsb{-1} \ot (v^{\sns{0}}\snsb{0})^{\sns{0}}\snsb{0}.
\end{split}
\end{align}
By applying  $\ve_{\mathcal{F}} \ot \Id_{\mathcal{U}} \ot \Id_{\mathcal{F}} \ot \ve_{\mathcal{U}} \ot \Id_{M}$ on both hand sides of \eqref{aux-28}, we get
\begin{equation}
(v^{\sns{0}}\snsb{-1})^{\pr{0}} \ot v^{\sns{-1}} \cdot (v^{\sns{0}}\snsb{-1})^{\pr{1}} \ot v^{\sns{0}}\snsb{0} = v\snsb{-1} \ot (v\snsb{0})^{\sns{-1}} \ot (v\snsb{0})^{\sns{0}}.
\end{equation}
Conversely, assume that \eqref{aux-27} holds for any $v \in V$. This results in
\begin{align}
\begin{split}
& v^{\sns{-2}} \ot (v^{\sns{0}}\snsb{-1})^{\pr{0}} \ot v^{\sns{-1}} \cdot (v^{\sns{0}}\snsb{-1})^{\pr{1}} \ot v^{\sns{0}}\snsb{0} \\
& = v^{\sns{-1}} \ot v^{\sns{0}}\snsb{-1} \ot (v^{\sns{0}}\snsb{0})^{\sns{-1}} \ot (v^{\sns{0}}\snsb{0})^{\sns{0}},
\end{split}
\end{align}
which implies  \eqref{aux-28} \ie the coassociativity of the $\mathcal{F} \acl \mathcal{U}$-coaction.
\end{proof}

\begin{corollary}
Let $(\Fg_1,\Fg_2)$ be a matched pair of Lie algebras and $V$ be an AYD module over the double crossed sum $\Fg_1 \bowtie \Fg_2$ with locally finite action and locally conilpotent coaction. Then $V$ has a left $R(\Fg_2) \acl U(\Fg_1)$-comodule structure by \eqref{auxy12}.
\end{corollary}

\begin{proof}
Since $V$ is a locally conilpotent left $\Fg_1$-comodule, it has a left $U(\Fg_1)$-comodule structure. On the other hand, being a locally finite right $\Fg_2$-module, $V$ is a left $R(\Fg_2)$-comodule \cite{Hoch74}.
By Proposition \ref{proposition-24} we have
\begin{equation}
(v \cdot u')\snsb{-1} \ot (v \cdot u')\snsb{0} = S(u'\ps{2}) \rt v\snsb{-1} \ot v\snsb{0} \cdot u'\ps{1},
\end{equation}
or equivalently
\begin{equation}
u'\ps{2} \rt (v \cdot u'\ps{1})\snsb{-1} \ot (v \cdot u'\ps{1})\snsb{0} = v\snsb{-1} \ot v\snsb{0} \cdot u'.
\end{equation}
Using the $R(\Fg_2)$-coaction on $V$ and $R(\Fg_2)$-coaction on $U(\Fg_1)$, we can translate the latter into
\begin{equation}
(v^{\sns{-1}} \cdot (v^{\sns{0}}\snsb{-1})^{\pr{1}})(u') (v^{\sns{0}}\snsb{-1})^{\pr{0}} \ot v^{\sns{0}}\snsb{0} = v\snsb{-1} \ot (v\snsb{0})^{\sns{0}}((v\snsb{0})^{\sns{-1}})(u').
\end{equation}
Finally, by the non-degenerate pairing between $U(\Fg_2)$ and $R(\Fg_2)$ we get
\begin{equation}
(v^{\sns{0}}\snsb{-1})^{\pr{0}} \ot v^{\sns{-1}} \cdot (v^{\sns{0}}\snsb{-1})^{\pr{1}} \ot v^{\sns{0}}\snsb{0} = v\snsb{-1} \ot (v\snsb{0})^{\sns{-1}} \ot (v\snsb{0})^{\sns{0}},
\end{equation}
\ie the $R(\Fg_2) \acl U(\Fg_1)$-coaction compatibility.
\end{proof}

Having understood the module and the comodule compatibilities over Lie-Hopf algebras, we now proceed to the YD-module compatibility conditions.

\begin{proposition}\label{aux-41}
Let $(\mathcal{F}, \mathcal{U})$ be a matched pair of Hopf algebras and $V$ be a right module and left comodule over $\mathcal{F} \acl \mathcal{U}$ such that via the corresponding module and comodule structures M becomes a YD-module over $\Uc$. Then $V$ is a YD-module over $\Fc\acl\Uc$ if and only if $V$ is a YD-module over  $\Fc$ via the corresponding module and comodule structures, and the conditions
\begin{align}\label{auxy8}
&(v \cdot f)\snsb{-1} \ot (v \cdot f)\snsb{0} = v\snsb{-1} \ot v\snsb{0} \cdot f, \\\label{auxy9}
&v^{\sns{-1}}f\ps{1} \ot v^{\sns{0}}f\ps{2} = v^{\sns{-1}}(v^{\sns{0}}\snsb{-1} \rt f\ps{1}) \ot v^{\sns{0}}\snsb{0} \cdot f\ps{2}, \\\label{auxy10}
&v^{\sns{-1}} \ot v^{\sns{0}} \cdot u = (u\ps{1})^{\pr{1}}(u\ps{2} \rt (v \cdot (u\ps{1})^{\pr{0}})^{\sns{-1}}) \ot (v \cdot (u\ps{1})^{\pr{0}})^{\sns{0}}, \\\label{auxy11}
&v\snsb{-1}u^{\pr{0}} \ot v\snsb{0} \cdot u^{\pr{1}} = v\snsb{-1}u \ot v\snsb{0}
\end{align}
are satisfied.
\end{proposition}

\begin{proof}
First we assume that $M$ is a YD module over $\mathcal{F} \acl \mathcal{U}$. Since $\Fc$ is a Hopf subalgebra of $\Fc \acl \Uc$, $M$ is a YD module over $\Fc$.

\medskip

Next, we prove the compatibilities \eqref{auxy8} to \eqref{auxy11}. Writing \eqref{aux-right-left-YD-module} for an arbitrary $f \acl 1 \in \mathcal{F} \acl \mathcal{U}$, we get
\begin{align}
\begin{split}
& (f\ps{2} \acl 1) \cdot ((v \cdot f\ps{1})^{\sns{-1}} \acl (v \cdot f\ps{1})^{\sns{0}}\snsb{-1}) \ot (v \cdot f\ps{1})^{\sns{0}}\snsb{0} \\
& = (v^{\sns{-1}} \acl v^{\sns{0}}\snsb{-1}) \cdot (f\ps{1} \acl 1) \ot v^{\sns{0}}\snsb{0} \cdot f\ps{2}.
\end{split}
\end{align}
Using the YD condition for $\mathcal{F}$ on the left hand side
\begin{align}\label{aux-33}
\begin{split}
& f\ps{2}(v \cdot f\ps{1})^{\sns{-1}} \acl (v \cdot f\ps{1})^{\sns{0}}\snsb{-1} \ot (v \cdot f\ps{1})^{\sns{0}}\snsb{0} \\
& = v^{\sns{-1}}f\ps{1} \acl (v^{\sns{0}}f\ps{2})\snsb{-1} \ot (v^{\sns{0}}f\ps{2})\snsb{0} \\
& = v^{\sns{-1}}(v^{\sns{0}}\snsb{-2} \rt f\ps{1}) \acl v^{\sns{0}}\snsb{-1} \ot v^{\sns{0}}\snsb{0} \cdot f\ps{2}.
\end{split}
\end{align}
Finally applying $\ve_{\mathcal{F}} \ot \Id_{\mathcal{U}} \ot \Id_V$ on both hand sides of \eqref{aux-33} we get \eqref{auxy8}. Similarly we apply $\Id_{\mathcal{F}} \ot \ve_{\mathcal{U}} \ot \Id_V$ to get \eqref{auxy9}.

\medskip

By the same argument,  the YD compatibility for an element of the form $1 \acl u \in \Fc\acl \Uc$,  followed by the YD compatibility on $\Uc$ yields \eqref{auxy10} and \eqref{auxy11}.

\medskip

Conversely, let us assume that $V$ is a right-left YD-module over $\Fc$ and the conditions \eqref{auxy8},\ldots,\eqref{auxy11} are satisfied. We will prove that the YD condition over $\mathcal{F} \acl \mathcal{U}$ holds for  the elements of the form $f \acl 1 \in \mathcal{F} \acl \mathcal{U}$ and  $1 \acl u \in \mathcal{F} \acl \mathcal{U}$. By \eqref{auxy9}, we have
\begin{align}
\begin{split}
& v^{\sns{-1}}f\ps{1} \acl (v^{\sns{0}}f\ps{2})\snsb{-1} \ot (v^{\sns{0}}f\ps{2})\snsb{0} \\
& = v^{\sns{-1}}(v^{\sns{0}}\snsb{-1} \rt f\ps{1}) \acl (v^{\sns{0}}\snsb{0} \cdot f\ps{2})\snsb{-1} \ot (v^{\sns{0}}\snsb{0} \cdot f\ps{2})\snsb{0},
\end{split}
\end{align}
which, using  \eqref{auxy8}, implies the YD compatibility for the elements of the form $f \acl 1 \in \mathcal{F} \acl \mathcal{U}$.

\medskip
 Next, by \eqref{auxy10} we have
\begin{align}
\begin{split}
& (u\ps{1})^{\pr{1}}(u\ps{2} \rt (v \cdot (u\ps{1})^{\pr{0}})^{\sns{-1}}) \acl \\
& \hspace{4cm} u\ps{3}(v \cdot (u\ps{1})^{\pr{0}})^{\sns{0}}\snsb{-1} \ot (v \cdot (u\ps{1})^{\pr{0}})^{\sns{0}}\snsb{0} \\
& = v^{\sns{-1}} \acl u\ps{2}(v^{\sns{0}} \cdot u\ps{1})\snsb{-1} \ot (v^{\sns{0}} \cdot u\ps{1})\snsb{0}, \end{split}
\end{align}
which amounts to the  YD compatibility condition for the elements of the form $1 \acl u \in \mathcal{F} \acl \mathcal{U}$ by using YD compatibility over $\mathcal{U}$ and \eqref{auxy11}.

\medskip
 Since YD condition is multiplicative, it is then satisfied for any $f \acl u \in \mathcal{F} \acl \mathcal{U}$, and hence we have proved that $V$ is YD module over $\mathcal{F} \acl \mathcal{U}$.
\end{proof}

\begin{proposition}\label{aux-46}
Let $(\Fg_1,\Fg_2)$ be a matched pair of finite dimensional Lie algebras, $V$ an AYD module over the double crossed sum $\Fg_1 \bowtie \Fg_2$ with locally finite action and locally conilpotent coaction. Then by the right action \eqref{aux-17} and the left coaction \eqref{auxy12}, $V$ becomes a right-left YD module over $R(\Fg_2) \acl U(\Fg_1)$.
\end{proposition}

\begin{proof}
We prove the claim  by verifying the conditions of Proposition \ref{aux-41}.
Since $V$ is an AYD module over $\Fg_1 \bowtie \Fg_2$ with a locally conilpotent coaction, it is an AYD module over $U(\Fg_1) \bowtie U(\Fg_2)$. In particular by Corollary \ref{corollary-58}, it is a left comodule over $U(\Fg_1) \bowtie U(\Fg_2)$ via the  coaction  \eqref{aux-corollary58-coaction}. By the coassociativity of this coaction, we have
\begin{equation}\label{aux-38}
v\sns{0}\snsb{-1} \ot v\sns{-1} \ot v\sns{0}\snsb{0} = v\snsb{-1} \ot v\snsb{0}\sns{-1} \ot v\snsb{0}\sns{0}.
\end{equation}
Thus, the application of $\Id_{U(\Fg_1)} \ot f \ot \Id_M$  on both sides of \eqref{aux-38} produces \eqref{auxy8}.

\medskip

Using \eqref{auxy6} and \eqref{aux-38} we get
\begin{align}
\begin{split}
& u'\ps{2}(v \cdot u'\ps{1})\sns{-1} \ot (v \cdot u'\ps{1})\sns{0} = \\
& (u'\ps{2} \lt (v \cdot u'\ps{1})\sns{0}\snsb{-1})(v \cdot u'\ps{1})\sns{-1} \ot (v \cdot u'\ps{1})\sns{0}\snsb{0} = \\
& (u'\ps{2} \lt (v \cdot u'\ps{1})\snsb{-1})(v \cdot u'\ps{1})\snsb{0}\sns{-1} \ot (v \cdot u'\ps{1})\snsb{0}\sns{0}.
\end{split}
\end{align}
Then applying $f \ot \Id_V$ to both sides and using the non-degenerate pairing between $R(\Fg_2)$ and $U(\Fg_2)$, we conclude \eqref{auxy9}.

\medskip

To verify  \eqref{auxy10}, we use the $U(\Fg_1) \bowtie U(\Fg_2)$-module compatibility on $V$, \ie for any $u \in U(\Fg_1)$, $u' \in U(\Fg_2)$ and $v \in V$,
\begin{equation}
(v \cdot u') \cdot u = (v \cdot (u'\ps{1} \rt u\ps{1})) \cdot (u'\ps{2} \lt u\ps{2}).
\end{equation}
Using the non-degenerate pairing between $R(\Fg_2)$ and $U(\Fg_2)$, we rewrite this equality as
\begin{align}
\begin{split}
& v^{\sns{-1}}(u')v^{\sns{0}} \cdot u = (v \cdot (u'\ps{1} \rt u\ps{1}))^{\sns{-1}}(u'\ps{2} \lt u\ps{2})(v \cdot (u'\ps{1} \rt u\ps{1}))^{\sns{0}} \\
& = u\ps{2} \rt (v \cdot (u'\ps{1} \rt u\ps{1}))^{\sns{-1}}(u'\ps{2})(v \cdot (u'\ps{1} \rt u\ps{1}))^{\sns{-1}} \\
& = (u\ps{1})^{\pr{1}}(u'\ps{1})(u\ps{2} \rt (v \cdot (u\ps{1})^{\pr{0}})^{\sns{-1}})(u'\ps{2})(v \cdot (u\ps{1})^{\pr{0}})^{\sns{0}} \\
& = [(u\ps{1})^{\pr{1}}(u\ps{2} \rt (v \cdot (u\ps{1})^{\pr{0}})^{\sns{-1}})](u')(v \cdot (u\ps{1})^{\pr{0}})^{\sns{0}},
\end{split}
\end{align}
which gives \eqref{auxy10}.

\medskip

Using the $U(\Fg_1) \bowtie U(\Fg_2)$-coaction compatibility \eqref{aux-38} together with \eqref{auxy4},  we have
\begin{align}
\begin{split}
& v\snsb{-1}u^{\pr{0}} \ot v\snsb{0} \cdot u^{\pr{1}} = v\snsb{-1}u^{\pr{0}}u^{\pr{1}}(v\snsb{0}\sns{-1}) \ot v\snsb{0}\sns{0} \\
& = v\snsb{-1}(v\snsb{0}\sns{-1} \rt u) \ot v\snsb{0}\sns{0} = v\sns{0}\snsb{-1}(v\sns{-1} \rt u) \ot v\sns{0}\snsb{0} \\
& = v\snsb{-1}u \ot v\sns{0},
\end{split}
\end{align}
which is \eqref{auxy11}.
\end{proof}

We are now ready to express the main result of this subsection.

\begin{theorem}
Let $(\mathcal{F}, \mathcal{U})$ be a matched pair of Hopf algebras such that $\Fc$ is commutative, $\Uc$ is cocommutative, and $\Big\langle\;, \Big\rangle:\mathcal{F} \times \mathcal{V} \to \mathbb{C}$ is a non-degenerate Hopf pairing. Then $V$ is an AYD module over $\Uc\bi\Vc$ if and only if $V$ is a YD module over $\Fc\acl \Uc$ such that it is a YD module over $\Uc$ by the corresponding module and comodule structures.
\end{theorem}

\begin{proof}
Let $V$ be a right-left YD-module over the bicrossed product Hopf algebra $\mathcal{F} \acl \mathcal{U}$ as well as a right-left YD module over $\Uc$. We first prove that $V$ is a right $\mathcal{U} \bowtie \mathcal{V}$-module.

\medskip

By Proposition \ref{aux-41}, we have \eqref{auxy10}.  Evaluating the both sides of this equality on an arbitrary $u' \in \mathcal{V}$ we get
\begin{equation}
(v \cdot u') \cdot u = (v \cdot (u'\ps{1} \rt u\ps{1})) \cdot (u'\ps{2} \lt u\ps{2}).
\end{equation}
This proves that $V$ is a right module over the double crossed product $\mathcal{U} \bowtie \mathcal{V}$.

\medskip

Next, we show that $V$ is a left $\mathcal{U} \bowtie \mathcal{V}$-comodule. This time using \eqref{auxy8}
 and the duality between right $\mathcal{F}$-action and left $\mathcal{V}$-coaction we get
\begin{equation}
f(v\sns{-1})v\sns{0}\snsb{-1} \ot v\sns{0}\snsb{0} = f(v\snsb{0}\sns{-1})v\snsb{-1} \ot v\snsb{0}\sns{0}.
\end{equation}
Since the pairing is non-degenerate, we conclude that $V$ is a left comodule over $\Uc\bi\Vc$.

\medskip

Finally using Corollary \ref{auxx-24}, we prove that AYD condition over $\mathcal{U} \bowtie \mathcal{V}$ is satisfied, \ie  we show that \eqref{auxy13} to \eqref{auxy16} are satisfied.

\medskip

Considering the Hopf  duality between $\mathcal{F}$ and  $\mathcal{V}$, the right $\mathcal{F} \acl \mathcal{U}$-module compatibility reads
\begin{equation}
f((v \cdot u)\sns{-1})(v \cdot u)\sns{0} = f(v\sns{-1} \lt u\ps{1})v\sns{0} \cdot u\ps{2}.
\end{equation}
Hence \eqref{auxy13} holds.

\medskip

By  \eqref{auxy11} and the Hopf duality between  $\Fc$ and  $\Vc$, we get
\begin{equation}
v\snsb{-1}u^{\pr{0}}u^{\pr{1}}(v\snsb{0}\sns{-1}) \ot v\snsb{0}\sns{0} = v\snsb{-1}(v\snsb{0}\sns{-1} \rt u) \ot v\snsb{0}\sns{0} = v\snsb{-1}u \ot v\snsb{0},
\end{equation}
which immediately imply \eqref{auxy14}.

\medskip

Evaluating the left $\mathcal{F} \acl \mathcal{U}$-coaction compatibility \eqref{aux-27} on an arbitrary $u' \in \mathcal{V}$, we get
\begin{equation}
u'\ps{2} \rt (v \cdot u'\ps{1})\snsb{-1} \ot (v \cdot u'\ps{1})\snsb{0} = v\snsb{-1} \ot v\snsb{0} \cdot u',
\end{equation}
which  immediately  implies \eqref{auxy15}.

\medskip

Finally,  evaluating the left hand side of \eqref{auxy9} on an arbitrary $u' \in \mathcal{V}$, we get
\begin{align}
\begin{split}
& LHS = f\ps{1}(u'\ps{2})(v \cdot u'\ps{1}) \cdot f\ps{2} = f\ps{1}(u'\ps{2})f\ps{2}((v \cdot u'\ps{1})\sns{-1}) (v \cdot u'\ps{1})\sns{0} \\
& = f(u'\ps{2}(v \cdot u'\ps{1})\sns{-1}) (v \cdot u'\ps{1})\sns{0} = f(v\sns{-1} \cdot u'\ps{1}) v\sns{0} \cdot u'\ps{2},
\end{split}
\end{align}
and
\begin{align}
\begin{split}
& RHS = (v \cdot u'\ps{1})\snsb{-1} \rt f\ps{1}(u'\ps{2})f\ps{2}((v \cdot u'\ps{1})\snsb{0}\sns{-1})(v \cdot u'\ps{1})\snsb{0}\sns{0} \\
& = f\ps{1}(u'\ps{2} \lt (v \cdot u'\ps{1})\snsb{-1})f\ps{2}((v \cdot u'\ps{1})\snsb{0}\sns{-1})(v \cdot u'\ps{1})\snsb{0}\sns{0} \\
& = f\ps{1}(u'\ps{4} \lt (v\sns{0} \cdot u'\ps{2})\snsb{-1})f\ps{2}(S(u'\ps{3})v\sns{-1}u'\ps{1})(v\sns{0} \cdot u'\ps{2})\snsb{0} \\
& = f([u'\ps{4} \lt (v\sns{0} \cdot u'\ps{2})\snsb{-1}]S(u'\ps{3})v\sns{-1}u'\ps{1})(v\sns{0} \cdot u'\ps{2})\snsb{0},
\end{split}
\end{align}
where on the third equality we use \eqref{auxy8}. As a result,
\begin{align}
\begin{split}
& v\sns{-1} \cdot u'\ps{1} \ot v\sns{0} \cdot u'\ps{2}  \\
&= [u'\ps{4} \lt (v\sns{0} \cdot u'\ps{2})\snsb{-1}]S(u'\ps{3})v\sns{-1}u'\ps{1} \ot (v\sns{0} \cdot u'\ps{2})\snsb{0}.
\end{split}
\end{align}
Using the cocommutativity of $\Vc$, we conclude \eqref{auxy16}.

\medskip

Conversely, let $V$ be an AYD-module over $\mathcal{U} \bowtie \mathcal{V}$. Then  $V$ is a  left comodule over $\Fc\acl \Uc$ by \eqref{auxy15} and a right module over $\Fc\acl \Uc$ by  \eqref{auxy13}.  So, by  Proposition \ref{aux-41} it suffices to verify \eqref{auxy8} to \eqref{auxy11}.

\medskip

Indeed, \eqref{auxy8} follows from the coaction compatibility over $\Uc \bowtie \Vc$. The condition \eqref{auxy9} is the consequence of  \eqref{auxy16}. The equation \eqref{auxy10} is obtained from   the module compatibility over $\Uc \bowtie \Vc$. Finally \eqref{auxy11} follows from \eqref{auxy14}.
\end{proof}

We conclude this subsection with a brief remark on the stability.

\begin{proposition}\label{aux-59}
Let $(\Fg_1,\Fg_2)$ be a matched pair of Lie algebras and $V$ be an AYD module over the double crossed sum $\Fg_1 \bowtie \Fg_2$ with locally finite action and locally conilpotent coaction. Assume also that $V$ is stable over $R(\Fg_2)$ and $U(\Fg_1)$. Then  $V$ is stable over $R(\Fg_2) \acl U(\Fg_1)$.
\end{proposition}

\begin{proof}
For an $v \in V$, using the $U(\Fg_1 \bowtie \Fg_2)$-comodule compatibility \eqref{aux-38}, we get
\begin{align}
\begin{split}
& (v^{\sns{0}})\snsb{0} \cdot (v^{\sns{-1}} \acl (v^{\sns{0}})\snsb{-1}) = ((v^{\sns{0}})\snsb{0} \cdot v^{\sns{-1}}) \cdot (v^{\sns{0}})\snsb{-1} \\
& = (v^{\sns{0}} \cdot v^{\sns{-1}})\snsb{0} \cdot (v^{\sns{0}} \cdot v^{\sns{-1}})\snsb{-1} = v\snsb{0} \cdot v\snsb{-1} = v.
\end{split}
\end{align}
\end{proof}

\subsection{A 4-dimensional SAYD module over the Schwarzian Hopf algebra}

In this subsection, we illustrate the theory of coefficients we have studied in this chapter in a nontrivial 4-dimensional  SAYD module over $\Hc_{1\rm S}\cop\cong R(\Cb)\acl U({g\ell(1)}^{\rm aff})$, the Schwarzian Hopf algebra introduced in \cite{ConnMosc98}.

\medskip

The non-triviality of this example is the non-trivially  of the $R(\Cb)$-action and the $U({g\ell(1)}^{\rm aff})$-coaction.

\medskip

We denote $R(\Cb)$, $U({g\ell(1)}^{\rm aff})$ and $\Hc_{1\rm S}\cop$ by $\Fc$, $\Uc$ and $\Hc$ respectively.

\medskip

Let us first recall the Lie algebra $s\ell(2)$ as a double crossed sum Lie algebra. We have $s\ell(2) = \Fg_1 \bowtie \Fg_2$, $\Fg_1 = \Cb\langle X,Y \rangle$, $\Fg_2 = \Cb\langle Z \rangle$, with Lie bracket
\begin{equation}
[Y,X] = X, \quad [Z,X] = Y, \quad [Z,Y] = Z.
\end{equation}

Let us take $V = S(s\ell(2)^\ast)\nsb{1}$. By Example \ref{aux-62}, $V$ is an SAYD module over $s\ell(2)$ via the coadjoint action and the Koszul coaction.

\medskip

Writing ${\Fg_2}^\ast = \Cb\langle \d_1 \rangle$, we have $\Fc = R(\Fg_2) = \mathbb{C}[\delta_1]$, $\Uc = U(\Fg_1)$ and finally $\Fc \acl \Uc \cong \Hc_{1 \rm S}\cop$. See \cite{ConnMosc98,MoscRang07} for more details.

\medskip

Next, we construct the $\Fc \acl \Uc$-(co)action explicitly and we verify that $\, ^{\sigma}V_{\delta}$ is an SAYD module over $\Fc \acl \Uc$. Here, $(\sigma,\delta)$ is the canonical modular pair in involution associated to the bicrossed product $\Fc \acl \Uc$. Let us compute $\sigma \in \Fc$ from the right $\Fc$-coaction on $\Uc$.

\medskip

In view of \eqref{aux-actions-matched-pair-Lie-algebras}, the action $\Fg_2 \lt \Fg_1$ is given by
\begin{equation}
Z \lt X = 0, \quad Z \lt Y = Z.
\end{equation}
Similarly, the action $\Fg_2 \rt \Fg_1$ is
\begin{equation}
Z \rt X = Y, \quad Z \rt Y = 0.
\end{equation}

Dualizing the left action $\Fg_2 \rt \Fg_1$, we have the $\Fc$-coaction on $\Uc$ as
\begin{align}
\begin{split}
& \Uc \to \Uc \ot \Fc, \quad u \mapsto u^{\pr{0}} \ot u^{\pr{1}} \\
& X \mapsto X \ot 1 + Y \ot \delta_1 \\
& Y \mapsto Y \ot 1.
\end{split}
\end{align}
Therefore,
\begin{equation}
\sigma = \det \left(
               \begin{array}{cc}
                 1 & \delta_1 \\
                 0 & 1 \\
               \end{array}
             \right) = 1.
\end{equation}

 On the other hand, by the Lie algebra structure of $\Fg_1 \cong {g\ell(1)}^{\rm aff}$, we have
\begin{equation}
\delta(X) = 0, \qquad \delta(Y) = 1.
\end{equation}

Next,  we express the $\Fc \acl \Uc$-coaction on $V = S({s\ell(2)}^\ast)\nsb{1}$ explicitly. A vector space basis of $V$ is given by $\{1_V, R^X, R^Y, R^Z\}$ and the $\Fg_1$-coaction (which is the Koszul coaction) is
\begin{equation}
V \to \Fg_1 \ot V, \quad 1_V \mapsto X \ot R^X + Y \ot R^Y, \quad R^i \mapsto 0, \quad i \in \{X,Y,Z\}.
\end{equation}
Note that the application of this coaction twice is zero, thus it is locally conilpotent. Then the corresponding $\Uc$ coaction is
\begin{align}
\begin{split}
& V \to \Uc \ot V, \quad v \mapsto v\snsb{-1} \ot v\snsb{0} \\
& 1_V \mapsto 1 \ot 1_V + X \ot R^X + Y \ot R^Y \\
& R^i \mapsto 1 \ot R^i, \quad i \in \{X,Y,Z\}.
\end{split}
\end{align}

To determine the left $\Fc$-coaction, we need to dualize the right $\Fg_2$-action. We have
\begin{equation}
1_V \lt Z = 0, \quad R^X \lt Z = 0, \quad R^Y \lt Z = R^X, \quad R^Z \lt Z = R^Y,
\end{equation}
implying
\begin{align}
\begin{split}
& V \to \Fc \ot V, \quad v \mapsto v^{\sns{-1}} \ot v^{\sns{0}} \\
& 1_V \mapsto 1 \ot 1_V \\
& R^X \mapsto 1 \ot R^X \\
& R^Y \mapsto 1 \ot R^Y + \delta_1 \ot R^X \\
& R^Z \mapsto 1 \ot R^Z + \delta_1 \ot R^Y + \frac{1}{2}\delta_1^2 \ot R^X.
\end{split}
\end{align}
As a result, $\Fc \acl \Uc$-coaction becomes
\begin{align}
\begin{split}
& V \to \Fc \acl \Uc \ot V, \quad v \mapsto v^{\sns{-1}} \acl v^{\sns{0}}\snsb{-1} \ot v^{\sns{0}}\snsb{0} \\
& 1_V \mapsto 1 \ot 1_V + X \ot R^X + Y \ot R^Y \\
& R^X \mapsto 1 \ot R^X \\
& R^Y \mapsto 1 \ot R^Y + \delta_1 \ot R^X \\
& R^Z \mapsto 1 \ot R^Z + \delta_1 \ot R^Y + \frac{1}{2}\delta_1^2 \ot R^X.
\end{split}
\end{align}

 Let us next determine the right $\Fc \acl \Uc$-action. It is enough to determine the $\Uc$-action and $\Fc$-action separately. The action of $\Uc$ is directly given by
\begin{align}
\begin{split}
& 1_V \lt X = 0, \quad 1_V \lt Y = 0 \\
& R^X \lt X = -R^Y, \quad R^X \lt Y = R^X \\
& R^Y \lt X = -R^Z, \quad R^Y \lt Y = 0 \\
& R^Z \lt X = 0, \quad R^Z \lt Y = -R^Z.
\end{split}
\end{align}

To be able to see the $\Fc$-action, we determine the $\Fg_2$-coaction. This follows from the Kozsul coaction on $V$, \ie
\begin{align}
\begin{split}
& V \to U(\Fg_2) \ot V, \quad v \mapsto v\sns{-1} \ot v\sns{0} \\
& 1_V \mapsto 1 \ot 1_V + Z \ot R^Z \\
& R^i \mapsto 1 \ot R^i, \quad i \in \{X,Y,Z\}.
\end{split}
\end{align}
Hence, the $\Fc$-action is given by
\begin{equation}
1_V \lt \delta_1 = R^Z, \quad R^i \lt \delta_1 = 0, \quad i \in \{X,Y,Z\}.
\end{equation}

We will now check  that $V$ is a YD module over the bicrossed product Hopf algebra $\Fc \acl \Uc$. It is straightforward to check that $V$ satisfies the conditions introduced in Lemma \ref{module on bicrossed product} and Lemma \ref{comodule on bicrossed product}; that is,  $V$ is a module and comodule on $\Fc \acl \Uc$, respectively. We proceed to the verification of the YD condition on the bicrossed product Hopf algebra $\Fc \acl \Uc$.

\medskip

By the multiplicative property of the YD condition, it is enough to check that the condition holds  for the elements $X, Y, \delta_1 \in \Fc \acl \Uc$.

\medskip

For simplicity of notation, we write the $\Fc \acl \Uc$-coaction as $v \mapsto v\pr{-1} \ot v\pr{0}$.

\medskip

We begin with $1_V \in V$ and $X \in \Fc \acl \Uc$. We have
\begin{align}
\begin{split}
& X\ps{2} \cdot (1_V \lt X\ps{1})\pr{-1} \ot (1_V \lt X\ps{1})\pr{0}  \\
&= (1_V \lt X)\pr{-1} \ot (1_V \lt X)\pr{0} + X 1_V\pr{-1} \ot 1_V\pr{0}  \\
&+ \delta_1(1_V \lt Y)\pr{-1} \ot (1_V \lt Y)\pr{0}  \\
&= X \ot 1_V + X^2 \ot R^X + XY \ot R^Y\\
& =X \ot 1_V + X^2 \ot R^X + YX \ot R^Y + X \ot R^X \lt X  \\
& +Y \ot R^Y \lt X + Y \ot 1_V \lt \delta_1\\
&= 1_V\pr{-1} X \ot 1_V\pr{0} + 1_V\pr{-1} \ot 1_V\pr{0} \lt X + 1_V\pr{-1} Y \ot 1_V\pr{0} \lt \delta_1 \\
& =1_V\pr{-1} X\ps{1} \ot 1_V\pr{0} \lt X\ps{2}.
\end{split}
\end{align}

We proceed to check the condition for  $1_V \in V$ and $Y \in \Fc \acl \Uc$. We have
\begin{align}
\begin{split}
& Y\ps{2} \cdot (1_V \lt Y\ps{1})\pr{-1} \ot (1_V \lt Y\ps{1})\pr{0}  \\
&= (1_V \lt Y)\pr{-1} \ot (1_V \lt Y)\pr{0} + Y 1_V\pr{-1} \ot 1_V\pr{0} \\
& =Y \ot 1_V + YX \ot R^X + Y^2 \ot R^Y\\
& =Y \ot 1_V + XY \ot R^X + Y^2 \ot R^Y + X \ot R^X \lt Y\\
& =1_V\pr{-1} Y \ot 1_V\pr{0} + 1_V\pr{-1} \ot 1_V\pr{0} \lt Y  \\
&= 1_V\pr{-1} Y\ps{1} \ot 1_V\pr{0} \lt Y\ps{2}.
\end{split}
\end{align}

For $1_V \in V$ and $\delta_1 \in \Fc \acl \Uc$, we get
\begin{align}
\begin{split}
& \delta_1\ps{2} (1_V \lt \delta_1\ps{1})\pr{-1} \ot (1_V \lt \delta_1\ps{1})\pr{0} \\
& =(1_V \lt \delta_1)\pr{-1} \ot (1_V \lt \delta_1)\pr{0} + \delta_1 1_V\pr{-1} \ot 1_V\pr{0}  \\
& =1 \ot R^Z + \delta_1 \ot R^Y + \frac{1}{2}\delta_1^2  \ot R^X + \delta_1 \ot 1_V + \delta_1X \ot R^X + \delta_1Y \ot R^Y\\
& =\delta_1 \ot 1_V + X\delta_1 \ot R^X + Y\delta_1 \ot R^Y + 1 \ot 1_V \lt \delta_1\\
&= 1_V\pr{-1} \delta_1 \ot 1_V\pr{0} + 1_V\pr{-1} \ot 1_V\pr{0} \lt \delta_1  \\
& =1_V\pr{-1} \delta_1\ps{1} \ot 1_V\pr{0} \lt \delta_1\ps{2}.
\end{split}
\end{align}

Next, we consider $R^X \in V$ and $X \in \Fc \acl \Uc$. In this case we have
\begin{align}
\begin{split}
& X\ps{2} (R^X \lt X\ps{1})\pr{-1} \ot (R^X \lt X\ps{1})\pr{0} = \\
& =(R^X \lt X)\pr{-1} \ot (R^X \lt X)\pr{0} + X R^X\pr{-1} \ot R^X\pr{0} \\
& +\delta_1 (R^X \lt Y)\pr{-1} \ot (R^X \lt Y)\pr{0}  \\
& =-1 \ot R^Y - \delta_1 \ot R^X + X \ot R^X + \delta_1 \ot R^X  \\
& =-1 \ot R^Y + X \ot R^X= X \ot R^X + 1 \ot R^X \lt X\\
& =R^X\pr{-1} X \ot R^X\pr{0} + R^X\pr{-1} \ot R^X\pr{0} \lt X + R^X\pr{-1} Y \ot R^X\pr{0} \lt \delta_1  \\
& =R^X\pr{-1} X\ps{1} \ot R^X\pr{0} \lt X\ps{2}.
\end{split}
\end{align}

For $R^X \in V$ and $Y \in \Fc \acl \Uc$, we obtain
\begin{align}
\begin{split}
& Y\ps{2} (R^X \lt Y\ps{1})\pr{-1} \ot (R^X \lt Y\ps{1})\pr{0} \\
& =(R^X \lt Y)\pr{-1} \ot (R^X \lt Y)\pr{0} + Y R^X\pr{-1} \ot R^X\pr{0}  \\
&= 1 \ot R^X + Y \ot R^X = Y \ot R^X + 1 \ot R^X \lt Y \\
& =R^X\pr{-1}  Y \ot R^X\pr{0} + R^X \ot R^X\pr{0} \lt Y  \\
& =R^X\pr{-1}  Y\ps{1} \ot R^X\pr{0} \lt Y\ps{2}.
\end{split}
\end{align}

Next, for $R^X \in V$ and $\delta_1 \in \Fc \acl \Uc$ we have
\begin{align}
\begin{split}
& \delta_1\ps{2} (R^X \lt \delta_1\ps{1})\pr{-1} \ot (R^X \lt \delta_1\ps{1})\pr{0}  \\
& =(R^X \lt \delta_1)\pr{-1} \ot (R^X \lt \delta_1)\pr{0} + \delta_1 R^X\pr{-1} \ot R^X\pr{0}  \\
& =\delta_1 \ot R^X=\delta_1 \ot R^X + 1 \ot R^X \lt \delta_1\\
& =R^X\pr{-1} \delta_1 \ot R^X\pr{0} + R^X\pr{-1} \ot R^X\pr{0} \lt \delta_1  \\
& =R^X\pr{-1} \delta_1\ps{1} \ot R^X\pr{0} \lt \delta_1\ps{2}.
\end{split}
\end{align}

We proceed to verify the YD compatibility condition for $R^Y \in V$. For $R^Y \in V$ and $X \in \Fc \acl \Uc$, we have
\begin{align}
\begin{split}
& X\ps{2} (R^Y \lt X\ps{1})\pr{-1} \ot (R^Y \lt X\ps{1})\pr{0}  \\
& =(R^Y \lt X)\pr{-1} \ot (R^Y \lt X)\pr{0} + X R^Y\pr{-1} \ot R^Y\pr{0}  \\
& +\delta_1 (R^Y \lt Y)\pr{-1} \ot (R^Y \lt Y)\pr{0}  \\
&= -1 \ot R^Z - \delta_1 \ot R^Y - \frac{1}{2}\delta_1^2 \ot R^X + X \ot R^Y + X\delta_1 \ot R^X\\
&= X \ot R^Y + \delta_1X \ot R^X + 1 \ot R^Y \lt X + \delta_1 \ot R^X \lt X\\
&= R^Y\pr{-1} X \ot R^Y\pr{0} + R^Y\pr{-1} \ot R^Y\pr{0} \lt X + R^Y\pr{-1} Y \ot R^Y\pr{0} \lt \delta_1  \\
& =R^Y\pr{-1} X\ps{1} \ot R^Y\pr{0} \lt X\ps{2}.
\end{split}
\end{align}

Similarly for $R^Y \in V$ and $Y \in \Fc \acl \Uc$, we have
\begin{align}
\begin{split}
& Y\ps{2} (R^Y \lt Y\ps{1})\pr{-1} \ot (R^Y \lt Y\ps{1})\pr{0}  \\
& =(R^Y \lt Y)\pr{-1} \ot (R^Y \lt Y)\pr{0} + Y R^Y\pr{-1} \ot R^Y\pr{0}  \\
& =Y \ot R^Y + Y\delta_1 \ot R^X\\
& =Y \ot R^Y + \delta_1Y \ot R^X + \delta_1 \ot R^X \lt Y\\
& =R^Y\pr{-1} Y \ot R^Y\pr{0} + R^Y\pr{-1} \ot R^Y\pr{0} \lt Y  \\
& =R^Y\pr{-1} Y\ps{1} \ot R^Y\pr{0} \lt Y\ps{2}.
\end{split}
\end{align}

Finally, for $R^Y \in V$ and $\delta_1 \in \Fc \acl \Uc$,
\begin{align}
\begin{split}
& \delta_1\ps{2} (R^Y \lt \delta_1\ps{1})\pr{-1} \ot (R^Y \lt \delta_1\ps{1})\pr{0}  \\
& =(R^Y \lt \delta_1)\pr{-1} \ot (R^Y \lt \delta_1)\pr{0} + \delta_1 R^Y\pr{-1} \ot R^Y\pr{0}  \\
& =\delta_1 \ot R^Y + \delta_1^2 \ot R^X\\
& =R^Y\pr{-1} \delta_1 \ot R^Y\pr{0} + R^Y\pr{-1} \ot R^Y\pr{0} \lt \delta_1  \\
& =R^Y\pr{-1} \delta_1\ps{1} \ot R^Y\pr{0} \lt \delta_1\ps{2}.
\end{split}
\end{align}

Now we check the YD condition for $R^Z \in V$. For $R^Z \in V$ and $X \in \Fc \acl \Uc$,
\begin{align}
\begin{split}
& X\ps{2} (R^Z \lt X\ps{1})\pr{-1} \ot (R^Z \lt X\ps{1})\pr{0}  \\
& =(R^Z \lt X)\pr{-1} \ot (R^Z \lt X)\pr{0} + X R^Z\pr{-1} \ot R^Z\pr{0} \\
& +\delta_1 (R^Z \lt Y)\pr{-1} \ot (R^Z \lt Y)\pr{0}  \\
& =X \ot R^Z + X\delta_1 \ot R^Y + \frac{1}{2}X\delta_1^2 \ot R^X - \delta_1 \ot R^Z - \delta_1^2 \ot R^Y - \frac{1}{2}\delta_1^3 \ot R^X\\
&= X \ot R^Z + \delta_1X \ot R^Y + \frac{1}{2}\delta_1^2X \ot R^X + \delta_1 \ot R^Y \lt X + \frac{1}{2}\delta_1^2 \ot R^X \lt X\\
&= R^Z\pr{-1} X \ot R^Z\pr{0} + R^Z\pr{-1} \ot R^Z\pr{0} \lt X + R^Z\pr{-1} Y \ot R^Z\pr{0} \lt \delta_1  \\
&= R^Z\pr{-1} X\ps{1} \ot R^Z\pr{0} \lt X\ps{2}.
\end{split}
\end{align}

Next, we consider $R^Z \in V$ and $Y \in \Fc \acl \Uc$. In this case,
\begin{align}
\begin{split}
& Y\ps{2} (R^Z \lt Y\ps{1})\pr{-1} \ot (R^Z \lt Y\ps{1})\pr{0}  \\
& =(R^Z \lt Y)\pr{-1} \ot (R^Z \lt Y)\pr{0} + Y R^Z\pr{-1} \ot R^Z\pr{0}  \\
& =-1 \ot R^Z - \delta_1 \ot R^Y - \frac{1}{2}\delta_1^2 \ot R^X + Y \ot R^Z + Y\delta_1 \ot R^Y + \frac{1}{2}Y\delta_1^2 \ot R^X\\
& =Y \ot R^Z + \delta_1Y \ot R^Y + \frac{1}{2}\delta_1^2Y \ot R^X + 1 \ot R^Z \lt Y + \frac{1}{2}\delta_1^2 \ot R^X \lt Y \\
& =R^Z\pr{-1} Y \ot R^Z\pr{0} + R^Z\pr{-1} \ot R^Z\pr{0} \lt Y  \\
&= R^Z\pr{-1} Y\ps{1} \ot R^Z\pr{0} \lt Y\ps{2}.
\end{split}
\end{align}

Finally we check the YD compatibility for $R^Z \in V$ and $\delta_1 \in \Fc \acl \Uc$. We have
\begin{align}
\begin{split}
& \delta_1\ps{2} (R^Z \lt \delta_1\ps{1})\pr{-1} \ot (R^Z \lt \delta_1\ps{1})\pr{0} \\
& =(R^Z \lt \delta_1)\pr{-1} \ot (R^Z \lt \delta_1)\pr{0} + \delta_1 R^Z\pr{-1} \ot R^Z\pr{0}  \\
& =\delta_1 \ot R^Z + \delta_1^2 \ot R^Y + \frac{1}{2}\delta_1^3 \ot R^X\\
&=R^Z\pr{-1} \delta_1 \ot R^Z\pr{0} + R^Z\pr{-1} \ot R^Z\pr{0} \lt \delta_1  \\
&= R^Z\pr{-1} \delta_1\ps{1} \ot R^Z\pr{0} \lt \delta_1\ps{2}.
\end{split}
\end{align}

We have proved  that $V$ is a YD module over the bicrossed product Hopf algebra $\Fc \acl \Uc = {\Hc_{1 \rm S}}^{\rm cop}$.

\medskip

Let us now check the stability condition. Since in this case $\sigma = 1$, $\,^{\sigma}V_{\delta}$ has the same coaction as $V$. For an element $v \ot 1_\Cb \in \,^\s V_\d$, let $v \ot 1_\Cb \mapsto (v \ot 1_{\Cb})\pr{-1} \ot (v \ot 1_{\Cb})\pr{0} \in \Fc \acl \Uc \ot \,^{\sigma}V_{\delta}$ denote the coaction. We have
\begin{align*}\notag
& (1_V \ot 1_{\mathbb{C}})\pr{0} \cdot (1_V \ot 1_{\mathbb{C}})\pr{-1} = (1_V \ot 1_{\mathbb{C}}) \cdot 1 + (R^X \ot 1_{\mathbb{C}}) \cdot X + (R^Y \ot 1_{\mathbb{C}}) \cdot Y \\\notag
& = 1_V \ot 1_{\mathbb{C}} + R^X \cdot X\ps{2}\delta(X\ps{1}) \ot 1_{\mathbb{C}} + R^Y \cdot Y\ps{2}\delta(Y\ps{1}) \ot 1_{\mathbb{C}} \\\notag
& = 1_V \ot 1_{\mathbb{C}} + R^X \cdot X \ot 1_{\mathbb{C}} + R^X \delta(X) \ot 1_{\mathbb{C}} + R^Y \cdot Y \ot 1_{\mathbb{C}} + R^Y \delta(Y) \ot 1_{\mathbb{C}} \\
& = 1_V \ot 1_{\mathbb{C}}.\\[.5cm]
& (R^X \ot 1_{\mathbb{C}})\pr{0} \cdot (R^X \ot 1_{\mathbb{C}})\pr{-1} = (R^X \ot 1_{\mathbb{C}}) \cdot 1
= R^X \ot 1_{\mathbb{C}}.\\[.5cm]
& (R^Y \ot 1_{\mathbb{C}})\pr{0} \cdot (R^Y \ot 1_{\mathbb{C}})\pr{-1} = (R^Y \ot 1_{\mathbb{C}}) \cdot 1 + (R^X \ot 1_{\mathbb{C}}) \cdot \delta_1
 = R^Y \ot 1_{\mathbb{C}}.\\[.5cm]
& (R^Z \ot 1_{\mathbb{C}})\pr{0} \cdot (R^Z \ot 1_{\mathbb{C}})\pr{-1} =\\
&\hspace{3cm} (R^Z \ot 1_{\mathbb{C}}) \cdot 1 + (R^Y \ot 1_{\mathbb{C}}) \cdot \delta_1 + (R^X \ot 1_{\mathbb{C}}) \cdot \frac{1}{2}\delta_1^2  = R^Z \ot 1_{\mathbb{C}}.
\end{align*}
Hence the stability is satisfied.

\medskip

We record this subsection in the following proposition.

\begin{proposition}
The four dimensional module-comodule  $$V_\d:=S(s\ell(2)^\ast)\nsb{1} \ot \Cb_{\delta}= \Cb\Big\langle 1_V, R^X,R^Y,R^Z\Big\rangle \ot \Cb_\d$$ is an SAYD module over the the Schwarzian Hopf algebra $\Hc_{\rm 1S}\cop$, via the action and coaction
$$
\left.
  \begin{array}{c|ccc}
    \lt & X & Y & \d_1 \\[.2cm]
    \hline
    &&&\\[-.2cm]
    \one & 0 & 0 & \bfR^Z \\[.1cm]
    \bfR^X & -\bfR^Y & 2\bfR^X & 0 \\[.1cm]
    \bfR^Y & -\bfR^Z & \bfR^Y & 0 \\[.1cm]
    \bfR^Z & 0 & 0 & 0 \\
  \end{array}
\right.\qquad
  \begin{array}{rl}
  &\Db: V_\d \longrightarrow \Hc_{\rm 1S}\cop \ot V_\d \\[.2cm]
\hline
  &\\[-.2cm]
& \one \mapsto 1 \ot \one + X \ot \bfR^X + Y \ot \bfR^Y \\[.1cm]
& \bfR^X \mapsto 1 \ot \bfR^X \\[.1cm]
& \bfR^Y \mapsto 1 \ot \bfR^Y + \delta_1 \ot \bfR^X \\[.1cm]
& \bfR^Z \mapsto 1 \ot \bfR^Z + \delta_1 \ot \bfR^Y + \frac{1}{2}\delta_1^2 \ot \bfR^X.
  \end{array}
$$

 Here, $\one := 1_V \ot \Cb_\d, \; \bfR^X := R^X \ot \Cb_\d, \; \bfR^Y := R^Y \ot \Cb_\d, \; \bfR^Z := R^Z \ot \Cb_\d$.
\end{proposition}

\subsection{AYD modules over the Connes-Moscovici Hopf algebras}

In this subsection we investigate  the finite dimensional SAYD modules over the Connes-Moscovici Hopf algebras $\Hc_n$. We first recall from \cite{MoscRang09} the bicrossed product decomposition of the Connes-Moscovici Hopf algebras.

\medskip

Let $\Diff(\Rb^n)$ denote the group of diffeomorphisms on $\Rb^n$. Via the splitting $\Diff(\Rb^n) = G \cdot N$, where $G$ is the group of affine transformation on $\Rb^n$ and
\begin{equation}
N = \Big\{\psi \in \Diff(\Rb^n) \; \Big| \; \psi(0) = 0, \; \psi'(0) = \Id\Big\},
\end{equation}
we have $\Hc_n = \Fc(N) \acl U(\Fg)$. Elements of the Hopf algebra $\Fc:=\Fc(N)$ are called the regular functions. They are the coefficients of the Taylor expansions at $0 \in \Rb^n$ of the elements of the group $N$. Here, $\Fg$ is the Lie algebra of the group $G$ and $\Uc:=U(\Fg)$ is the universal enveloping algebra of $\Fg$.

\medskip

On the other hand for the Lie algebra $\Fa_n$ of formal vector fields on $\Rb^n$,  following \cite{Fuks}, let us denote the subalgebra of the vector fields
\begin{equation}
\sum_{i=1}^n f_i\frac{\p}{\p x^i},
\end{equation}
such that $f_1, \ldots, f_n$ belongs to the $(k+1)$st power of the maximal ideal of the ring of formal power series by $\Fl_k$, $k \geq -1$. Then
\begin{equation}
\Fa_n = \Fl_{-1} \supseteq \Fl_{0} \supseteq \Fl_{1} \supseteq \ldots
\end{equation}
such that
\begin{equation}
[\Fl_p,\Fl_q] \subseteq \Fl_{p+q}.
\end{equation}
Furthermore,
\begin{equation}
g\ell(n) = \Fl_0/\Fl_1, \quad \Fl_{-1}/\Fl_0 \cong \mathbb{R}^n, \quad\text{and} \quad \Fl_{-1}/\Fl_0 \oplus \Fl_0/\Fl_1 \cong g\ell(n)^{\rm aff}.
\end{equation}

As a result, setting $\Fg := g\ell(n)^{\rm aff}$ and $\Fn := \Fl_1$, the Lie algebra  $\Fa$ admits the decomposition $\Fa = \Fg \oplus \Fn$, and hence we have a matched pair of Lie algebras $(\Fg,\Fn)$. The Hopf algebra  $\Fc(N)$ is isomorphic with $R(\Fn)$ via the non-degenerate pairing found in \cite{ConnMosc98}:
\begin{equation}
\Big\langle  \a^i_{ j_1,\ldots,j_p}, Z^{k_1, \ldots, k_q}_l\Big\rangle= \d^p_q\d^i_l\d_{j_1,\ldots,j_p}^{k_1, \ldots, k_q}.
\end{equation}
Here
\begin{equation}
\a^i_{ j_1,\ldots,j_p}(\psi)= {\left.\frac{\p^p}{\p x^{j_p}\ldots \p x^{j_1}}\right|}_{x=0}\psi^i(x),
\end{equation}
and
\begin{equation}
Z^{k_1, \ldots, k_q}_l= x^{k_1}\ldots x^{k_q}\frac{\p}{\p x^l}.
\end{equation}

\medskip

Let $\d$ be the trace of the adjoint representation of $\Fg$ on itself. Then by \cite{ConnMosc98}, $\Cb_\d$ is a SAYD module over the Hopf algebra $\Hc_n$.

\begin{lemma}\label{lemma-(co)action-trivial}
For any YD module over $\Hc_n$, the action of $\Uc$ and the coaction of $\Fc$ are trivial.
\end{lemma}

\begin{proof}
Let $V$ be a finite dimensional YD module over $\mathcal{H}_n = \Fc \acl \Uc$. By the same argument in Proposition \ref{aux-41}, $V$ is a module over $\Fa$. However, we know that $\Fa$ has no nontrivial finite dimensional representation  by  \cite{Fuks}. As a result, the $\Uc$ action and the $\Fc$-coaction on $V$ are trivial.
\end{proof}

Let us introduce the isotropy subalgebra $\Fg_0\subset \Fg$ by
\begin{equation}
\Fg_0 = \Big\{X \in \Fg\,\Big|\, Y \rt X = 0, \forall Y \in \Fn\Big\} \subseteq \Fg.
\end{equation}
By the construction of $\Fa$ it is obvious that $\Fg_0$ is generated by $Z^i_j$, and hence $\Fg_0\cong g\ell(n)$. By the definition of the  coaction $\Db_\Uc:\Uc\ra \Uc\ot \Fc$, we see that $U(\Fg_0)=\Uc^{co\Fc}$.

\begin{lemma}\label{lemma-coaction-lands}
For any finite dimensional YD module $V$ over  $\Hc_n$, the coaction $$\Db:V\ra \Hc_n\ot V$$ lands in $U(\Fg_0)\ot V.$
\end{lemma}

\begin{proof}
By Lemma \ref{lemma-(co)action-trivial} we know that $\Uc$-action and $\Fc$-coaction on $V$ are trivial. Hence, the left coaction $V \to \Fc\acl \Uc \ot V$ becomes $v \mapsto 1 \acl v\snsb{-1} \ot v\snsb{0}$. The coassociativity of the coaction
\begin{equation}
1 \acl v\snsb{-2} \ot 1 \acl v\snsb{-1} \ot v\snsb{0} = 1 \acl (v\snsb{-2})^{\pr{0}} \ot (v\snsb{-2})^{\pr{1}} \acl v\snsb{-1} \ot v\snsb{0}
\end{equation}
implies that
\begin{equation}
v \mapsto v\snsb{-1} \ot v\snsb{0} \in \Uc^{co\Fc} \ot V =U(\Fg_0) \ot V.
\end{equation}
\end{proof}

\begin{lemma}\label{lemma-coaction-trivial}
Let $V$ be a finite dimensional YD module over the Hopf algebra $\Hc_n$. Then the coaction of $\Hc_n$ on $V$ is trivial.
\end{lemma}

\begin{proof}
By  Lemma \ref{lemma-coaction-lands} we know that the coaction of $\Hc_n$ on $V$ lands in $U(\Fg_0)\ot V$. Since $U(\Fg_0)$ is a Hopf subalgebra of $\Hc_n$, it is obvious that $V$ is an AYD module over $U(\Fg_0)$. Since $\Fg_0$ is finite dimensional,  $V$ becomes an AYD module over $\Fg_0$.

\medskip

Let us express  the $\Fg_0$-coaction for an arbitrary basis element $v^i \in V$ as
\begin{equation}
v^i \mapsto v^i\nsb{-1} \ot v^i\nsb{0} = \alpha^{ip}_{kq}Z^q_p \ot v^k \in \Fg_0 \ot V.
\end{equation}
Then AYD condition over $\Fg_0$ becomes
\begin{equation}
\alpha^{ip}_{kq}[Z^q_p,Z] \ot v^k = 0.
\end{equation}
Choosing an arbitrary $Z = Z^{p_0}_{q_0} \in g\ell(n) = \Fg_0$, we get
\begin{equation}
\alpha^{ip_0}_{kq}Z^q_{q_0} \ot v^k - \alpha^{ip}_{kq_0}Z^{p_0}_p \ot v^k = 0,
\end{equation}
or equivalently
\begin{equation}
\alpha^{ip_0}_{kq_0}(Z^{q_0}_{q_0} - Z^{p_0}_{p_0}) + \sum_{q \neq q_0}\alpha^{ip_0}_{kq}Z^q_{q_0} - \sum_{p \neq p_0}\alpha^{ip}_{kq_0}Z^{p_0}_p = 0.
\end{equation}
Thus for $n \geq 2$ we have proved that the $\Fg_0$-coaction is trivial. Hence its lift  to a  $U(\Fg_0)$-coaction is trivial. On the other hand, it is guaranteed by the Proposition \ref{proposition-categories-equivalent} that this is precisely the $U(\Fg_0)$-coaction we have started with. This proves that the $\Uc$ coaction, and hence the $\Hc_n$ coaction on $V$ is trivial.

\medskip

On the other hand,  for $n = 1$, the YD condition for $X \in \Hc_1$ reads, in view of the triviality of the action of $g\ell(1)^{\rm aff}$,
\begin{equation}\label{aux-coaction-n=1}
Xv\pr{-1} \ot v\pr{0} + Z(v \cdot \d_1)\pr{-1} \ot (v \cdot \d_1)\pr{0} = v\pr{-1}X \ot v\pr{0}.
\end{equation}
By Lemma \ref{lemma-coaction-lands} we know that the coaction lands in $U(g\ell(1))$. Together with the relation $[Z,X] = X$, \eqref{aux-coaction-n=1} forces the $\Hc_1$-coaction (and also the action) to be trivial.
\end{proof}

\begin{lemma}\label{lemma-action-trivial}
Let $V$ be a finite dimensional YD module over the Hopf algebra $\Hc_n$. Then the action of $\Hc_n$ on $V$ is trivial.
\end{lemma}

\begin{proof}
By Lemma \ref{lemma-(co)action-trivial} we know that the action of $\Hc_n$ on $V$ is concentrated on the action of $\Fc$ on $V$. So it suffices to prove that this action is trivial.

\medskip

For an arbitrary $v \in V$ and $1 \acl X_k \in \mathcal{H}_n$, we write the YD compatibility. First we calculate
\begin{align}
\begin{split}
& \Delta^2(1 \acl X_k) = \\
& (1 \acl 1) \ot (1 \acl 1) \ot (1 \acl X_k) + (1 \acl 1) \ot (1 \acl X_k) \ot (1 \acl 1) \\
& + (1 \acl X_k) \ot (1 \acl 1) \ot (1 \acl 1) + (\delta^p_{qk} \acl 1) \ot (1 \acl Y_p^q) \ot (1 \acl 1) \\
& + (1 \acl 1) \ot (\delta^p_{qk} \acl 1) \ot (1 \acl Y_p^q) + (\delta^p_{qk} \acl 1) \ot (1 \acl 1) \ot (1 \acl Y_p^q).
\end{split}
\end{align}
Since, by  Lemma \ref{lemma-coaction-trivial}, the coaction of $\Hc_n$ on $V$ is trivial, the AYD condition  can be written as
\begin{align}
\begin{split}
& (1 \acl 1) \ot v \cdot X_k = S(1 \acl X_k) \ot v + 1 \acl 1 \ot v \cdot X_k + 1 \acl X_k \ot v  \\
&+ \delta^p_{qk} \acl 1 \ot v \cdot Y_p^q - 1 \acl Y_p^q \ot v \cdot \delta^p_{qk} - \delta^p_{qk} \acl Y_p^q \ot v  \\
& =\delta^p_{qk} \acl 1 \ot v \cdot Y_p^q + Y_i^j \rt \delta^i_{jk} \acl 1 \ot v - 1 \acl Y_p^q \ot v \cdot \delta^p_{qk}.
\end{split}
\end{align}
Therefore,
\begin{equation}
v \cdot \delta^p_{qk} = 0.
\end{equation}
Finally, by the module compatibility on a bicrossed product $\Fc \acl \Uc$, we get
\begin{equation}
(v \cdot X_l) \cdot \delta^p_{qk} = v \cdot (X_l \rt \delta^p_{qk}) + (v \cdot \delta^p_{qk}) \cdot X_l,
\end{equation}
which in turn, by using one more time the triviality of the $U(\Fg_1)$-action on $V$, implies
\begin{equation}
v \cdot \delta^p_{qkl} = 0.
\end{equation}
Similarly we have
\begin{equation}
v \cdot \delta^p_{qkl_1 \ldots l_s} = 0, \quad \forall s.
\end{equation}
This proves that the   $\Fc$-action and hence the $\Hc_n$ action on $V$ is trivial.
\end{proof}

Now we prove the main result of this section.

\begin{theorem}
The only finite dimensional AYD module over the Connes-Moscovici Hopf algebra $\mathcal{H}_n$  is $\Cb_\d$.
\end{theorem}

\begin{proof}
By Lemma \ref{lemma-action-trivial} and Lemma \ref{lemma-coaction-trivial} we know that the only finite dimensional YD module on $\Hc_n$ is the trivial one. On the other hand, by \cite{Stai} we know that the category of AYD modules and the category of YD modules over a Hopf algebra $H$ are equivalent provided $H$ has a modular pair in involution $(\d,\s)$. In fact, the equivalence functor between these  two categories is simply given by
\begin{equation}
^H\mathcal{YD}_H \to \; ^H\mathcal{AYD}_H, \quad  V \mapsto \,{^\s}V_\d:=V\ot\;^\s\Cb_\d.
\end{equation}
Since by the result of Connes-Moscovici  in \cite{ConnMosc98} the Hopf algebra $\Hc_n$ admits a modular pair in involution $(\d,1)$,  we conclude that the only finite dimensional AYD module on $\Hc_n$ is $\Cb_\d$.
\end{proof}

\chapter{Hopf-cyclic cohomology}\label{chapter-hopf-cyclic-cohomology}


This chapter is devoted to the study of the Hopf-cyclic cohomology of Lie-Hopf algebras with coefficients in SAYD modules. To this end we first focus on the commutative representative Hopf algebras. Then we deal with the cohomology of the noncommutative geometric Hopf algebras with induced coefficients. Finally we identify the Hopf-cyclic cohomology of the Lie-Hopf algebra $R(\Fg_2)\acl U(\Fg_1)$ with coefficients in full generality with the Lie algebra cohomology of $\Fg_1\bowtie \Fg_2$.

\medskip

More explicitly, we first establish a van Est type isomorphism (extending the results in \cite{ConnMosc98,ConnMosc01,MoscRang09}) between the Hopf-cyclic cohomology of a Lie-Hopf algebra with coefficients in an induced SAYD module and the Lie algebra cohomology of the corresponding Lie algebra relative to a Levi subalgebra with coefficients in the original module.

\medskip

We then extend our van Est type isomorphism in the presence of the nontrivial Lie algebra coaction. In  contrast to the induced case, the isomorphism is  proved at the first level of natural spectral sequences,  rather than at the level of complexes.


\section{ Hopf cyclic cohomology of representative Hopf algebras}

In this section, we compute the Hopf-cyclic cohomology of the commutative Hopf algebras $R(G)$, $R(\Fg)$ and $\mathscr{P}(G)$ with coefficients.

\medskip

Let a coalgebra $C$ and an algebra $A$  be in duality, \ie there exists a pairing
\begin{equation}\label{aux-pairing-algebra-coalgebra}
\langle\; ,\rangle: C\ot A\ra \Cb
\end{equation}
compatible with product and coproduct in the sense that
\begin{equation}\label{pairing-property}
\langle c,ab\rangle=  \langle c\ps{1},a\rangle\langle c\ps{2},b\rangle, \quad \langle c,1\rangle=\ve(c).
\end{equation}
Using this duality, we can turn a bicomodule $V$ over $C$ into a bimodule over $A$ via
\begin{equation}
a \cdot v:= \langle v\ns{1},a\rangle v\ns{0}, \quad v\cdot a:= \langle v\ns{-1},a\rangle v\ns{0}.
\end{equation}

Now we define the map
\begin{align}\label{map-theta-Alg}
 \begin{split}
&\t_{(C,A)}: V\ot C^{\ot q}\ra \Hom(A^{\ot q},V),\\
&\t_{(C,A)}(v\ot c^1\odots c^q)(a_1 \odots a_q)= \langle c^1, a_1\rangle \cdots \langle c^q, a_q\rangle v.
\end{split}
 \end{align}

\begin{lemma}
For any algebra $A$  and  coalgebra $C$ with a pairing \eqref{aux-pairing-algebra-coalgebra} and any $C$-bicomodule $V$, the map $\t_{(C,A)}$ defined in \eqref{map-theta-Alg} is a map of complexes between the Hochschild complex of the coalgebra $C$ with coefficients in the bicomodule $V$ and the Hochschild complex of the algebra $A$ with coefficients in the $A$-bimodule induced by $V$.
\end{lemma}

\begin{proof}
The claim follows directly from the pairing property \eqref{pairing-property}.
\end{proof}

Let a commutative Hopf algebra $\Fc$ and the universal enveloping algebra $U(\Fg)$ of a Lie algebra $\Fg$ be in Hopf duality via
\begin{equation}
\langle,\rangle:\Fc\ot U(\Fg)\ra \Cb.
\end{equation}
In addition, let us assume that $\Fg=\Fh\ltimes \Fl$,  where the Lie subalgebra  $\Fh$ is reductive,  every finite dimensional representation of $\Fh$ is semisimple and $\Fl$ is an ideal.

\medskip

For a $\Fg$-module $V$ we observe that the inclusion $\Fl\hookrightarrow \Fg$ induces a map
   \begin{equation}
   \pi_\Fl: \Hom(U(\Fg)^{\ot q},V)\ra \Hom(U(\Fl)^{\ot l}, V)
   \end{equation}
of Hochschild complexes. Moreover, it is known that the antisymmetrization map
   \begin{align}\label{map-antisymmetrization}
   &\a: \Hom(U(\Fl)^{\ot q},V)\ra W^q(\Fl,V):= V\ot \wdg^{q}\Fl^\ast,\\\notag
    &\a(\om)(X_1,\ldots,X_q)=\sum_{\s\in {S_q}}(-1)^\s \om(X_{\s(1)},\ldots,X_{\s(q)})
   \end{align}
is a map of complexes between Hochschild cohomology of $U(\Fl)$ with coefficients in $V$ and the Lie algebra cohomology of $\Fl$ with coefficients in $V$.

\medskip

We then use the fact that $\Fh$ acts  semisimply  to decompose the Lie algebra cohomology complex $W(\Fl,V)$ into the weight spaces
\begin{equation}
W^\bullet(\Fl,V)= \bigoplus _{\mu\in \Fh^\ast}W^\bullet_\mu(\Fl,V).
\end{equation}
Since  $\Fh$ acts on  $\Fl$ by derivations, we observe that each $W_\mu^\bullet(\Fl,V)$ is  a complex on its own and the projection $\pi_\mu: W^{\bullet}(\Fl,V)\ra W^{\bullet}_\mu(\Fl,V)$ is a map of complexes. Composing $\t_{(\Fc,U(\Fg))}$, $\pi_\Fl$ and $\pi_\mu$ we get a map of complexes
   \begin{equation}
   \t_{\Fc,U(\Fg),\Fl,\mu}:=\pi_\mu\circ\pi_\Fl\circ\t_{(\Fc,U(\Fg))}: C^\bullet_{\rm coalg}(\Fc,V)\ra W^\bullet_\mu(\Fl,V).
   \end{equation}

\begin{definition}
Let a commutative Hopf algebra $\Fc$ be in Hopf duality with $U(\Fg)$. A decomposition $\Fg=\Fh\ltimes \Fl$ of Lie algebras is called a $\Fc$-Levi decomposition if the map $\t_{\Fc,U(\Fg),\Fl,0}$ is a quasi isomorphism for any $\Fc$-comodule $V$.
\end{definition}

\begin{theorem}\label{Theorem-F-Levi}
Let a commutative Hopf algebra $\Fc$ be in duality with $U(\Fg)$ and let $\Fg=\Fh\ltimes \Fl$ be an $\Fc$-Levi decomposition. Then the map  $\t_{\Fc,U(\Fg),\Fl,0}$ induces an isomorphism between Hopf cyclic cohomology of $\Fc$ with coefficients in $V$ and the relative Lie algebra cohomology of $\Fh\subseteq \Fg$ with coefficients in $V$. In other words,
\begin{equation}
HP^\bullet(\Fc,V)\cong \bigoplus_{\bullet=\,i\;{\rm mod\, 2}}H^i(\Fg,\Fh,V).
\end{equation}
\end{theorem}

\begin{proof}
We first recall from \cite{KhalRang03} that for any commutative Hopf algebra $\Fc$ and trivial comodule, the Connes boundary $B$ vanishes in the level of Hochschild cohomology. The same proof works for any comodule and hence we have
\begin{equation}
HP^\bullet(\Fc,V)\cong \bigoplus_{\bullet= \,i\;{\rm mod\, 2}} H^{i}_{\rm coalg}(\Fc,V).
\end{equation}
Since $\Fg=\Fh\ltimes \Fl$ is assumed to be a $\Fc$-Levi decomposition, the  map of complexes  $\t_{\Fc,U(\Fg),\Fl,0}$ induces the desired isomorphism in the level of cohomologies.
\end{proof}

\subsection{Hopf-cyclic cohomology of $R(G)$}

In this subsection, we compute the Hopf cyclic cohomology of the commutative Hopf algebra $\Fc:=R(G)$, the Hopf algebra of all representative functions on a Lie group $G$,  with coefficients in a right comodule $V$ over $\Fc$.

\medskip

Let $V$ be a right comodule over $\Fc$, or equivalently a representative left $G$-module. Let us recall from \cite{HochMost62} that a representative $G$-module is a locally finite $G$-module such that for any finite-dimensional $G$-submodule $W$ of $V$, the induced representation of $G$ on $W$ is continuous.

\medskip

The representative $G$-module  $V$ is called representatively injective if for every exact sequence
\begin{equation}
\xymatrix{ 0\ar[r]& A\ar[r]^\rho \ar[d]_{\a}& B\ar[r]\ar@{.>}[dl]^{\b}& C\ar[r]&0\\
&V&&&}
\end{equation}
of $G$-module homomorphisms between representative $G$-modules $A$, $B$, and $C$, and for every $G$-module homomorphism $\a: A\ra V$, there is a $G$-module homomorphism  $\b:B\ra V$ such that $\b\circ\rho=\a$.

\medskip

A representatively injective resolution of the representative  $G$-module $V$ is an exact sequence of $G$-module homomorphisms
\begin{equation}
\xymatrix{ 0\ar[r]& V\ar[r]& X_0\ar[r] & X_1\ar[r]& \cdots}\;,
\end{equation}
where each $X_i$ is a representatively injective  $G$-module.

\medskip

Finally, the representative cohomology group of $G$ with value in the representative $G$-module $V$ is defined to be the cohomology of
\begin{equation}
\xymatrix{ X_0^G\ar[r] & X_1^G\ar[r]& \cdots}\;,
\end{equation}
where $X_i^G$ are the elements of $X_i$ which are fixed by $G$.  We denote this cohomology by $H_{\rm rep}(G,V)$.

\begin{proposition}\label{Propositio-rep-coalgebra}
For any Lie group $G$ and any representative $G$-module $V$, the representative cohomology groups of $G$ with value in $V$ coincide with the coalgebra cohomology groups of the coalgebra $R(G)$ with coefficients in the induced comodule by $V$.
\end{proposition}

\begin{proof}
 In \cite{HochMost61} it is shown  that
\begin{equation}
\xymatrix{ V\ar[r]^{d_{-1}~~} &V\ot \Fc\ar[r]^{d_0~~}&V\ot \Fc^{\ot 2}\ar[r]^{~~d_1}&\cdots }\;,
\end{equation}
is a representatively injective resolution for the representative $G$-module $V$.
Here $G$ acts on $V\ot \Fc^{\ot n}$ by
\begin{equation}
\g\cdot(v\ot f^1\odots f^q)= \g v\ot f^1\cdot \g^{-1}\odots f^q\cdot \g^{-1},
\end{equation}
where $G$ acts on $\Fc$ by right translation. We view $V\ot \Fc^{\ot q}$ as a group cochain with value in the $G$-module $V$ by embedding $V\ot \Fc^{\ot q}$ into the space ${\rm Map}_{\rm cont}(G^{\times q},V)$ of all continuous maps from $G^{\times\, q}$ to $V$,  via
 \begin{equation}
 (v\ot f^1\odots f^q)(\g_1,\ldots, \g_q)=f^1(\g_1)\cdots f^q(\g_q)v.
 \end{equation}
The coboundaries $d_i$ are defined by
\begin{align}
\begin{split}
&d_{-1}(v)(\g)=v,\\
&d_i(\phi)(\g_1,\ldots,\g_{q+1})= \sum_{i=0}^{q+1} \phi(\g_0,\ldots, \widehat{\g}_i,\ldots, \g_{q}).
\end{split}
\end{align}
We then identify $(V\ot \Fc^{\ot q})^G$ with $V\ot \Fc^{\ot q-1}$ via
\begin{equation}
v\ot f^1\odots f^q\mapsto \ve(f^1)v\ot f^2 f^3\ps{1}\cdots f^q\ps{1}\ot f^3\ps{2}\cdots f^q\ps{2}\odots f^{q-1}\ps{q}f^q\ps{q}\ot f^q\ps{q+1}.
\end{equation}
The complex of the $G$-fixed part of the resolution is
 \begin{equation}
\xymatrix{  V\ar[r]^{\d_0~~} &V\ot \Fc\ar[r]^{~~\d_1}&\cdots}\;,
\end{equation}
where the coboundaries are defined by
\begin{align}
\begin{split}
&\d_0: V\ra V\ot \Fc,\quad  \quad \d_0(v)=  v\ns{0}\ot v\ns{1}-v\ot 1,\\
&\d_i: V\ot\Fc^{\ot \,q}\ra V\ot \Fc^{\ot \,q+1}, \quad i \geq 1,\\
&\d_i(v\ot f^1\odots f^q)=  v\ns{0}\ot v\ns{1}\ot f^1\odots f^q+\\
&\sum_{i=1}^q (-1)^i v\ot f^1\odots f^i\ps{1}\ot f^i\ps{2}\odots f^q + (-1)^{q+1}v\ot f^1\odots f^q\ot 1.
\end{split}
\end{align}
This is the complex that computes the coalgebra cohomology of $\Fc$ with coefficients in the $\Fc$-comodule $V$.
\end{proof}

Let us recall from \cite{HochMost61} that a nucleus of a Lie group $G$ is  a simply connected solvable closed normal Lie subgroup $L$ of $G$ such that $G/L$ is reductive. Then $G/L$ has  a faithful representation and every finite dimensional analytic representation of $G/L$ is semisimple.    In this case $G=S\ltimes L$, where $S:=G/L$ is reductive.  Let, in addition,  $\Fs\subseteq\Fg$ be the Lie algebras of  $S$ and $G$ respectively.

\medskip

For a representative $G$-module $V$ let us define the map
\begin{align}\label{aux-map-D-Lie-group}
\begin{split}
&\Dc_{\rm Gr}: V\ot \Fc^{\ot q}\ra C^q(\Fg,\Fs,V),\\
&\Dc_{\rm Gr}(v\ot f^{1}\odots f^q)(X_1,\ldots,X_q)\\
 &~~~~~~~~~~=\sum_{\mu\in S_q}(-1)^\mu\mdt{1}\cdots\;\mdt{q}f^1(\exp(t_1X_{\mu(1)}))\cdots f^q(\exp(t_qX_{\mu(q)}))v.
\end{split}
\end{align}

\begin{theorem} \label{theorem-R(G)-cohomology}
Let $G$ be a  Lie group, $V$ be a representative $G$-module, $L$ be a  nucleus of $G$ and $\Fs\subset \Fg$ be the Lie algebras of $S:=G/L$ and $G$ respectively. Then the map \eqref{aux-map-D-Lie-group} induces an isomorphism of vector spaces between the periodic Hopf cyclic cohomology of $R(G)$ and the relative Lie algebra cohomology of $\Fs\subseteq \Fg$ with coefficients in  $\Fg$-module induced by $V$. In other words,
\begin{equation}
HP^\bullet(R(G),V)\cong \bigoplus_{\bullet=i\;\;{\rm mod \; 2}}H^i(\Fg,\Fs,V).
\end{equation}
\end{theorem}

\begin{proof}
We know that $R(G)$ and $U(\Fg)$ are in Hopf duality via
   \begin{equation}
   \langle f,X\rangle =\dt f(\exp(tX)), \quad X\in\Fg, f\in R(G).
   \end{equation}
As a result, we see that $\Dc_{\rm Gr} = \t_{R(G),U(\Fg),\Fl,0}$  and hence by Proposition \ref{Propositio-rep-coalgebra} we have a map of complexes between the complex of representative group cochains of $G$ with value in $V$ and the relative Chevalley-Eilenberg complex of the Lie algebras $\Fs\subseteq\Fg$ with  coefficients in the $\Fg$-module induced by $V$.

\medskip

Moreover, by \cite[Theorem 10.2]{HochMost62}, $\Dc_{\rm Gr}$ induces a quasi-isomorphism. So $\Fg=\Fs\ltimes\Fl$ is a $R(G)$-Levi decomposition and hence the result follows from Theorem \ref{Theorem-F-Levi}.
\end{proof}

\subsection{Hopf-cyclic cohomology of $R(\Fg)$}

In this subsection we compute the Hopf cyclic cohomology of the commutative Hopf algebra $R(\Fg)$ of representative functions on $U(\Fg)$ with coefficients in $V$. To this end,  we will need a cohomology theory similar to the representative cohomology of Lie groups.

\medskip

Let $\Fg$ be a  Lie algebra and $V$  be  a locally finite $\Fg$-module, equivalently an $R(\Fg)$-comodule. The $\Fg$-module  $V$ is called injective if for every exact sequence
\begin{equation}
\xymatrix{ 0\ar[r]& A\ar[r]^\rho \ar[d]_{\a}& B\ar[r]\ar@{.>}[dl]^{\b}& C\ar[r]&0\\
&V&&&}
\end{equation}
of $\Fg$-module homomorphisms between $\Fg$-modules $A$, $B$, and $C$, and for every $\Fg$-module homomorphism $\a: A\ra V$, there is a $\Fg$-module homomorphism  $\b:B\ra V$ such that $\b\circ\rho=\a$.

\medskip

An  injective resolution of the $\Fg$-module $V$ is an exact sequence of $\Fg$-module homomorphisms
\begin{equation}
\xymatrix{ 0\ar[r]& V\ar[r]& X_0\ar[r] & X_1\ar[r]& \cdots}\;,
\end{equation}
where each $X_i$ is an injective  $\Fg$-module.

\medskip

The representative cohomology groups of $G$ with value in the $\Fg$-module $V$ are then defined to be the cohomology groups of
\begin{equation}
\xymatrix{ X_0^\Fg\ar[r] & X_1^\Fg\ar[r]& \cdots}\;,
\end{equation}
where
\begin{equation}
 X_i^\Fg=\{\xi\in X_i\mid X\xi=0\quad \forall\, X\in \Fg \}.
 \end{equation}
We denote this cohomology by $H_{\rm rep}(\Fg,V)$.

\medskip

Since the category of locally finite $\Fg$-modules and the category of $R(\Fg)$-comodules are equivalent,  $H^\bullet_{\rm rep}(\Fg,V)$ is the same as ${\rm Cotor}^\bullet_{R(\Fg)}(V,\Cb)$, which  is nothing but    $H^\bullet_{\rm coalg}(R(\Fg),V)$.

\medskip

Let $\Fl$ be the radical of $\Fg$, \ie the unique maximal solvable ideal of $\Fg$. Levi decomposition implies that  $\Fg=\Fs\ltimes \Fl$. Here $\Fs$ is a  semisimple subalgebra of $\Fg$, and it is called a Levi subalgebra .

\medskip

We now consider the map
\begin{align}\label{aux-map-D-Lie-alg}
\begin{split}
&\Dc_{\rm Alg}: V\ot R(\Fg)^{\ot q}\ra (V\ot\wedge^{q} \Fl^\ast)^\Fs,\\
&\Dc_{\rm Alg}(v\ot f^1\odots f^q)(X_1,\ldots,X_q)= \sum_{\s\in S_q}(-1)^\s f^1(X_{\s(1)})\dots f^q(X_{\s(q)})v.
\end{split}
\end{align}

\begin{proposition}\label{Proposition-R(g)-cohomology}
Let $\Fg$ be a Lie algebra and $\Fg=\Fs\ltimes \Fl$ be a   Levi decomposition.  Then  for any finite dimensional $\Fg$-module $V$, the map \eqref{aux-map-D-Lie-alg} induces an isomorphism between the representative cohomology $H_{\rm rep}(\Fg,V)$ and the relative  Lie algebra cohomology $H(\Fg,\Fs,V)$.
\end{proposition}

\begin{proof}
We know that $R(\Fg)$ and $U(\Fg)$ are in (nondegenerate) Hopf duality via
   \begin{equation}
   \langle f,u\rangle = f(u), \quad u\in U(\Fg), f\in R(\Fg).
   \end{equation}
As a result, $\Dc_{\rm Alg} = \t_{R(\Fg),U(\Fg),\Fl,0}$ and hence \eqref{aux-map-D-Lie-alg} is a map of complexes.

\medskip

Now let $G$ be the  simply connected Lie group of the Lie algebra $\Fg$. The Levi decomposition  $\Fg=\Fs\ltimes \Fl$ induces a nucleus decomposition  $G=S\ltimes L$. Since $G$ is simply connected the representations of $\Fg$ and $G$ coincide, and any injective resolution of $\Fg$ is induced by an representatively injective  resolution of $G$. Therefore, the obvious map $H_{\rm rep}(G,V)\ra H_{\rm rep}(\Fg,V)$ is surjective.

\medskip

Since $\Dc_{\rm Gr}: H_{\rm rep}(G,V)\ra H(\Fg,\Fs,V)$ factors through $\Dc_{\rm Alg}: H_{\rm rep}(\Fg,V)\ra H(\Fg,\Fs,V)$, when $V$ is finite dimensional the latter map is an isomorphism.
\end{proof}

Now we state the main result of this subsection.

\begin{theorem}\label{theorem-cohomology-R(g)}
Let $\Fg$ be a finite dimensional Lie algebra and $\Fg=\Fs\ltimes \Fl$ be a   Levi decomposition. Then for any   finite dimensional $\Fg$-module $V$, the map \eqref{aux-map-D-Lie-alg} induces an isomorphism of vector spaces between the periodic Hopf cyclic cohomology of $R(\Fg)$ with coefficients in the comodule induced by $V$, and the relative Lie algebra cohomology of $\Fs\subseteq \Fg$ with coefficients in $V$. In other words,
\begin{equation}
HP^\bullet(R(\Fg),V)\cong \bigoplus_{\bullet=\,i\;{\rm mod \; 2}}H^i(\Fg,\Fs,V).
\end{equation}
\end{theorem}

\begin{proof}
By Proposition \ref{Proposition-R(g)-cohomology},  $\Fg=\Fs\ltimes\Fl$ is a $R(\Fg)$-Levi decomposition. Hence the proof follows from Theorem \ref{Theorem-F-Levi}.
\end{proof}

\subsection{Hopf-cyclic cohomology of $\mathscr{P}(G)$}

In this subsection, we compute the Hopf cyclic cohomology of $\mathscr{P}(G)$, the Hopf algebra of polynomial functions on an affine algebraic group $G$.

\medskip

Let $V$ be a finite dimensional polynomial right $G$-module. The polynomial  group cohomology of $G$ with coefficients in $V$ is defined to be the cohomology of the complex
\begin{equation}\label{aux-polynomial-complex}
\xymatrix{ C^0_{\rm pol}(G,V)\ar[r]^\d& C^1_{\rm pol}(G,V)\ar[r]^{~~~\d} &\cdots }
\end{equation}
where $C^0_{\rm pol}(G,V)=V$, and
\begin{equation}
C^q_{\rm pol}(G,V)=\left\{ \phi:G^{\times\,q}\ra V\mid \phi\;\; \text{is polynomial}\right\}, \quad q \geq 1.
\end{equation}
The coboundary $\d$ is the group cohomology coboundary
 \begin{align}
 \begin{split}
 &\d:V\ra C^1_{\rm pol}(G,V), \quad \d(v)(\g)=v- v\cdot \g,\\
 &\d: C^q_{\rm pol}(G,V)\ra C^{q+1}_{\rm pol}(G,V), \quad q\geq 1, \\
& \d(\phi)(\g_1,\ldots,\g_{q+1})=\d(\phi)(\g_2,\ldots,\g_{q+1})+\\
 &\sum_{i=1}^q (-1)^i \phi(\g_1, \ldots, \g_i\g_{i+1}, \ldots, \g_{q+1})+ (-1)^{q+1}\phi(\g_1,\ldots, \g_q)\cdot\g_{q+1}
 \end{split}
 \end{align}

We can identify $C^q_{\rm pol}(G,V)$ with $V\ot \mathscr{P}(G)^{\ot q}$ via
 \begin{equation}
 v\ot f^1\odots f^q(\g_1, \ldots,\g_q):= f^1(\g_1)\cdots f^q(\g_q)v
 \end{equation}
to observe that the coboundary $\d$ is identified with the Hochschild coboundary of the coalgebra $\mathscr{P}(G)$ with values in the bicomodule $V$, whose right comodule structure is trivial and left comodule structure is induced by the right $G$-module structure.

\medskip

The  cohomology of the complex \eqref{aux-polynomial-complex} is denoted by $H_{\rm pol}(G,V)$. The cohomology $H_{\rm }(G,V)$ can also be calculated by means of polynomially injective resolutions \cite{Hoch61}.

\medskip

A  polynomial representation $V$ of an affine algebraic group $G$ is called polynomially  injective if for any exact sequence
\begin{equation}
\xymatrix{ 0\ar[r]& A\ar[r]^\rho \ar[d]_{\a}& B\ar[r]\ar@{.>}[dl]^{\b}& C\ar[r]&0\\
&V&&&}
\end{equation}
of polynomial modules over $G$, there is a polynomial $G$-module homomorphism  $\b:B\ra V$ such that $\b\circ\rho=\a$.

\medskip

A polynomially injective  resolution for a polynomial $G$-module $V$ is an exact sequence
    \begin{equation}
\xymatrix{ 0\ar[r]&V\ar[r]& X_0\ar[r] & X_1\ar[r]& \cdots}
\end{equation}
of polynomially injective modules over $G$.

\medskip

It is shown in \cite{Hoch61} that $H_{\rm pol}(G,V)$ is computed by the cohomology of the $G$-fixed part of any  polynomially injective resolution \ie the by the complex
   \begin{equation}
 \xymatrix{  X_0^G\ar[r] & X_1^G\ar[r]& \cdots}\;.
\end{equation}
A natural  polynomially injective  resolution of a polynomial $G$-module $V$ is the resolution of the polynomial differential forms with values in $V$. This is $V\ot \wdg^\bullet\Fg^\ast\ot \mathscr{P}(G)$, where $G$ acts by  $\g\cdot (u\ot f)= u\ot (f\cdot \g^{-1})$, and $G$ acts on $\mathscr{P}(G)$ by right translations.

\medskip

This yields the map of complexes
\begin{align}\label{aux-map-D-polynomial}
\begin{split}
&\Dc_{\rm Pol}: C^q_{\rm pol}(G,V)\ra C^q(\Fg,\Fg^{\rm red},V),\\
&\Dc_{\rm Pol}(v\ot f^1\odots f^q)(X_1,\ldots, X_q)\\
&~~~~~~~~~~~~~~~~~~~~~~~~~~=\sum_{\s\in S_q} (-1)^\s (X_{\s(1)}\cdot f^1)(e)\cdots (X_{\s(q)}\cdot f^q)(e)\;v,
\end{split}
\end{align}
where we identify the Lie algebra $\Fg$ of $G$ with the  left $G$-invariant derivations of $\mathscr{P}(G)$, and $G= G^{\rm red}\ltimes G^{\rm u}$ is a Levi decomposition of an affine algebraic group $G$, \ie $G^{\rm u}$ is the unipotent radical of $G$ and $G^{\rm red}$ is the maximal reductive subgroup of $G$, and finally  $\Fg^{\rm red}$ and $\Fg^{\rm u}$ denote the Lie algebras of $G^{\rm red}$ and $G^{\rm u}$ respectively.

\begin{theorem}\label{theorem-cohomology-P(G)}
Let $G$ be an affine algebraic group and $V$ be a finite dimensional polynomial $G$-module. Let $G= G^{\rm red}\ltimes G^{\rm u}$ be a Levi decomposition of $G$ and $\Fg^{\rm red}\subseteq \Fg$ be the Lie algebras of $G^{\rm red}$ and $G$ respectively. Then the map \eqref{aux-map-D-polynomial} induces an isomorphism between the periodic Hopf-cyclic cohomology of $\mathscr{P}(G)$ with coefficients in the comodule induced by $V$, and the  Lie algebra cohomology of $\Fg$ relative to $\Fg^{\rm red}$ with coefficients in the $\Fg$-module $V$. In other words,
\begin{equation}
HP^\bullet(\mathscr{P}(G),V)\cong \bigoplus_{\bullet\,=\,i \;{\rm mod}\, 2} H^i(\Fg,\Fg^{\rm red},V)
\end{equation}
\end{theorem}

\begin{proof}
It is shown in \cite{Hoch61} that $V\ot \wdg^\bullet\Fg^\ast\ot \mathscr{P}(G)$ is a polynomially injective resolution for $V$. The comparison between this resolution and the standard resolution $V\ot \mathscr{P}(G)^{\ot\,\bullet+1}$ results in the map \eqref{aux-map-D-polynomial}.  The proof of \cite[Theorem 2.2]{KumaNeeb06} shows that  the map \eqref{aux-map-D-polynomial} is an isomorphism between  $H^\bullet_{\rm pol}(G,V)$ and $H^\bullet(\Fg,\Fg_{\rm red},V)$. On the other hand, in view of the natural pairing
\begin{equation}
\langle u\,,\, f\rangle = (u\cdot f)(e)
\end{equation}
between $\mathscr{P}(G)$ and $U(\Fg)$, we have $\Dc_{\rm Pol} = \t_{\mathscr{P}(G),U(\Fg),\Fg^{\rm u},0}$. This shows that $\Fg=\Fg^{\rm red}\ltimes \Fg^{\rm u}$ is a $\mathscr{P}(G)$-Levi decomposition. We then apply Theorem \ref{Theorem-F-Levi} to finish the proof.
\end{proof}

\section{Hopf-cyclic cohomology of Lie-Hopf algebras with induced coefficients}

In this section we extend the machinery for computing the Hopf cyclic cohomology of bicrossed product Hopf algebras developed in \cite{MoscRang09,MoscRang11} to compute the Hopf cyclic cohomology of  the geometric bicrossed product Hopf algebras $R(G_2)\acl U(\Fg_1)$, $R(\Fg_2)\acl U(\Fg_1)$ and $\mathscr{P}(G_2)\acl U(\Fg_1)$ with induced coefficients.

\subsection{Bicocyclic module associated to Lie-Hopf algebras}

Let $\Fg$ be a Lie algebra and  $\Fc$ a commutative $\Fg$-Hopf algebra. We denote the bicrossed product Hopf algebra $\Fc\acl U(\Fg)$ by $\Hc$.

\medskip

Let the character $\d$ and the group-like $\s$ be the canonical modular pair in involution defined by \eqref{aux-delta} and \eqref{aux-sigma}. In addition, let $V$ be an induced $(\Fg,\Fc)$-module and $^\s{V}_\d$ be the associated  SAYD module over $\Hc$ defined by \eqref{aux-action-SAYD-over-Lie-Hopf} and \eqref{aux-coaction-SAYD-over-Lie-Hopf}.

\medskip

The Hopf algebra $\Uc:=U(\Fg)$ acts on $\;^\s{V}_\d\ot \Fc^{\ot q}$ by
\begin{equation}
(v\ot \td f)\cdot u= \d_\Fg(u\ps{1}) S(u\ps{2})\cdot v\ot S(u\ps{3})\bullet \td f,
\end{equation}
where $\td f:= f^1\odots f^q$ and the left action of $\Uc$ on $\Fc^{\ot q}$ is defined by
\begin{align} \label{u-bullet}
\begin{split}
& u\bullet( f^1\odots f^n):=\\
&  u\ps{1}\ns{0}\rt f^1\ot u\ps{1}\ns{1}u\ps{2}\ns{0}\rt f^2\odots
u\ps{1}\ns{n-1}\dots u\ps{n-1}\ns{1} u\ps{n}\rt  f^n.
\end{split}
 \end{align}

Following \cite{MoscRang09,MoscRang11}, we then define a bicocyclic module
\begin{equation}
C(\Uc,\Fc,\,^\s{V}_\d) = \bigoplus_{p,q\geq 0}C^{p,q}(\Uc,\Fc,\,^\s{V}_\d),
\end{equation}
where
\begin{equation}
C^{p,q}(\Uc,\Fc,\,^\s{V}_\d):=\,^\s{V}_\d\ot \Uc^{\ot p}\ot \Fc^{\ot q}, \qquad p,q\ge 0,
\end{equation}
whose horizontal morphisms are given by
\begin{align}
\begin{split}
&\hd_0(v\ot \td{u}\ot \td{f})= v\ot 1\ot u^1\ot\dots\ot u^p\ot  \td{f}\\
&\hd_j(v\ot \td{u}\ot \td{f})= v\ot u^1\ot\dots\ot
\Delta(u^i)\ot\dots \ot u^p\ot  \td{f}\\
&\hd_{p+1}(v\ot \td{u}\ot \td{f})=v\ot u^1\ot \dots \ot
u^p\ot 1\ot \td{f}\\
 &\hs_j(v\ot \td u\ot \td f)=v\ot u^1\ot \dots \ot \epsilon(u^{j+1})\ot\dots\ot u^p\ot
 \td{f}\\
&\hta(v\ot \td{u}\ot  \td{f})=\\
&\d_\Fg(u\ps{1})S(u\ps{2})\cdot v\ot S(u^1\ps{4})\cdot(u^2\ot\dots\ot u^p\ot 1)\ot S(u^1\ps{3})\bullet \td{f},
 \end{split}
\end{align}
while the vertical morphisms are
\begin{align}
\begin{split}
&\vd_0(v\ot \td{u}\ot \td{f})= v\ot \td u\ot  1\ot \td{f},\\
 &\vd_j(v\ot \td{u}\ot \td{f})= v \ot \td u \ot    f^1\odots\Delta(f^j)\odots f^q,\\
&\vd_{q+1}(v\ot \td{u}\ot \td{f})= v\ns{0} \ot \td u\ns{0}\ot \td{f}\ot S(\td u\ns{1}v\ns{1})\s,\\
 &\vs_j(v\ot \td u\ot  \td f)=v\ot \td u\ot  f^1\odots \epsilon(f^{j+1})\odots f^n, \\
 &\vta(v\ot \td{u}\ot  \td{f}) =\\
 & v\ns{0}\ot \td u\ns{0} \ot  S(f^1)\cdot(f^2\odots f^n\ot S(\td u\ns{1}v\ns{1})\s) .
  \end{split}
\end{align}

By definition,  a bicocyclic module is  a bigraded module whose  rows and columns form  cocyclic modules on their own Hochschild coboundary and Connes boundary maps. These boundaries and coboundaries are denoted by $\hB$, $\vB$, $\hb$, and $\vb$, as demonstrated in the diagram
\begin{align}\label{UF}
\begin{xy} \xymatrix{  \vdots\ar@<.6 ex>[d]^{\uparrow B} & \vdots\ar@<.6 ex>[d]^{\uparrow B}
 &\vdots \ar@<.6 ex>[d]^{\uparrow B} & &\\
^\s{V}_\d \ot \Uc^{\ot 2} \ar@<.6 ex>[r]^{\hb}\ar@<.6
ex>[u]^{  \uparrow b  } \ar@<.6 ex>[d]^{\uparrow B}&
  ^\s{V}_\d\ot \Uc^{\ot 2}\ot \Fc   \ar@<.6 ex>[r]^{\hb}\ar@<.6 ex>[l]^{\hB}\ar@<.6 ex>[u]^{  \uparrow b  }
   \ar@<.6 ex>[d]^{\uparrow B}&^\s{V}_\d\ot \Uc^{\ot 2}\ot\Fc^{\ot 2}
   \ar@<.6 ex>[r]^{~~\hb}\ar@<.6 ex>[l]^{\hB}\ar@<.6 ex>[u]^{  \uparrow b  }
   \ar@<.6 ex>[d]^{\uparrow B}&\ar@<.6 ex>[l]^{~~\hB} \hdots&\\
^\s{V}_\d \ot \Uc \ar@<.6 ex>[r]^{\hb}\ar@<.6 ex>[u]^{  \uparrow b  }
 \ar@<.6 ex>[d]^{\uparrow B}&  ^\s{V}_\d \ot \Uc \ot\Fc \ar@<.6 ex>[r]^{\hb}
 \ar@<.6 ex>[l]^{\hB}\ar@<.6 ex>[u]^{  \uparrow b  } \ar@<.6 ex>[d]^{\uparrow B}
 &^\s{V}_\d\ot \Uc \ot \Fc^{\ot 2}  \ar@<.6 ex>[r]^{~~\hb}\ar@<.6 ex>[l]^{\hB}\ar@<.6 ex>[u]^{  \uparrow b  }
  \ar@<.6 ex>[d]^{\uparrow B}&\ar@<.6 ex>[l]^{~~\hB} \hdots&\\
^\s{V}_\d  \ar@<.6 ex>[r]^{\hb}\ar@<.6 ex>[u]^{  \uparrow b  }&
^\s{V}_\d\ot\Fc \ar@<.6 ex>[r]^{\hb}\ar[l]^{\hB}\ar@<.6
ex>[u]^{  \uparrow b  }&^\s{V}_\d\ot\Fc^{\ot 2}  \ar@<.6
ex>[r]^{~~\hb}\ar@<.6 ex>[l]^{\hB}\ar@<1 ex >[u]^{  \uparrow b  }
&\ar@<.6 ex>[l]^{~~\hB} \hdots& .}
\end{xy}
\end{align}

We  note that the standard Hopf-cyclic cohomology complex $C^\bullet(\Hc,
\;^\s{V}_\d)$ can be identified with the diagonal subcomplex  $\FZ^{\bullet,\bullet}(\Hc,\Fc,\,^\s{V}_\d)$ of $C^{\bullet,\bullet}(\Uc,\Fc,\,^\s{V}_\d)$ via
\begin{align}\label{aux-PSI-bicrossed}
 \begin{split}
 & \, \Psi_{\acl}: \FZ^{\bullet,\bullet}(\Hc,\Fc,\,^\s{V}_\d) \ra  C^\bullet(\Hc, \;^\s{V}_\d) \\
&\Psi_{\acl}(v\ot u^1\odots u^n\ot f^1\odots f^n) =\\
& v\ot f^1\acl u^1\ns{0}\ot f^2u^1\ns{1}\acl u^2\ns{0}\odots f^n u^1\ns{n-1} \dots u^{n-1}\ns{1}\acl u^n ,
 \end{split}
\end{align}
and its inverse
\begin{align}\label{aux-PSI-bicrossed-inverse}
 \begin{split}
 & \Psi^{-1}_{\acl}:C^\bullet(\Hc,\;^\s{V}_\d) \ra \FZ^{\bullet,\bullet}(\Hc,\Fc,\,^\s{V}_\d) \\
&\Psi^{-1}_{\acl}(v\ot   f^1\acl u^1\ot \dots\ot f^n\acl u^n) =\\
& v\ot u^1\ns{0}\odots u^{n-1}\ns{0}\ot u^n\ot f^1\ot\\
&\ot f^2S(u^1\ns{n-1})\ot f^3S(u^1\ns{n-2}u^2\ns{n-2}) \odots
f^nS(u^1\ns{1} \dots u^{n-1}\ns{1})  .
 \end{split}
\end{align}

On the other hand, the bicocyclic module \eqref{UF} can be further reduced to
\begin{equation}
C^{\bullet, \bullet}(\Fg,\Fc,\, ^\s{V}_\d):=\,^\s{V}_\d\ot \wedge^{\bullet}\Fg\ot \Fc^{\ot \bullet} ,
\end{equation}
by replacing the tensor algebra of $\Uc (\Fg)$ with the exterior algebra of  $\Fg$.  In order to see this we first note that the action \eqref{u-bullet}, when restricted to $\Fg$, reads
 \begin{align} \label{bullet}
 \begin{split}
&X\bullet (f^1 \ot \cdots \ot f^q) = \\
& X\ps{1}\ns{0}\rt f^1\ot X\ps{1}\ns{1}(X\ps{2}\ns{0}\rt f^2)\odots X\ps{1}\ns{q-1}\dots X\ps{q-1}\ns{1}(X\ps{q}\rt f^q).
\end{split}
\end{align}
Next, let
\begin{equation}
\,  \p_{\Fg} : \;^\s{V}_\d\ot \wg^p\Fg\ot\Fc^{\ot q} \ra \;^\s{V}_\d\ot \wg^{p-1}\Fg\ot\Fc^{\ot q}
\end{equation}
be the Lie algebra homology boundary with coefficients in the $\Fg$ module $\;^\s{V}_\d\ot\Fc^{\ot q}$,
\begin{equation} \label{Liebdact}
(v\ot\td f)\lt X= v\d_\Fg(X)\ot \td{f}\; - \; X\cdot v\ot \td{f}\;-\;v\ot X\bullet \td f.
\end{equation}
From the antisymmetrization map
\begin{align} \label{antsym1}
&\td \a_{\Fg}: \;^\s{V}_\d\ot
\wg^p\Fg \ot \Fc^{\ot q}\ot\ra \;^\s{V}_\d\ot \Uc^{\ot p} \ot \Fc^{\ot q} ,
\qquad \td\a_{\Fg}= \a\ot  \Id, \\ \notag
&\a(v \ot X^1\wdots X^p)= \frac{1}{p!} \sum_{\s\in S_p}(-1)^\s  v \ot X^{\s(1)}\odots X^{\s(p)} ,
\end{align}
we have \cite[Proposition 7]{ConnMosc98}
\begin{equation}
\vb \circ \td \a_{\Fg} = 0, \qquad \vB \circ \td \a_{\Fg} = \td \a_{\Fg} \circ \p_{\Fg}.
\end{equation}
On the other hand, since $\Fc$ is commutative, the coaction $\Db:\Fg\ra \Fg\ot \Fc$ extends to a unique coaction $\Db_{\Fg}:\wg^p\Fg\ra \wg^p\Fg\ot \Fc$ for any $p \geq 1$. Tensoring this with the right coaction on $\,^\s{V}_\d$, we obtain
\begin{align}
\begin{split} \label{wgcoact}
& \Db_{^\s{V}_\d\ot\wdg\Fg}: \,^\s{V}_\d \ot \wdg^p\Fg \ra \,^\s{M}_\d \ot \wdg^p\Fg \ot \Fc , \qquad \forall p \geq 1,\\
&\Db_{^\s{V}_\d\ot\wdg\Fg}(v\ot X^1\wdots X^p) =\\
&\,v\ns{0}\ot X^1\ns{0}\wdots X^p\ns{0}\ot
\s^{-1}v\ns{1}X^1\ns{1}\dots X^p\ns{1} .
\end{split}
\end{align}
\medskip
This leads to
\begin{align} \label{b-cobd}
\begin{split}
& b_\Fc: \,^\s{V}_\d \ot \wdg^p\Fg \ot \Fc^{\ot\,q} \ra \,^\s{V}_\d \ot \wdg^p\Fg \ot \Fc^{\ot\,q+1} \\
&b_\Fc(v\ot\a\ot   f^1\odots f^q )=v \ot\a\ot  1\ot f^1\odots f^q\\
&+\sum_{i = 1}^q (-1)^i v \ot\a\ot  f^1\odots \D(f^i)\odots f^q\\
&+(-1)^{q+1} v\ns{0}\ot\a\ns{0}\ot   f^1\odots f^q\ot  S(\a\ns{1})S(v\ns{1})\s  ,
\end{split}
\end{align}
which satisfies
\begin{equation}
\hb \circ \td \a_{\Fg} = \td \a_{\Fg} \circ b_\Fc.
\end{equation}
Similarly, the horizontal boundary
\begin{equation}\label{B-bd}
B_\Fc= \left(\sum_{i=0}^{q-1}(-1)^{(q-1)i}\tau_\Fc^{i}\right) \s \tau_\Fc,
\end{equation}
where
\begin{align}
\begin{split}
&\tau_\Fc(v\ot \a\ot f^1\odots f^q) = \\
& v\ns{0}\ot\a\ns{0}\ot S(f^1)\cdot(f^2\odots f^q\ot S(\a\ns{1})S(v\ns{1})\s)
\end{split}
\end{align}
and
\begin{equation}
\s(v\ot\a\ot f^1\odots f^q)= \ve(f^q)v\ot\a\ot f^1\odots f^{q-1}.
\end{equation}
We then have
\begin{equation}
\hB \circ \td \a_{\Fg} = \td \a_{\Fg} \circ B_\Fc.
\end{equation}
However, since $\Fc$ is commutative and acts on $\;^\s{V}_\d\ot \wg^q\Fg$ trivially, by \cite[Theorem 3.22]{KhalRang03},  $B_\Fc$ vanishes in Hochschild cohomology and therefore can be omitted.

\medskip

As a result, we arrive at the bicomplex
\begin{align}\label{UF+}
\begin{xy} \xymatrix{  \vdots\ar[d]^{\p_{\Fg}} & \vdots\ar[d]^{\p_{\Fg}}
 &\vdots \ar[d]^{\p_{\Fg}} & &\\
\;^\s{V}_\d \ot \wg^2\Fg \ar[r]^{b_{\Fc}} \ar[d]^{\p_{\Fg}}& \;^\s{V}_\d\ot \wg^2\Fg\ot\Fc  \ar[r]^{b_{\Fc}} \ar[d]^{\p_{\Fg}}
 &\;^\s{V}_\d\ot \wg^2\Fg\ot\Fc^{\ot 2} \ar[r]^{~~~~~~b_{\Fc}}   \ar[d]^{\p_{\Fg}}& \hdots&\\
\;^\s{V}_\d \ot \Fg \ar[r]^{b_{\Fc}} \ar[d]^{\p_{\Fg}}&  \;^\s{V}_\d\ot \Fg\ot\Fc \ar[r]^{b_{\Fc}} \ar[d]^{\p_{\Fg}}&\;^\s{V}_\d\ot \Fg\ot\Fc^{\ot 2}
   \ar[d]^{\p_{\Fg}} \ar[r]^{~~~~~~~b_{\Fc}}& \hdots&\\
\;^\s{V}_\d \ar[r]^{b_{\Fc}}&  \;^\s{V}_\d\ot\Fc \ar[r]^{b_{\Fc}}&\;^\s{V}_\d\ot\Fc^{\ot 2} \ar[r]^{~~~~~b_{\Fc}} & \hdots&  . }
\end{xy}
\end{align}

Our discussion can be summarized by the following analogue of \cite[Proposition 3.21]{MoscRang09}.

\begin{proposition} \label{mixCE}
The map \eqref{antsym1} induces a quasi-isomorphism between the total complexes $\Tot C^{\bullet, \bullet}(\Fg, \Fc, \,^\s{V}_\d )$ and $\Tot C^{\bullet, \bullet} (\Uc, \Fc, \;^\s{V}_\d)$.
\end{proposition}

\medskip

We next convert the Lie algebra homology into Lie algebra cohomology by the Poincar\'e isomorphism
\begin{align}\label{theta}
\begin{split}
&\FD_{\Fg} = \Id\ot \Fd_{\Fg} \ot \Id : V\ot\wedge^q\Fg^\ast\ot \Fc^{\ot p} \ra
  \;V\ot  \wedge^{N}\Fg^\ast \ot\wedge^{N-q}\Fg  \ot \Fc^{\ot p} \\
 &\Fd_{\Fg}(\eta) \, = \,   \varpi^\ast\ot \iota(\eta) \varpi,
 \end{split}
\end{align}
where $\varpi$ is a covolume element and $\varpi^\ast $ is the dual volume element.

\medskip

For any $\lambda \in \Fg^\ast $, the contraction operator
$\iota(\lambda) : \wedge^{\bullet}\Fg \ra \wedge^{\bullet-1}\Fg$ is
the unique derivation of degree $-1$ on the exterior algebra of $\Fg$,
\begin{equation}
\iota(\lambda) (X) \, = \, \langle \lambda, X \rangle , \qquad \fl \, X \in \Fg ,
\end{equation}
while for $\eta = \lambda_1 \wdots \lambda_q \in \wedge^{q}\Fg^\ast$,
 $\iota(\eta):\wedge^{\bullet}\Fg \ra \wedge^{\bullet-q}\Fg$ is given by
\begin{equation}
\iota (\lambda_1 \wdots \lambda_q) := \iota (\lambda_q) \circ \ldots
 \circ \iota(\lambda_1) , \qquad \fl \, \lambda_1, \ldots , \lambda_q \in \Fg^\ast .
\end{equation}

The coadjoint action of $\Fg$ on $\Fg^\ast$ induces the action
\begin{align}
\ad^\ast (X) \varpi^\ast = \d_\Fg(X) \varpi^\ast , \qquad \fl \, X \in \Fg,
\end{align}
on $\wg^{N}\Fg^\ast$. Hence we identify $\wg^N\Fg^\ast$ with $\;\Cb_\d$ as $\Fg$-modules.

\medskip

We now set
\begin{equation}\label{aux-left-F-coaction-g-dual}
\Db_\Fg^\ast:\Fg^\ast\ra \Fc \ot \Fg^\ast, \quad \Db_\Fg^\ast(\t^i)= \sum_{j =1}^N f^i_j\ot \t^j,
\end{equation}
which can be seen as the transpose of the  original right coaction \eqref{aux-first-order-matrix-coefficients}. Let us check that it is a coaction. We know from \eqref{aux-comultiplication-matrix-coefficients} that $\D(f_j^i)= f^i_k\ot f^k_j$. Hence,
\begin{equation}
((\Id\ot \Db_{\Fg}^\ast)\circ\Db_\Fg^\ast)(\t^i)= \sum_{j,k}f^i_k\ot f^k_j\ot \t^j=
((\D\ot\Id)\circ\Db_\Fg^\ast)(\t^i).
\end{equation}

We extend this coaction on $\wdg ^\bullet \Fg^\ast$ diagonally via
\begin{align}\label{coaction-g-ast}
\begin{split}
\a\ns{-1}\ot \a\ns{0}= \sum_{1\le l_j\le m} f_{l_1}^{i_1}\cdots f_{l_k}^{i_k}\ot \t^{l_1}\wdots \t^{l_k},\quad \a= \t^{i_1}\wdots \t^{i_k},
\end{split}
\end{align}
which is a coaction as a result of the commutativity of $\Fc$. In particular, we have $\Db_\Fg^\ast(\varpi^\ast)=\s\ot \varpi^\ast$ and hence
\begin{equation}
\wdg^N\Fg^\ast= \;^\s{\Cb}_\d
\end{equation}
as a right module and left comodule.

\medskip

Using  the antipode of $\Fc$ we turn \eqref{aux-left-F-coaction-g-dual} into a right coaction $\Db_{\Fg^\ast}:\Fg^\ast\ra \Fg^\ast \ot \Fc $. We then apply this right coaction to endow $V\ot \wdg^p\Fg^\ast$ with a right coaction
\begin{align}
\begin{split}
&\Db_{V\ot\wdg\Fg^\ast}: V\ot \wdg^p\Fg^\ast\ra V\ot \wdg^p\Fg^\ast\ot \Fc,\\
&\Db_{V\ot\wdg\Fg^\ast}(v\ot \a)= v\ns{0}\ot \a\ns{0}\ot v\ns{1}S(\a\ns{-1}).
\end{split}
\end{align}

\begin{lemma}
The Poincar\'e isomorphism connects $\Db_{^\s{V}_\d\ot\wdg\Fg}$ and $\Db_{V\ot\wdg\Fg^\ast}$ as follows:
\begin{equation}\label{F-coaction-g*}
\Db_{V\ot\wdg\Fg^\ast}(v\ot\om)=\FD_{\Fg}^{-1}\circ\Db_{^\s{V}_\d\ot\wdg\Fg}\circ\FD_{\Fg}(v\ot\om).
\end{equation}
\end{lemma}

\begin{proof}
Without loss of generality let $\om:=\t^{p+1}\wdots\t^N$. We observe that
\begin{equation}
\FD_{\Fg}(v\ot\om)= v\ot\varpi^\ast\ot X_1\wdots X_p.
\end{equation}
Applying $\Db_{^\s{V}_\d\ot\wdg\Fg}$ we get,
\begin{align}\label{proof-FD-Db}
\begin{split}
&\Db_{^\s{V}_\d\ot\wdg\Fg}(\FD_{\Fg}(v\ot\om)) =\\
&\underset{1\le l_t,i_s\le N}{\sum v\ns{0}\ot} \t^{i_1}\wdots \t^{i_N}\ot X_{l_1}\wdots X_{l_p}\ot v\ns{1}S(f^1_{i_1})\cdots S(f_{i_N}^N)f^{l_{1}}_1\cdots f^{l_p}_p =\\
&\sum_{1\le l_t,i_s\le N}v\ns{0}\ot \t^{l_1}\wdots\t^{l_p}\wg\t^{i_{p+1}}\wdots \t^{i_N}\ot\\
&~~~~~~~~~~~~~~~~~~~~~~~~~~~~~~~~~~~~~~~ X_{l_1}\wdots X_{l_p}\ot v\ns{1}S(f^{p+1}_{i_{p+1}})\cdots S(f_{i_N}^N) =\\
&\sum_{\mu\in S_N}(-1)^\mu v\ns{0}\ot\varpi^\ast \ot X_{\mu(1)}\wdots X_{\mu(p)}\ot m\ns{1}S(f^{p+1}_{\mu(p+1)})\cdots S(f^{N}_{\mu(N)}) =\\
&\underset{\underset{\;\mu\in S_{N-p}}{1\le l_1< \ldots< l_p\le N} }{\sum}(-1)^\mu v\ns{0}\ot\varpi^\ast \ot X_{l_1}\wdots X_{l_p}\ot v\ns{1}S(f^{p+1}_{j_{\mu(1)}})\cdots S(f^{N}_{j_{\mu(N-p)}}),\\
\end{split}
\end{align}
where in the last equality $\{j_1<j_2<\ldots <j_{N-p}\}$ denotes the complement of $\{l_1<l_2<\ldots< l_p\}$ in $\{1, \ldots, N\}$.

\medskip

On the other hand,
\begin{align}
\begin{split}
&\FD_\Fg(\Db_{V\ot\wdg\Fg^\ast}(v\ot\om)) =\\
&\underset{1\le l_t\le N}{\sum}\FD_\Fg( v\ns{0}\ot \t^{l_{1}}\wdots\t^{l_{N-p}}\ot v\ns{1}S(f^{p+1}_{l_1}\cdots f^{N}_{l_{N-p}})) =\\
&\underset{\underset{\;\mu\in S_{N-p}}{1\le j_1<\ldots <j_{N-p}\le N,}}{\sum}(-1)^\mu\FD_\Fg( v\ns{0}\ot \t^{j_1}\wdots\t^{j_{N-p}}\ot m\ns{1}S(f^{p+1}_{j_{\mu(1)}}\cdots f^{N}_{j_{\mu(N-p)}}) =\\
&\underset{\underset{\;\mu\in S_{N-p}}{1\le l_1< \ldots< l_p\le N} }{\sum}(-1)^\mu v\ns{0}\ot\varpi^\ast \ot X_{l_1}\wdots X_{l_p}\ot v\ns{1}S(f^{p+1}_{j_{\mu(1)}})\cdots S(f^{N}_{j_{\mu(N-p)}}).
\end{split}
\end{align}
\end{proof}

Finally, by the transfer of structure we pass from  the bicomplex  \eqref{UF+} to
 \begin{align}\label{UF+*}
\begin{xy} \xymatrix{  \vdots & \vdots
 &\vdots &&\\
 V\ot \wdg^2\Fg^\ast  \ar[u]^{\p_{\Fg^\ast}}\ar[r]^{b^\ast_\Fc~~~~~~~}&  V\ot \wdg^2\Fg^\ast\ot\Fc \ar[u]^{\p_{\Fg^\ast}} \ar[r]^{b^\ast_\Fc}& V\ot \wdg^2\Fg^\ast\ot\Fc^{\ot 2} \ar[u]^{\p_{\Fg^\ast}} \ar[r]^{~~~~~~~~~b^\ast_\Fc} & \hdots&  \\
 V\ot \Fg^\ast  \ar[u]^{\p_{\Fg^\ast}}\ar[r]^{b^\ast_\Fc~~~~~}& V\ot \Fg^\ast\ot\Fc \ar[u]^{\p_{\Fg^\ast}} \ar[r]^{b^\ast_\Fc}& V\ot  \Fg^\ast\ot \Fc^{\ot 2} \ar[u]^{\p_{\Fg^\ast}} \ar[r]^{~~~~~b^\ast_\Fc }& \hdots&  \\
   V\ar[u]^{\p_{\Fg^\ast}}\ar[r]^{b^\ast_\Fc~~~~~~~}& V\ot \Fc \ar[u]^{\p_{\Fg^\ast}}\ar[r]^{b^\ast_\Fc}& V\ot \Fc^{\ot 2} \ar[u]^{\p_{\Fg^\ast}} \ar[r]^{~~~~~b^\ast_\Fc} & \hdots&. }
\end{xy}
\end{align}
The vertical coboundary
\begin{equation}
\p_{\Fg^\ast}:C^{p,q}(\Fg^\ast,  \Fc, V)\ra C^{p+1,q}(\Fg^\ast,  \Fc, V)
\end{equation}
is the Lie algebra cohomology coboundary of the Lie algebra $\Fg$ with coefficients in the right $\Fg$-module $V\ot\Fc^{\ot p}$,
\begin{align}
\begin{split}
& ( v\ot f^1\odots f^p) \blacktriangleleft X=\\
&- X\cdot v\ot f^1\odots f^p\;-\; v\ot X\bullet(f^1\odots f^p).
\end{split}
\end{align}
The horizontal coboundary is
\begin{align}
\begin{split}
& b^*_\Fc:C^{p,q}(\Fg^\ast,  \Fc, V)\ra C^{p,q+1}(\Fg^\ast,  \Fc, V) \\
&b^\ast_\Fc(v\ot\a\ot f^1\odots f^q)=\\
&v\ot\a\ot 1\ot f^1\odots f^q +\sum_{i=1}^q(-1)^i v\ot\a\ot f^1\odots \D(f^i)\odots f^q\\
&+(-1)^{q+1}v\ns{0}\ot \a\ns{1}\ot f^1\odots f^q\ot S(v\ns{1})\a\ns{-1}.
\end{split}
\end{align}

We summarize our discussion in the following proposition.

\begin{proposition} \label{mixCE*}
Composition of the maps \eqref{antsym1} and \eqref{theta} induces a quasi-isomorphism between the total complexes  $\Tot C^{\bullet, \bullet} (\Uc, \Fc, \;^\s{V}_\d) $ and $\Tot C^{\bullet, \bullet} (\Fg^\ast,  \Fc,V) $.
\end{proposition}

\subsection{ Hopf-cyclic cohomology of $R(G_2)\acl U(\Fg_1)$, \\ $R(\Fg_2)\acl U(\Fg_1)$ and $\mathscr{P}(G_2)\acl U(\Fg_1)$}

Using the machinary developed in the previous subsection, we now compute the Hopf-cyclic cohomology of the noncommutative geometric Hopf algebras $R(G_2)\acl U(\Fg_1)$, $R(\Fg_2)\acl U(\Fg_1)$ and $\mathscr{P}(G_2)\acl U(\Fg_1)$ with induced coefficients.

\medskip

Let $(\Fg_1,\Fg_2)$ be a matched pair of Lie algebras, $\Fa = \Fg_1 \bowtie \Fg_2$ be their double crossed sum and $\Fh\subseteq \Fg_2$ be a $\Fg_1$-invariant subalgebra. Then we can see that $\Fa/\Fh\cong \Fg_1\oplus \Fg_2/\Fh$. In addition, let $\Fh$ act on $\Fa/\Fh$ by the induced adjoint action \ie
\begin{equation}\label{adjoint}
\ad_\z( Z\oplus \bar\x)= \overline{[0\oplus \z, Z\oplus \x]}= \z\rt Z\oplus\overline{ [\z,\x]}.
\end{equation}
Finally let us denote $\Fg_2/\Fh$ by $\Fl$.

\medskip

In order for the Chevalley-Eilenberg coboundary of $\Fg_1$ with coefficients in $V\ot \wdg^q\Fl^\bullet$ to be $\Fh$-linear, we assume that the action of  $\Fh$ on $\Fg_1$  is given by derivations. Then we introduce the bicomplex
\begin{equation}\label{g-1-g-2-bicomplex}
\xymatrix{\vdots&\vdots&\vdots&\\
(V\ot \wdg^2\Fg_1^\ast)^\Fh \ar[r]^{\hP}\ar[u]^{\vP}&(V\ot \wdg^2\Fg_1^\ast\ot \Fl^\ast)^\Fh\ar[r]^{\hP} \ar[u]^{\vP}& (V\ot \wdg^2\Fg_1^\ast\ot \wdg^2\Fl^\ast)^\Fh\ar[r]^{~~~~~~~~~~~\hP}\ar[u]^{\vP}&\cdots\\
(V\ot \Fg_1^\ast)^\Fh \ar[r]^{\hP}\ar[u]^{\vP}&(V\ot \Fg_1^\ast\ot \Fl^\ast)^\Fh\ar[r]^{\hP}\ar[u]^{\vP} & (V\ot \Fg_1^\ast\ot \wdg^2\Fl^\ast)^\Fh\ar[r]^{~~~~~~~~~~~\hP}\ar[u]^{\vP}&\cdots\\
V^\Fh \ar[r]^{\hP}\ar[u]^{\vP}&(V\ot \Fl^\ast)^\Fh\ar[r]^{\hP} \ar[u]^{\vP}& (V\ot \wdg^2\Fl^\ast)^\Fh\ar[r]^{~~~~~~~~~~\hP}\ar[u]^{\vP}&\cdots}
\end{equation}
The horizontal coboundary $\hP$ is the relative Lie algebra cohomoloy coboundary of  $\Fh\subseteq \Fg_2$ with coefficients in $V\ot\wdg^p\Fg_1^\ast$. The vertical coboundary $\vP$ is the Lie algebra cohomology coboundary of $\Fg_1$ with coefficients on $V\ot\wdg^q\Fl^\ast$. Since the action of $\Fh$ on $\Fg_1$ is given by derivations, the vertical coboundary is well-defined.

\medskip

We can identify the total complex
\begin{equation}
{\Tot}^\bullet(\Fg_1^\ast,(\Fg_2/\Fh)^\ast,V) = \underset{p+q=\bullet}{\bigoplus} \left(V\ot \wdg^p\Fg_1^\ast\ot \wdg^q \Fl^\ast\right)^\Fh
\end{equation}
of the bicomplex \eqref{g-1-g-2-bicomplex} with the relative Lie algebra cohomology complex $W(\Fa, \Fh, V)$ via the map
\begin{align}
\begin{split}
&\natural:  C^n(\Fg_1\bowtie\Fg_2, \Fh, V)\ra \bigoplus_{p+q=n} (V\ot \wdg^p\Fg^\ast\ot \wdg^q\Fl^\ast)^\Fh \\
& \natural(\omega)(Z_1,\ldots, Z_p\mid \z_1, \ldots, \z_q)=\omega(Z_1\oplus 0,\ldots, Z_p\oplus 0, 0\oplus\z_1, \ldots, 0\oplus \z_q),
\end{split}
\end{align}
and its inverse
\begin{align}
\begin{split}
&\natural^{-1}(\mu\ot\nu)(Z_1\oplus\z_1, \dots,Z_{p+q}\oplus\z_{p+q}) \\
&=\sum_{\s\in Sh(p,q)}(-1)^{\s}\mu(Z_{\s(1)}, \dots,Z_{\s(p)})\nu(\z_{\s(p+1)}, \dots, \z_{\s(p+q)}),
\end{split}
\end{align}
where $Sh(p,q)$ denotes the set of all $(p,q)$-shuffles, \cite{Maclane-book}.

\begin{lemma}\label{lemma-natural}
The map $\nr$ is an isomorphism of complexes.
\end{lemma}

\begin{proof}
The proof is analogue of \cite[Lemma 2.7]{MoscRang11} and follows from the observation that, in view of \eqref{adjoint}, the vertical and horizontal coboundaries of \eqref{g-1-g-2-bicomplex} are the restrictions of the Chevalley-Eilenberg coboundary of $\Fa$ with coefficients.
\end{proof}

\begin{definition}
Let $(\Fg_1,\Fg_2)$ be a matched pair of Lie algebras and a $\Fg_1$-Hopf algebra $\Fc$ be in Hopf duality with $U(\Fg_2)$. Then we say  $\Fc$ is $(\Fg_1, \Fg_2)$-related  if
 \begin{enumerate}
   \item  The pairing is $U(\Fg_1)$-balanced, \ie
\begin{equation}\label{g-balenced}
\langle u', u\rt f \rangle= \langle u' \lt u, f \rangle, \qquad f\in\Fc, \;u'\in U(\Fg_2),\; u\in U(\Fg_1).
\end{equation}
\item The action of $U(\Fg_2)$ on $U(\Fg_1)$ is compatible with the coaction of $\Fc$ on $U(\Fg_1)$ via the pairing, \ie
\begin{equation}\label{g-cobalenced}
 u\ns{0}\langle u', u\ns{1}\rangle= u'\rt u,\qquad u\in U(\Fg_1), \;\; u'\in U(\Fg_2).
\end{equation}
 \end{enumerate}
\end{definition}

\begin{lemma}
For a matched pair $(\Fg_1,\Fg_2)$ of Lie algebras and for $q\geq 1$,
\begin{align}\label{*-action}
\begin{split}
&(u'^1\odots u'^q)\ast u=\\
&u'^1\lt (u'^2\ps{1}\cdots u'^q\ps{1}\rt u\ps{1})\odots u'^{q-1}\ps{q-1}\lt (u'^q\ps{q-1}\rt u\ps{q-1})\ot u'^q\ps{q}\lt u\ps{q}
\end{split}
\end{align}
defines a right action of $U(\Fg_1)$ on $U(\Fg_2)^{\ot q}$.
\end{lemma}

\begin{proof}
We need to prove that for $\td u':= u'^1\odots u'^q\in U(\Fg_2)^{\ot q}$, and $u^1, u^2 \in U(\Fg_1)$, we have $(\td u'\ast u^1)\ast u^2= \td u' \ast(u^1u^2)$.

\medskip

Indeed, using the fact that $U(\Fg_2)$ is $U(\Fg_1)$-module coalgebra and \eqref{aux-mutual-pair-1} we observe that
\begin{align}\label{proof-*-action}
\begin{split}
&((u'^1\ot u'^2)\ast u^1)\ast u^2=  (u'^1 \lt(u'^2\ps{1}\rt u^1\ps{1})\ot u'^2\ps{2} \lt u^1\ps{2})\ast u^2\\
& = (u'^1 \lt(u'^2\ps{1}\rt u^1\ps{1})\lt ((u'^2\ps{2} \lt u^1\ps{2})\ps{1}\rt u^2\ps{1})\ot (u'^2\ps{2} \lt u^1\ps{2})\ps{2}\lt u^2\ps{2})\\
&= u'^1 \lt((u'^2\ps{1}\rt u^1\ps{1})((u'^2\ps{2} \lt u^1\ps{2})\rt u^2\ps{1}))\ot (u'^2\ps{3} \lt u^1\ps{3}u^2\ps{2})\\
&= u'^1\lt ( u'^2\ps{1}\rt (u^1\ps{1}u^2\ps{2}))\ot u'^2\ps{2}\lt (u^1\ps{2}u^2\ps{2})= (u'^1\ot u'^2)\ast u^1u^2.
\end{split}
\end{align}
\end{proof}

\begin{remark}{\rm
Furthermore, for $\td u'\in U(\Fg_2)^{\ot m}$ and $\td u''= u''^1\odots u''^l\in U(\Fg_2)^{\ot l}$, we have
\begin{equation}
(\td u'\ot \td u'')\ast u= \td u'\ast(u''^1\ps{1}\dots u''^l\ps{1}\rt u\ps{1})\ot (u''^1\ps{2}\odots u''^l\ps{2})\ast u\ps{2}.
\end{equation}
}\end{remark}

\begin{proposition}\label{Proposition-theta-linear}
Let $(\Fg_1, \Fg_2)$ be a matched pair of Lie algebras and $\Fc$ be a $(\Fg_1, \Fg_2)$-related Hopf algebra. Then the map $\t_{\Fc, U(\Fg_2)}$ defined in \eqref{map-theta-Alg} is $U(\Fg_1)$-linear provided $\Fg_1$ acts on $\Fc^{\ot q}$ by $\bullet$ defined in \eqref{bullet}, and on $U(\Fg_2)^{\ot q}$ by $\ast$ defined in \eqref{*-action}.
\end{proposition}

\begin{proof}
Without loss of generality we assume that $V=\Cb$. We use the pairing properties  \eqref{g-balenced} and \eqref{g-cobalenced} to observe that
\begin{align*}
&\t_{(\Fc,U(\Fg_2))}(u\bullet (f^1\odots f^q))(u'^1\odots u'^q)\\
&= \langle u'^1,u\ps{1}\ns{0}\rt f^1\rangle \langle u'^2,u\ps{1}\ns{1}(u\ps{2}\ns{0}\rt f^1)\rangle\cdots\\
&~~~~~~~~~~~~~~~~~~~~~~~~~~~~~\cdots \langle u'^q,u\ps{1}\ns{q-1}\cdots u\ps{q-1}\ns{1} (u\ps{q+1}\rt f^q)\rangle\\
&= \langle u'^1,  u\ps{1}\ns{0}\rt f^1\rangle \langle u'^2\ps{1},u\ps{1}\ns{1}\rangle\langle u'^2\ps{2},u\ps{2}\ns{0}\rt f^2\rangle\cdots\\
&\cdots \langle u'^q\ps{1}\,,\, u\ps{1}\ns{q-1}\rangle\langle u'^q\ps{2}\,,\, u\ps{2}\ns{q-2}\rangle\cdots \langle u'^q\ns{q-1}\,,\, u\ps{q-1}\ns{1}\rangle\langle u'^q\ps{q}\,,\, u\ps{q+1}\rt f^q\rangle\\
&= \langle u'^1,  u\ps{1}\ns{0}\rt f^1\rangle \langle u'^2\ps{1}\cdots u'^q\ps{1},u\ps{1}\ns{1}\rangle \langle u'^2\ps{2},u\ps{2}\ns{0}\rt f^2\rangle\\
&\langle u'^3\ps{2}\cdots u'^q\ps{2}\,,\, u\ps{2}\ns{1} \rangle\cdots\\
& \cdots\langle u'^{q-1}\ps{q-1},u\ps{q-1}\ns{0}\rt f^{q-1}\rangle\langle u'^q\ps{q-1}\,,\, u\ps{q-1}\ns{1} \rangle\langle u'^q\ps{q}, u\ps{q}\rt f^q\rangle
\end{align*}
\begin{align*}
&= \langle u'^1,  (  u'^2\ps{1}\cdots u'^q\ps{1}\rt  u\ps{1}  )\rt f^1\rangle \langle u'^2\ps{2},
( u'^3\ps{2}\cdots u'^q\ps{2}\rt u\ps{2})\rt f^2\rangle\cdots\\
& \cdots\langle u'^{q-1}\ps{q-1}, ( u'^q\ps{q-1}\rt u\ps{q-1})\rt f^{q-1}\rangle\langle u'^q\ps{q}, u\ps{q}\rt f^q\rangle\\
&= \langle u'^1\lt (  u'^2\ps{1}\cdots u'^q\ps{1} \rt u\ps{1})\,,\, f^1\rangle \langle u'^2\ps{2}\lt ( u'^3\ps{2}\cdots u'^q\ps{2}\rt u\ps{2})\,,\, f^2\rangle\cdots\\
& \cdots\langle u'^{q-1}\ps{q-1}\lt( u'^q\ps{q-1}\rt u\ps{q-1})\,,\, f^{q-1}\rangle\langle u'^q\ps{q}\lt u\ps{q}\,,\, f^q\rangle\\
&= \t_{(\Fc,U(\Fg_2))}(f^1\odots f^q)((u'^1\odots u'^q)\ast u),
\end{align*}
for any $u \in U(\Fg_1),\, u'^1,\ldots, u'^q \in U(\Fg_2)$ and $f^1,\ldots, f^q \in \Fc$.
\end{proof}

\begin{proposition}\label{proposition-antisymmetriza-linear}
Let  $(\Fg_1,\Fg_2)$ be a matched pair of Lie algebras. Assume that $\Fg_1$ acts on $U(\Fg_2)^{\ot q}$ by $\ast$ defined in \eqref{*-action}, and on $\wg^q\Fg_2$ by the intrinsic right action of $\Fg_1$ on $\Fg_2$. Then the antisymmetrization map \eqref{map-antisymmetrization} is a  $\Fg_1$-linear map  of complexes between the normalized Hochschild complex of $U(\Fg_2)$ and the Chevalley-Eilenberg complex of $\Fg_2$ with coefficients.
\end{proposition}

\begin{proof}
We already know that the map \eqref{map-antisymmetrization} is a  map of complexes. So we need to prove that it is $\Fg_1$-linear.

\medskip

Then for any $\xi^1, \ldots, \xi^q \in \Fg_2$, any $X\in \Fg_1$ and any normalized Hochschild cochain $\phi$ the claim follows from
\begin{align}
\begin{split}
&\phi((\xi^1\odots \z^q)\ast X)\\
&=\phi(\xi^1\lt (\xi^2\ps{1}\cdots \xi^q\ps{1}\rt X)\ot \xi^2\ps{2}\lt (\xi^3\ps{2}\cdots \xi^q\ps{2}\rt 1)\ot\cdots\\
&~~~\cdots\ot \xi^{q-1}\ps{q-1}\lt (\xi^q\ps{q-1}\rt 1)\ot \xi^q\ps{q}+\\
&~~~~~~~\xi^1\lt (\xi^2\ps{1}\cdots \xi^q\ps{1}\rt 1)\ot \xi^2\ps{2}\lt (\xi^3\ps{2}\cdots \xi^q\ps{2}\rt X)\ot\cdots\\
&~~~~~~~~~~~~~\cdots\ot \xi^{q-1}\ps{q-1}\lt (\xi^q\ps{q-1}\rt 1)\ot \xi^q\ps{q}+\cdots\\
&~~~~~~~~~~~~~~~~~~~\cdots + \xi^1\lt (\xi^2\ps{1}\cdots \xi^q\ps{1}\rt 1)\ot \xi^2\ps{2}\lt (\xi^3\ps{2}\cdots \xi^q\ps{2}\rt 1)\ot\cdots\\
&~~~~~~~~~~~~~~~~~~~~~~~~~~~~~~~~\cdots\ot \xi^{q-1}\ps{q-1}\lt (\xi^q\ps{q-1}\rt 1)\ot \xi^q\ps{q}\lt X)\\
&=\sum _{i=1}^q  \phi(\xi^1\ot \xi^2\odots \xi^i\lt (\xi^{i+1}\ps{1}\cdots \xi^q\ps{1}\rt X)\ot \xi^{i+1}\ps{2}\odots \xi^q\ps{2})\\
&=\underset{\xi^{i+1}\ps{1}=\cdots=\xi^{q}\ps{1}=1}{\sum _{i=1}^q  \phi(\xi^1\ot \xi^2\odots} \xi^i\lt (\xi^{i+1}\ps{1}\cdots \xi^q\ps{1}\rt X)\ot \xi^{i+1}\ps{2}\odots \xi^q\ps{2})\\
&+\underset{\xi^{j}\ps{2}=1~~\text{for some }\; i+1\le j\le q}{\sum _{i=1}^q \phi( \xi^1\ot \xi^2\odots} \xi^i\lt (\xi^{i+1}\ps{1}\cdots \xi^q\ps{1}\rt X)\ot \xi^{i+1}\ps{2}\odots \xi^q\ps{2})\\
&=\sum_{i=1}^q \phi(\xi^1\ot \xi^2\odots \xi^i\lt  X\ot \xi^{i+1}\odots \xi^q )+0.
\end{split}
\end{align}
\end{proof}

Now let $(\Fg_1, \Fg_2)$ be a  matched pair of Lie algebras, $\Fc$  be a $(\Fg_1,\Fg_2)$-related Hopf algebra and let $\Fg_2= \Fh\ltimes \Fl$, where $\Fh$ is reductive and every $\Fh$-module is semisimple. We then have the map
\begin{align}\label{VE}
\begin{split}
& \Vc: V\ot \wdg^q\Fg_1^\ast\ot \Fc^{\ot q}\ra (V\ot\wdg^p \Fg_1\ot \wdg^q\Fl^\ast)^\Fh \\
&\Vc(v\ot\om\ot f^1\odots f^q)( X^1,\ldots,X^p\mid \xi^1,\ldots, \xi^q)\\
&=\om(X^1,\ldots X^p)\sum_{\s\in S_q}(-1)^\s \langle \xi^{\s(1)}\,,\, f^1\rangle \cdots
\langle \xi^{\s(q)}\,,\, f^q\rangle v.
\end{split}
\end{align}
between the bicomplexes \eqref{UF+*} and  \eqref{g-1-g-2-bicomplex}.

\begin{theorem}\label{Theorem-main}
Let $(\Fg_1,\Fg_2)$ be a matched pair of Lie algebras, $\Fc$ be a $(\Fg_1,\Fg_2)$-related Hopf algebra, $\Fg_2=\Fh\ltimes \Fl$ be a $\Fc$-Levi decomposition such that $\Fh$ is $\Fg_1$-invariant and the natural action of $\Fh$ on $\Fg_1$ is given by derivations. Then for any $\Fc$-comodule and $\Fg_1$-module $V$, the map \eqref{VE} is a map of bicomplexes and induces an isomorphism between Hopf-cyclic cohomology of $\Fc\acl U(\Fg_1)$ with coefficients in $^\s{V}_\d$ and the Lie algebra cohomology of $\Fg_1\bowtie \Fg_2$ relative to $\Fh$ with coefficients in the $\Fg_1\bowtie \Fg_2$-module induced by $V$. In other words,
\begin{equation}
HP^\bullet(\Fc\acl U(\Fg_1),\, ^\s{V}_\d)\cong \bigoplus_{i=\,\bullet\;{\rm mod 2}} H^i(\Fg_1\bowtie \Fg_2, \Fh,\;V).
\end{equation}
\end{theorem}

\begin{proof}
First we shall prove that $\Vc$ commutes with both horizontal and vertical coboundaries of the bicomplexes \eqref{UF+*} and  \eqref{g-1-g-2-bicomplex}.

\medskip

The commutation of $\Vc$ with horizontal coboundaries is  guaranteed by the fact that $\t_{\Fc,U(\Fg_2),\Fl, \mu}$ is a map of complexes for $V\ot \wdg^p\Fg_1^\ast$.

\medskip

By the definition of the Chevalley-Eilenberg coboundary, to prove that $\Vc$ commutes with the vertical coboundaries amounts  to showing that  $\Vc$ is $\Fg_1$-linear, as $\Vc$ does  not have any effect on  $V\ot\wdg^p\Fg_1^\ast$. Since $\Vc$ is the composition of the antisymmetrization map $\a$ and $\t_{(\Fc, U(\Fg_2))}$, this in turn is guaranteed by Proposition \ref{Proposition-theta-linear} and Proposition \ref{proposition-antisymmetriza-linear}.

\medskip

Since $\Fg_2=\Fh\ltimes \Fl$ is a  $\Fc$-Levi decomposition, $\Vc$ induces an isomorphism  in the level of the $E_1$-terms of the spectral sequences of the bicomplexes \eqref{UF+*} and  \eqref{g-1-g-2-bicomplex}. Hence we conclude that $\Vc$ induces an  isomorphism in the level of  cohomology of  total complexes. This observation finishes the proof, since  the total complex of \eqref{UF+*} computes the Hopf-cyclic cohomology of $\Fc \acl U(\Fg_1)$ by Proposition \ref{mixCE*} and the total complex of \eqref{g-1-g-2-bicomplex} computes the relative Lie algebra cohomology.
\end{proof}

\begin{corollary}
Let $(G_1,G_2)$ be a matched pair of Lie groups, $L$ be a nucleus of $G_2$ and let $\Fh$, $ \Fg_1$, $\Fg_2$ and $\Fl$ denote the Lie algebras of $H:=G/L$, $G_1$, $G_2$ and $L$ respectively. Let us also assume that $\Fh$ is $\Fg_1$-invariant and the natural action of $\Fh$ on $\Fg_1$ is given by derivations. Then for any representative $G_1\bowtie G_2$-module $V$, we have
\begin{equation}
HP^\bullet(R(G_2)\acl U(\Fg_1),\; ^\s{V}_\d)\;\;\cong \bigoplus_{i=\bullet \, {\rm mod\,2}} H^i(\Fg_1\bowtie\Fg_2,\Fh,V).
\end{equation}
\end{corollary}

\begin{proof}
We prove that the criteria of Theorem \ref{Theorem-main} are satisfied for $\Fc:=R(G_2)$. By Theorem \ref{theorem-R(G_2)-Lie-Hopf}, $R(G_2)$ is a $\Fg_1$-Hopf algebra.

\medskip

The Hopf pairing between $R(G_2)$ and $U(\Fg_2)$ is given by evaluation, $\langle f\,,\, \xi\rangle :=\dt f(\exp(t\xi))$ and hence it is naturally $U(\Fg_1)$-balanced and compatible with the $R(G_2)$-coaction on $U(\Fg_1)$. Thus, $R(G_2)$ is $(\Fg_1,\Fg_2)$-related.

\medskip

Finally, it follows from \cite[Theorem 10.2]{HochMost57} that the map $\t_{R(G_2),U(\Fg_2),\Fl,0}:H^\bullet_{\rm rep}(G_2,V) \to H^\bullet(\Fg_2,\Fh,V)$ is an isomorphism. Hence, $\Fg_2=\Fh\ltimes\Fl$ is a $R(G_2)$-Levi decomposition.
\end{proof}

\begin{corollary}\label{corollary-matched-pair-Lie-algebra-cohomology}
Let $(\Fg_1,\Fg_2)$ be a matched pair of Lie algebras and  $\Fg_2=\Fh\ltimes \Fl$ be a Levi decomposition such that $\Fh$ is $\Fg_1$-invariant and the action of $\Fh$ on $\Fg_1$ is given by derivations. Then for any finite dimensional $\Fg_1\bowtie \Fg_2$-module $V$ we have
\begin{equation}
HP^\bullet(R(\Fg_2)\acl U(\Fg_1),\; ^\s{V}_\d)\;\;\cong \bigoplus_{i=\bullet \,{\rm mod \,2}} H^i(\Fg_1\bowtie\Fg_2,\Fh,V).
\end{equation}
\end{corollary}

\begin{proof}
We prove that the criteria of Theorem \ref{Theorem-main} is satisfied for $\Fc:=R(\Fg_2)$, which is a $\Fg_1$-Hopf algebra by Proposition \ref{Proposition-matched-Lie-Hopf-Lie}.

\medskip

We know that $\Fc$ and $U(\Fg_2)$ are in a Hopf pairing via evaluation. This pairing is by definition $U(\Fg_1)$-balanced and it is evident that the coaction of $R(\Fg_2)$ on $U(\Fg_1)$ is compatible with the pairing. So $R(\Fg_2)$ is $(\Fg_1,\Fg_2)$-related.

\medskip

Letting $G_2$ to be the simply connected Lie group of $\Fg_2$, the Levi decomposition $\Fg_2=\Fh\ltimes \Fl$ induces a nucleus decomposition $G_2 = H \ltimes L$. Since $G_2$ is simply connected, the representations of $\Fg_2$ and $G_2$ coincide and any injective resolution of $\Fg_2$ is induced by a representatively injective resolution of $G_2$, \ie the obvious map $H^\bullet_{\rm rep}(G_2,V) \to H^\bullet(\Fg_2,V)$ is surjective.

\medskip

Finally, since the isomorphism $\t_{R(G_2),U(\Fg_2),\Fl,0}:H^\bullet_{\rm rep}(G_2,V) \to H^\bullet(\Fg_2,\Fh,V)$ factors through $\t_{R(\Fg_2),U(\Fg_2),\Fl,0}:H^\bullet_{\rm rep}(\Fg_2,V) \to H^\bullet(\Fg_2,\Fh,V)$, the latter is an isomorphism. Therefore, $\Fg_2=\Fh\ltimes \Fl$ is a $R(\Fg_2)$-Levi decomposition.
\end{proof}

\begin{corollary}
Let $(G_1,G_2)$ be a matched pair of affine algebraic groups, $G_2= G_2^{\rm red}\rtimes G_2^{\rm u}$ be a Levi decomposition of $G_2$ and let $\Fg_1$, $\Fg_2$, $\Fg_2^{\rm red}$ and $\Fg_2^{\rm u}$ be the Lie algebras of $G_1$, $G_2$, $G_2^{\rm red}$ and $G_2^{\rm u}$ respectively. Let us also assume that $\Fg_2^{\rm red}$ is $\Fg_1$ invariant and the natural action of $\Fg_2^{\rm red}$ on $\Fg_1$ is given by derivations. Then for any finite dimensional polynomial representation $V$ of $G_1\bowtie G_2$, we have
\begin{equation}
HP^\bullet(\mathscr{P}(G_2)\acl U(\Fg_1),\; ^\s{V}_\d)\;\;\cong \bigoplus_{i=\bullet\,{\rm mod\,2}} H^i(\Fg_1\bowtie\Fg_2,\Fg_2^{\rm red},V).
\end{equation}
\end{corollary}

\begin{proof}
Firstly, $\mathscr{P}(G_2)$ is a $\Fg_1$-Hopf algebra by Theorem \ref{theorem-affine-algebraic-Lie-Hopf}.

\medskip

Secondly, the Hopf pairing $\langle f\,,\, u'\rangle =(u' \cdot f)(e)$ between $\mathscr{P}(G_2)$ and $U(\Fg_2)$ is obviously $U(\Fg_1)$-balanced and compatible with the $\mathscr{P}(G_2)$-coaction on $U(\Fg_1)$ by the very definition of this coaction. Hence $\mathscr{P}(G_2)$ is  $(\Fg_1,\Fg_2)$-related.

\medskip

Finally, the map $\t_{\mathscr{P}(G_2),U(\Fg_2),\Fg_2^{\rm u},0}:H^\bullet_{\rm pol}(G_2,V) \to H^\bullet(\Fg_2,\Fg_2^{\rm red},V)$ is an isomorphism, \cite[Theorem 2.2]{KumaNeeb06}. Hence $\Fg_2=\Fg_2^{\rm red}\ltimes \Fg_2^{\rm u}$ is a $\mathscr{P}(G_2)$-Levi decomposition.
\end{proof}

\section{Hopf-cyclic cohomology with coefficients}

In this section our aim is to compute the Hopf-cyclic cohomology of the bicrossed product Hopf algebra $R(\Fg_2)\acl U(\Fg_1)$ with coefficients in an SAYD module associated to both a representation and a corepresentation of the Lie algebra $\Fg_1 \bowtie \Fg_2$.

\medskip

We first compute the periodic Hopf-cyclic cohomology of the universal enveloping algebra $U(\Fg)$ with coefficients in a SAYD module over the algebra $\Fg$. This is a generalization of \cite[Proposition 7]{ConnMosc98}. In the second step, using the Poincar\'e dual complex we compute the periodic Hopf-cyclic cohomology of the Hopf algebra $R(\Fg)$ of the representative functions with coefficients in a unimodular SAYD module over the Lie algebra $\Fg$. This result generalizes Theorem \ref{theorem-cohomology-R(g)}. Finally we generalize Theorem \ref{Theorem-main} and compute the periodic Hopf-cyclic cohomology of $R(\Fg_2)\acl U(\Fg_1)$ with coefficients in a unimodular SAYD module over the Lie algebra $\Fg_1 \bowtie \Fg_2$.

\medskip

In this general case, our van Est isomorphism is no longer on the level of complexes. Instead, it is on the level of $E_1$-terms of the spectral sequences associated to certain filtrations of the corresponding bicomplexes.

\subsection{Cyclic cohomology of Lie algebras}

In this subsection we show that for $V$, a SAYD module over a Lie algebra $\Fg$, the periodic cyclic cohomology of $\Fg$ with coefficients in $V$ and the periodic cyclic cohomology of the universal enveloping algebra $U(\Fg)$ with coefficients in the corresponding SAYD over $U(\Fg)$ are isomorphic.

\medskip

If $V$ is a SAYD module over a Lie algebra $\Fg$, then we have the differential complex $(C(\Fg,V),\p_{\rm CE}+\p_{\rm K})$ where
\begin{equation}\label{aux-complex-cyclic-Lie-algebra}
C(\Fg,V)= \bigoplus_{n \geq 0}C_n(\Fg,V), \quad C_n(\Fg,V):= \wedge^n \Fg\ot V
\end{equation}
together with the Chevalley-Eilenberg boundary
\begin{equation}
\p_{\rm CE}:C_{n+1}(\Fg,V) \to C_n(\Fg,V),
\end{equation}
and the Koszul coboundary
\begin{equation}
\partial_{\rm K}:C_n(\Fg,V) \to C_{n+1}(\Fg,V), \quad Y_1 \wdots Y_n \ot v \mapsto  v\nsb{-1} \wg Y_1 \wdots Y_n \ot v\nsb{0}.
\end{equation}

\begin{proposition}\label{7}
The space $(C_{\bullet}(\Fg,V), \p_{\rm CE}+\p_{\rm K})$ is a differential complex if and only if $V$ is stable  right $\widetilde{\Fg}$-module.
\end{proposition}

\begin{proof}
On one hand we have
\begin{align}
\begin{split}
& \p_{\rm CE}(\p_{\rm K}(Y_0 \wdots Y_n \ot v)) = \sum_i \p_{\rm CE}(X_i \wg Y_0 \wdots Y_n \ot v \lt \theta^i)  \\
& =\sum_i Y_0 \wdots Y_n \ot (v \lt \theta^i) \cdot X_i  \\
& +\sum_{i,j} (-1)^{j+1} X_i \wg Y_0 \wdots \widehat{Y}_j \wdots Y_n \ot (v \lt \theta^i) \cdot Y_j  \\
& +\sum_{i,j} (-1)^{j+1} [X_i,Y_j] \wg Y_0 \wdots \widehat{Y}_j \wdots Y_n \ot v \lt \theta^i  \\
& +\sum_{i,j} (-1)^{j+k} [Y_j,Y_k]\wg X_i \wg Y_0 \wdots \widehat{Y}_j \wdots \widehat{Y}_k \wdots Y_n \ot v \lt \theta^i,
\end{split}
\end{align}
and on the other hand
\begin{align}
\begin{split}
& \p_{\rm K}(\p_{\rm CE}(Y_0 \wdots Y_n \ot v)) = \sum_{j} (-1)^j \p_{\rm K}(Y_0 \wdots \widehat{Y}_j \wdots Y_n \ot v \cdot Y_j)  \\
& +\sum_{j,k} (-1)^{j+k} \p_{\rm K}([Y_j,Y_k] \wg Y_0 \wdots \widehat{Y}_j \wdots \widehat{Y}_k \wdots Y_n \ot v)  \\
& =\sum_{i,j} (-1)^j X_i \wg Y_0 \wdots \widehat{Y}_j \wdots Y_n \ot (v \wg Y_j) \lt \theta^i  \\
& +\sum_{i,j,k} (-1)^{j+k+1} [Y_j,Y_k] \wg X_i \wg Y_0 \wdots \widehat{Y}_j \wdots \widehat{Y}_k \wdots Y_n \ot v \lt \theta^i.
\end{split}
\end{align}
Therefore, $C(\Fg,V)$ is a differential complex if and only if
\begin{align}\label{8}
\begin{split}
& (\p_{\rm CE} \circ \p_{\rm K} + \p_{\rm K} \circ \p_{\rm CE})(Y_0 \wdots Y_n \ot v)  \\
& =\sum_i Y_0 \wdots Y_n \ot (v \lt \theta^i) \cdot X_i  \\
& +\sum_{i,j} (-1)^{j+1} X_i \wg Y_0 \wdots \widehat{Y}_j \wdots Y_n \ot [(v \lt \theta^i) \cdot Y_j - (v \cdot Y_j) \lt \theta^i]  \\
& +\sum_{i,j} (-1)^{j+1} [X_i,Y_j] \wg Y_0 \wdots \widehat{Y}_j \wdots Y_n \ot v \lt \theta^i = 0.
\end{split}
\end{align}
Now, if we assume that  $(C^{\bullet}(\Fg,V),\p_{\rm CE}+\p_{\rm K} )$ is a differential complex, then we obtain the stability condition \eqref{aux-unimodular-stable} evaluating   \eqref{8} on $1 \ot v$. Similarly we observe that $V$ is a $\widetilde{\Fg}$-module by evaluating \eqref{8} on $Y \ot v$.

\medskip

The reverse argument is straightforward.
\end{proof}

\begin{definition}
Let $\Fg$ be a Lie algebra and $V$ be a right-left SAYD module over $\Fg$. We call the cohomology of the  complex $(C(\Fg,V), \p_{\rm CE} + \p_{\rm K})$ the  cyclic   cohomology of the Lie algebra $\Fg$ with coefficients in the SAYD module $V$, and denote it by  $HC^{\bullet}(\Fg,V)$. Similarly we denote its periodic cyclic cohomology by $HP^{\bullet}(\Fg,V)$.
\end{definition}

The following analogue of \cite[Proposition 7]{ConnMosc98} is the main result in this subsection.

\begin{theorem}\label{g-U(g) spectral sequence}
Let $\Fg$ be a Lie algebra and $V$ be a SAYD module over the Lie algebra $\Fg$ with locally conilpotent coaction. Then the periodic cyclic cohomology of $\Fg$ with coefficients in  $V$ is the same as the periodic cyclic cohomology of $U(\Fg)$ with coefficients in the corresponding SAYD module $V$ over $U(\Fg)$. In short,
\begin{equation}
HP^\bullet(\Fg, V) \cong HP^\bullet(U(\Fg), V).
\end{equation}
\end{theorem}

\begin{proof}
The total coboundary of $C(\Fg, V)$ is $\p_{\rm CE} + \p_{\rm K}$ while the total coboundary of the complex $C(U(\Fg), V)$ computing the  cyclic cohomology of $U(\Fg)$ is $B + b$.

\medskip

The filtration on $V$ via \cite[Lemma 6.2]{JaraStef} induces a filtration on each of these bicomplexes. Hence, there are spectral sequences associated to these filtrations and we compare the $E_1$-terms of these spectral sequences. To this end, we first show that the coboundaries respect this filtration.

\medskip

By \cite[Lemma 6.2]{JaraStef}, each $F_pV$ is a submodule of $V$. Thus, the Lie algebra homology boundary $\p_{\rm CE}$ respects the filtration. As for $\p_{\rm K}$ we notice that
\begin{equation}
\p_{\rm K}(X_1 \wdots X_n \ot v) = v\nsb{-1} \wg X_1 \wdots X_n \ot v\nsb{0}
\end{equation}
for any $v \in F_pV$. Since
\begin{equation}
\Db(v) = v\snsb{-1} \ot v\snsb{0} = 1 \ot v + v\nsb{-1} \ot v\nsb{0} + \sum_{k \geq 2}\theta_k^{-1}(v\nsb{-k} \cdots v\nsb{-1}) \ot v\nsb{0}
\end{equation}
we observe  that $v\nsb{-1} \wg X_1 \wdots X_n \ot v\nsb{0} \in \wedge^{n+1}\Fg \ot F_{p-1}V$.  Because  $F_{p-1}V \subseteq F_pV$, we conclude that  $\p_{\rm K}$ respects the filtration.

\medskip

Since the Hochschild coboundary $b:C^n(U(\Fg),V)\ra C^{n+1}(U(\Fg),V)$ is the alternating sum of the coface operators, it suffices to check that each coface operator preserves the filtration. This  is evident for all cofaces except the last one. In order to deal with the last coface operator we take  $v \in F_pV$ and  write
\begin{equation}
v\snsb{-1} \ot v\snsb{0} = 1 \ot v + v\ns{-1} \ot v\ns{0}, \qquad v\ns{-1} \ot v\ns{0} \in \Fg \ot F_{p-1}V.
\end{equation}
We then have
\begin{equation}
\p_{n+1}(v \ot u^1 \ot \cdots \ot u^n) = v\snsb{0} \ot u^1 \ot \cdots \ot u^n \ot v\snsb{-1} \in F_pV \ot U(\Fg)^{\ot\, n+1}.
\end{equation}
Hence, we can say that $b$ respects the filtration.

\medskip

For the cyclic operator, the result follows once again from the fact that $F_p$ is a $\Fg$-module. Indeed, for $v \in F_pV$ we have
\begin{align}
\begin{split}
& \tau(v \ot u^1 \ot \cdots \ot u^n) = v\snsb{0} \cdot u^1\ps{1} \ot S(u^1\ps{2}) \cdot (u^2 \ot \cdots \ot u^n \ot v\snsb{-1}) \\
& \hspace{4cm} \in F_pV \ot U(\Fg)^{\ot\, n}.
\end{split}
\end{align}

Finally the extra degeneracy operator
\begin{equation}
\sigma_{-1}(v \ot u^1 \ot \cdots \ot u^n) = v \cdot u^1\ps{1} \ot S(u^1\ps{2}) \cdot (u^2 \ot \cdots \ot u^n) \in F_pV \ot U(\Fg)^{\ot\, n}
\end{equation}
preserves the filtration as $F_p$ is a $\Fg$-module and the coaction preserves the filtration. As a result, we can say that the Connes' boundary $B:C^n(U(\Fg),V) \to C^{n-1}(U(\Fg),V)$ respects the filtration.

\medskip

Now, the $E_1$-term of the spectral sequence associated to the filtration $(F_pV)_{p \geq 0}$ computing the periodic cyclic cohomology of the Lie algebra $\Fg$ is known to be of the form
\begin{equation}
E_1^{j,\,i}(\Fg) = H^{i+j}(F_{j+1}C(\Fg, V)/F_jC(\Fg, V), [\p_{\rm CE} + \p_{\rm K}])
\end{equation}
where, $[\p_{\rm CE} + \p_{\rm K}]$ is the induced coboundary operator on the quotient complex. By the obvious identification
\begin{equation}
F_{j+1}C(\Fg,V)/F_jC(\Fg,V) \cong C(\Fg,F_{j+1}V / F_jV) = C(\Fg,(V / F_jV)^{\rm co\,\Fg}),
\end{equation}
we see that
\begin{equation}
E_1^{j,\,i}(\Fg) = H^{i+j}(C(\Fg,(V / F_jV)^{\rm co\,U(\Fg)}), [0]+[\p_{\rm CE}]),
\end{equation}
as $\p_{\rm K}(F_{j+1}C(\Fg,V)) \subseteq F_jC(\Fg,V)$.

\medskip

 Similarly,
\begin{equation}
E_1^{j,\,i}(U(\Fg)) = H^{i+j}(C(U(\Fg),(V / F_jV)^{\rm co\,U(\Fg)}), [b + B]).
\end{equation}
Finally, by \cite[Proposition 7]{ConnMosc98} a quasi-isomorphism $[\alpha]:E_1^{j,\,i}(\Fg) \to E_1^{j,\,i}(U(\Fg))$, $\forall i,j$ is induced by the anti-symmetrization map
\begin{equation}
\alpha:C(\Fg,(V / F_jV)^{\rm co\,\Fg}) \to C(U(\Fg),(V / F_jV)^{\rm co\,U(\Fg)}).
\end{equation}
\end{proof}

\begin{remark}{\rm
In case the $\Fg$-module $V$ has a trivial $\Fg$-comodule structure, the coboundary $\p_{\rm K} = 0$ and
\begin{equation}
HP^{\bullet}(\Fg, V) = \bigoplus_{n = \,\bullet \; {\rm mod\,2}} H_n(\Fg, V).
\end{equation}
In this case, the above Theorem becomes \cite[Proposition 7]{ConnMosc98}.
}\end{remark}

\subsection{Hopf-cyclic cohomology of $R(\Fg)$}

In this subsection, we compute the Hopf-cyclic cohomology of the Hopf algebra $R(\Fg)$ with coefficients in a SAYD module over $\Fg$ in terms of the relative Lie algebra cohomology.

\medskip

Let $V$ be a locally finite $\Fg$-module and locally conilpotent $\Fg$-comodule.  Then similar to \ref{aux-U(g)-dual-coaction}, $V$ has a left $R(\Fg)$-comodule structure $\Db: V \to R(\Fg) \ot V$ defined by
\begin{equation}
v \mapsto v^{\sns{-1}} \ot v^{\sns{0}}, \quad \text{if}\quad v^{\sns{-1}}(u)v^{\sns{0}} = v \cdot u
\end{equation}
for any $v \in V$ and any $u \in U(\Fg)$.

\medskip

On the other hand, the left $U(\Fg)$-comodule structure on $V$ induces a  right $R(\Fg)$-module structure
\begin{align}
\begin{split}
& V \ot R(\Fg) \to V \\
& v \ot f \mapsto v \cdot f := f(v\snsb{-1})v\snsb{0}
\end{split}
\end{align}
just as in \eqref{aux-17}.

\begin{proposition}\label{proposition-AYD-R(g)}
Let $V$ be locally finite as a $\Fg$-module and locally conilpotent as a  $\Fg$-comodule. If $V$ is an AYD module over $\Fg$, then $V$ is an AYD module over $R(\Fg)$.
\end{proposition}

\begin{proof}
Since $V$ is an AYD module over the Lie algebra $\Fg$ with a locally conilpotent $\Fg$-coaction, it is an AYD over $U(\Fg)$.

\medskip

Let $v \in V$, $f \in R(\Fg)$ and $u \in U(\Fg)$. On one side of the AYD condition we  have
\begin{equation}
\Db(v \cdot f)(u) = (v \cdot f) \cdot u = f(v\snsb{-1})v\snsb{0} \cdot u,
\end{equation}
and on the other side,
\begin{align}
\begin{split}
& (S(f\ps{3})v^{\sns{-1}}f\ps{1})(u)v^{\sns{0}} \cdot f\ps{2} = \\
& S(f\ps{3})(u\ps{1})v^{\sns{-1}}(u\ps{2})f\ps{1}(u\ps{3})f\ps{2}((v^{\sns{0}})\snsb{-1})(v^{\sns{0}})\snsb{0} = \\
& f\ps{2}(S(u\ps{1}))v^{\sns{-1}}(u\ps{2})f\ps{1}(u\ps{3}(v^{\sns{0}})\snsb{-1})(v^{\sns{0}})\snsb{0} = \\
& v^{\sns{-1}}(u\ps{2})f(u\ps{3}(v^{\sns{0}})\snsb{-1}S(u\ps{1}))(v^{\sns{0}})\snsb{0} = \\
& f(u\ps{3}(v^{\sns{-1}}(u\ps{2})v^{\sns{0}})\snsb{-1}S(u\ps{1}))(v^{\sns{-1}}(u\ps{2})v^{\sns{0}})\snsb{0} = \\
& f(u\ps{3}(v \cdot u\ps{2})\snsb{-1}S(u\ps{1}))(v \cdot u\ps{2})\snsb{0} = \\
& f(u\ps{3}S(u\ps{2}\ps{3})v\snsb{-1}u\ps{2}\ps{1}S(u\ps{1}))v\snsb{0} \cdot u\ps{2}\ps{2} = \\
& f(u\ps{5}S(u\ps{4})v\snsb{-1}u\ps{2}S(u\ps{1}))v\snsb{0} \cdot u\ps{3} = \\
& f(v\snsb{-1})v\snsb{0} \cdot u,
\end{split}
\end{align}
where we used the AYD condition on $U(\Fg)$ in the sixth equality. This proves that $V$ is an AYD module over $R(\Fg)$.
\end{proof}

If $V$ is a unimodular SAYD module over a Lie algebra $\Fg$, then we have a differential complex $(W(\Fg,V),d_{\rm CE}+d_{\rm K})$, called the perturbed Koszul complex \cite{KumaVerg}. Here
\begin{equation}
W(\Fg,V) = \bigoplus_{n \geq 0}W^n(\Fg,V), \quad W^n(\Fg,V) := \wedge^n \Fg^* \ot V,
\end{equation}
together with the Chevalley-Eilenberg coboundary
\begin{equation}
d_{\rm CE}:W^n(\Fg,V) \to W^{n+1}(\Fg,V),
\end{equation}
and the Koszul boundary
\begin{equation}
d_{\rm K}:W^n(\Fg,V) \to W^{n-1}(\Fg,V), \quad \alpha \ot v \mapsto \iota({v\nsb{-1}})(\alpha) \ot v\nsb{0}.
\end{equation}

This complex is in fact the Poincar\'e dual of the complex \eqref{aux-complex-cyclic-Lie-algebra}. In order to see this, we first note that the coefficients will be appropriate under the Poincar\'e isomorphism.

\begin{proposition}\label{12}
A vector space $V$ is a unimodular stable  right  $\widetilde\Fg$-module if and only if  $V \ot \mathbb{C}_{\delta}$ is a  stable right $\widetilde\Fg$-module.
\end{proposition}

\begin{proof}
Indeed, if $V$ is unimodular stable right  $\widetilde\Fg$-module, that is $\sum_i (v \lt X_i) \cdot \theta^i = 0$,  for any $v\in V$, then
\begin{align}
\begin{split}
& \sum_i ((v \ot 1_{\mathbb{C}}) \cdot \theta^i) \lt X^i = \sum_i (v \cdot \theta^i) \cdot X_i \ot 1_{\mathbb{C}} + v\delta(X_i) \ot 1_{\mathbb{C}}  \\
& =\sum_i (v \cdot X_i) \lt \theta^i \ot 1_{\mathbb{C}} = 0,
\end{split}
\end{align}
which proves that $V\ot \Cb_{\d}$ is stable. Similarly we observe that for $1 \leq i,j \leq N$,
\begin{align}
\begin{split}
& ((v \ot 1_{\mathbb{C}}) \cdot X_j) \lt \theta^i = (v \cdot X_j \ot 1_{\mathbb{C}} + v\delta(X_j) \ot 1_{\mathbb{C}}) \lt \theta^i = \\
& ((v \cdot X_j) \lt \theta^i + v\delta(X_j) \lt \theta^i) \ot 1_{\mathbb{C}} = \\
& (v \lt (X_j \rt \theta^i) + (v \lt \theta^i) \cdot X_j + v\delta(X_j) \lt \theta^i) \ot 1_{\mathbb{C}} = \\
& (v \ot 1_{\mathbb{C}}) \lt (X_j \rt \theta^i) + ((v \ot 1_{\mathbb{C}}) \lt \theta^i) \cdot X_j
\end{split}
\end{align}
\ie $V \ot \mathbb{C}_{\delta}$ is a right  $\widetilde\Fg$-module.

\medskip

The converse argument is similar.
\end{proof}

Let us now briefly recall the Poincar\'e isomorphism
\begin{equation}
 \FD_P:\wedge^k \Fg^* \to \wedge^{N-k}\Fg, \qquad  \eta \mapsto \iota(\eta)\varpi,
\end{equation}
where $\varpi = X_1 \wedge \cdots \wedge X_N$ is the covolume element of $\Fg$.  By definition
  $\iota(\theta^i): \wedge^{\bullet} \Fg \to \wedge^{\bullet-1}\Fg$ is given  by
\begin{equation}
\langle \iota(\theta^i)\xi, \theta^{j_1} \wedge \cdots \wedge \theta^{j_{r-1}} \rangle := \langle \xi, \theta^i \wedge \theta^{j_1} \wedge \cdots \wedge \theta^{j_{r-1}} \rangle, \quad  \xi \in \wedge^r\Fg.
\end{equation}
Finally, for $\eta = \theta^{i_1} \wedge \cdots \wedge \theta^{i_k}$, the interior multiplication $\iota(\eta): \wedge^\bullet\Fg \to \wedge^{\bullet-k}\Fg $ is defined by
\begin{equation}
\iota(\eta) := \iota(\theta^{i_k}) \circ \cdots \circ \iota(\theta^{i_1}).
\end{equation}

\begin{proposition}
Let $V$ be a stable right  $\widetilde\Fg$-module. Then the Poincar\'e isomorphism induces a map of complexes between the complex $W(\Fg,V \ot \mathbb{C}_{-\delta})$ and the complex $C(\Fg,V)$.
\end{proposition}

\begin{proof}
Let us first introduce the notation $\widetilde{V}:=V \ot \mathbb{C}_{-\delta}$. We can identify $\widetilde{V}$ with $V$ as a vector space, but with the right $\Fg$-module structure deformed as in $v \lt X := v \cdot X - v\delta(X)$.

\medskip

We prove the commutativity of the (co)boundaries via the (inverse) Poincar\'e isomorphism, \ie
\begin{align}
\begin{split}
& \FD_P^{-1}: \wedge^p \Fg \ot V \to \wedge^{N-p}\Fg^* \ot \widetilde{V} \\
& \xi \ot v \mapsto \FD_P^{-1}(\xi \ot v),
\end{split}
\end{align}
where for an arbitrary $\eta \in \wg^{N-p}\Fg$
\begin{equation}
\langle \eta, \FD_P^{-1}(\xi \ot v) \rangle := \langle \eta\xi, \omega^* \rangle v.
\end{equation}
Here, $\omega^* \in \wg^N\Fg^*$ is the volume form.

\medskip

The commutativity of the diagram
\begin{equation}
\xymatrix {
 \ar[d]_{\FD_P^{-1}} \wedge^p\Fg \ot V  \ar[r]^{\p_{\rm CE}} &  \wedge^{p-1}\Fg \ot V   \ar[d]^{\FD_P^{-1}} \\
 \wedge^{N-p}\Fg^* \ot \widetilde{V}   \ar[r]_{d_{\rm CE}} &  \wedge^{N-p+1}\Fg^* \ot \widetilde{V}
}
\end{equation}
follows from the Poincar\'e duality in Lie algebra homology - cohomology,
 \cite[Section VI.3]{Knap}. For the commutativity of the diagram
\begin{equation}
\xymatrix {
 \ar[d]_{\FD_P^{-1}} \wedge^p\Fg \ot V  \ar[r]^{\p_{\rm K}} &  \wedge^{p+1}\Fg \ot V   \ar[d]^{\FD_P^{-1}} \\
 \wedge^{N-p}\Fg^* \ot \widetilde{V}   \ar[r]_{d_{\rm K}} &  \wedge^{N-p-1}\Fg^* \ot \widetilde{V}
}
\end{equation}
we take an arbitrary $\xi \in \wg^p\Fg$, $\eta \in \wg^{N-p-1}\Fg$ and $v \in V$. Then
\begin{align}
\begin{split}
& \FD_P^{-1}(\p_{\rm K}(\xi \ot v))(\eta) = \langle \eta X_i \xi, \omega^* \rangle v \lt \theta^i = \\
& (-1)^{N-p-1} \langle X_i \eta \xi, \omega^* \rangle v \lt \theta^i = (-1)^{N-p-1}d_{\rm K}(\FD_P^{-1}(\xi \ot v))(\eta).
\end{split}
\end{align}
\end{proof}

As an example, let us compute $\widetilde{HP}(s\ell(2),V)$ for $V = S(s\ell(2)^\ast)\nsb{1}$ by demonstrating  explicit representatives of the cohomology classes.

\medskip

Being an SAYD module over $U(s\ell(2))$, the space $V$ admits the filtration $(F_pV)_{p \in \Zb}$ from \cite{JaraStef}. Explicitly,
\begin{equation}\label{aux-filtration-on-V}
F_0V = \{R^X,R^Y,R^Z\}, \quad F_pV = \left\{\begin{array}{cc}
                                              V & p \geq 1 \\
                                              0 & p<0.
                                            \end{array}
\right.
\end{equation}
The induced filtration on the complex is
\begin{equation}
F_j(W(s\ell(2),V)) := W(s\ell(2),F_jV),
\end{equation}
and the $E_1$-term of the associated spectral sequence is
\begin{align}
\begin{split}
& E_1^{j,i}(s\ell(2),V) = H^{i+j}(W(s\ell(2),F_jV)/W(s\ell(2),F_{j-1}V)) \\
& \hspace{2.2cm} \cong H^{i+j}(W(s\ell(2),F_jV/F_{j-1}V)).
\end{split}
\end{align}

Since $F_jV/F_{j-1}V$ has trivial $s\ell(2)$-coaction, the boundary $d_{\rm K}$ vanishes on the quotient complex $W(s\ell(2),F_jV/F_{j-1}V)$ and hence
\begin{equation}
E_1^{j,i}(s\ell(2),V) = \bigoplus_{i+j\, \cong\, \bullet \; {\rm mod\,2}}H^{\bullet}(s\ell(2),F_jV/F_{j-1}V).
\end{equation}
In particular,
\begin{equation}
E_1^{0,i}(s\ell(2),V) = H^i(W(s\ell(2),F_0V)) \cong \bigoplus_{i\, \cong \,\bullet \; {\rm mod\,2}}H^{\bullet}(s\ell(2,\Cb),F_0V) = 0.
\end{equation}
The last equality follows from Whitehead's theorem, as $F_0V$ is an irreducible $s\ell(2)$-module of dimension greater than 1. For $j = 1$ we have $V/F_0V \cong \Cb$ and hence
\begin{equation}
E_1^{1,i}(s\ell(2),V) = \bigoplus_{i+1\, \cong \,\bullet \; {\rm mod\,2}}H^{\bullet}(s\ell(2)),
\end{equation}
which gives two cohomology classes as a result of Whitehead's 1st and 2nd lemmas, \cite{Knap}. Finally, by $F_pV = V$ for $p \ge 1$, we have $E_1^{j,i}(s\ell(2),V) = 0$ for $j \geq 2$.

\medskip

Let us now decompose the complex as
\begin{equation}
W(s\ell(2),V) = V^{\rm even}(s\ell(2),V) \oplus W^{\rm odd}(s\ell(2),V),
\end{equation}
where
\begin{align}
\begin{split}
& W^{\rm even}(s\ell(2),V) = V \oplus (\wedge^2 {s\ell(2)}^* \ot V), \\
& W^{\rm odd}(s\ell(2),V) = ({s\ell(2)}^* \ot V) \oplus (\wedge^3 {s\ell(2)}^* \ot V).
\end{split}
\end{align}

Let us take $1_V \in W^{\rm even}(s\ell(2),V)$. We can immediately observe that $d_{\rm CE}(1_V) = 0$ as well as $d_{\rm K}(1_V) = 0$. On the other hand, on the level of spectral sequence the class $1_V \in W^{\rm even}(s\ell(2),V)$ descends to the nontrivial class in $H^0(s\ell(2))$. Hence, it is a representative of the  even cohomology class.

\medskip

Secondly, we consider
\begin{equation}\label{aux-odd-class-sl(2)-V}
(2\theta^X \ot R^Z - \theta^Y \ot R^Y \; , \; \theta^X \wg \theta^Y \wg \theta^Z \ot 1_V) \in W^{\rm odd}(s\ell(2,\Cb),V),
\end{equation}
where $\{\theta^X,\theta^Y,\theta^Z\}$ is the dual basis corresponding to the basis $\{X,Y,Z\}$ of $s\ell(2)$. Let us show that \eqref{aux-odd-class-sl(2)-V} is a $d_{\rm CE} + d_{\rm K}$-cocycle. It is immediate that
\begin{equation}
d_{\rm CE}(\theta^X \wg \theta^Y \wg \theta^Z \ot 1_V) = 0.
\end{equation}
As for the Koszul differential,
\begin{align}
\begin{split}
& d_{\rm K}(\theta^X \wg \theta^Y \wg \theta^Z \ot 1_V) = \\
& \iota_X(\theta^X \wg \theta^Y \wg \theta^Z) \ot R^X + \iota_Y(\theta^X \wg \theta^Y \wg \theta^Z) \ot R^Y + \iota_Z(\theta^X \wg \theta^Y \wg \theta^Z) \ot R^Z = \\
& \theta^Y \wg \theta^Z \ot R^X - \theta^X \wg \theta^Z \ot R^Y + \theta^X \wg \theta^Y \ot R^Z.
\end{split}
\end{align}

 On the other hand, we have
\begin{align}
\begin{split}
& d_{\rm CE}(\theta^X \ot R^Z) = \theta^X \wg \theta^Y \ot R^Z - \theta^Y \wg \theta^X \ot R^Z \cdot Y - \theta^Z \wg \theta^X \ot R^Z \cdot Z \\
& = \theta^X \wg \theta^Z \ot R^Y,
\end{split}
\end{align}
and
\begin{align}
\begin{split}
& d_{\rm CE}(\theta^Y \ot R^Y) = \theta^X \wg \theta^Z \ot R^Y - \theta^X \wg \theta^Y \ot R^Y \cdot X - \theta^Z \wg \theta^Y \ot R^Z \cdot Z \\
& = \theta^X \wg \theta^Y \ot R^Z + \theta^X \wg \theta^Z \ot R^Y + \theta^Y \wg \theta^Z \ot R^X.
\end{split}
\end{align}
Therefore,
\begin{equation}
d_{\rm CE}(2\theta^X \ot R^Z - \theta^Y \ot R^Y) = - \theta^X \wg \theta^Y \ot R^Z + \theta^X \wg \theta^Z \ot R^Y - \theta^Y \wg \theta^Z \ot R^X.
\end{equation}
We also have
\begin{equation}
d_{\rm K}(2\theta^X \ot R^Z - \theta^Y \ot R^Y) = 2R^ZR^X - R^YR^Y = 0 - 0 = 0.
\end{equation}
Therefore,
\begin{equation}
(d_{\rm CE} + d_{\rm K})((2\theta^X \ot R^Z - \theta^Y \ot R^Y\;,\;\theta^X \wg \theta^Y \wg \theta^Z \ot 1_V)) = 0.
\end{equation}
Finally we note that $(2\theta^X \ot R^Z - \theta^Y \ot R^Y \; , \; \theta^X \wg \theta^Y \wg \theta^Z \ot 1_M)$ descends, in the $E_1$-level of the spectral sequence,  to the cohomology class represented by the nontrivial 3-cocycle $$\theta^X \wg \theta^Y \wg \theta^Z.$$ Hence, it represents the odd cohomology class.

\medskip

Let us record our discussion in the following proposition.

\begin{proposition}\label{proposition-periodic-cyclic-sl(2)-V}
The periodic cyclic cohomology of the Lie algebra $s\ell(2)$ with coefficients in SAYD module  $V:=S({s\ell(2)}^\ast)\nsb{1}$  is represented by
\begin{align}
&\widetilde{HP}^{\rm even}(s\ell(2),V)=\Cb\Big\langle  1_V\Big\rangle, \\
&\widetilde{HP}^{\rm odd}(s\ell(2),V)=\Cb\Big\langle (2\theta^X \ot R^Z - \theta^Y \ot R^Y\;,\;\theta^X\wedge \theta^Y\wedge \theta^Z \ot 1_V) \Big\rangle.
\end{align}
\end{proposition}

We next introduce  the relative perturbed Koszul complex as
\begin{equation}\label{aux-relative-Koszul}
W(\Fg, \Fh, V) = \Big\{f \in W(\Fg, V)\, \Big|\, \iota(Y)f = 0, \iota(Y)(d_{\rm CE}f) = 0 , \forall\, Y \in \Fh\Big\}.
\end{equation}

It follows from the following lemma that \eqref{aux-relative-Koszul} is a subcomplex of the perturbed Koszul complex.

\begin{lemma}\label{lemma-relative-Koszul}
Let $\Fg$ be a Lie algebra, $\Fh \subseteq \Fg$ be a subalgebra and $V$ be a unimodular SAYD module over $\Fg$. Then we have
\begin{equation}
d_{\rm K}(W(\Fg, \Fh, V)) \subseteq W(\Fg, \Fh, V).
\end{equation}
\end{lemma}

\begin{proof}
For any $\a \ot v \in W^{n+1}(\Fg, \Fh, V)$ and any $Y \in \Fh$,
\begin{align}
\begin{split}
& \iota(Y)(d_{\rm K}(\a \ot v)) = \iota(Y)((-1)^n\iota(v\nsb{-1})\a \ot v\nsb{0}) = (-1)^n\iota(Y)\iota(v\nsb{-1})\a \ot v\nsb{0} \\
& = (-1)^{n-1}\iota(v\nsb{-1})\iota(Y)\a \ot v\nsb{0} = d_{\rm K}(\iota(Y)\a \ot v) = 0
\end{split}
\end{align}
Similarly, using $d_{\rm CE} \circ d_{\rm K} + d_{\rm K} \circ d_{\rm CE} = 0$,
\begin{align}
\begin{split}
& \iota(Y)(d_{\rm CE}(d_{\rm K}(\a \ot v))) = - \iota(Y)(d_{\rm K} \circ d_{\rm CE}(\a \ot v)) \\
& = - d_{\rm K}(\iota(Y)d_{\rm CE}(\a \ot v)) = 0.
\end{split}
\end{align}
\end{proof}

\begin{definition}
Let $\Fg$ be a Lie algebra, $\Fh \subseteq \Fg$ be a Lie subalgebra and $V$ be a unimodular SAYD module over $\Fg$. We call the homology of the complex $(W(\Fg, \Fh, V), d_{\rm CE} + d_{\rm K})$ the relative cyclic cohomology of the Lie algebra $\Fg$ relative to the Lie subalgebra $\Fh$ with coefficients in $V$. We use the notation $\widetilde{HC}^\bullet(\Fg,\Fh,V)$. Similarly, we use the notation $\widetilde{HP}^\bullet(\Fg,\Fh,V)$ to denote the periodic cyclic cohomology.
\end{definition}

In the case of the trivial Lie algebra coaction, the above cohomology becomes the relative Lie algebra cohomology.

\medskip

We next prove the following generalization of Theorem \ref{theorem-cohomology-R(g)}.

\begin{theorem}\label{theorem-g-R(g)spectral sequence}
Let $\Fg$ be a Lie algebra and $\Fg = \Fh \ltimes \Fl$ be a Levi decomposition. Let $V$ be a unimodular SAYD module over $\Fg$ as a locally finite $\Fg$-module and a locally conilpotent $\Fg$-comodule. Assume also that $V$ is stable over $R(\Fg)$.  Then the cyclic cohomology of $\Fg$ relative to the subalgebra $\Fh \subseteq \Fg$ with coefficients in $V$ is the same as the periodic cyclic cohomology of $R(\Fg)$ with coefficients in $V$. In short,
\begin{equation}
HP^\bullet(R(\Fg), V) \cong \widetilde{HP}^\bullet(\Fg, \Fh, V)
\end{equation}
\end{theorem}

\begin{proof}
Since $V$ is a unimodular stable AYD module over $\Fg$, by Lemma \ref{lemma-relative-Koszul} the relative perturbed Koszul complex $(W(\Fg, \Fh, V), d_{\rm CE} + d_{\rm K})$ is well defined. On the other hand, since $V$ is locally finite as a $\Fg$-module and locally conilpotent as a $\Fg$-comodule, it is an AYD module over $R(\Fg)$ by Proposition \ref{proposition-AYD-R(g)}. Together with the assumption that $V$ is stable over $R(\Fg)$, the Hopf-cyclic complex $(C(R(\Fg),V), b + B)$ is well defined.

\medskip

Since $V$ is a unimodular SAYD module over  $\Fg$, $V_{\d}:= V\ot\Cb_{\d}$ is an SAYD module over $\Fg$, where $\d$ is the trace of the adjoint representation of the Lie algebra $\Fg$ on itself. Therefore, by \cite[Lemma 6.2]{JaraStef} we have the filtration $V = \cup_{p \in \Zb}F_pV$ defined as $F_0V = V^{\rm co\,U(\Fg)}$ and inductively
\begin{equation}
F_{p+1}V/F_pV = (V/F_pV)^{\rm co\,U(\Fg)}.
\end{equation}
This filtration naturally induces an analogous filtration on the complexes by
\begin{equation}
F_jW(\Fg, \Fh, V) = W(\Fg, \Fh, F_jV), \quad \text{and} \quad F_jC(R(\Fg),V) = C(R(\Fg),F_jV).
\end{equation}

We now show that the (co)boundary maps $d_{\rm CE},d_{\rm K},b,B$ respect this filtration. To do so for $d_{\rm K}$ and $d_{\rm CE}$, it  suffices to show that the $\Fg$-action and $\Fg$-coaction on $V$ respect the filtration, which is done in  the proof of Theorem \ref{g-U(g) spectral sequence}. Similarly, to show that the Hochschild coboundary $b$ and the Connes boundary map $B$ respect the filtration we need to show that $R(\Fg)$-action and $R(\Fg)$-coaction respects the filtration.

\medskip

Indeed, for an element $v \in F_pV$, writing the $U(\Fg)$-coaction as
\begin{equation}
v\snsb{-1} \ot v\snsb{0} = 1 \ot v + v\ns{-1} \ot v\ns{0}, \qquad v\ns{-1} \ot v\ns{0} \in U(\Fg) \ot F_{p-1}V,
\end{equation}
we get
\begin{equation}
v \cdot f = \ve(f)v + f(v\ns{-1})v\ns{0} \in F_pV
\end{equation}
for any $f \in R(\Fg)$. This proves that the $R(\Fg)$-action respects the filtration. To prove that $R(\Fg)$-coaction respects the filtration, we first write the coaction on $v \in F_pV$ as
\begin{equation}
v \mapsto \sum_i f^i \ot v_i \in R(\Fg) \ot V.
\end{equation}
By \cite[Lemma 1.1]{Hoch-book} there are elements $u_j \in U(\Fg)$ such that $f^i(u_j) = \d^i_j$. Hence, for any $v_{i_0}$ we have
\begin{equation}
v_{i_0} = \sum_i f^i(u_{i_0}) v_i = v\cdot u_{i_0} \in F_pV.
\end{equation}
We have proved that the $R(\Fg)$-coaction respects the filtration.

\medskip

Next, we write the $E_1$ terms of the associated spectral sequences. We have
\begin{align}
\begin{split}
& E_1^{j,i}(\Fg,\Fh,V) = H^{i+j}(F_jW(\Fg, \Fh, V)/F_{j-1}W(\Fg, \Fh, V)) \\
& = H^{i+j}(W(\Fg, \Fh, F_jV/F_{j-1}V)) = \bigoplus_{i+j = n \, {\rm mod\,2}}H^n(\Fg, \Fh, F_jV/F_{j-1}V),
\end{split}
\end{align}
where in the last equality we used the fact that $F_{j}V/F_{j-1}V$ has trivial $\Fg$-coaction.

\medskip

Similarly we have
\begin{align}
\begin{split}
& E_1^{j,i}(R(\Fg),V) = H^{i+j}(F_jC(R(\Fg),V)/F_{j-1}C(R(\Fg),V)) \\
& = H^{i+j}(C(R(\Fg),F_jV/F_{j-1}V)) = \bigoplus_{i+j = n\, {\rm mod\,2}}H^n_{\rm coalg}(R(\Fg),F_jV/F_{j-1}V)),
\end{split}
\end{align}
where in the last equality we may use Theorem \ref{Theorem-F-Levi} due to the fact that $F_jV/F_{j-1}V$ has trivial $\Fg$-coaction, hence trivial $R(\Fg)$-action.

\medskip

Finally, under the hypotheses of the theorem, a quasi-isomorphism between the $E_1$-terms is given by Theorem \ref{Theorem-main}.
\end{proof}

\begin{remark}{\rm
If $V$ has trivial $\Fg$-comodule structure, then $d_{\rm K} = 0$ and hence the above theorem gives Theorem \ref{Theorem-main}.
}\end{remark}

\subsection{Hopf-cyclic cohomology of $R(\Fg_2) \acl U(\Fg_1)$}

In this subsection we compute the Hopf-cyclic cohomology of the bicrossed product Hopf algebra $R(\Fg_2) \acl U(\Fg_1)$, associated to a matched pair $(\Fg_1,\Fg_2)$ of Lie algebras, in terms of the relative Lie algebra cohomology.

\medskip

Let $V$ be an AYD module over a double crossed sum Lie algebra $\Fg_1 \bowtie \Fg_2$ with a locally finite $\Fg_1 \bi \Fg_2$-action and a locally conilpotent $\Fg_1 \bi \Fg_2$-coaction. Then by Proposition \ref{aux-46}, $V$ is a YD module over the bicrossed product Hopf algebra $R(\Fg_2) \acl U(\Fg_1)$.

\medskip

Let also $(\delta,\sigma)$ be the canonical modular pair in invalution for the Hopf algebra $R(\Fg_2) \acl U(\Fg_1)$ (Theorem \ref{theorem-MPI}). Then $\, ^{\sigma}V_{\delta} := V \ot \, ^{\sigma}\mathbb{C}_{\delta}$ is an AYD module over the bicrossed product Hopf algebra $R(\Fg_2) \acl U(\Fg_1)$.

\medskip

Finally, let us assume that $V$ is stable over $R(\Fg_2)$ and $U(\Fg_1)$. Then $\, ^{\sigma}V_{\delta}$ is stable if and only if
\begin{align}
\begin{split}
& v \ot 1_{\mathbb{C}} = (v^{\sns{0}}\snsb{0} \ot 1_{\mathbb{C}}) \cdot (v^{\sns{-1}} \acl v^{\sns{0}}\snsb{-1})(\sigma \acl 1) \\
& = (v^{\sns{0}}\snsb{0} \ot 1_{\mathbb{C}}) \cdot (v^{\sns{-1}} \acl 1)(1 \acl v^{\sns{0}}\snsb{-1})(\sigma \acl 1) \\
& = (v^{\sns{0}}\snsb{0} \cdot v^{\sns{-1}} \ot 1_{\mathbb{C}}) \cdot (1 \acl v^{\sns{0}}\snsb{-1})(\sigma \acl 1) \\
& = ((v^{\sns{0}}\snsb{0} \cdot v^{\sns{-1}}) \cdot v^{\sns{0}}\snsb{-1}\delta(v^{\sns{0}}\snsb{-2}) \ot 1_{\mathbb{C}}) \cdot (\sigma \acl 1) \\
& = v\snsb{0}\delta(v\snsb{-1})\sigma \ot 1_{\mathbb{C}},
\end{split}
\end{align}
where in the fourth equality we have used Proposition \ref{aux-59}. In other words, with $V$ satisfying the hypothesis of the Proposition \ref{aux-59}, $\, ^{\sigma}V_{\delta}$ is stable if and only if
\begin{equation}\label{aux-47}
v\snsb{0}\delta(v\snsb{-1})\sigma = v.
\end{equation}

\begin{theorem}\label{aux-63}
Let $(\Fg_1, \Fg_2)$ be a matched pair of Lie algebras and $\Fg_2 = \Fh \ltimes \Fl$ be a Levi decomposition such that $\Fh$ is $\Fg_1$-invariant and $\Fh$ acts on $\Fg_1$ by derivations. Let $V$ be a unimodular SAYD module over $\Fg_1 \bowtie \Fg_2$ with a locally finite action and a locally conilpotent coaction. Finally assume that $\, ^{\sigma}V_{\delta}$ is stable. Then we have
\begin{equation}
HP(R(\Fg_2) \acl U(\Fg_1),\, ^{\sigma}V_{\delta}) \cong \widetilde{HP}(\Fg_1 \bowtie \Fg_2, \Fh, V).
\end{equation}
\end{theorem}

\begin{proof}
Let $C(U(\Fg_1) \cl R(\Fg_2),\, ^{\sigma}V_{\delta})$ be the complex computing the Hopf-cyclic cohomology of the bicrossed product Hopf algebra $\mathcal{H} =  R(\Fg_2) \acl U(\Fg_1)$ with coefficients in the SAYD module $^{\sigma}V_{\delta}$.

\medskip

By \cite[Theorem 3.16]{MoscRang09}, this complex is quasi-isomorphic with the total complex of the mixed complex $\FZ(\mathcal{H},R(\Fg_2);\, ^{\sigma}V_{\delta})$ whose cylindrical structure is given by the horizontal operators
\begin{align}\label{horizontal-operators}
\begin{split}
& \hd_{0}(v \ot \td{f} \ot \td{u}) = v \ot 1 \ot f^1 \odots f^p \ot \td{u} \\
& \hd_{i}(v \ot \td{f} \ot \td{u}) = v \ot f^1 \odots \D(f^i)\odots f^p \ot \td{u} \\
& \hd_{p+1}(v \ot \td{f} \ot \td{u}) = v\pr{0} \ot f^1 \odots f^p \ot \td{u}^{\pr{-1}}v\pr{-1} \rt 1_{R(\Fg_2)} \ot \td{u}^{\pr{0}} \\
& \overset{\ra}{\sigma}_{j}(v \ot \td{f} \ot \td{u}) = v \ot f^1 \odots \ve(f^{j+1}) \odots f^p \ot \td{u} \\
&\hta(v \ot \td{f} \ot \td{u}) = v\pr{0} f^1\ps{1} \ot S(f^1\ps{2}) \cdot (f^2 \odots f^p \ot \td{u}^{\pr{-1}}v\pr{-1} \rt 1_{R(\Fg_2)} \ot \td{u}^{\pr{0}}),
\end{split}
\end{align}
and the vertical operators
\begin{align}\label{vertical-operators}
\begin{split}
&\vd_{0}(v \ot \td{f} \ot \td{u}) = v \ot \td{f} \ot \dot 1 \ot u^1\odots  u^q \\
&\vd_{i}(v \ot \td{f} \ot \td{u}) = v \ot \td{f} \ot u^0 \odots \D(u^i) \odots u^q \\
&\vd_{q+1}(v \ot \td{f} \ot \td{u}) = v\pr{0} \ot \td{f} \ot u^1 \odots u^q \ot \wbar{v\pr{-1}} \\
&\vs_j(v \ot \td{f} \ot \td{u}) = v \ot \td{f} \ot u^1 \odots \ve(u^{j+1}) \odots u^q \\
&\vta(v \ot \td{f} \ot \td{u}) = v\pr{0} u^1\ps{4}S^{-1}(u^1\ps{3} \rt 1_{R(\Fg_2)}) \ot \\
& S(S^{-1}(u^1\ps{2}) \rt 1_{R(\Fg_2)}) \cdot \left( S^{-1}(u^1\ps{1}) \rt \td{f} \ot S(u^1\ps{5}) \cdot (u^2 \odots u^q \ot \wbar{v\pr{-1}})
\right)
\end{split}
\end{align}
for any $v \in \, ^{\sigma}V_{\delta}$.

\medskip

Here,
\begin{equation}
U(\Fg_1)^{\ot \, q} \to \mathcal{H} \ot U(\Fg_1)^{\ot \, q}, \quad \td{u} \mapsto \td{u}^{\pr{-1}} \ot \td{u}^{\pr{0}}
\end{equation}
arises from the left $\mathcal{H}$-coaction on $U(\Fg_1)$ that coincides with the original $R(\Fg_2)$-coaction \cite[Proposition 3.20]{MoscRang09}. On the other hand, we have $U(\Fg_1) \cong \mathcal{H} \ot_{R(\Fg_2)} \mathbb{C} \cong \mathcal{H}/\mathcal{H}R(\Fg_2)^+$ as coalgebras via the map $(f \acl u) \ot_{R(\Fg_2)} 1_{\mathbb{C}} \mapsto \ve(f)u$, and $\wbar{f \acl u} = \ve(f)u$ denotes the corresponding class.

\medskip

Since $V$ is a unimodular SAYD module over $\Fg_1 \bowtie \Fg_2$, it admits the filtration $(F_pV)_{p \in \mathbb{Z}}$ defined as before. We recall from Proposition \ref{proposition-comodule-doublecrossed-sum} that $\Fg_1 \bowtie \Fg_2$-coaction is the summation of $\Fg_1$-coaction and $\Fg_2$-coaction. Therefore, since the $\Fg_1 \bi \Fg_2$-coaction respects the filtration, we conclude that the $\Fg_1$-coaction and $\Fg_2$-coaction respect the filtration. Similarly, since the $\Fg_1 \bi \Fg_2$-action respects the filtration, we conclude that the $\Fg_1$-action and $\Fg_2$-action respects the filtration. Finally, by a similar argument to that in the proof of Theorem \ref{theorem-g-R(g)spectral sequence} we can say that the $R(\Fg_2)$-action and $R(\Fg_2)$-coaction respect the filtration.

\medskip

As a result, the (co)boundary maps $d_{\rm CE}$ and $d_{\rm K}$ of the complex $W(\Fg_1 \bi \Fg_2, \Fh, V)$, and $b$ and $B$ of $C(U(\Fg_1) \cl R(\Fg_2), \, ^{\sigma}V_{\delta})$ respect the filtration.

\medskip

Next, we proceed to the $E_1$ terms of the associated spectral sequences. We have
\begin{align}\label{aux-48}
\begin{split}
& E_1^{j,i}(R(\Fg_2) \acl U(\Fg_1), \, ^{\sigma}V_{\delta})  \\
&= H^{i+j}(F_jC(U(\Fg_1) \cl R(\Fg_2),\, ^{\sigma}V_{\delta})/F_{j-1}C(U(\Fg_1) \cl R(\Fg_2),\, ^{\sigma}V_{\delta})),
\end{split}
\end{align}
where
\begin{align}
\begin{split}
& F_jC(U(\Fg_1) \cl R(\Fg_2),\, ^{\sigma}V_{\delta})/F_{j-1}C(U(\Fg_1) \cl R(\Fg_2),\, ^{\sigma}V_{\delta})  \\
& =C(U(\Fg_1) \cl R(\Fg_2), F_j \, ^{\sigma}V_{\delta}/F_{j-1} \, ^{\sigma}V_{\delta}).
\end{split}
\end{align}
Since $$F_j \, ^{\sigma}V_{\delta}/F_{j-1} \, ^{\sigma}V_{\delta} = \, ^{\sigma}(F_jV/F_{j-1}V)_{\delta} =: \, ^{\sigma}\wbar{V}_{\delta}$$ has trivial $\Fg_1 \bowtie \Fg_2$-comodule structure, its $U(\Fg_1)$-comodule structure and $R(\Fg_2)$-module structure are also trivial. Therefore, $\, ^{\sigma}\wbar{V}_{\delta}$ is an $R(\Fg_2)$-SAYD in the sense of \cite{MoscRang09}. In this case, by \cite[Proposition 3.16]{MoscRang09}, $\FZ(\mathcal{H},R(\Fg_2);\, ^{\sigma}V_{\delta})$ is a bicocyclic module and the cohomology in \eqref{aux-48} is computed from the total complex of the bicocyclic complex
\begin{align}\label{bicocyclic-bicomplex}
\begin{xy}  \xymatrix{  \vdots\ar@<.6 ex>[d]^{\uparrow B} & \vdots\ar@<.6 ex>[d]^{\uparrow B}
 &\vdots \ar@<.6 ex>[d]^{\uparrow B} & &\\
 ^\s{\overline{V}}_\d \ot U(\Fg_1)^{\ot 2} \ar@<.6 ex>[r]^{\hb}\ar@<.6
ex>[u]^{  \uparrow b  } \ar@<.6 ex>[d]^{\uparrow B}&
  ^\s{\overline{V}}_\d\ot U(\Fg_1)^{\ot 2}\ot  R(\Fg_2)   \ar@<.6 ex>[r]^{\hb}\ar@<.6 ex>[l]^{\hB}\ar@<.6 ex>[u]^{  \uparrow b  }
   \ar@<.6 ex>[d]^{\uparrow B} & ^\s{\overline{V}}_\d\ot U(\Fg_1)^{\ot 2}\ot R(\Fg_2)^{\ot 2}
   \ar@<.6 ex>[r] \ar@<.6 ex>[l]^{~~\hB} \ar@<.6 ex>[u]^{  \uparrow b  }
   \ar@<.6 ex>[d]^{\uparrow B}& \ar@<.6 ex>[l] \hdots&\\
^\s{\overline{V}}_\d \ot U(\Fg_1) \ar@<.6 ex>[r]^{\hb}\ar@<.6 ex>[u]^{  \uparrow b  }
 \ar@<.6 ex>[d]^{\uparrow B}&  ^\s{\overline{V}}_\d \ot U(\Fg_1) \ot R(\Fg_2) \ar@<.6 ex>[r]^{\hb}
 \ar@<.6 ex>[l]^{\hB}\ar@<.6 ex>[u]^{  \uparrow b  } \ar@<.6 ex>[d]^{\uparrow B}
 &^\s{\overline{V}}_\d\ot U(\Fg_1) \ot  R(\Fg_2)^{\ot 2}  \ar@<.6 ex>[r] \ar@<.6 ex>[l]^{\hB}\ar@<.6 ex>[u]^{  \uparrow b  }
  \ar@<.6 ex>[d]^{\uparrow B}&\ar@<.6 ex>[l] \hdots&\\
^\s{\overline{V}}_\d  \ar@<.6 ex>[r]^{\hb}\ar@<.6 ex>[u]^{  \uparrow b  }&
^\s{\overline{V}}_\d\ot R(\Fg_2) \ar@<.6 ex>[r]^{\hb}\ar[l]^{\hB}\ar@<.6
ex>[u]^{  \uparrow b  }&^\s{\overline{V}}_\d\ot R(\Fg_2)^{\ot 2}  \ar@<.6
ex>[r]^{~~\hb}\ar@<.6 ex>[l]^{\hB}\ar@<1 ex >[u]^{  \uparrow b  }
&\ar@<.6 ex>[l]^{~~\hB} \hdots&}
\end{xy}
\end{align}

 Moreover, by Proposition \ref{mixCE} the total complex of the bicomplex \eqref{bicocyclic-bicomplex} is quasi-isomorphic to the total complex of the bicomplex
\begin{align}
\begin{xy} \xymatrix{  \vdots & \vdots
 &\vdots &&\\
 \;^\s{\overline{V}}_\d\ot \wg^2{\Fg_1}^\ast  \ar[u]^{d_{\rm CE}}\ar[r]^{b^\ast_{R(\Fg_2)}\;\;\;\;\;\;\;\;\;\;}&  \;^\s{\overline{V}}_\d\ot \wg^2{\Fg_1}^\ast\ot R(\Fg_2) \ar[u]^{d_{\rm CE}} \ar[r]^{b^\ast_R(\Fg_2)}& \;^\s{\overline{V}}_\d\ot \wg^2{\Fg_1}^\ast\ot R(\Fg_2)^{\ot 2} \ar[u]^{d_{\rm CE}} \ar[r]^{\;\;\;\;\;\;\;\;\;\;\;\;\;\;\;\;\;\;\;\;\;\;\;b^\ast_{R(\Fg_2)}} & \hdots&  \\
 \;^\s{\overline{V}}_\d\ot {\Fg_1}^\ast  \ar[u]^{d_{\rm CE}}\ar[r]^{b^\ast_{R(\Fg_2)}\;\;\;\;\;\;\;\;\;\;\;\;\;}& \;^\s{\overline{V}}_\d\ot {\Fg_1}^\ast\ot R(\Fg_2) \ar[u]^{d_{\rm CE}} \ar[r]^{b^\ast_{R(\Fg_2)}}& \;^\s{\overline{V}}_\d\ot  {\Fg_1}^\ast\ot R(\Fg_2)^{\ot 2} \ar[u]^{d_{\rm CE}} \ar[r]^{\;\;\;\;\;\;\;\;\;\;\;\;\;\;\;\;\;\;\;b^\ast_R(\Fg_2) }& \hdots&  \\
   \;^\s{\overline{V}}_\d\ar[u]^{d_{\rm CE}}\ar[r]^{b^\ast_{R(\Fg_2)}~~~~~~~}& \;^\s{\overline{V}}_\d\ot R(\Fg_2) \ar[u]^{d_{\rm CE}}\ar[r]^{b^\ast_R(\Fg_2)}& \;^\s{\overline{V}}_\d\ot R(\Fg_2)^{\ot 2} \ar[u]^{d_{\rm CE}} \ar[r]^{\;\;\;\;\;\;\;\;\;\;\;\;\;\;\;b^\ast_{R(\Fg_2)}} & \hdots& }
\end{xy}
\end{align}
where $b^\ast_{R(\Fg_2)}$ is the coalgebra Hochschild coboundary with coefficients in the $R(\Fg_2)$-comodule $\;^\s{\overline{V}}_\d \ot \wg^q{\Fg_1}^\ast$.

\medskip

Similarly,
\begin{align}\label{aux-49}
\begin{split}
& E_1^{j,i}(\Fg_1 \bowtie \Fg_2,\Fh,V) = H^{i+j}(F_jW(\Fg_1 \bowtie \Fg_2,\Fh,V)/F_{j-1}W(\Fg_1 \bowtie \Fg_2,\Fh,V)) \\
& = H^{i+j}(W(\Fg_1 \bowtie \Fg_2,\Fh,F_jV/F_{j-1}V)) = \bigoplus_{i+j = n \, {\rm mod\,2}}H^n(\Fg_1 \bi \Fg_2, \Fh, F_jV/F_{j-1}V)
\end{split}
\end{align}
where the last equality follows from the fact that $F_jV/F_{j-1}V$ has trivial $\Fg_1 \bowtie \Fg_2$-comodule structure.

\medskip

Finally, the quasi isomorphism between the $E_1$-terms \eqref{aux-48} and \eqref{aux-49} is given by the Corollary \ref{corollary-matched-pair-Lie-algebra-cohomology}.
\end{proof}

\begin{remark}
{\rm
In case of trivial $\Fg_1 \bowtie \Fg_2$-coaction, this Theorem becomes Corollary \ref{corollary-matched-pair-Lie-algebra-cohomology}. In this case, the $U(\Fg_1)$-coaction and $R(\Fg_2)$-action are trivial, therefore the condition \eqref{aux-47} is obvious.}
\end{remark}

\section{Computations}

\subsection{Computation of $HP(\Hc_{1\rm S}, V_\d)$}

In this subsection we compute the cohomology of the complex $C(\Uc \cl \Fc, V_{\delta})$ with $V = S(s\ell(2)^\ast)\nsb{1}$, which  computes  the periodic Hopf cyclic cohomology
\begin{equation}
HP(\Hc_{1\rm S}\cop,V_{\delta}) = HP(\Fc \acl \Uc,V_{\delta}).
\end{equation}
We note that since $V_{\delta}$ is also an SAYD module over $U(s\ell(2))$ with the same action and coaction due to the unimodularity of $s\ell(2)$, we have
\begin{equation}\label{aux-filtration-on-V-d}
F_0V_{\delta} = F_0V \ot \Cb_{\delta} = \Cb\Big\langle R^X,R^Y,R^Z\Big\rangle \ot \mathbb{C}_{\delta}, \quad F_pV_{\delta} = \left\{\begin{array}{cc}
                         V_{\delta} & p \geq 1 \\
                         0 & p<0.
                       \end{array}
\right.
\end{equation}

We will first derive a Cartan type homotopy formula for Hopf cyclic cohomology, as in \cite{MoscRang07}. One notes that in \cite{MoscRang07} the SAYD module was one dimensional. We have to upgrade the homotopy formula to fit our situation. To this end, let
\begin{equation}
D_Y:\Hc \to \Hc, \quad D_Y(h) := hY.
\end{equation}
Obviously, $D_Y$ is an $\Hc$-linear coderivation. Hence the operators
\begin{align}
\begin{split}
& \mathcal{L}_{D_Y}:C^n(\Uc \cl \Fc,V_{\delta}) \to C^n(\Uc \cl \Fc,V_{\delta}) \\
& \Lc_{D_Y}(v\ot_\Hc c^0\odots c^n)=\sum_{i=0}^{n}v\ot c^0_\Hc \odots D_Y(c_i) \odots c^n,
\end{split}
\end{align}
\begin{align}
\begin{split}
& e_{D_Y}:C^n(\Uc \cl \Fc,V_{\delta}) \to C^{n+1}(\Uc \cl \Fc,V_{\delta}) \\
& e_{D_Y}(v\ot_\Hc c^0\odots c^n)=(-1)^{n}v\pr{0}\ot_\Hc c^0\ps{2}\ot c^1\odots c^n\ot v\pr{-1}D_Y(c^0\ps{1}),
\end{split}
\end{align}
and
\begin{align}
\begin{split}
& E_{D_Y}:C^n(\Uc \cl \Fc,V_{\delta}) \to C^{n-1}(\Uc \cl \Fc,V_{\delta}) \\
&E_{D_Y}^{j,i}(V\ot_\Hc c^0\odots c^{n})=\\
&=(-1)^{n(i+1)}\ve(c^0)v\pr{0}\ot_\Hc c^{n-i+2}\odots c^{n+1}\ot
v\pr{-1}c_1\odots m\pr{-(j-i)}c^{j-i} \\
&\ot v\pr{-(j-i+1)}D_Y(c^{j-i+1})\ot v\pr{-(j-i+2)}c^{j-i+2}\odots
v\pr{-(n-i+1)}c^{n-i+1},
\end{split}
\end{align}
satisfy, by \cite[Proposition 3.7]{MoscRang07},
\begin{equation}
[E_{D_Y} + e_{D_Y}, b + B] = \mathcal{L}_{D_Y}.
\end{equation}

We next obtain an analogue of \cite[Lemma 3.8]{MoscRang07}.

\begin{lemma}\label{aux-68}
We have
\begin{equation}
\mathcal{L}_{D_Y} = I - \widetilde{\ad}Y,
\end{equation}
where
\begin{equation}
\widetilde{\ad}Y(v_{\delta} \ot \widetilde{f} \ot \widetilde{u}) = v_{\delta} \ot \widetilde{\ad}Y(\widetilde{f} \ot \widetilde{u}) - (v \cdot Y)_{\delta} \ot \widetilde{f} \ot \widetilde{u},
\end{equation}
and $v_{\delta} := v \ot 1_{\mathbb{C}}$.
\end{lemma}

\begin{proof}
Let us first recall the isomorphism
\begin{equation}
\Theta := \Phi_2 \circ \Phi_1 \circ \Psi:C^{\bullet}_{\Hc}(\Uc \cl \Fc,V_{\delta}) \to \FZ^{\bullet,\bullet}
\end{equation}
of cocyclic modules. By \cite{MoscRang09}, we know that
\begin{align}\label{aux-76}
\begin{split}
& \Psi(v_{\delta} \ot_{\Hc} u^0 \cl f^0 \ot \ldots \ot u^n \cl f^n)  \\
& =v_{\delta} \ot_{\mathcal{H}} {u^0}^{\pr{-n-1}}f^0 \ot \ldots \ot {u^0}^{\pr{-1}} \ldots \ot {u^n}^{\pr{-1}}f^n \ot {u^0}^{\pr{0}} \ot \ldots \ot {u^n}^{\pr{0}}, \\
& \Psi^{-1}(v_{\delta} \ot_{\mathcal{H}} f^0 \ot \ldots \ot f^n \ot u^0 \ot \ldots \ot u^n)  \\
& =v_{\delta} \ot_{\mathcal{H}} {u^0}^{\pr{0}} \cl S^{-1}({u^0}^{\pr{-1}})f^0 \ot \ldots \ot {u^n}^{\pr{0}} \cl S^{-1}({u^0}^{\pr{-n-1}}{u^1}^{\pr{-n}} \ldots {u^n}^{\pr{-1}})f^n,
\end{split}
\end{align}

\begin{align}\label{aux-77}
\begin{split}
& \Phi_1(v_{\delta} \ot_{\Hc} f^0 \ot \ldots \ot f^n \ot u^0 \ot \ldots \ot u^n)  \\
& =v_{\delta} \cdot u^0\ps{2} \ot_{\Fc} S^{-1}(u^0\ps{1}) \rt (f^0 \ot \ldots \ot f^n) \ot S(u^0\ps{3}) \cdot (u^1 \ot \ldots \ot u^n), \\
& \Phi_1^{-1}(v_{\delta} \ot_{\Fc} f^0 \ot \ldots \ot f^n \ot u^1 \ot \ldots \ot u^n)  \\
& =v_{\delta} \ot_{\Hc} f^0 \ot \ldots \ot f^n \ot 1_{U(\Fg_1)} \ot u^1 \ot \ldots \ot u^n,
\end{split}
\end{align}
and
\begin{align}\label{aux-78}
\begin{split}
& \Phi_2(v_{\delta} \ot_{\Fc} f^0 \ot \ldots \ot f^n \ot u^1 \ot \ldots \ot u^n)  \\
& =v_{\delta} \cdot f^0\ps{1} \ot S(f^0\ps{2}) \cdot (f^1 \ot \ldots \ot f^n \ot u^1 \ot \ldots \ot u^n), \\
& \Phi_2^{-1}(v_{\delta} \ot f^1 \ot \ldots \ot f^n \ot u^1 \ot \ldots \ot u^n)  \\
&= v_{\delta} \ot_{\Fc} 1_{\Fc} \ot f^1 \ot \ldots \ot f^n \ot u^1 \ot \ldots \ot u^n.
\end{split}
\end{align}

Here, the left $\mathcal{H}$-coaction on $\Uc$ is the one corresponding to the right $\Fc$-coaction, namely
\begin{equation}\label{aux-65}
u^{\pr{-1}} \ot u^{\pr{0}} = S(u^{\pr{1}}) \ot u^{\pr{0}}.
\end{equation}

We also recall that
\begin{align}
\begin{split}
& \Phi:\mathcal{H} = \Fc \acl \Uc \to \Uc \cl \Fc \\
& \Phi(f \acl u) = u^{\pr{0}} \cl fu^{\pr{1}}\\
& \Phi^{-1}(u \cl f) = fS^{-1}(u^{\pr{1}}) \acl u^{\pr{0}}.
\end{split}
\end{align}

Therefore, we have
\begin{align}
\begin{split}
& \Theta \circ \mathcal{L}_{D_Y} \circ \Theta^{-1} (v_{\delta} \ot f^1 \ot \ldots \ot f^n \ot u^1 \ot \ldots \ot u^n)  \\
& =\Theta \circ \mathcal{L}_{D_Y} \circ \Psi^{-1} \circ \Phi_1^{-1} \circ \Phi_2^{-1} (v_{\delta} \ot f^1 \ot \ldots \ot f^n \ot u^1 \ot \ldots \ot u^n)  \\
&= \Theta \circ \mathcal{L}_{D_Y} \circ \Psi^{-1} \circ \Phi_1^{-1} (v_{\delta} \ot_{\Fc} 1_{\Fc} \ot f^1 \ot \ldots \ot f^n \ot u^1 \ot \ldots \ot u^n)  \\
& =\Theta \circ \mathcal{L}_{D_Y} \circ \Psi^{-1} (v_{\delta} \ot_{\mathcal{H}} 1_{\Fc} \ot f^1 \ot \ldots \ot f^n \ot 1_{\Uc} \ot u^1 \ot \ldots \ot u^n)  \\
& =\Theta \circ \mathcal{L}_{D_Y} (v_{\delta} \ot_{\mathcal{H}} 1_{\Uc} \cl 1_{\Fc} \ot {u^1}^{\pr{0}} \cl S^{-1}({u^1}^{\pr{-1}}) \rt f^1 \ot \ldots \\
& \ldots \ot {u^n}^{\pr{0}} \cl S^{-1}({u^1}^{\pr{-n}} \ldots {u^n}^{\pr{-1}}) \rt f^n)  \\
& =\Theta \circ \mathcal{L}_{D_Y} (v_{\delta} \ot_{\mathcal{H}} 1_{\Uc} \cl 1_{\Fc} \ot {u^1}^{\pr{0}} \cl {u^1}^{\pr{1}}f^1 \ot \ldots \ot {u^n}^{\pr{0}} \cl {u^n}^{\pr{1}} \ldots {u^1}^{\pr{n}}f^n),
\end{split}
\end{align}
where in the last equality we have used \eqref{aux-65}. In order to apply $\mathcal{L}_{D_Y}$, we make the observation that
\begin{align}
\begin{split}
& \Phi D_Y \Phi^{-1} (u \cl f) = \Phi (fS^{-1}(u^{\pr{1}}) \acl u^{\pr{0}}Y)  \\
& =(u^{\pr{0}}Y)^{\pr{0}} \cl fS^{-1}(u^{\pr{1}})(u^{\pr{0}}Y)^{\pr{1}}  \\
& ={u^{\pr{0}}\ps{1}}^{\pr{0}}Y^{\pr{0}} \cl fS^{-1}(u^{\pr{1}}){u^{\pr{0}}\ps{1}}^{\pr{1}}(u^{\pr{0}}\ps{2} \rt Y^{\pr{1}})  \\
& ={u\ps{1}}^{\pr{0}}Y^{\pr{0}} \cl fS^{-1}({u\ps{1}}^{\pr{2}}{u\ps{2}}^{\pr{1}}){u\ps{1}}^{\pr{1}}({u\ps{2}}^{\pr{0}} \rt Y^{\pr{1}})  \\
& =uY \cl f,
\end{split}
\end{align}
using the action-coaction compatibilities of a bicrossed product. Hence,
\begin{align}
\begin{split}
& \Theta \circ \mathcal{L}_{D_Y} (v_{\delta} \ot_{\mathcal{H}} 1_{\Uc} \cl 1_{\Fc} \ot {u^1}^{\pr{0}} \cl {u^1}^{\pr{1}}f^1 \ot \ldots \ot {u^n}^{\pr{0}} \cl {u^n}^{\pr{1}} \ldots {u^1}^{\pr{n}}f^n)  \\
&= \Theta (v_{\delta} \ot_{\mathcal{H}} Y \cl 1_{\Fc} \ot {u^1}^{\pr{0}} \cl {u^1}^{\pr{1}}f^1 \ot \ldots \ot {u^n}^{\pr{0}} \cl {u^n}^{\pr{1}} \ldots {u^1}^{\pr{n}}f^n)  \\
&+ \sum_{i = 1}^n \Theta (v_{\delta} \ot_{\mathcal{H}} 1_{\Uc} \cl 1_{\Fc} \ot {u^1}^{\pr{0}} \cl {u^1}^{\pr{1}}f^1 \ot \ldots \\
& \ldots \ot {u^i}^{\pr{0}}Y \cl {u^i}^{\pr{1}} \ldots {u^1}^{\pr{i}}f^i \ot \ldots \ot {u^n}^{\pr{0}} \cl {u^n}^{\pr{1}} \ldots {u^1}^{\pr{n}}f^n).
\end{split}
\end{align}

We notice that
\begin{align}
\begin{split}
& \Psi (v_{\delta} \ot_{\mathcal{H}} Y \cl 1_{\Fc} \ot {u^1}^{\pr{0}} \cl {u^1}^{\pr{1}}f^1 \ot \ldots \ot {u^n}^{\pr{0}} \cl {u^n}^{\pr{1}} \ldots {u^1}^{\pr{n}}f^n)  \\
& =v_{\delta} \ot_{\mathcal{H}} Y^{\pr{-n-1}} \cdot 1_{\Fc} \ot Y^{\pr{-n}}({u^1}^{\pr{0}})^{\pr{-n}}{u^1}^{\pr{1}}f^1 \ot \ldots \\
& \ldots \ot Y^{\pr{-1}}({u^1}^{\pr{0}})^{\pr{-1}} \ldots ({u^n}^{\pr{0}})^{\pr{-1}}{u^n}^{\pr{1}} \ldots \\
& {u^1}^{\pr{n}}f^n \ot Y^{\pr{0}} \ot ({u^1}^{\pr{0}})^{\pr{0}} \odots ({u^n}^{\pr{0}})^{\pr{0}} \\
& =v_{\delta} \ot_{\mathcal{H}} 1_{\Fc} \ot f^1 \ot \ldots \ot f^n \ot Y \ot u^1 \ot \ldots \ot u^n,
\end{split}
\end{align}
where in the last equality we have used \eqref{aux-65}. Similarly,
\begin{align}
\begin{split}
& \Psi (v_{\delta} \ot_{\mathcal{H}} 1_{\Uc} \cl 1_{\Fc} \ot {u^1}^{\pr{0}} \cl {u^1}^{\pr{1}}f^1 \ot \ldots \\
& \ldots \ot {u^i}^{\pr{0}}Y \cl {u^i}^{\pr{1}} \ldots {u^1}^{\pr{i}}f^i \ot \ldots \ot {u^n}^{\pr{0}} \cl {u^n}^{\pr{1}} \ldots {u^1}^{\pr{n}}f^n) \\
&= v_{\delta} \ot_{\mathcal{H}} 1_{\Fc} \ot f^1 \ot \ldots \ot f^n \ot 1_{\Uc} \ot u^1 \ot \ldots u^iY \ot \ldots \ot u^n.
\end{split}
\end{align}

Therefore,
\begin{align}
\begin{split}
& \Theta (v_{\delta} \ot_{\mathcal{H}} Y \cl 1_{\Fc} \ot {u^1}^{\pr{0}} \cl {u^1}^{\pr{1}}f^1 \ot \ldots \ot {u^n}^{\pr{0}} \cl {u^n}^{\pr{1}} \ldots {u^1}^{\pr{n}}f^n)  \\
& +\sum_{i = 1}^n \Theta (v_{\delta} \ot_{\mathcal{H}} 1_{\Uc} \cl 1_{\Fc} \ot {u^1}^{\pr{0}} \cl {u^1}^{\pr{1}}f^1 \ot \ldots \\
& \ldots \ot {u^i}^{\pr{0}}Y \cl {u^i}^{\pr{1}} \ldots {u^1}^{\pr{i}}f^i \ot \ldots \ot {u^n}^{\pr{0}} \cl {u^n}^{\pr{1}} \ldots {u^1}^{\pr{n}}f^n)  \\
& =\Phi_2 \circ \Phi_1 (v_{\delta} \ot_{\mathcal{H}} 1_{\Fc} \ot \widetilde{f} \ot Y \ot \widetilde{u} + v_{\delta} \ot_{\mathcal{H}} 1_{\Fc} \ot \widetilde{f} \ot 1_{\Uc} \ot \widetilde{u} \cdot Y)  \\
& =\Phi_2 (v_{\delta} \cdot Y\ps{2} \ot_{\Fc} S^{-1}(Y\ps{1}) \rt (1_{\Fc} \ot \widetilde{f}) \ot S(Y\ps{3}) \cdot \widetilde{u} + v_{\delta} \ot_{\Fc} 1_{\Fc} \ot \widetilde{f} \ot \widetilde{u} \cdot Y).
\end{split}
\end{align}

Considering the fact that $Y \in \mathcal{H}$ is primitive, and hence $\ad Y(f) = [Y,f] = Y \rt f$, we conclude that
\begin{align}
\begin{split}
& \Phi_2 (v_{\delta} \cdot Y\ps{2} \ot_{\Fc} S^{-1}(Y\ps{1}) \rt (1_{\Fc} \ot \widetilde{f}) \ot S(Y\ps{3}) \cdot \widetilde{u} + m_{\delta} \ot_{\Fc} 1_{\Fc} \ot \widetilde{f} \ot \widetilde{u} \cdot Y)  \\
& =\Phi_2 (- v_{\delta} \ot_{\Fc} 1_{\Fc} \ot \ad Y(\widetilde{f}) \ot \widetilde{u} - v_{\delta} \ot_{\Fc} 1_{\Fc} \ot \widetilde{f} \ot Y \cdot \widetilde{u}  \\
& +v_{\delta} \cdot Y \ot_{\Fc} 1_{\Fc} \ot \widetilde{f} \ot \widetilde{u} + v_{\delta} \ot_{\Fc} 1_{\Fc} \ot \widetilde{f} \ot \widetilde{u} \cdot Y)  \\
& =v_{\delta} \cdot Y \ot_{\Fc} \widetilde{f} \ot \widetilde{u} - v_{\delta} \ot_{\Fc} \ad Y(\widetilde{f}) \ot \widetilde{u} - v_{\delta} \ot_{\Fc} \widetilde{f} \ot \ad Y(\widetilde{u}).
\end{split}
\end{align}

Finally, we recall that $v_{\delta} \cdot Y = (v \cdot Y\ps{1})_{\delta}\delta(Y\ps{2}) = (v \cdot Y)_{\delta} + v_{\delta}$ to finish the proof.
\end{proof}

\begin{lemma}\label{aux-69}
The operator $\widetilde{\ad}Y$ commutes with the horizontal operators \eqref{horizontal-operators} and the vertical operators \eqref{vertical-operators}.
\end{lemma}

\begin{proof}
We start with the horizontal operators. For the first horizontal coface, we have
\begin{align}
\begin{split}
& \hd_{0}(\widetilde{\ad}Y(v_{\delta} \ot \widetilde{f} \ot \widetilde{u})) = \hd_0(v_{\delta} \ot \ad Y(\widetilde{f} \ot \widetilde{u}) - (v \cdot Y)_{\delta} \ot \widetilde{f} \ot \widetilde{u}) \\
& =v_{\delta} \ot 1 \ot \ad Y(\widetilde{f} \ot \widetilde{u}) - (v \cdot Y)_{\delta} \ot 1 \ot \widetilde{f} \ot \widetilde{u} \\
&= \widetilde{\ad}Y (\hd_0(v_{\delta} \ot \widetilde{f} \ot \widetilde{u}))),
\end{split}
\end{align}
where $\widetilde{f} = f^1 \odots f^p$ and $\widetilde{u} = u^1 \odots u^q$.

\medskip

For  $\hd_{i}$ with $1 \leq i \leq n$, the commutativity  is a consequence of $\ad Y \circ \D = \D \circ \ad Y$ on $\Fc$. To see this, we notice that
\begin{align}
\begin{split}
& \D(\ad Y(f)) = \D(Y \rt f) = {Y\ps{1}}^{\pr{0}} \rt f\ps{1} \ot {Y\ps{1}}^{\pr{1}}(Y\ps{2} \rt f\ps{2}) \\
& = \ad Y(f\ps{1}) \ot f\ps{2} + f\ps{1} \ot \ad Y(f\ps{2}) = \ad Y(\D(f)).
\end{split}
\end{align}

For the commutation with the last horizontal coface operator, we proceed as follows. First we observe that
\begin{align}
\begin{split}
& \widetilde{\ad}Y(\hd_{n+1}(v_{\delta} \ot \widetilde{f} \ot \widetilde{u})) = \widetilde{\ad}Y({v\pr{0}}_{\delta} \ot \widetilde{f} \ot \widetilde{u}^{\pr{-1}}v\pr{-1} \rt 1_{\Fc} \ot \widetilde{u}^{\pr{0}})  \\
& ={v\pr{0}}_{\delta} \ot \ad Y(\widetilde{f}) \ot \widetilde{u}^{\pr{-1}}v\pr{-1} \rt 1_{\Fc} \ot \widetilde{u}^{\pr{0}} \\
& +{v\pr{0}}_{\delta} \ot \widetilde{f} \ot \ad Y(\widetilde{u}^{\pr{-1}}v\pr{-1} \rt 1_{\Fc})\ot \widetilde{u}^{\pr{0}} \\
&+{v\pr{0}}_{\delta} \ot \widetilde{f} \ot \widetilde{u}^{\pr{-1}}v\pr{-1} \rt 1_{\Fc} \ot \ad Y(\widetilde{u}^{\pr{0}}) \\
&-(v\pr{0} \cdot Y)_{\delta} \ot \widetilde{f} \ot \widetilde{u}^{\pr{-1}}v\pr{-1} \rt 1_{\Fc} \ot \widetilde{u}^{\pr{0}}.
\end{split}
\end{align}

Next, for any $h = g \acl u \in \mathcal{H}$ and $f \in \Fc$, we have on the one hand that
\begin{equation}
\ad Y(h \rt f) = \ad Y(g(u \rt f)) = \ad Y(g)(u \rt f) + g(Yu \rt f),
\end{equation}
and on the other hand that
\begin{align*}
&\ad Y(h) \rt f = (\ad Y(g) \acl u + g \acl Yu - g \acl uY) \rt f \\
&\hspace{1.9cm}= \ad Y(g)(u \rt f) + g(Yu \rt f) - g(uY \rt f).
\end{align*}

In other words,
\begin{equation}
\ad Y(h \rt f) = \ad Y(h) \rt f + h \rt \ad Y(f).
\end{equation}

Therefore we have
\begin{align}
\begin{split}
& {v\pr{0}}_{\delta} \ot \widetilde{f} \ot \ad Y(\widetilde{u}^{\pr{-1}}v\pr{-1} \rt 1_{\Fc})\ot \widetilde{u}^{\pr{0}} \\
& ={v\pr{0}}_{\delta} \ot \widetilde{f} \ot \ad Y(\widetilde{u}^{\pr{-1}})v\pr{-1} \rt 1_{\Fc}\ot \widetilde{u}^{\pr{0}} \\
& +{v\pr{0}}_{\delta} \ot \widetilde{f} \ot \widetilde{u}^{\pr{-1}}\ad Y(v\pr{-1}) \rt 1_{\Fc}\ot \widetilde{u}^{\pr{0}}.
\end{split}
\end{align}

Recalling \eqref{aux-65} and the coaction - multiplication compatibility \eqref{aux-matched-pair-4} on a bicrossed product, we observe that
\begin{align}
\begin{split}
& {\ad Y(u)}^{\pr{-1}} \ot {\ad Y(u)}^{\pr{0}} = S({\ad Y(u)}^{\pr{1}}) \ot {\ad Y(u)}^{\pr{0}} \\
& =S({Y\ps{1}}^{\pr{1}}(Y\ps{2} \rt u^{\pr{1}})) \ot {Y\ps{1}}^{\pr{0}}u^{\pr{0}} - S({u\ps{1}}^{\pr{1}}(u\ps{2} \rt Y^{\pr{1}})) \ot {u\ps{1}}^{\pr{0}}Y^{\pr{0}} \\
& =S(u^{\pr{1}}) \ot Yu^{\pr{0}} + S(Y \rt u^{\pr{1}}) \ot u^{\pr{0}} - S(u^{\pr{1}}) \ot u^{\pr{0}}Y,
\end{split}
\end{align}
where we have used $Y^{\pr{0}} \ot Y^{\pr{1}} = Y \ot 1$.

\medskip

By \cite[Lemma 1.1]{MoscRang11} we also have $S(Y \rt f) = Y \rt S(f)$ for any $f \in \Fc$. Hence we can conclude that
\begin{equation}\label{aux-72}
{\ad Y(u)}^{\pr{-1}} \ot {\ad Y(u)}^{\pr{0}} = \ad Y(u^{\pr{-1}}) \ot u^{\pr{0}} + u^{\pr{-1}} \ot \ad Y(u^{\pr{0}}),
\end{equation}
which implies immediately that
\begin{equation}
{\ad Y(\widetilde{u})}^{\pr{-1}} \ot {\ad Y(\widetilde{u})}^{\pr{0}} = \ad Y(\widetilde{u}^{\pr{-1}}) \ot \widetilde{u}^{\pr{0}} + \widetilde{u}^{\pr{-1}} \ot \ad Y(\widetilde{u}^{\pr{0}}).
\end{equation}

Finally, by the right-left AYD compatibility \eqref{aux-SAYD-condition} of $V$ over $\Hc$ we have
\begin{equation}
(v \cdot Y)\pr{-1} \ot (v \cdot Y)\pr{0} = v\pr{-1} \ot v\pr{0} \cdot Y - \ad Y({v\pr{-1}}) \ot v\pr{0}.
\end{equation}

So, $\widetilde{\ad}Y$ commutes with the last horizontal coface  $\hd_{n+1}$ since
\begin{align}
\begin{split}
& \widetilde{\ad}Y(\hd_{n+1}(v_{\delta} \ot \widetilde{f} \ot \widetilde{u})) = {v\pr{0}}_{\delta} \ot \ad Y(\widetilde{f}) \ot \widetilde{u}^{\pr{-1}}v\pr{-1} \rt 1_{\Fc} \ot \widetilde{u}^{\pr{0}}  \\
& +{v\pr{0}}_{\delta} \ot \widetilde{f} \ot {\ad Y(\widetilde{u})}^{\pr{-1}}v\pr{-1} \rt 1_{\Fc}\ot {\ad Y(\widetilde{u})}^{\pr{0}} \\
& -{(v \cdot Y)\pr{0}}_{\delta} \ot \widetilde{f} \ot \widetilde{u}^{\pr{-1}}(v \cdot Y)\pr{0} \rt 1_{\Fc} \ot \widetilde{u}^{\pr{0}} \\
&= \hd_{n+1}(\widetilde{\ad}Y(v_{\delta} \ot \widetilde{f} \ot \widetilde{u})).
\end{split}
\end{align}

The commutation $\hs_j \circ \widetilde{\ad}Y = \widetilde{\ad}Y \circ \hs_j$ with the horizontal codegeneracy operators is evident.

\medskip

We now consider the horizontal cyclic operator. Let us first note that
\begin{equation}
v_{\delta} \cdot f = (v \cdot f)_{\delta}, \qquad f\in \Fc.
\end{equation}
We then have
\begin{align}
\begin{split}
& \widetilde{\ad}Y(\hta(v_{\delta} \ot \widetilde{f} \ot \widetilde{u})) \\
& =\widetilde{\ad}Y(({v\pr{0}} \cdot f^1\ps{1})_{\delta} \ot S(f^1\ps{2}) \cdot (f^2 \odots f^p \ot \widetilde{u}^{\pr{-1}}v\pr{-1} \rt 1_{\Fc} \ot \widetilde{u}^{\pr{0}})) \\
& =({v\pr{0}} \cdot f^1\ps{1})_{\delta} \ot \ad Y(S(f^1\ps{2})) \cdot (f^2 \odots f^p \ot \widetilde{u}^{\pr{-1}}v\pr{-1} \rt 1_{\Fc} \ot \widetilde{u}^{\pr{0}}) \\
&+ ({v\pr{0}} \cdot f^1\ps{1})_{\delta} \ot S(f^1\ps{2}) \cdot (\ad Y(f^2 \odots f^p) \ot \widetilde{u}^{\pr{-1}}v\pr{-1} \rt 1_{\Fc} \ot \widetilde{u}^{\pr{0}}) \\
& +({v\pr{0}} \cdot f^1\ps{1})_{\delta} \ot S(f^1\ps{2}) \cdot (f^2 \odots f^p \ot \ad Y(\widetilde{u}^{\pr{-1}})v\pr{-1} \rt 1_{\Fc} \ot \widetilde{u}^{\pr{0}}) \\
&+ ({v\pr{0}} \cdot f^1\ps{1})_{\delta} \ot S(f^1\ps{2}) \cdot (f^2 \odots f^p \ot \widetilde{u}^{\pr{-1}}\ad Y(v\pr{-1}) \rt 1_{\Fc} \ot \widetilde{u}^{\pr{0}})\\
&+ ({v\pr{0}} \cdot f^1\ps{1})_{\delta} \ot S(f^1\ps{2}) \cdot (f^2 \odots f^p \ot \widetilde{u}^{\pr{-1}}v\pr{-1} \rt 1_{\Fc} \ot \ad Y(\widetilde{u}^{\pr{0}})) \\
&- (({v\pr{0}} \cdot f^1\ps{1}) \cdot Y)_{\delta} \ot S(f^1\ps{2}) \cdot (f^2 \odots f^p \ot \widetilde{u}^{\pr{-1}}v\pr{-1} \rt 1_{\Fc} \ot \widetilde{u}^{\pr{0}}).
\end{split}
\end{align}

By the commutativity of $\ad Y$ with the left $\mathcal{H}$-coaction on $\Uc$ and the antipode, we can immediately conclude that
\begin{align}
\begin{split}
& \widetilde{ad}Y(\hta(v_{\delta} \ot \widetilde{f} \ot \widetilde{u})) \\
& =({v\pr{0}} \cdot f^1\ps{1})_{\delta} \ot S(adY(f^1\ps{2})) \cdot (f^2 \odots f^p \ot \widetilde{u}^{\pr{-1}}v\pr{-1} \rt 1_{\Fc} \ot \widetilde{u}^{\pr{0}}) \\
& +({v\pr{0}} \cdot f^1\ps{1})_{\delta} \ot S(f^1\ps{2}) \cdot (adY(f^2 \odots f^p) \ot \widetilde{u}^{\pr{-1}}v\pr{-1} \rt 1_{\Fc} \ot \widetilde{u}^{\pr{0}}) \\
&+ ({v\pr{0}} \cdot f^1\ps{1})_{\delta} \ot S(f^1\ps{2}) \cdot (f^2 \odots f^p \ot {adY(\widetilde{u})}^{\pr{-1}}v\pr{-1} \rt 1_{\Fc} \ot {adY(\widetilde{u})}^{\pr{0}}) \\
&+ ({v\pr{0}} \cdot f^1\ps{1})_{\delta} \ot S(f^1\ps{2}) \cdot (f^2 \odots f^p \ot \widetilde{u}^{\pr{-1}}adY(v\pr{-1}) \rt 1_{\Fc} \ot \widetilde{u}^{\pr{0}}) \\
&- (({v\pr{0}} \cdot f^1\ps{1}) \cdot Y)_{\delta} \ot S(f^1\ps{2}) \cdot (f^2 \odots f^p \ot \widetilde{u}^{\pr{-1}}v\pr{-1} \rt 1_{\Fc} \ot \widetilde{u}^{\pr{0}}).
\end{split}
\end{align}

Then by the module compatibility \eqref{aux-25} over the bicrossed product $\mathcal{H} = \Fc \acl \Uc$, we have
\begin{equation}
(v \cdot Y) \cdot f = (v \cdot f) \cdot Y + v \cdot \ad Y(f).
\end{equation}
Therefore,
\begin{align}
\begin{split}
& \widetilde{\ad}Y(\hta(v_{\delta} \ot \widetilde{f} \ot \widetilde{u})) \\
&= ({v\pr{0}})_{\delta} \cdot \ad Y(f^1\ps{1}) \ot S(f^1\ps{2}) \cdot (f^2 \odots f^p \ot \widetilde{u}^{\pr{-1}}v\pr{-1} \rt 1_{\Fc} \ot \widetilde{u}^{\pr{0}})  \\
& +{v\pr{0}}_{\delta} \cdot f^1\ps{1} \ot S(\ad Y(f^1\ps{2})) \cdot (f^2 \odots f^p \ot \widetilde{u}^{\pr{-1}}v\pr{-1} \rt 1_{\Fc} \ot \widetilde{u}^{\pr{0}}) \\
&+ {v\pr{0}}_{\delta} \cdot f^1\ps{1} \ot S(f^1\ps{2}) \cdot (\ad Y(f^2 \odots f^p) \ot \widetilde{u}^{\pr{-1}}v\pr{-1} \rt 1_{\Fc} \ot \widetilde{u}^{\pr{0}}) \\
& +{v\pr{0}}_{\delta} \cdot f^1\ps{1} \ot S(f^1\ps{2}) \cdot (f^2 \odots f^p \ot {\ad Y(\widetilde{u})}^{\pr{-1}}v\pr{-1} \rt 1_{\Fc} \ot {\ad Y(\widetilde{u})}^{\pr{0}}) \\
& -{(v \cdot Y)\pr{0}}_{\delta} \cdot f^1\ps{1} \ot S(f^1\ps{2}) \cdot (f^2 \odots f^p \ot \widetilde{u}^{\pr{-1}}(v \cdot Y)\pr{-1} \rt 1_{\Fc} \ot \widetilde{u}^{\pr{0}}).
\end{split}
\end{align}

Finally, by the commutativity $\ad Y \circ \D = \D \circ \ad Y$ on $\Fc$ we have
\begin{align}
\begin{split}
& \widetilde{\ad}Y(\hta(v_{\delta} \ot \widetilde{f} \ot \widetilde{u})) \\
& ={v\pr{0}}_{\delta} \cdot {\ad(f^1)}\ps{1} \ot S({\ad(f^1)}\ps{2}) \cdot (f^2 \odots f^p \ot \widetilde{u}^{\pr{-1}}v\pr{-1} \rt 1_{\Fc} \ot \widetilde{u}^{\pr{0}}) \\
&+ {v\pr{0}}_{\delta} \cdot f^1\ps{1} \ot S(f^1\ps{2}) \cdot (\ad Y(f^2 \odots f^p) \ot \widetilde{u}^{\pr{-1}}v\pr{-1} \rt 1_{\Fc} \ot \widetilde{u}^{\pr{0}}) \\
& +{v\pr{0}}_{\delta} \cdot f^1\ps{1} \ot S(f^1\ps{2}) \cdot (f^2 \odots f^p \ot {\ad Y(\widetilde{u})}^{\pr{-1}}v\pr{-1} \rt 1_{\Fc} \ot {\ad Y(\widetilde{u})}^{\pr{0}}) \\
& -{(v \cdot Y)\pr{0}}_{\delta} \cdot f^1\ps{1} \ot S(f^1\ps{2}) \cdot (f^2 \odots f^p \ot \widetilde{u}^{\pr{-1}}(v \cdot Y)\pr{-1} \rt 1_{\Fc} \ot \widetilde{u}^{\pr{0}}) \\
& = \hta(\widetilde{\ad}Y(v_{\delta} \ot \widetilde{f} \ot \widetilde{u})).
\end{split}
\end{align}

We continue with the vertical operators. We see that
\begin{equation}
\vd_i \circ \widetilde{\ad}Y = \widetilde{\ad}Y \circ \vd_i, \quad 0 \leq i \leq n
\end{equation}
are similar to their  horizontal counterparts. One notes that this time the commutativity $\ad Y \circ \D = \D \circ \ad Y$ on $\Uc$ is needed.

\medskip

Commutativity with the last vertical coface operator follows, as in the horizontal case, from the AYD compatibility on $V$ over $\Hc$. Indeed,
\begin{align}
\begin{split}
& \widetilde{\ad}Y(\vd_{n+1}(v_{\delta} \ot \widetilde{f} \ot \widetilde{u})) = \widetilde{\ad}Y({v\pr{0}}_{\delta} \ot \widetilde{f} \ot \widetilde{u} \ot \wbar{v\pr{-1}}) \\
& ={v\pr{0}}_{\delta} \ot \ad Y(\widetilde{f} \ot \widetilde{u}) \ot \wbar{v\pr{-1}} + {v\pr{0}}_{\delta} \ot \widetilde{f} \ot \widetilde{u} \ot \wbar{\ad Y(v\pr{-1})} \\
& - {(v\pr{0} \cdot Y)}_{\delta} \ot \widetilde{f} \ot \widetilde{u} \ot \wbar{v\pr{-1}} \\
&= {v\pr{0}}_{\delta} \ot \ad Y(\widetilde{f} \ot \widetilde{u}) \ot \wbar{v\pr{-1}} - {(v \cdot Y)\pr{0}}_{\delta} \ot \widetilde{f} \ot \widetilde{u} \ot \wbar{(v \cdot Y)\pr{-1}}  \\
& =\vd_{n+1}(\widetilde{\ad}Y(v_{\delta} \ot \widetilde{f} \ot \widetilde{u})).
\end{split}
\end{align}

Finally, we show the commutativity of $\widetilde{\ad}Y$  with the vertical cyclic operator. First, we notice that we can rewrite it as
\begin{align}
\begin{split}
&\vta(v_{\delta} \ot \widetilde{f} \ot \widetilde{u}) = ({v\pr{0}}_{\delta} \cdot u^1\ps{4}) \cdot S^{-1}(u^1\ps{3} \rt 1_{\Fc}) \, \ot \\
& S(S^{-1}(u^1\ps{2}) \rt 1_{\Fc}) \cdot \left( S^{-1}(u^1\ps{1}) \rt \widetilde{f} \ot S(u^1\ps{5}) \cdot (u^2 \odots u^q \ot \wbar{v\pr{-1}})
\right) \\
& = {v\pr{0}}_{\delta} \cdot u^1\ps{2} \ot S^{-1}(u^1\ps{1}) \rt \widetilde{f} \ot S(u^1\ps{3}) \cdot (u^2 \odots u^q \ot \wbar{v\pr{-1}}) \\
& = {(v\pr{0} \cdot u^1\ps{3})}_{\delta}\delta(u^1\ps{2}) \ot S^{-1}(u^1\ps{1}) \rt \widetilde{f} \ot S(u^1\ps{4}) \cdot (u^2 \odots u^q \ot \wbar{v\pr{-1}}).
\end{split}
\end{align}

Therefore we have
\begin{align}
\begin{split}
& \widetilde{\ad}Y(\vta(v_{\delta} \ot \widetilde{f} \ot \widetilde{u})) \\
& =\widetilde{\ad}Y({v\pr{0}}_{\delta} \cdot u^1\ps{2} \ot S^{-1}(u^1\ps{1}) \rt \widetilde{f} \ot S(u^1\ps{3}) \cdot (u^2 \odots u^q \ot \wbar{v\pr{-1}})) \\
&= {v\pr{0}}_{\delta} \cdot u^1\ps{2} \ot \ad Y(S^{-1}(u^1\ps{1}) \rt \widetilde{f}) \ot S(u^1\ps{3}) \cdot (u^2 \odots u^q \ot \wbar{v\pr{-1}}) \\
&+ {v\pr{0}}_{\delta} \cdot u^1\ps{2} \ot S^{-1}(u^1\ps{1}) \rt \widetilde{f} \ot \ad Y(S(u^1\ps{3})) \cdot (u^2 \odots u^q \ot \wbar{v\pr{-1}}) \\
&+ {v\pr{0}}_{\delta} \cdot u^1\ps{2} \ot S^{-1}(u^1\ps{1}) \rt \widetilde{f} \ot S(u^1\ps{3}) \cdot (\ad Y(u^2 \odots u^q) \ot \wbar{v\pr{-1}}) \\
&+ {v\pr{0}}_{\delta} \cdot u^1\ps{2} \ot S^{-1}(u^1\ps{1}) \rt \widetilde{f} \ot S(u^1\ps{3}) \cdot (u^2 \odots u^q \ot \wbar{\ad Y(v\pr{-1})}) \\
&- {(v\pr{0} \cdot u^1\ps{3}Y)}_{\delta}\delta(u^1\ps{2}) \ot S^{-1}(u^1\ps{1}) \rt \widetilde{f} \ot S(u^1\ps{4}) \cdot (u^2 \odots u^q \ot \wbar{v\pr{-1}}).
\end{split}
\end{align}

Extending
\begin{equation}\label{aux-71}
\ad Y(h \rt f) = \ad Y(h) \rt f + h \rt \ad Y(f),
\end{equation}
to
\begin{equation}
\ad Y(h \rt \widetilde{f}) = \ad Y(h) \rt \widetilde{f} + h \rt \ad Y(\widetilde{f}),
\end{equation}
where $\widetilde{f} = f^1\odots f^p$, we have
\begin{align}
\begin{split}
{v\pr{0}}_{\delta} \cdot u^1\ps{2} \ot \ad Y(S^{-1}(u^1\ps{1}) \rt \widetilde{f}) \ot S(u^1\ps{3}) \cdot (u^2 \odots u^q \ot \wbar{v\pr{-1}}) \\
={v\pr{0}}_{\delta} \cdot u^1\ps{2} \ot S^{-1}(\ad Y(u^1\ps{1})) \rt \widetilde{f} \ot S(u^1\ps{3}) \cdot (u^2 \odots u^q \ot \wbar{v\pr{-1}}) \\
+{v\pr{0}}_{\delta} \cdot u^1\ps{2} \ot S^{-1}(u^1\ps{1}) \rt \ad Y(\widetilde{f}) \ot S(u^1\ps{3}) \cdot (u^2 \odots u^q \ot \wbar{v\pr{-1}}).
\end{split}
\end{align}

Next, we observe that
\begin{align}
\begin{split}
& - {(v\pr{0} \cdot u^1\ps{3}Y)}_{\delta}\delta(u^1\ps{2}) \ot S^{-1}(u^1\ps{1}) \rt \widetilde{f} \ot S(u^1\ps{4}) \cdot (u^2 \odots u^q \ot \wbar{v\pr{-1}}) \\
&= {(v\pr{0} \cdot \ad Y(u^1\ps{3}))}_{\delta}\delta(u^1\ps{2}) \ot S^{-1}(u^1\ps{1}) \rt \widetilde{f} \ot S(u^1\ps{4}) \cdot (u^2 \odots u^q \ot \wbar{v\pr{-1}}) \\
&- {(v\pr{0} \cdot Y)}_{\delta} \cdot u^1\ps{2} \ot S^{-1}(u^1\ps{1}) \rt \widetilde{f} \ot S(u^1\ps{3}) \cdot (u^2 \odots u^q \ot \wbar{v\pr{-1}}),
\end{split}
\end{align}
where,
\begin{equation}\label{aux-73}
(v \cdot \ad Y(u\ps{2}))_{\delta} \delta(u\ps{1}) = v_{\delta} \cdot \ad Y(u).
\end{equation}

Therefore we have
\begin{align}
\begin{split}
& \widetilde{\ad}Y(\vta(v_{\delta} \ot \widetilde{f} \ot \widetilde{u})) \\
& ={v\pr{0}}_{\delta} \cdot u^1\ps{2} \ot S^{-1}(u^1\ps{1}) \rt \ad Y(\widetilde{f}) \ot S(u^1\ps{3}) \cdot (u^2 \odots u^q \ot \wbar{v\pr{-1}}) \\
&+ {v\pr{0}}_{\delta} \cdot u^1\ps{2} \ot S^{-1}(\ad Y(u^1\ps{1})) \rt \widetilde{f} \ot S(u^1\ps{3}) \cdot (u^2 \odots u^q \ot \wbar{v\pr{-1}}) \\
&+ {v\pr{0}}_{\delta} \cdot \ad Y(u^1\ps{2}) \ot S^{-1}(u^1\ps{1}) \rt \widetilde{f} \ot S(u^1\ps{4}) \cdot (u^2 \odots u^q \ot \wbar{v\pr{-1}}) \\
&+ {v\pr{0}}_{\delta} \cdot u^1\ps{2} \ot S^{-1}(u^1\ps{1}) \rt \widetilde{f} \ot S(\ad Y(u^1\ps{3})) \cdot (u^2 \odots u^q \ot \wbar{v\pr{-1}}) \\
&+ {v\pr{0}}_{\delta} \cdot u^1\ps{2} \ot S^{-1}(u^1\ps{1}) \rt \widetilde{f} \ot S(u^1\ps{3}) \cdot (\ad Y(u^2 \odots u^q) \ot \wbar{v\pr{-1}}) \\
&- {(v \cdot Y)\pr{0}}_{\delta} \cdot u^1\ps{2} \ot S^{-1}(u^1\ps{1}) \rt \widetilde{f} \ot S(u^1\ps{3}) \cdot (u^2 \odots u^q \ot \wbar{(v \cdot Y)\pr{-1}}).
\end{split}
\end{align}

Then the commutativity $\D \circ \ad Y = \ad Y \circ \D$ on $\Uc$ finishes the proof, since
\begin{align}
\begin{split}
& \widetilde{\ad}Y(\vta(v_{\delta} \ot \widetilde{f} \ot \widetilde{u})) \\
&= {v\pr{0}}_{\delta} \cdot u^1\ps{2} \ot S^{-1}(u^1\ps{1}) \rt \ad Y(\widetilde{f}) \ot S(u^1\ps{3}) \cdot (u^2 \odots u^q \ot \wbar{v\pr{-1}}) \\
&+ {v\pr{0}}_{\delta} \cdot \ad Y(u^1)\ps{2} \ot S^{-1}(\ad Y(u^1)\ps{1}) \rt \widetilde{f} \ot S(\ad Y(u^1)\ps{3}) \cdot (u^2 \ot \ldots \\
& \ldots \ot u^q \ot \wbar{v\pr{-1}}) \\
&+ {v\pr{0}}_{\delta} \cdot u^1\ps{2} \ot S^{-1}(u^1\ps{1}) \rt \widetilde{f} \ot S(u^1\ps{3}) \cdot (\ad Y(u^2 \odots u^q) \ot \wbar{v\pr{-1}}) \\
&- {(v \cdot Y)\pr{0}}_{\delta} \cdot u^1\ps{2} \ot S^{-1}(u^1\ps{1}) \rt \widetilde{f} \ot S(u^1\ps{3}) \cdot (u^2 \odots u^q \ot \wbar{(v \cdot Y)\pr{-1}}) \\
& = \vta(\widetilde{\ad}Y(v_{\delta} \ot \widetilde{f} \ot \widetilde{u})).
\end{split}
\end{align}
\end{proof}

For the generators $X,Y,\delta_1 \in \mathcal{H}$, it is already known that
\begin{equation}
\ad Y(Y) = 0, \quad \ad Y(X) = X, \quad \ad Y(\delta_1) = \delta_1.
\end{equation}

We recall here the action of $Y \in s\ell(2,\Cb)$:
\begin{equation}
1_V \lt Y = 0, \quad R^X \lt Y = R^X, \quad R^Y \lt Y = 0, \quad R^Z \lt Y = -R^Z.
\end{equation}

Hence, for the weight defined by
\begin{align}\label{aux-weight}
\begin{split}
& \nm{Y} = 0, \quad \nm{X} = 1, \quad \nm{\delta_1} = 1, \\
& \nm{1_V} = 0, \quad \nm{R^X} = -1, \quad \nm{R^Y} = 0, \quad \nm{R^Z} = 1
\end{split}
\end{align}
on the cyclic complex, the operator $\widetilde{\ad}Y$ has the property
\begin{equation}
\widetilde{\ad}Y(v_{\delta} \ot \widetilde{f} \ot \widetilde{u}) = \nm{v_{\delta} \ot \widetilde{f} \ot \widetilde{u}}v_{\delta} \ot \widetilde{f} \ot \widetilde{u},
\end{equation}
where $\nm{v_{\delta} \ot \widetilde{f} \ot \widetilde{u}} := \nm{v} + \nm{\widetilde{f}} + \nm{\widetilde{u}}$.

\medskip

As a result, the operator $\widetilde{\ad}Y$ acts as a weight operator. Extending \eqref{aux-weight} to the cocyclic complex $\FZ(\Hc,\Fc,V_\d)$, we have
\begin{equation}
\FZ^{\bullet, \bullet} = \bigoplus_{k \in \mathbb{Z}}\FZ[k]^{\bullet, \bullet},
\end{equation}
where
\begin{equation}
\FZ[k] = \Big\{v_{\delta} \ot \widetilde{f} \ot \widetilde{u} \in \FZ(\Hc,\Fc,V_\d) \, \Big| \, \nm{v_{\delta} \ot \widetilde{f} \ot \widetilde{u}}= k\Big\}.
\end{equation}

As a result of Lemma \ref{aux-69}, we can say that $\FZ[k]$ is a subcomplex for any $k \in \mathbb{Z}$, and hence the cohomology inherits the grading. Thus,
\begin{equation}
HP(\mathcal{H},V_{\delta}) = \bigoplus_{k \in \mathbb{Z}}H(\FZ[k]).
\end{equation}

Moreover, using Lemma \ref{aux-68} we get the following analogue of Corollary 3.10 in \cite{MoscRang07}.

\begin{corollary}\label{corollary-weight-1}
The cohomology is captured by the weight 1 subcomplex, \ie
\begin{equation}
H(\FZ[1]) = HP(\mathcal{H},V_{\delta}), \quad H(\FZ[k]) = 0, \quad k \neq 1.
\end{equation}
\end{corollary}

\begin{proposition}
The odd and even periodic Hopf cyclic cohomologies of $\Hc_{1\rm S }\cop$ with coefficients in $V_\d$ are both one dimensional. Their classes are ``approximately" given by the cocycles
\begin{align}
&c^{\rm odd}=\one \ot \d_1 \in E_1^{1,{\rm odd}}\\
&c^{\rm even} = \one \ot X \ot Y - \one \ot Y \ot X + \one \ot Y \ot \delta_1Y \in E_1^{1,{\rm even}}.
\end{align}
in the  $E_1$-term of a natural spectral sequence associated to $V_\d$. Here, $\one := 1_V \ot \Cb_\d$.
\end{proposition}

\begin{proof}
We have seen that all cohomology classes are concentrated in the weight 1 subcomplex.
On the other hand, the $E_1$-term of the spectral sequence associated to the filtration \eqref{aux-filtration-on-V-d} on $V_{\delta}$ is
\begin{equation}
E_1^{j,i}(\mathcal{H},V_{\delta}) = H^{i+j}(C(\Uc \cl \Fc, F_jV_{\delta} / F_{j-1}V_{\delta})),
\end{equation}
where $F_{0}V_{\delta} / F_{-1}V_{\delta} \cong F_{0}V_{\delta}$, $F_{1}V_{\delta} / F_0V_{\delta} \cong \mathbb{C}_{\delta}$ and $F_{j+1}V_{\delta} / F_jV_{\delta} = 0$ for $j \geq 1$.

\medskip

Therefore,
\begin{align}
\begin{split}
& E_1^{0,i}(\mathcal{H},V_{\delta}) = 0, \quad E_1^{1,i}(\mathcal{H},V_{\delta}) = H^i(C(\Uc \cl \Fc, \mathbb{C}_{\delta})), \\
& E_1^{j,i}(\mathcal{H},V_{\delta}) = 0, \quad j \geq 1.
\end{split}
\end{align}

So the spectral sequence collapses at the $E_2$-term and we get
\begin{equation}
E_2^{0,i}(\mathcal{H},V_{\d}) \cong E_{\infty}^{0,i}(\mathcal{H},V_{\d}) = 0,
\end{equation}
\begin{equation}\label{aux-70}
E_2^{1,i}(\mathcal{H}, V_\d) \cong E_{\infty}^{1,i}(\mathcal{H}) = F_1H^i(C(\Uc \cl \Fc, V_{\delta})) / F_0H^i(C(\Uc \cl \Fc, V_{\delta})),
\end{equation}
and
\begin{equation}
E_2^{j,i}(\mathcal{H}, V_\d) \cong E_{\infty}^{j,i}(\mathcal{H},V_\d) = 0, \quad j \geq 2.
\end{equation}

By definition of the induced filtration on the cohomology groups, we have
\begin{align}
\begin{split}
& F_1H^i(C(\Uc \cl \Fc, V_{\delta})) = H^i(C(\Uc \cl \Fc, F_1V_{\delta})) = \\
& H^i(C(\Uc \cl \Fc, V_{\delta})),
\end{split}
\end{align}
and
\begin{align}
\begin{split}
& F_0H^i(C(\Uc \cl \Fc, V_{\delta})) = H^i(C(\Uc \cl \Fc, F_0V_{\delta})) \cong \\
& H^i(W(s\ell(2,\Cb), F_0V)) = 0,
\end{split}
\end{align}
where the last equality follows from Whitehead's theorem.
\end{proof}

\subsubsection{Construction of a representative cocycle for the odd class}

We first compute the odd cocycle in the total complex  ${\rm{Tot}}^{\bullet}(\Fc,\Uc,V_\d)$  of the bicomplex  \eqref{bicocyclic-bicomplex}.
Let us  recall the total mixed complex
\begin{equation}
{\rm{Tot}}^{\bullet}(\Fc,\Uc,V_\d):=\bigoplus_{p+q =\, \bullet}V_\d \ot \Fc^{\ot\;p} \ot \Uc^{\ot\;q},
\end{equation}
with the operators
\begin{align}
& \hb_p=\sum_{i=0}^{p+1} (-1)^{i}\hd_i, \qquad \vb_q=\sum_{i=0}^{q+1} (-1)^{i}\vd_i, \qquad b_T=\sum_{p+q=n}\hb_p +(-1)^p \vb_q \\
& \hB_p=(\sum_{i=0}^{p-1}(-1)^{(p-1)i}\hta^i)\overset{\ra}{\sigma}_{p-1}\hta, \quad \vB_q=(\sum_{i=0}^{q-1}(-1)^{(q-1)i}\vta^i)\uparrow\hspace{-4pt}\sigma_{q-1}\vta, \\\notag
& \hspace{3cm} \quad B_T=\sum_{p+q=n}\hB_p+(-1)^p\vB_q.
\end{align}

\begin{proposition}
Let
\begin{equation}
c':=\one \ot \d_1 \in V_\d \ot \Fc
\end{equation}
and
\begin{equation}
c''':= \bfR^Y \ot X + 2\bfR^Z \ot Y \in V_\d \ot \Uc.
\end{equation}
Then $c'+c''' \in {\rm{Tot}}^1(\Fc,\Uc,V_\d)$ is a Hochschild cocycle.
\end{proposition}

\begin{proof}
We start with the element $c':=\one \ot \d_1 \in V_\d \ot \Fc .$

\medskip

The equality $\vb(c')=0$ is immediate. Next, we observe that
\begin{align}
\begin{split}
& \hb(c') = -\bfR^X \ot \d_1 \ot X - \bfR^Y \ot \d_1 \ot Y \\
& = -\bfR^X \ot \d_1 \ot X + \bfR^Y \ot \d_1 \ot Y + \bfR^X \ot {\d_1}^2 \ot Y +\\
& - \bfR^X \ot {\d_1}^2 \ot Y - 2\bfR^Y \ot \d_1 \ot Y \\
& = \vb(\bfR^Y \ot X+ 2\bfR^Z \ot Y).
\end{split}
\end{align}
So, for $c''':=\bfR^Y \ot X + 2\bfR^Z \ot Y \in V_\d \ot \Uc$, we have $\hb(c') - \vb(c''') = 0$. Finally we notice that $\hb(c''')=0$.
\end{proof}

\begin{proposition}
The element $c'+c''' \in {\rm{Tot}}^1(\Fc,\Uc,V_\d)$ is a Connes cycle.
\end{proposition}

\begin{proof}
Using the action of $\Fc$ and $\Uc$ on $V_\d$, we directly conclude that on the one hand we have $\vB(c') = \bfR^Z$, and on the other hand  $\hB(c''') = -\bfR^Z$.
\end{proof}

Our next step is to send this cocycle to the cyclic complex $C^1(\Hc,V_\d)$. This is a two step process. We first use the Alexander-Whitney map
\begin{equation}\label{AW}
\, AW:= \bigoplus_{p+q=n} AW_{p,q}: \Tot^n(\Fc,\Uc,V_\d)\ra \FZ^{n,n},
\end{equation}
\begin{equation*}
AW_{p,q}: \Fc^{\ot p}\ot  \Uc^{\ot q}\longrightarrow \Fc^{\ot
p+q}\ot  \Uc^{\ot p+q}
\end{equation*}
\begin{equation*}
AW_{p,q}=(-1)^{p+q}\underset{
\text{p\;times}}{\underbrace{\vd_0\vd_0\dots
\vd_0}}\hd_n\hd_{n-1}\dots \hd_{p+1} \, .
\end{equation*}
to pass to the diagonal complex $\FZ(\Hc,\Fc,V_\d)$. It is readily checked that
\begin{equation}
AW_{1,0}(c') = - \one \ot \d_1 \ot 1 - \bfR^X \ot \d_1 \ot X - \bfR^Y \ot \d_1 \ot Y,
\end{equation}
as well as
\begin{equation}
AW_{0,1}(c''') = - \bfR^Y \ot 1 \ot X - 2\bfR^Z \ot 1 \ot Y.
\end{equation}

Summing up we get
\begin{equation}
c^{\rm odd}_{\rm diag} := - \one \ot \d_1 \ot 1 - \bfR^X \ot \d_1 \ot X - \bfR^Y \ot \d_1 \ot Y - \bfR^Y \ot 1 \ot X - 2\bfR^Z \ot 1 \ot Y.
\end{equation}

Finally, via the quasi-isomorphism \eqref{aux-PSI-bicrossed} we carry $c^{\rm odd}_{\rm diag} \in D^2(\Uc,\Fc,V_\d)$ to
\begin{equation}\label{c-odd}
c^{\rm odd} = - \Big(\one \ot \delta_1 + \bfR^X \ot \delta_1X+ \bfR^Y \ot (X+\delta_1Y) + 2 \bfR^Z \ot Y \Big) \in C^1(\Hc, V_\d).
\end{equation}

\begin{proposition}
The element $c^{\rm odd}$ defined in \eqref{c-odd} is a Hochschild cocycle.
\end{proposition}

\begin{proof}
We first calculate its images under the Hochschild coboundary map $b:C^1(\mathcal{H},V_{\delta}) \to C^2(\mathcal{H},V_{\delta})$. We have
\begin{align}
\begin{split}
& b(\one \ot \delta_1) = \one \ot 1_{\mathcal{H}} \ot \delta_1 - \one \ot \D(\delta_1) + \one \ot \delta_1 \ot 1 + \bfR^Y \ot \delta_1 \ot Y + \bfR^X \ot \delta_1 \ot X \\
& = \bfR^Y \ot \delta_1 \ot Y + \bfR^X \ot \delta_1 \ot X, \\[.2cm]
& b(\bfR^Y \ot X) = \bfR^Y \ot 1_{\mathcal{H}} \ot X - \bfR^Y \ot \D(X) + \bfR^Y \ot X \ot 1_{\mathcal{H}} + \bfR^X \ot X \ot \delta_1 \\
& = \bfR^X \ot X \ot \delta_1 - \bfR^Y \ot Y \ot \delta_1, \\[.2cm]
& b(\bfR^X \ot \delta_1X) = \bfR^X \ot 1_{\mathcal{H}} \ot \delta_1X - \bfR^X \ot \D(\delta_1X) + \bfR^X \ot \delta_1X \ot 1_{\mathcal{H}} \\
& = - \bfR^X \ot \delta_1 \ot X - \bfR^X \ot \delta_1Y \ot \delta_1 - \bfR^X \ot X \ot \delta_1 - \bfR^X \ot Y \ot {\delta_1}^2, \\[.2cm]
& b(\bfR^Y \ot \delta_1Y) = \bfR^Y \ot 1_{\mathcal{H}} \ot \delta_1Y - \bfR^Y \ot \D(\delta_1Y) + \bfR^Y \ot \delta_1Y \ot 1_{\mathcal{H}} + \bfR^X \ot \delta_1Y \ot \delta_1 \\
& = \bfR^X \ot \delta_1Y \ot \delta_1 - \bfR^Y \ot \delta_1 \ot Y - \bfR^Y \ot Y \ot \delta_1, \\[.2cm]
& b(\bfR^Z \ot Y) = \bfR^Z \ot 1_{\mathcal{H}} \ot Y - \bfR^Z \ot \D(Y) + \bfR^Z \ot Y \ot 1_{\mathcal{H}} \\
& + \bfR^Y \ot Y \ot \delta_1 + \frac{1}{2} \bfR^X \ot Y \ot {\delta_1}^2 \\
& = \bfR^Y \ot Y \ot \delta_1 + \frac{1}{2} \bfR^X \ot Y \ot {\delta_1}^2.
\end{split}
\end{align}

Now, summing up we get
\begin{equation}
b(\one \ot \delta_1 + \bfR^X \ot \delta_1X+ \bfR^Y \ot (X+\delta_1Y) + 2 \bfR^Z \ot Y) = 0.
\end{equation}
\end{proof}

\begin{proposition}
The Hochschild cocycle $c^{\rm odd}$ defined in \eqref{c-odd} vanishes under the Connes boundary map.
\end{proposition}

\begin{proof}
The Connes boundary is defined on the normalized bicomplex by the formula
\begin{equation}
B = \sum_{i = 0}^n (-1)^{ni}\tau^i \circ \sigma_{-1},
\end{equation}
where
\begin{equation}
\sigma_{-1}(v_{\delta} \ot h^1 \ot \ldots \ot h^{n+1}) = v_{\delta} \cdot h^1\ps{1} \ot S(h^1\ps{2}) \cdot (h^2 \ot \ldots \ot h^{n+1})
\end{equation}
is the extra degeneracy. Accordingly,
\begin{align}
\begin{split}
& B(\one \ot \delta_1 + \bfR^Y \ot X + \bfR^X \ot \delta_1X + \bfR^Y \ot \delta_1Y + 2 \cdot \bfR^Z \ot Y) = \\
& \one \cdot \delta_1 + \bfR^Y \cdot X + \bfR^X \cdot \delta_1X + \bfR^Y \cdot \delta_1Y + 2 \cdot \bfR^Z \cdot Y = \\
& \bfR^Z - \bfR^Z = 0.
\end{split}
\end{align}
\end{proof}

\subsubsection{Construction of  a representative cocycle for the even class}

Secondly, we will compute the even cocycle in the total complex  ${\rm{Tot}}^{\bullet}(\Fc,\Uc,V_\d)$  of the bicomplex  \eqref{bicocyclic-bicomplex}.

\begin{proposition}
Let
\begin{align}\label{aux-c-even}
\begin{split}
& c := \one \ot X \ot Y - \one \ot Y \ot X - \bfR^X \ot XY \ot X - \bfR^X \ot Y \ot X^2 + \bfR^Y \ot XY \ot Y \\
& \hspace{1.2cm} + \bfR^Y \ot X \ot Y^2 - \bfR^Y \ot Y \ot X \in V_\d \ot \Uc^{\ot\;2}
\end{split}
\end{align}
and
\begin{align}
\begin{split}
& c'' := - \bfR^X \ot \d_1 \ot XY^2 + \frac{2}{3}\bfR^X \ot {\d_1}^2 \ot Y^3 +\frac{1}{3}\bfR^Y \ot \d_1 \ot Y^3 \\
& \hspace{1.2cm} - \frac{1}{4}\bfR^X \ot {\d_1}^2 \ot Y^2 - \frac{1}{2}\bfR^Y \ot \d_1 \ot Y^2 \in V_\d \ot \Fc \ot \Uc.
\end{split}
\end{align}
Then  $c+c'' \in {\rm{Tot}}^2(\Fc,\Uc,V_\d)$ is a Hochschild cocycle.
\end{proposition}

\begin{proof}
We start with the element \eqref{aux-c-even}. It is immediate that $\hb(c) = 0$ . To be able to compute $\vb(c)$, we need to determine the $\Fc$-coaction
\begin{align}
\begin{split}
& \bfR^X \ot (XY)^{\pr{-1}}X^{\pr{-1}} \ot (XY)^{\pr{0}} \ot X^{\pr{0}} \\
& + \bfR^X \ot (X^2)^{\pr{-1}} \ot Y \ot (X^2)^{\pr{0}} - {\bfR^Y}\pr{0} \ot {\bfR^Y}\pr{-1}(XY)^{\pr{-1}} \ot (XY)^{\pr{0}} \ot Y \\
& - {\bfR^Y}\pr{0} \ot {\bfR^Y}\pr{-1}X^{\pr{-1}} \ot X^{\pr{0}} \ot Y^2 + {\bfR^Y}\pr{0} \ot {\bfR^Y}\pr{-1}X^{\pr{-1}} \ot Y \ot X^{\pr{0}}.
\end{split}
\end{align}

Hence, observing that
\begin{align}
\begin{split}
& \Db(X^2) = (X^2)^{\pr{0}} \ot (X^2)^{\pr{1}} = {X\ps{1}}^{\pr{0}}X^{\pr{0}} \ot {X\ps{1}}^{\pr{1}}(X\ps{2} \rt X^{\pr{1}}) \\
& = X^2 \ot 1 + 2XY \ot \d_1 + X \ot \d_1 + Y^2 \ot {\d_1}^2 + \frac{1}{2}Y \ot {\d_1}^2,
\end{split}
\end{align}
and
\begin{equation}
\Db(XY) = (XY)^{\pr{0}} \ot (XY)^{\pr{1}} = X^{\pr{0}}Y \ot X^{\pr{1}}, \quad XY \mapsto XY \ot 1 + Y^2 \ot \d_1,
\end{equation}
we have
\begin{align}
\begin{split}
& \vb_0(c) = -\bfR^X \ot \d_1 \ot Y^2 \ot X - \bfR^X \ot \d_1 \ot XY \ot Y + \bfR^X \ot {\d_1}^2 \ot Y^2 \ot Y \\
& - 2\bfR^X \ot \d_1 \ot Y \ot XY - \bfR^X \ot \d_1 \ot Y \ot X + \bfR^X \ot {\d_1}^2 \ot Y \ot Y^2 \\
& + \frac{1}{2}\bfR^X \ot {\d_1}^2 \ot Y \ot Y - \bfR^X \ot \d_1 \ot XY \ot Y + \bfR^Y \ot \d_1 \ot Y^2 \ot Y \\
& + \bfR^X \ot {\d_1}^2 \ot Y^2 \ot Y - \bfR^X \ot \d_1 \ot X \ot Y^2 + \bfR^Y \ot \d_1 \ot Y \ot Y^2 \\
& + \bfR^X \ot {\d_1}^2 \ot Y \ot Y^2 + \bfR^X \ot \d_1 \ot Y \ot X - \bfR^Y \ot \d_1 \ot Y \ot Y \\
& - \bfR^X \ot {\d_1}^2 \ot Y \ot Y.
\end{split}
\end{align}

It is now clear that
\begin{align}
\begin{split}
& \vb(c) = \hb\big(\bfR^X \ot \d_1 \ot XY^2 - \frac{2}{3}\bfR^X \ot {\d_1}^2 \ot Y^3 -\frac{1}{3}\bfR^Y \ot \d_1 \ot Y^3 \\
& \hspace{1.6cm} + \frac{1}{4}\bfR^X \ot {\d_1}^2 \ot Y^2 + \frac{1}{2}\bfR^Y \ot \d_1 \ot Y^2 \big).
\end{split}
\end{align}

Therefore, for the element
\begin{align}
\begin{split}
& c'' := - \bfR^X \ot \d_1 \ot XY^2 + \frac{2}{3}\bfR^X \ot {\d_1}^2 \ot Y^3 +\frac{1}{3}\bfR^Y \ot \d_1 \ot Y^3 \\
& \hspace{1.2cm} - \frac{1}{4}\bfR^X \ot {\d_1}^2 \ot Y^2 - \frac{1}{2}\bfR^Y \ot \d_1 \ot Y^2,
\end{split}
\end{align}
we have $\hb(c'') + \vb(c) = 0$.

\medskip

Finally we observe that
\begin{align}
\begin{split}
& \vb(c'') = \bfR^X \ot \d_1 \ot \d_1 \ot Y^3 - \frac{4}{3}\bfR^X \ot \d_1 \ot \d_1 \ot Y^3 + \frac{1}{3}\bfR^X \ot \d_1 \ot \d_1 \ot Y^3 \\
& \hspace{1.5cm} + \frac{1}{2}\bfR^X \ot \d_1 \ot \d_1 \ot Y^2 - \frac{1}{2}\bfR^X \ot \d_1 \ot \d_1 \ot Y^2 = 0.
\end{split}
\end{align}
\end{proof}

\begin{proposition}
The element $c+c'' \in {\rm{Tot}}^2(\Fc,\Uc,V_\d)$ vanishes under the  Connes boundary map.
\end{proposition}

\begin{proof}
As above, we start with \eqref{aux-c-even}. To apply $\hB$, it suffices to consider the horizontal extra degeneracy operator $\overset{\ra}{\sigma}_{-1}:=\overset{\ra}{\sigma}_{1}\hta$. We have,
\begin{equation}
\overset{\ra}{\sigma}_{-1}(\one \ot X \ot Y) = \one \cdot X\ps{1} \ot S(X\ps{2})Y = -\one \ot XY,
\end{equation}
and
\begin{equation}
\overset{\ra}{\sigma}_{-1}(\one \ot Y \ot X) = \one \cdot Y\ps{1} \ot S(Y\ps{2})X = \one \ot X - \one \ot YX = -\one \ot XY.
\end{equation}
We have just proved that  $\overset{\ra}{\sigma}_{-1}(\one \ot X \ot Y - \one \ot Y \ot X) = 0$. For the other  terms in $c \in V_\d \ot \Uc^{\ot\;2}$ we proceed as follows:
\begin{align}
\begin{split}
& \overset{\ra}{\sigma}_{-1}(\bfR^X \ot XY \ot X) = \bfR^X \cdot X\ps{1}Y\ps{1} \ot S(Y\ps{2})S(X\ps{2})X \\
& = \bfR^X \cdot XY \ot X + \bfR^X \ot YX^2 - \bfR^X \cdot X \ot YX - \bfR^X \cdot Y \ot X^2 \\
& = \bfR^Y \ot XY + \bfR^X \ot YX^2 - 2\bfR^X \ot X^2,
\end{split}
\end{align}
\begin{align}
\begin{split}
& \overset{\ra}{\sigma}_{-1}(\bfR^Y \ot XY \ot Y) \\
& =\bfR^Y \cdot XY \ot Y + \bfR^Y \ot YXY - \bfR^Y \cdot X \ot Y^2 - \bfR^Y \cdot Y \ot XY \\
& =\bfR^Y \ot XY^2 + \bfR^Z \ot Y^2,
\end{split}
\end{align}
\begin{equation}
\overset{\ra}{\sigma}_{-1}(\bfR^X \ot Y \ot X^2) = 2\bfR^X \ot X^2 - \bfR^X \ot YX^2,
\end{equation}
\begin{equation}
\overset{\ra}{\sigma}_{-1}(\bfR^Y \ot X \ot Y^2) = -\bfR^Z \ot Y^2 - \bfR^Y \ot XY^2,
\end{equation}
and finally
\begin{equation}
\overset{\ra}{\sigma}_{-1}(\bfR^Y \ot Y \ot X) = \bfR^Y \ot X - \bfR^Y \ot YX = - \bfR^Y \ot XY.
\end{equation}

This way we prove
\begin{align}
\begin{split}
&\overset{\ra}{\sigma}_{-1}(- \bfR^X \ot XY \ot X - \bfR^X \ot Y \ot X^2 + \bfR^Y \ot XY \ot Y \\
& +\bfR^Y \ot X \ot Y^2 - \bfR^Y \ot Y \ot X) = 0.
\end{split}
\end{align}

Hence we conclude $\hB(c)=0$.

\medskip

Since the action of $\d_1$ on $F_0V_\d = \Cb\langle \bfR^X, \bfR^Y, \bfR^Z \rangle$ is trivial, $\vB(c'')=0$. The horizontal counterpart $\hB(c'') = 0$ follows from the observations
\begin{equation}
(\bfR^X \ot {\d_1}^2) \cdot Y = \bfR^X \cdot Y\ps{1} \ot S(Y\ps{2}) \rt {\d_1}^2 = 0
\end{equation}
and
\begin{equation}
(\bfR^Y \ot \d_1) \cdot Y = \bfR^Y \cdot Y\ps{1} \ot S(Y\ps{2}) \rt \d_1 = 0.
\end{equation}
\end{proof}

Next, we send the element $c+c'' \in {\rm Tot}^2(\Fc, \Uc,V_\d)$ to the cyclic complex $C(\Hc,V_\d)$. As before, in the first step we use the Alexander-Whitney map to land in the diagonal complex $\FZ(\Hc,\Fc,V_\d)$. To this end, we have
\begin{align}
\begin{split}
& AW_{0,2}(c) = \; \vd_0\vd_0(c) \\
& = \one \ot 1 \ot 1 \ot X \ot Y - \one \ot 1 \ot 1 \ot Y \ot X - \bfR^X \ot 1 \ot 1 \ot XY \ot X \\
&  - \bfR^X\ot 1 \ot 1 \ot Y \ot X^2 + \bfR^Y \ot 1 \ot 1 \ot XY \ot Y + \bfR^Y \ot 1 \ot 1 \ot X \ot Y^2 \\
& - \bfR^Y \ot 1 \ot 1 \ot Y \ot X,
\end{split}
\end{align}
and
\begin{align}
\begin{split}
& AW_{1,1}(c) =\; \vd_0\hd_1(c) \\
& = - \bfR^X \ot 1 \ot \d_1 \ot XY^2 \ot 1 + \frac{2}{3}\bfR^X \ot 1 \ot {\d_1}^2 \ot Y^3 \ot 1 +\frac{1}{3}\bfR^Y \ot 1 \ot \d_1 \ot Y^3 \ot 1 \\
& - \frac{1}{4}\bfR^X \ot 1 \ot {\d_1}^2 \ot Y^2 \ot 1 - \frac{1}{2}\bfR^Y \ot 1 \ot \d_1 \ot Y^2 \ot 1.
\end{split}
\end{align}

As a result, we obtain the element
\begin{align}
\begin{split}
& c^{{\rm even}}_{{\rm diag}} = \one \ot 1 \ot 1 \ot X \ot Y - \one \ot 1 \ot 1 \ot Y \ot X - \bfR^X \ot 1 \ot 1 \ot XY \ot X \\
& - \bfR^X\ot 1 \ot 1 \ot Y \ot X^2 + \bfR^Y \ot 1 \ot 1 \ot XY \ot Y + \bfR^Y \ot 1 \ot 1 \ot X \ot Y^2 \\
& - \bfR^Y \ot 1 \ot 1 \ot Y \ot X - \bfR^X \ot 1 \ot \d_1 \ot XY^2 \ot 1 + \frac{2}{3}\bfR^X \ot 1 \ot {\d_1}^2 \ot Y^3 \ot 1 \\
& +\frac{1}{3}\bfR^Y \ot 1 \ot \d_1 \ot Y^3 \ot 1 - \frac{1}{4}\bfR^X \ot 1 \ot {\d_1}^2 \ot Y^2 \ot 1 - \frac{1}{2}\bfR^Y \ot 1 \ot \d_1 \ot Y^2 \ot 1.
\end{split}
\end{align}

In the second step, we use the map \eqref{aux-PSI-bicrossed} to obtain
\begin{align}\label{even-cocycle}
\begin{split}
& c^{{\rm even}}:= \Psi\big(c^{{\rm even}}_{{\rm diag}}\big) = \one \ot X \ot Y - \one \ot Y \ot X + \one \ot Y \ot \d_1Y - \bfR^X \ot XY \ot X \\
& - \bfR^X \ot Y^2 \ot \d_1X - \bfR^X \ot Y \ot X^2 + \bfR^Y \ot XY \ot Y + \bfR^Y \ot Y^2 \ot \d_1Y \\
& + \bfR^Y \ot X \ot Y^2 + \bfR^Y \ot Y \ot \d_1Y^2 - \bfR^Y \ot Y \ot X - \bfR^X \ot XY^2 \ot \d_1 \\
& - \frac{1}{3}\bfR^X \ot Y^3 \ot {\d_1}^2 + \frac{1}{3} \bfR^Y \ot Y^3 \ot \d_1 - \frac{1}{4} \bfR^X \ot Y^2 \ot {\d_1}^2 - \frac{1}{2} \bfR^Y \ot Y^2 \ot \d_1,\\
\end{split}
\end{align}
in  $ C^2(\Hc, V_\d)$.

\begin{proposition}
The element  $c^{\rm even}$ defined in \eqref{even-cocycle} is a Hochschild cocycle.
\end{proposition}

\begin{proof}
We first recall that
\begin{align}
\begin{split}
& b(\one \ot X \ot Y - \one \ot Y \ot X + \one \ot Y \ot \d_1Y) = \\
& -\bfR^X \ot X \ot Y \ot X - \bfR^Y \ot X \ot Y \ot Y + \bfR^X \ot Y \ot X \ot X \\
& + \bfR^Y \ot Y \ot X \ot Y - \bfR^X \ot Y \ot \d_1Y \ot X - \bfR^Y \ot Y \ot \d_1Y \ot Y.
\end{split}
\end{align}

Next we compute
\begin{align}
\begin{split}
& b(\bfR^X \ot XY \ot X) = - \bfR^X \ot X \ot Y \ot X - \bfR^X \ot Y \ot X \ot X \\
& - \bfR^X \ot Y^2 \ot \d_1 \ot X - \bfR^X \ot Y \ot \d_1Y \ot X + \bfR^X \ot XY \ot Y \ot \d_1, \\[.2cm]
\end{split}
\end{align}

\begin{align}
\begin{split}
& b(\bfR^X \ot Y^2 \ot \d_1X) = - 2\bfR^X \ot Y \ot Y \ot \d_1X + \bfR^X \ot Y^2 \ot \d_1 \ot X  \\
& + \bfR^X \ot Y^2 \ot X \ot \d_1 + \bfR^X \ot Y^2 \ot \d_1Y \ot \d_1 + \bfR^X \ot Y^2 \ot Y \ot {\d_1}^2, \\[.2cm]
\end{split}
\end{align}

\begin{align}
\begin{split}
& b(\bfR^X \ot Y \ot X^2) = 2 \bfR^X \ot Y \ot X \ot X + \bfR^X \ot Y \ot XY \ot \d_1 \\
& + \bfR^X \ot Y \ot Y \ot X\d_1 + \bfR^X \ot Y \ot YX \ot \d_1 \\
& + \bfR^X \ot Y \ot Y \ot \d_1X + \bfR^X \ot Y \ot Y^2 \ot {\d_1}^2, \\[.2cm]
\end{split}
\end{align}

\begin{align}
\begin{split}
& b(\bfR^Y \ot XY \ot Y) = - \bfR^Y \ot X \ot Y \ot Y - \bfR^Y \ot Y \ot X \ot Y  \\
& - \bfR^Y \ot Y^2 \ot \d_1 \ot Y - \bfR^Y \ot Y \ot \d_1Y \ot Y - \bfR^X \ot XY \ot Y \ot \d_1, \\[.2cm]
\end{split}
\end{align}

\begin{align}
\begin{split}
& b(\bfR^Y \ot Y^2 \ot \d_1Y) = -2 \bfR^Y \ot Y \ot Y \ot \d_1Y + \bfR^Y \ot Y^2 \ot \d_1 \ot Y \\
& + \bfR^Y \ot Y^2 \ot Y \ot \d_1 - \bfR^X \ot Y^2 \ot \d_1Y \ot \d_1,
\end{split}
\end{align}

as well as

\begin{align}
\begin{split}
& b(\bfR^Y \ot X \ot Y^2) = \\
& - \bfR^Y \ot Y \ot \d_1 \ot Y^2 + 2\bfR^Y \ot X \ot Y \ot Y - \bfR^X \ot X \ot Y^2 \ot \d_1, \\[.2cm]
\end{split}
\end{align}

\begin{align}
\begin{split}
& b(\bfR^Y \ot Y \ot \d_1Y^2) = \bfR^Y \ot Y \ot \d_1 \ot Y^2 + \bfR^Y \ot Y \ot Y^2 \ot \d_1  \\
& + 2\bfR^Y \ot Y \ot \d_1Y \ot Y + 2\bfR^Y \ot Y \ot Y \ot \d_1Y - \bfR^X \ot Y \ot \d_1Y^2 \ot \d_1, \\[.2cm]
\end{split}
\end{align}

\begin{align}
\begin{split}
& b(\bfR^Y \ot Y \ot X) = \bfR^Y \ot Y \ot Y \ot \d_1 - \bfR^X \ot Y \ot X \ot \d_1, \\[.2cm]
\end{split}
\end{align}

\begin{align}
\begin{split}
& b(\bfR^X \ot XY^2 \ot \d_1) = - \bfR^X \ot X \ot Y^2 \ot \d_1 - \bfR^X \ot Y^2 \ot X \ot \d_1 \\
& - 2 \bfR^X \ot XY \ot Y \ot \d_1 -2 \bfR^X \ot Y \ot XY \ot \d_1 \\
& - 2\bfR^X \ot Y^2 \ot \d_1Y \ot \d_1 - \bfR^X \ot Y^3 \ot \d_1 \ot \d_1 \\
& - \bfR^X \ot Y \ot \d_1Y^2 \ot \d_1, \\[.2cm]
\end{split}
\end{align}

\begin{align}
\begin{split}
& b(\bfR^X \ot Y^3 \ot {\d_1}^2) = \\
& -3 \bfR^X \ot Y^2 \ot Y \ot {\d_1}^2 - 3\bfR^X \ot Y \ot Y^2 \ot {\d_1}^2 + 2 \bfR^X \ot Y^3 \ot \d_1 \ot \d_1, \\[.2cm]
\end{split}
\end{align}

\begin{align}
\begin{split}
& b(\bfR^Y \ot Y^3 \ot \d_1) = \\
& -3 \bfR^Y \ot Y^2 \ot Y \ot \d_1 - 3 \bfR^Y \ot Y \ot Y^2 \ot \d_1 - \bfR^X \ot Y^3 \ot \d_1 \ot \d_1, \\[.2cm]
\end{split}
\end{align}

\begin{align}
\begin{split}
& b(\bfR^X \ot Y^2 \ot {\d_1}^2) = -2 \bfR^X \ot Y \ot Y \ot {\d_1}^2 + 2 \bfR^X \ot Y^2 \ot \d_1 \ot \d_1, \\[.2cm]
\end{split}
\end{align}

and finally

\begin{align}
\begin{split}
& b(\bfR^Y \ot Y^2 \ot \d_1) = -2 \bfR^Y \ot Y \ot Y \ot \d_1 - \bfR^X \ot Y^2 \ot \d_1 \ot \d_1.
\end{split}
\end{align}
Summing up, we get the result.
\end{proof}

\begin{proposition}
The Hochschild cocycle $c^{\rm even}$ defined in  \eqref{even-cocycle} vanishes under the  Connes boundary map.
\end{proposition}

\begin{proof}
We will first prove that the extra degeneracy operator $\s_{-1}$ vanishes on
\begin{align}
\begin{split}
& c := \one \ot X \ot Y - \one \ot Y \ot X + \one \ot Y \ot \d_1Y - \bfR^X \ot XY \ot X \\
& - \bfR^X \ot Y^2 \ot \d_1X - \bfR^X \ot Y \ot X^2 + \bfR^Y \ot XY \ot Y + \bfR^Y \ot Y^2 \ot \d_1Y \\
& + \bfR^Y \ot X \ot Y^2 + \bfR^Y \ot Y \ot \d_1Y^2 - \bfR^Y \ot Y \ot X \in C^2(\Hc, V_\d).
\end{split}
\end{align}
We observe that
\begin{align}
\begin{split}
&\s_{-1}(\one \ot X \ot Y - \one \ot Y \ot X + \one \ot Y \ot \d_1Y) = 0,\\[.2cm]
&\s_{-1}(\bfR^X \ot XY \ot X) = \bfR^Y \ot XY + \bfR^X \ot X^2Y - \bfR^X \ot \d_1XY^2,\\[.2cm]
&\s_{-1}(\bfR^X \ot Y^2 \ot \d_1X) = \bfR^X \ot \d_1XY^2,\\[.2cm]
&\s_{-1}(\bfR^X \ot Y \ot X^2) = -\bfR^X \ot X^2Y,\\[.2cm]
&\s_{-1}(\bfR^Y \ot XY \ot Y) = \bfR^Z \ot Y^2 + \bfR^Y \ot XY^2 - \bfR^Y \ot \d_1Y^3,\\[.2cm]
&\s_{-1}(\bfR^Y \ot Y^2 \ot \d_1Y) = \bfR^Y \ot \d_1Y^3,\\[.2cm]
&\s_{-1}(\bfR^Y \ot X \ot Y^2) = -\bfR^Z \ot Y^2 - \bfR^Y \ot XY^2 + \bfR^Y \ot \d_1Y^3,\\[.2cm]
&\s_{-1}(\bfR^Y \ot Y \ot \d_1Y^2) = - \bfR^Y \ot \d_1Y^3,\\[.2cm]
&\s_{-1}(\bfR^Y \ot Y \ot X) = -\bfR^Y \ot XY.
\end{split}
\end{align}

As a result, we obtain $\s_{-1}(c) = 0$. In the second step, we prove that the Connes boundary map $B$ vanishes on
\begin{align}
\begin{split}
& c'' := - \bfR^X \ot XY^2 \ot \d_1 - \frac{1}{3}\bfR^X \ot Y^3 \ot {\d_1}^2 + \frac{1}{3} \bfR^Y \ot Y^3 \ot \d_1 \\
& - \frac{1}{4} \bfR^X \ot Y^2 \ot {\d_1}^2 - \frac{1}{2} \bfR^Y \ot Y^2 \ot \d_1 \in C^2(\Hc, V_\d).
\end{split}
\end{align}

Indeed, as in this case $B = (\Id - \tau) \circ \s_{-1}$, it suffices to observe that
\begin{align}
\begin{split}
&\s_{-1}(\bfR^X \ot XY^2 \ot \d_1)\\
& = - \bfR^Y \ot \d_1Y^2 - \bfR^X \ot \d_1XY^2 - \frac{1}{2}\bfR^X \ot
 {\d_1}^2Y^2 + \bfR^X \ot {\d_1}^2Y^3,\\[.2cm]
&\s_{-1}(\bfR^X \ot Y^3 \ot {\d_1}^2) = - \bfR^X \ot {\d_1}^2Y^3,\\[.2cm]
&\s_{-1}(\bfR^Y \ot Y^3 \ot \d_1) = - \bfR^Y \ot \d_1Y^3,\\[.2cm]
&\s_{-1}(\bfR^X \ot Y^2 \ot {\d_1}^2) = \bfR^X \ot {\d_1}^2Y^2,\\[.2cm]
&\s_{-1}(\bfR^Y \ot Y^2 \ot \d_1) = \bfR^Y \ot \d_1Y^2,\\[.2cm]
\end{split}
\end{align}
together with
\begin{align}
\begin{split}
&\tau(\bfR^Y \ot \d_1Y^2) = - \bfR^Y \ot \d_1Y^2 - \bfR^X \ot {\d_1}^2Y^2,\\[.2cm]
&\tau(\bfR^X \ot {\d_1}^2Y^2) = \bfR^X \ot {\d_1}^2Y^2,\\[.2cm]
&\tau(\bfR^Y \ot \d_1Y^3) = \bfR^Y \ot \d_1Y^3 + \bfR^X \ot {\d_1}^2Y^3,\\[.2cm]
&\tau(\bfR^X \ot {\d_1}^2Y) = - \bfR^Y \ot {\d_1}^2Y,\\[.2cm]
&\tau(\bfR^X \ot {\d_1}^2Y^3) = - \bfR^Y \ot {\d_1}^2Y^3,\\[.2cm]
&\tau(\bfR^X \ot \d_1XY^2) = \bfR^Y \ot \d_1Y^2 + \bfR^X \ot \d_1XY^2 + \frac{1}{2}\bfR^X \ot {\d_1}^2Y^2 - \bfR^X \ot {\d_1}^2Y^3.
\end{split}
\end{align}
\end{proof}

We summarize our results in the following theorem.

\begin{theorem}
The odd and even periodic Hopf cyclic cohomologies of the Schwarzian Hopf algebra $\Hc_{\rm 1S}\cop$ with coefficients in the 4-dimensional SAYD module $V_\d = S({s\ell(2,\Cb)}^\ast)\nsb{1}$ are given by
\begin{equation}
HP^{\rm odd}(\Hc_{\rm 1S}\cop,V_\d) = \Cb \Big\langle \one \ot \delta_1 + \bfR^X \ot \delta_1X + \bfR^Y \ot (X+ \delta_1Y) + 2 \bfR^Z \ot Y \Big\rangle,
\end{equation}
and
\begin{align}
\begin{split}
& HP^{\rm even}(\Hc_{\rm 1S}\cop,V_\d) = \Cb \Big\langle \one \ot X \ot Y - \one \ot Y \ot X + \one \ot Y \ot \d_1Y \\
&  - \bfR^X \ot XY \ot X - \bfR^X \ot Y^2 \ot \d_1X  - \bfR^X \ot Y \ot X^2 + \bfR^Y \ot XY \ot Y \\
& + \bfR^Y \ot Y^2 \ot \d_1Y + \bfR^Y \ot X \ot Y^2 + \bfR^Y \ot Y \ot \d_1Y^2 - \bfR^Y \ot Y \ot X \\
& - \bfR^X \ot XY^2 \ot \d_1- \frac{1}{3}\bfR^X \ot Y^3 \ot {\d_1}^2 + \frac{1}{3} \bfR^Y \ot Y^3 \ot \d_1 \\
& - \frac{1}{4} \bfR^X \ot Y^2 \ot {\d_1}^2 - \frac{1}{2} \bfR^Y \ot Y^2 \ot \d_1 \Big\rangle.
\end{split}
\end{align}
\end{theorem}

\begin{remark}{\rm
We note that we have also verified that
\begin{equation}
HP(\Hc_{\rm 1S}\cop,V_\d) \cong \widetilde{HP}(s\ell(2),V)
\end{equation}
which in turn is a consequence of Theorem \ref{aux-63}.
}\end{remark}

\subsection{Hopf-cyclic cohomology of the projective Hopf algebra}

In this subsection we identify the periodic cyclic cohomology of the projective Hopf algebra in terms of the cohomology of a (perturbed) Weil complex. More explicitly,
\begin{equation}
HP(\Hc_{n\rm Proj},\,\Cb_\d) = H(W(g\ell(n),V_{n\rm Proj})),
\end{equation}
for some (unimodular) SAYD module $V_{n\rm Proj}$ over $g\ell(n)$. Here $(1,\d)$ is the canonical MPI associated to the Hopf algebra $\Hc_{n\rm Proj}$ via Theorem \ref{theorem-MPI}. In view of the fact that the projective Hopf algebra $\Hc_{n\rm Proj}$ is a quotient of the Connes-Moscovici Hopf algebra $\Hc_n$, this result is an analogue of
\begin{equation}
HP(\Hc_n,\,\Cb_\d) = H_{\rm GF}(\Fa_n) = H(W(g\ell(n),V_n), \qquad V_n := S(g\ell(n)^\ast)\nsb{n},
\end{equation}
where $\Fa_n$ is the Lie algebra of formal vector fields on $\Rb^n$.

\medskip

As the Hopf algebra $\Hc_{n\rm Proj}$ is constructed from the Lie algebra decomposition $pg\ell(n) = g\ell(n)^{\rm aff} \oplus \Rb^n$, we know that $\Hc_{n\rm Proj} = R(\Rb^n)\acl U(g\ell(n)^{\rm aff})$. Considering the Levi decomposition of the abelian Lie algebra $\Rb^n$,
\begin{equation}
\Rb^n = 0 \ltimes \Rb^n,
\end{equation}
by Theorem \ref{aux-63} we have
\begin{equation}
HP(\Hc_{n\rm Proj},\,\Cb_\d) = \widetilde{HP}(pg\ell(n),\Cb) = H(W(pg\ell(n),\Cb)),
\end{equation}
where $\Cb$ is the trivial (unimodular) SAYD module over $pg\ell(n)$.

\medskip

Let
\begin{equation}
\left\{\t^i,\t^j_k,\t_l\,|\,1\leq i,j,k,l\leq n\right\}
\end{equation}
be a dual basis of $pg\ell(n)$. Regarding $\t^i$'s and $\t_l$'s as coefficients, we define
\begin{equation}
V_{n\rm Proj} = \left\{(\t^{i_1}\wdots\t^{i_r})\wg(\t_{l_1}\wdots\t_{l_s}) \,|\, r,s\geq0 \right\}.
\end{equation}
The space $V_{n\rm Proj}$ is a $g\ell(n)$-module under the coadjoint action. Explicitly,
\begin{equation}
\one \cdot Y_p^q = 0, \quad \t^i \cdot Y_p^q = \d^i_p\t^q, \quad \t_l \cdot Y_p^q = -\d^q_l\t^p.
\end{equation}
Now, for the sake of avoiding confusion, let us denote the dual basis elements of $g\ell(n)$ by ${\widehat{Y}}^p_q$ while defining the action
\begin{equation}
\t^i \lt {\widehat{Y}}^p_q = 0= \t_l \lt {\widehat{Y}}^p_q, \quad \one \lt {\widehat{Y}}^p_q = -\left(\t^p\wg\t_q + \sum_a\d^p_q\t^a\wg\t_a\right)
\end{equation}
of the abelian Lie algebra $g\ell(n)^\ast$. It is fairly straightforward to observe that this is indeed an action.

\begin{lemma}
The space $V_{n\rm Proj}$ is a right-left AYD module over $g\ell(n)$.
\end{lemma}
\begin{proof}
In view of Proposition \ref{23} we will prove that $V_{n\rm Proj}$ is a module over $\widetilde{g\ell(n)} = g\ell(n)^\ast \rtimes g\ell(n)$. To this end, we need to check that
\begin{equation}
v \cdot [{\widehat{Y}}^a,Y_b] = \sum_c C^a_{bc} v \lt{\widehat{Y}}^c
\end{equation}
for
\begin{equation}
a = \left(\begin{array}{c}
                                                          p \\
                                                          q
                                                        \end{array}
\right), \quad b= \left(\begin{array}{c}
                  r \\
                  s
                \end{array}
\right), \quad c = \left(\begin{array}{c}
               k \\
               l
             \end{array}
\right).
\end{equation}
Since we can write
\begin{equation}
C^a_{bc} = {\widehat{Y}}^a([Y_b,Y_c]),
\end{equation}
we observe that we need to check that
\begin{equation}\label{aux-AYD-check}
v \cdot [{\widehat{Y}}^p_q,Y_r^s] = \sum_{k,l} {\widehat{Y}}^p_q([Y_r^s,Y_k^l]) v \lt{\widehat{Y}}^k_l.
\end{equation}
In case $v = \t^i$ or $v = \t_l$, both the left and the right hand sides of \eqref{aux-AYD-check} are zero.  Hence it will suffice to consider $v = \one$, in which case the left hand side is
\begin{align}
\begin{split}
& \one \cdot [{\widehat{Y}}^p_q,Y_r^s] = \left(\one \lt {\widehat{Y}}^p_q\right) \cdot Y_r^s - \left(\one \cdot Y_r^s\right) \lt {\widehat{Y}}^p_q \\
& = \left(\one \lt {\widehat{Y}}^p_q\right) \cdot Y_r^s \\
& = -\left(\t^p\wg\t_q + \sum_a\d^p_q\t^a\wg\t_a\right) \cdot Y_r^s \\
& = - \d^p_r\t^s\wg\t_q + \d^s_q\t^p\wg\t_r,
\end{split}
\end{align}
and the right hand side is
\begin{align}
\begin{split}
& \sum_{k,l} {\widehat{Y}}^p_q([Y_r^s,Y_k^l]) \one \cdot{\widehat{Y}}^k_l = \sum_{k,l} {\widehat{Y}}^p_q(\d^s_kY_r^l - \d^l_rY_k^s) \one \lt{\widehat{Y}}^k_l \\
& = \one \lt \left(\d^p_r{\widehat{Y}}^s_q - \d^s_q{\widehat{Y}}^p_r\right) \\
& = -\left(\d^p_r \t^s\wg\t_q + \sum_a \d^p_r\d^s_q\t^a\wg\t_a\right) + \left(\d^s_q \t^p\wg\t_r + \sum_a \d^s_q\d^p_r\t^a\wg\t_a\right)\\
& = - \d^p_r\t^s\wg\t_q + \d^s_q\t^p\wg\t_r.
\end{split}
\end{align}
\end{proof}

\begin{lemma}
The space $V_{n\rm Proj}$ is (unimodular) stable over $g\ell(n)$.
\end{lemma}
\begin{proof}
In view of \eqref{aux-unimodular-stable} we need to observe that
\begin{equation}
v \cdot \left(\sum_{k,l} {\widehat{Y}}^k_l \rtimes Y_k^l\right) = 0, \qquad \forall v \in V_{n\rm Proj}.
\end{equation}
For $v = \t^i$ or $v = \t_l$ this is immediate. For $v = \one$ we have
\begin{align}
\begin{split}
& \one \cdot \left(\sum_{k,l} {\widehat{Y}}^k_l \rtimes Y_k^l\right) = - \sum_{k,l} \left(\t^k\wg\t_l + \sum_a \d^k_l\t^a\wg\t_a\right) \cdot Y_k^l \\
& -\sum_{k,l}\left(\t^l\wg\t_l - \t^k\wg\t_k\right) =0.
\end{split}
\end{align}
\end{proof}

Let us fix the notation
\begin{align}
\begin{split}
& \t^J_K := \t^{j_1}_{k_1}\wdots\t^{j_s}_{k_s}, \qquad s \geq 0 \\
& \t^I:=\t^{i_1}\wdots\t^{i_t}, \qquad t \geq 0 \\
& \t_L:=\t_{l_1}\wdots\t_{l_r}, \qquad r \geq 0
\end{split}
\end{align}
for use in the next lemma.

\begin{lemma}
The map
\begin{equation}
id:W(pg\ell(n,\Rb),\Cb) \to W(g\ell(n,\Rb),V_{n\rm Proj}), \qquad \t^J_K\t^I\t_L \mapsto \t^J_K\ot\t^I\t_L
\end{equation}
commutes with the differentials.
\end{lemma}
\begin{proof}
Let $v = \left(\t^{j_1}_{k_1}\wdots\t^{j_s}_{k_s}\right)\wg\left(\t^{i_1}\wdots\t^{i_t}\right)\wg\left(\t_{l_1}\wdots\t_{l_r}\right)$ such that $s \geq 0$ and $t+r>0$. Then for $v \in W(pg\ell(n),\Cb)$ we have
\begin{align}
\begin{split}
& \left(d_{\rm CE}+d_{\rm K}\right)(v) = d_{\rm CE}(v) \\
& = d_{\rm CE}\left(\t^{j_1}_{k_1}\wdots\t^{j_s}_{k_s}\right)\wg\left(\t^{i_1}\wdots\t^{i_t}\right)\wg\left(\t_{l_1}\wdots\t_{l_r}\right) \\
& + (-1)^s\left(\t^{j_1}_{k_1}\wdots\t^{j_s}_{k_s}\right)\wg d_{\rm CE}\left(\t^{i_1}\wdots\t^{i_t}\right)\wg\left(\t_{l_1}\wdots\t_{l_r}\right) \\
&+(-1)^{s+r}\left(\t^{j_1}_{k_1}\wdots\t^{j_s}_{k_s}\right)\wg\left(\t^{i_1}\wdots\t^{i_t}\right)\wg d_{\rm CE}\left(\t_{l_1}\wdots\t_{l_r}\right) \\
& = d_{\rm CE}\left(\t^{j_1}_{k_1}\wdots\t^{j_s}_{k_s}\right)\wg\left(\t^{i_1}\wdots\t^{i_t}\right)\wg\left(\t_{l_1}\wdots\t_{l_r}\right) \\
& -\sum_{a,b}\t^a_b\left(\t^{j_1}_{k_1}\wdots\t^{j_s}_{k_s}\right)\wg\left[\left(\t^{i_1}\wdots\t^{i_t}\right)\wg\left(\t_{l_1}\wdots\t_{l_r}\right)\right]\cdot Y_a^b
\end{split}
\end{align}
On the other hand, since
\begin{equation}
\left(d_{\rm CE}+d_{\rm K}\right)(id(v)) = d_{\rm CE}(id(v))
\end{equation}
in case of $t+r >0$, we immediately conclude that
\begin{equation}
\left(d_{\rm CE}+d_{\rm K}\right)\circ id = id \circ \left(d_{\rm CE}+d_{\rm K}\right), \qquad \text{for}\,\,t+r>0.
\end{equation}
For $t=r=0$ we consider $w = \t^{j_1}_{k_1}\wdots\t^{j_s}_{k_s}$, $s \geq 0$. First, for $\t^j_k \in W(pg\ell(n),\Cb)$,
\begin{equation}
d_{\rm CE}\left(\t^j_k\right) =\left(d_{\rm CE}+d_{\rm K}\right)\left(\t^j_k\right)= -\sum_a \t^j_a\wg\t^a_k - \t^j\wg\t_k - \sum_a \d^j_k\t^a\wg\t_a.
\end{equation}
Now since
\begin{equation}
\left(d_{\rm CE}\right)\circ id \left(\t^j_k\right) = -\sum_a \t^j_a\wg\t^a_k
\end{equation}
and
\begin{align}
\begin{split}
& \left(d_{\rm K}\right)\circ id \left(\t^j_k\right) = \sum_{p,q} \iota(Y_p^q)\left(\t^j_k\right) \ot \one\lt {\widehat{Y}}^p_q \\
& - \t^j\wg\t_k - \sum_a \d^j_k\t^a\wg\t_a,
\end{split}
\end{align}
we have
\begin{align}
\xymatrix {
 \ar[d]_{d_{\rm CE}+d_{\rm K}} \t^j_k  \ar[r]^{id} &  \t^j_k   \ar[d]^{d_{\rm CE}+d_{\rm K}} \\
 -\sum_a \t^j_a\wg\t^a_k - \t^j\wg\t_k - \sum_a \d^j_k\t^a\wg\t_a   \ar[r]^{\,\,\,\,\,\,\,\,\,\,\,\,\,\,\,\,\,\,\,\,\,\,\,\,\,\,id} &  d_{\rm CE}\left(\t^j_k\right) + d_{\rm K}\left(\t^j_k\right).
}
\end{align}
Finally for the element $w \in W(pg\ell(n),\Cb)$, the commutativity of the above diagram follows from the multiplicativity of the differentials $d_{\rm CE}$ and $d_{\rm K}$.
\end{proof}

\chapter{Future Research}\label{chapter-future-research}


A foliation on a smooth manifold $M$ is an involutive subbundle $L \subseteq TM$ of the tangent bundle $TM$. Then the normal bundle $Q$ of the foliation is given by the short exact sequence
\begin{equation}
\xymatrix{
0 \ar[r] & L \ar[r] & TM \ar[r]^p & Q\ar[r]&0
}\end{equation}
of bundles. The codimension $n$ of the foliation is the dimension of the normal bundle $Q$.

\medskip

The normal bundle $Q$ of a foliation is canonically equipped with a partial connection, called the Bott connection, written
\begin{equation}\label{aux-Bott-connection}
\nb_Xs = p[X,Y]
\end{equation}
for any local section $s \in \G(M,Q)$, local vector field $X \in \chi(L)$ and a local vector field $Y \in \chi(M)$ such that $p(Y) = s$. The partial connection is flat if
\begin{equation}
\nb_X\nb_Y - \nb_Y\nb_X - \nb_{[X,Y]} = 0.
\end{equation}
Then it is possible to carry the discussion to the level of principal bundles. A principal bundle $P \to M$ with structure group $G$ is foliated if $P$ is equipped with a flat partial connection \cite{KambTond71,Moli71}. For instance, if $L = TM$, then this definition of foliated bundle reduces to that of a flat bundle. If, on the other hand $L = 0$, then we have the ordinary principal $G$-bundle.

\medskip

Next, an adapted connection on the normal bundle $Q$ is a covariant derivative that reduces to \eqref{aux-Bott-connection} on $\chi(L)$. From the principal bundle point of view, an adapted connection on $Q$ corresponds to a connection 1-form on the frame bundle $F(Q)$. Let $L \subseteq TM$ be a $G$-foliation and $\om$ be the corresponding connection 1-form on the $G$-reduction $P$ of $F(Q)$. Then an adapted connection on $Q$ is called basic if
\begin{equation}
L_{\widetilde{X}}\om = 0
\end{equation}
for any partial horizontal lift to $P$ of $X \in \chi(L)$.

\medskip

Finally, a codimension $n$ foliation $L \subseteq TM$ of a manifold $M$ is called a $G$-foliation if there exists a $G$ reduction $P$ of the frame bundle $F(Q)$, \ie $F(Q) = P \times_G {\rm GL}(n)$, such that the canonical foliated bundle structure of $F(Q)$ arises from a foliated bundle structure of $P$. See \cite[Proposition 1.8]{KambTond78} or \cite{Duch-thesis} for other equivalent characterizations.

\medskip

It is known that submersions form simple examples of foliations. More explicitly, a submersion endows the domain manifold with a foliation whose leaves are the connected components of the inverse images of the points in the target manifold. If the manifold on the range is equipped with a $G$-structure \cite{Cher66}, then such a foliation is an example of a $G$-foliation. Other examples involve (oriented) Riemannian foliations, volume preserving foliations, conformal foliations, spin foliations and almost complex foliations \cite{KambTond78}.

\medskip

The characteristic classes of $G$-foliations are defined by making use of the notion of adapted connection. Let $L \subseteq TM$ be a $G$-foliation of codimension $n$, and $\om$ be the connection 1-form on the $G$-reduction $P$ corresponding to an adapted connection on $Q$. Then the Weil homomorphism defines a $G$-DG-homomorphism
\begin{equation}
k(\om):W(\Fg) \to \Om(P)
\end{equation}
from the Weil algebra of the Lie algebra $\Fg$ of $G$ to the de Rham complex of $P$. This homomorphism vanishes on the differential ideal $I:=S^{n+1}(\Fg^\ast)$ \cite[Corollary 4.27]{KambTond-book}. As a result, the Weil homomorphism induces
\begin{equation}
k(\om):W(\Fg)\nsb{n} \to \Om(P), \qquad W(\Fg)\nsb{n}:= W(\Fg)/I,
\end{equation}
and hence a characteristic homomorphism
\begin{equation}
k_\ast:H(W(\Fg)\nsb{n}) \to H_{\rm dR}(P).
\end{equation}
Moreover, in case of the normal bundle of the foliation to be trivialized, we have
\begin{equation}\label{aux-characteristic-homomorphism}
k_\ast:H(W(\Fg)\nsb{n}) \to H_{\rm dR}(M).
\end{equation}
With this at hand, we seek a noncommutative analogue of this construction. In other words, we want to replace the de Rham cohomology of the manifold $M$ with the cyclic cohomology of the algebra $\Ac$ and the characteristic homomorphism \eqref{aux-characteristic-homomorphism} with a Hopf-cyclic characteristic homomorphism.

\medskip

In view of the isomorphisms
\begin{equation}
H(W(g\ell(n))\nsb{n}) \cong H_{\rm GF}(\Fa_n) \cong HP(\Hc_n),
\end{equation}
between the $n$-truncated Weil algebra of $g\ell(n)$, the Gelfand-Fuks cohomology $\Fa_n$, the Lie algebra of formal vector fields, and the periodic cyclic cohomology of $\Hc_n$, the Connes-Moscovici Hopf algebra, a noncommutative characteristic homomorphism is given by Connes and Moscovici in \cite{ConnMosc98}:
\begin{align}\label{aux-Connes-Moscovici-characteristic-map}
\begin{split}
& \chi_\tau:C^n_{\Hc_1}(\Hc_1) \to C^n(\Ac), \\
& \chi_\tau(h^1 \odots h^n)(a_0 \odots a_n) := \tau(a_0h^1(a_1)\ldots h^n(a_n)),
\end{split}
\end{align}
where $\Ac = C^\infty_c(F(\Rb^n)) \rtimes \Diff(\Rb^n)$. Using this characteristic homomorphism, the characteristic classes of transversely orientable foliations of codimension 1 are realized in the cyclic cohomology of the algebra $\Ac$, see Connes and Moscovici \cite{ConnMosc}. The classes are the transverse fundamental class $[{\rm TF}] \in HC^2(\Ac)$,
\begin{equation}
{\rm TF}(f^0U^\ast_{\psi_0},f^1U^\ast_{\psi_1},f^2U^\ast_{\psi_2}) = \left\{\begin{array}{cc}
                                                                        \int_{F\Rb}f^0\widetilde{\psi_0}^\ast(df^1)\widetilde{\psi_0}^\ast\widetilde{\psi_1}^\ast(df^2), & \psi_2\psi_1\psi_0 = \Id \\
                                                                        0, & \psi_2\psi_1\psi_0 \neq \Id,
                                                                      \end{array}\right.
\end{equation}
and the Godbillon-Vey class $[{\rm GV}] \in HC^1(\Ac)$,
\begin{equation}
{\rm GV}(f^0U^\ast_{\psi_0},f^1U^\ast_{\psi_1}) = \left\{\begin{array}{cc}
                                                                        \int_{F\Rb}f^0\widetilde{\psi_0}^\ast(f^1)y\frac{\psi_0''(x)}{\psi_0'(x)}\frac{dx\wg dy}{y^2}, & \psi_1\psi_0 = \Id \\
                                                                        0, & \psi_1\psi_0 \neq \Id.
                                                                      \end{array}\right.
\end{equation}

We start with the truncated Weil complex $W(\Fg)\nsb{n}$ viewed as a Hopf-cyclic complex. Indeed, we have
\begin{equation}
W(\Fg)\nsb{n} = W(\Fg,V_n), \qquad V_n = S(\Fg^\ast)\nsb{n},
\end{equation}
where $V_n$ is a (unimodular) SAYD module over the Lie algebra $\Fg$. Moreover, by the Poncar\'e isomorphism between the perturbed Koszul complex $W(\Fg,V_n)$ and the complex $C(\Fg,V_n)$ of \eqref{aux-complex-cyclic-Lie-algebra}, as well as Theorem \ref{g-U(g) spectral sequence}, we can see the truncated Weil complex as a Hopf-cyclic complex.

\medskip

Once this is understood, we can then look for an algebra $\Ac$ on which $\Fg$ acts as derivations. Hence
the cup product construction \cite{Rang08}
\begin{align}\label{aux-cup-product-Rangipour}
\begin{split}
& \cup:HC^p(U(\Fg),V_n) \ot HC^q_{U(\Fg)}(\Ac,V_n) \to HC^{p+q}(\Ac) \\
& (x \cup \vp)(a_0 \odots a_{p+q}) := \\
& \sum_{\s \in Sh(p,q)} (-1)^\s \p_{\wbar{\s}(p)} \ldots \p_{\wbar{\s}(1)}\vp(\langle \p_{\wbar{\s}(p+q)} \ldots \p_{\wbar{\s}(p+1)} x, a_0 \odots a_{p+q} \rangle).
\end{split}
\end{align}
produces a characteristic homomorphism
\begin{equation}
\xymatrix{
H(W(\Fg,V)) \ar[r]^{\FD_P} & H(C(\Fg,V)) \ar[r]^\cong& HC(U(\Fg),V) \ar[r]^{\;\;\;\;\;\chi_\vp} & HC(\Ac),
}
\end{equation}
as a composition of the Poincar\'e isomorphism, the isomorphism given by Theorem \ref{g-U(g) spectral sequence}, and
\begin{align}
\begin{split}
\chi_\vp: HC^p(U(\Fg),V) \to HC^{p+q}(\Ac), \qquad [\vp] \in HC^q_{U(\Fg)}(\Ac,V_n).
\end{split}
\end{align}

We illustrate this construction in codimension 1. Letting $\Ac = C^\infty_c(F(\Rb))\rtimes\Diff(\Rb)$, $\Uc = U(g\ell(1))$ and $V_1 = S(g\ell(1)^\ast)\nsb{1}$, a 2-dimensional SAYD module over the Hopf algebra $\Uc$, we define
\begin{equation}
\vp :V_1 \ot \Ac^{\ot\;2} \to \Cb, \quad \vp((\a\one + \b\t) \ot a_0 \ot a_1) = \a\tau(a_0X(a_1)) - \b\tau(a_0\d_1(a_1)).
\end{equation}
It is then straightforward to check that $[\vp] \in HC^1_\Uc(\Ac,V_1)$ is a nontrivial $\Uc$-equivariant cyclic class.

\medskip

On the other hand, the cohomology $H(W(g\ell(1))\nsb{1})$ of the truncated Weil complex is represented by the cocycles $[\one] \in H^0(W(g\ell(1))\nsb{1})$ and $[R \ot Y] \in H^1(W(g\ell(1))\nsb{1})$. By Poincar\'e duality, the cohomology $H(C(g\ell(1), V_1))$ is generated by the cocycles $[\one \ot Y]\in H^1(C(g\ell(1), V_1))$ and $[R]\in H^0(C(g\ell(1), V_1))$. Moreover, by Theorem \ref{g-U(g) spectral sequence} we can find the cocycle representatives of the cohomology $HC(\Uc,V_1)$. They are $[R] \in HC^0(\Uc,V_1)$ and $[\one \ot Y + \frac{1}{2}R \ot Y^2] \in HC^1(\Uc,V_1)$.

\medskip

Next we use the cup product construction \eqref{aux-cup-product-Rangipour}. Cupping with the cyclic 1-cocycle $\vp \in HC^1_\Uc(\Ac,V_1)$ yields a characteristic map
\begin{equation}
\chi_\vp:C_\Uc^p(\Uc,V) \to C^{p+1}(\Ac), \qquad p = 0,1.
\end{equation}
Then we have
\begin{equation}
[\chi_\vp(\one \ot Y + \frac{1}{2}\t \ot Y^2)] = [{\rm TF}],
\end{equation}
and
\begin{equation}
[\chi_\vp(\t)] = [{\rm GV}].
\end{equation}
That is, we obtain the transverse fundamental class and the Godbillon-Vey class.

\medskip

In future, we would like to investigate a conceptual understanding on the equivariant cyclic cohomology $HC_\Uc(A,V)$ of an algebra $A$ with coefficients. Recovering the results of Connes and Moscovici in \cite{ConnMosc} by obtaining the fundamental class and the Godbillon-Vey class, proves that our method results in a computationally powerful noncommutative  characteristic homomorphism.

\bibliographystyle{amsplain}
\bibliography{Rangipour-Sutlu-References}{}

\end{document}